\documentclass[10pt]{article}
\usepackage{amssymb}
\usepackage{mathrsfs}
\usepackage{iftex}
\usepackage{dsfont} 
\usepackage{graphicx}
\usepackage{enumerate}
\usepackage{array}
\usepackage{url} 
\usepackage{amsmath}
\usepackage{tcolorbox}
\usepackage{subcaption}
\usepackage{bm}
\usepackage{hyperref}
\usepackage{arydshln}
\usepackage{enumitem}
\usepackage{caption}
\usepackage{comment}

\newcommand{\Fn}{{\rm F}}
\newcommand{\cOp}{\mathcal{O}_p}
\newcommand{\pwk}{\pi_{\mathsf{w}}}
\let\hat\widehat
\let\tilde\widetilde

\newcommand{\ba}{\bm{a}}
\newcommand{\bb}{\bm{b}}

\newcommand{\bd}{\bm{d}}
\newcommand{\be}{\bm{e}}

\newcommand{\bg}{\bm{g}}

\newcommand{\br}{\bm{r}}
\newcommand{\bs}{\bm{s}}

\newcommand{\bu}{\bm{u}}
\newcommand{\bv}{\bm{v}}

\newcommand{\bx}{\bm{x}}
\newcommand{\by}{\bm{y}}
\newcommand{\bz}{\bm{z}}

\newcommand{\Ib}{\mathbf{I}}

\newcommand{\Lb}{\mathbf{L}}

\newcommand{\Wb}{\mathbf{W}}

\newcommand{\bA}{\bm{A}}
\newcommand{\bB}{\bm{B}}
\newcommand{\bC}{\bm{C}}
\newcommand{\bD}{\bm{D}}
\newcommand{\bE}{\bm{E}}

\newcommand{\bG}{\bm{G}}
\newcommand{\bH}{\bm{H}}
\newcommand{\bI}{\bm{I}}
\newcommand{\bJ}{\bm{J}}

\newcommand{\bL}{\bm{L}}
\newcommand{\bM}{\bm{M}}

\newcommand{\bP}{\bm{P}}
\newcommand{\bQ}{\bm{Q}}
\newcommand{\bR}{\bm{R}}
\newcommand{\bS}{\bm{S}}
\newcommand{\bT}{\bm{T}}
\newcommand{\bU}{\bm{U}}
\newcommand{\bV}{\bm{V}}

\newcommand{\bX}{\bm{X}}
\newcommand{\bY}{\bm{Y}}
\newcommand{\bZ}{\bm{Z}}

\newcommand{\cA}{\mathcal{A}}
\newcommand{\cB}{\mathcal{B}}

\newcommand{\cE}{\mathcal{E}}

\newcommand{\cH}{\mathcal{H}}

\newcommand{\cK}{\mathcal{K}}
\newcommand{\cL}{\mathcal{L}}

\newcommand{\cN}{\mathcal{N}}
\newcommand{\cO}{\mathcal{O}}

\newcommand{\cQ}{\mathcal{Q}}
\newcommand{\cR}{\mathcal{R}}
\newcommand{\cS}{{\mathcal{S}}}

\newcommand{\EE}{\mathbb{E}}

\newcommand{\PP}{\mathbb{P}}

\newcommand{\RR}{\mathbb{R}}
\newcommand{\SSS}{\mathbb{S}}

\newcommand{\bbeta}{\bm{\beta}}
\newcommand{\bgamma}{\bm{\gamma}}
\newcommand{\bdelta}{\bm{\delta}}

\newcommand{\btheta}{\bm{\theta}}

\newcommand{\bmu}{\bm{\mu}}

\newcommand{\bphi}{\bm{\phi}}

\newcommand{\bpsi}{\bm{\psi}}

\newcommand{\bGamma}{\bm{\Gamma}}
\newcommand{\bDelta}{\bm{\Delta}}
\newcommand{\bTheta}{\bm{\Theta}}
\newcommand{\bLambda}{\bm{\Lambda}}

\newcommand{\bSigma}{\bm{\Sigma}}

\newcommand{\bPhi}{\bm{\Phi}}
\newcommand{\bPsi}{\bm{\Psi}}
\newcommand{\bOmega}{\bm{\Omega}}

\newcommand{\argmin}{\mathop{\mathrm{argmin}}}
\newcommand{\argmax}{\mathop{\mathrm{argmax}}}

\newcommand{\tr}{\mathop{\mathrm{tr}}}

\newcommand*{\zero}{{\bm 0}}
\newcommand*{\one}{{\bm 1}}

\def\T{{ \intercal }}

\ifx\BlackBox\undefined
\newcommand{\BlackBox}{\rule{1.5ex}{1.5ex}}  
\fi

\ifx\QED\undefined
\def\QED{~\rule[-1pt]{5pt}{5pt}\par\medskip}
\fi

\ifx\proof\undefined
\newenvironment{proof}{\par\noindent{\bf Proof\ }}{\hfill\BlackBox\\[2mm]}
\fi

\newtheorem{assumption}{Assumption}

\ifx\theorem\undefined
\newtheorem{theorem}{Theorem}
\fi
\ifx\example\undefined
\newtheorem{example}{Example}
\fi
\ifx\property\undefined
\newtheorem{property}{Property}
\fi
\ifx\lemma\undefined
\newtheorem{lemma}{Lemma}
\fi
\ifx\proposition\undefined
\newtheorem{proposition}{Proposition}
\fi
\ifx\remark\undefined
\newtheorem{remark}{Remark}
\fi
\ifx\corollary\undefined
\newtheorem{corollary}{Corollary}
\fi
\ifx\definition\undefined
\newtheorem{definition}{Definition}
\fi
\ifx\conjecture\undefined
\newtheorem{conjecture}{Conjecture}
\fi
\ifx\fact\undefined
\newtheorem{fact}{Fact}
\fi
\ifx\claim\undefined
\newtheorem{claim}{Claim}
\fi
\ifx\cond\undefined
\newtheorem{cond}{Condition}
\fi

\ifPDFTeX
  \usepackage[T1]{fontenc}
  \usepackage[utf8]{inputenc}
  \usepackage{textcomp} 
\else 
  \usepackage{unicode-math}
  \defaultfontfeatures{Scale=MatchLowercase}
  \defaultfontfeatures[\rmfamily]{Ligatures=TeX,Scale=1}
\fi
\usepackage{lmodern}
\ifPDFTeX\else  
\fi
\IfFileExists{upquote.sty}{\usepackage{upquote}}{}
\IfFileExists{microtype.sty}{
  \usepackage[]{microtype}
  \UseMicrotypeSet[protrusion]{basicmath} 
}{}
\makeatletter
\@ifundefined{KOMAClassName}{
  \IfFileExists{parskip.sty}{%
    \usepackage{parskip}
  }{
    \setlength{\parindent}{2pt}
    \setlength{\parskip}{6pt plus 2pt minus 1pt}}
}{
  \KOMAoptions{parskip=half}}
\makeatother
\usepackage{xcolor}
\makeatletter
\ifx\paragraph\undefined\else
  \let\oldparagraph\paragraph
  \renewcommand{\paragraph}{
    \@ifstar
      \xxxParagraphStar
      \xxxParagraphNoStar
  }
  \newcommand{\xxxParagraphStar}[1]{\oldparagraph*{#1}\mbox{}}
  \newcommand{\xxxParagraphNoStar}[1]{\oldparagraph{#1}\mbox{}}
\fi
\ifx\subparagraph\undefined\else
  \let\oldsubparagraph\subparagraph
  \renewcommand{\subparagraph}{
    \@ifstar
      \xxxSubParagraphStar
      \xxxSubParagraphNoStar
  }
  \newcommand{\xxxSubParagraphStar}[1]{\oldsubparagraph*{#1}\mbox{}}
  \newcommand{\xxxSubParagraphNoStar}[1]{\oldsubparagraph{#1}\mbox{}}
\fi
\makeatother

\usepackage{longtable,booktabs,array}
\usepackage{calc} 
\usepackage{etoolbox}
\makeatletter
\patchcmd\longtable{\par}{\if@noskipsec\mbox{}\fi\par}{}{}
\makeatother
\IfFileExists{footnotehyper.sty}{\usepackage{footnotehyper}}{\usepackage{footnote}}
\makesavenoteenv{longtable}
\usepackage{graphicx}
\makeatletter
\def\maxwidth{\ifdim\Gin@nat@width>\linewidth\linewidth\else\Gin@nat@width\fi}
\def\maxheight{\ifdim\Gin@nat@height>\textheight\textheight\else\Gin@nat@height\fi}
\makeatother
\setkeys{Gin}{width=\maxwidth,height=\maxheight,keepaspectratio}
\makeatletter
\def\fps@figure{htbp}
\makeatother

\makeatletter
\@ifpackageloaded{caption}{}{\usepackage{caption}}
\AtBeginDocument{%
\ifdefined\contentsname
  \renewcommand*\contentsname{Table of contents}
\else
  \newcommand\contentsname{Table of contents}
\fi
\ifdefined\listfigurename
  \renewcommand*\listfigurename{List of Figures}
\else
  \newcommand\listfigurename{List of Figures}
\fi
\ifdefined\listtablename
  \renewcommand*\listtablename{List of Tables}
\else
  \newcommand\listtablename{List of Tables}
\fi
\ifdefined\figurename
  \renewcommand*\figurename{Figure}
\else
  \newcommand\figurename{Figure}
\fi
\ifdefined\tablename
  \renewcommand*\tablename{Table}
\else
  \newcommand\tablename{Table}
\fi
}
\@ifpackageloaded{float}{}{\usepackage{float}}
\floatstyle{ruled}
\@ifundefined{c@chapter}{\newfloat{codelisting}{h}{lop}}{\newfloat{codelisting}{h}{lop}[chapter]}
\floatname{codelisting}{Listing}

\makeatother
\makeatletter
\@ifpackageloaded{caption}{}{\usepackage{caption}}
\@ifpackageloaded{subcaption}{}{\usepackage{subcaption}}
\makeatother

\ifLuaTeX
  \usepackage{selnolig}  
\fi
\usepackage[]{natbib}
\usepackage{bookmark}

\IfFileExists{xurl.sty}{\usepackage{xurl}}{} 
\hypersetup{
    colorlinks=true,
    citecolor=blue,
    linkcolor=blue,
    urlcolor=blue
}

\begin{document}

\addtocontents{toc}{\protect\setcounter{tocdepth}{-1}}

\def\spacingset#1{\renewcommand{\baselinestretch}
{#1}\small\normalsize} \spacingset{1.70}

\title{\bf  Marginal Maximum Likelihood Estimation and Asymptotic Theory for Latent Variable Models in High Dimensions}
  \author{Chengyu Cui and Gongjun Xu\hspace{.2cm}\\
    Department of Statistics, University of Michigan}
    \date{}
  \maketitle
\vspace{0.5em}

\spacingset{1.15}
\begingroup
\setlength{\leftmargini}{5em}%
\begin{abstract}

This work addresses a longstanding gap in the statistical foundations of marginal maximum likelihood estimation for high-dimensional latent variable models.
Marginal maximum likelihood estimation is widely used to fit latent variable models across the social sciences, ecology, and machine learning. Despite its broad use, rigorous asymptotic theory for nonlinear models remains limited when both the sample size and the number of observed variables diverge.
The gap arises largely from the fact that integration over the latent variables creates a nonlinear objective that tightly couples the high-dimensional model parameters.
To address this issue, we first show that the marginal likelihood exhibits multiple nearly flat directions even at the true parameter, in contrast to the behavior in the fixed-dimensional case. Building on this geometric characterization, we develop new techniques to establish consistency and asymptotic normality for the marginal estimator of the high-dimensional parameters. For the latent variables, we provide frequentist and Bayesian uncertainty quantification, proving asymptotic normality of the maximum a posteriori estimator and a Bernstein–von Mises-type result for the plug-in posterior. Together, these results provide rigorous foundations for marginal estimation and latent-variable inference in high-dimensional models.
\end{abstract}
\endgroup
\vspace{0.5em}

\spacingset{1}
{\it Keywords: } {Marginal likelihood; random effects; asymptotic normality; posterior inference.}

\spacingset{1.19}

\section{Introduction}\label{sec_intro}
Latent variable models~\citep{skrondal2004generalized,bartholomew2011latent} provide a powerful statistical framework for representing dependence in high-dimensional observations through a small number of unobserved variables. They specify a data-generating process in which a high-dimensional observation $\bX = (X_1,\ldots,X_q)^\T$ is generated conditionally on a latent variable $\bU\in\RR^K$, with $K$ usually much smaller than $q$. Parameterized by $\btheta = (\btheta_1^\T,\dots,\btheta_q^\T)^\T$, the model takes the form
\begin{equation}\label{eq_glm_link}
    \left\{\begin{aligned}&\bU\sim\PP_U,\\& X_j\mid \bU \;{\sim}\;\; p_j\big\{\cdot\mid \eta(\bU;\btheta_j)\big\}\quad\text{independently for}\quad j=1,\dots,q\end{aligned}\right.
\end{equation}
where $\PP_U$ denotes the latent distribution, $\eta(\bU;\btheta_j)$ is a linear function of $\bU$ parameterized by $\btheta_j$, and $p_j(\cdot\mid\cdot)$ is a specified conditional density or mass function allowed to vary across $j$. This framework has received substantial attention across many fields:
\begin{enumerate}[leftmargin=0.4cm]
\item\; In the social sciences, latent variable models are widely used. In psychological and educational assessment, examples include item response theory and factor analysis, where the latent variables represent unobserved abilities, traits, or attitudes~\citep{holland1990dutch,rizopoulos2008generalized,reckase2009}. Closely related latent-variable formulations also appear in political science, where latent positions summarize legislators' ideological preferences~\citep{treier2008democracy}. For network and social-interaction data, latent variables are used as node positions or memberships to explain network formation and dependence~\citep{hoff2002latent,bickel2013asymptotic}; 
\item\; In ecology, latent variables are useful for capturing dependence among species or sampling sites in joint species distribution  models~\citep{hui2015model,warton2015so,kidzinski2022generalized}. Site-level latent scores provide low-dimensional representations of unmeasured ecological traits that jointly shape species abundances. This representation supports model-based ordination, comparison and clustering of sampling sites, and identification of species characteristic of different ecological regimes;
\item\; In machine learning, latent variable models are popular in learning low-dimensional representations. Notable examples include variational autoencoders~\citep{kingma2013auto,makhzani2015adversarial}, deep generative models~\citep{salakhutdinov2015learning}, and causal representation learning~\citep{scholkopf2021toward,wang2024desiderata}. These models usually posit that high-dimensional observations arise from a comparatively low-dimensional set of latent variables through a complicated, often nonlinear map, and their inference often requires marginalizing over unobserved variables.
\end{enumerate}
Despite differences in their domain-specific formulations, these models share several central features, including a low-dimensional latent representation underlying high-dimensional data, typically nonlinear observation models, rotational nonidentifiability, and objective functions obtained by integrating over the latent variables. The formulation in
\eqref{eq_glm_link} provides a common abstraction of these features.

Marginal likelihood estimation is one of the most popular approaches to fitting latent variable models~\citep{skrondal2004generalized,bartholomew2011latent}. Suppose we observe $n$ independent samples $\bX_1,\dots,\bX_n$ from \eqref{eq_glm_link}, where $\bX_i=(X_{i1},\dots,X_{iq})^\T$, and suppose that the latent distribution $\PP_U$ admits a density $\pi(\bu)$. Conditional on $\{\bU_i\}_{i\le n}$, assume the variables $X_{i1},\dots,X_{iq}$ are independent and follow \eqref{eq_glm_link}. The negative marginal log-likelihood is
\begin{equation}
    \cL(\btheta\mid\bX) \; = \; - \sum_{i=1}^n \log \int_{\RR^K} \prod_{j=1}^qp_j\big\{X_{ij}\mid\eta(\bu;\btheta_j)\big\}\pi(\bu)d\bu.\label{eq_marginal_likelihood}
\end{equation} Here, $\bu$ is the integration variable and the argument of $\eta(\cdot ;\btheta_j)$.
Except in special conjugate or Gaussian cases, the integrals in \eqref{eq_marginal_likelihood} have no closed form. One immediate challenge associated with marginal estimation is thus computational. This has motivated a large literature on numerical and approximate methods, including EM algorithms~\citep{dempster1977maximum,bock1981marginal}, Laplace approximation~\citep{breslow1993approximate,huber2004estimation}, and adaptive quadrature~\citep{moustaki2000generalized,Bianconcini2012AsymptoticPO}. More recent work has developed scalable variational and stochastic-optimization methods for fitting latent variable models in high-dimensional settings~\citep[e.g.][]{cai2010high,hui2017variational,oka2024boosting}.



Despite these developments in computational methods, the theoretical properties of the marginal maximum likelihood estimator remain much less understood.
In fact, in regimes where both $n$ and $q$ diverge, even the consistency of marginal maximum likelihood estimation has not been established, and the estimator's practical performance is supported largely by empirical evidence~\citep{baker2004item,reckase2009}. Related work avoids these integrals by treating the latent variables as fixed parameters, thereby addressing a different estimation problem~\citep{bai2003inferential,chen2019joint,Wang2018Maximum}.
Moreover, uncertainty quantification for the model parameters commonly relies on the Fisher information matrix~\citep{oakes1999direct,bartholomew2011latent}. In high dimensions, evaluating it becomes computationally infeasible because the matrix dimension grows with $q$ and its evaluation involves $n$ integrals without convenient closed forms. Statistical guarantees for latent-variable recovery are also scarce. In psychometrics, it has long been conjectured that, as the number of items $q$ grows, the posterior distribution is approximately normal~\citep{holland1990dutch}, and existing results address it only partially under the impractical assumption that the model parameters are known~\citep{chang1993asymptotic,kornely2022asymptotic}.

Establishing statistical properties for marginal maximum likelihood estimation in high dimensions is challenging for several reasons. First, the objective to minimize is a sum of $n$ integrals as in~\eqref{eq_marginal_likelihood}, each of which generally lacks a closed form when $p_{j}(\cdot\mid\cdot)$ is nonlinear, as is typical in many applications.
Related likelihood integrals arise in generalized linear mixed models (GLMMs)~\citep{breslow1993approximate,mcculloch2008generalized}, for which substantial asymptotic theory has been developed. For instance, \citet{jiang2022usable}, \citet{lyu2024increasing}, \citet{jiang2025asymp} and related work have resolved several longstanding questions concerning likelihood-based inference for GLMMs. By contrast, the problem considered here is quite different.
Unlike GLMMs where the parameter dimension is typically fixed, the parameter dimension in our setting grows with $q$. Moreover, each integral in \eqref{eq_marginal_likelihood} couples the entire parameter vector through the nonlinear product $\prod_{j=1}^qp_j\big\{X_{ij}\mid\eta(\bu;\btheta_j)\big\}$, which prevents the objective from being decomposed into lower-dimensional estimation problems.

Beyond the difficulty of integral approximation, a second and more fundamental challenge lies in the geometry of the marginal log-likelihood. A standard route to asymptotic theory for an estimator $\hat\btheta$ begins with a Taylor expansion around the true parameter $\btheta^*$:
    \begin{equation}
          \partial_{\btheta}\cL(\hat\btheta\mid\bX) \; - \; \partial_{\btheta}\cL(\btheta^*\mid\bX) \; \approx\; \partial^2_{\btheta\btheta}\cL(\btheta^*\mid \bX)\ (\hat\btheta - \btheta^*)  ,\label{eq_expand_informal}
    \end{equation}
    and then inverts the Hessian matrix $\partial^2_{\btheta\btheta}\cL(\btheta^*|\bX) $ to control $\hat\btheta - \btheta^*$~\citep{van2000asymptotic}. In our setting, this route breaks down. First, a Laplace approximation of the Hessian yields a leading block-diagonal term, together with additional terms that collectively form a large dense matrix whose operator norm is comparable to that of the leading term. As a result, 
    the dense and non-negligible off-diagonal terms leave no mathematically convenient way to invert the full Hessian. Second, as we show in Lemma~\ref{lemma_convexity} below, the Hessian has an ill-conditioned spectrum as $q$ grows. Its inverse therefore diverges in operator norm, preventing uniform control via direct Taylor expansion.



\subsection{Our Contributions}

To address these challenges, we develop new tools for analyzing the high-dimensional marginal likelihood and establish a comprehensive asymptotic theory for marginal maximum likelihood estimation. Our contributions are summarized as follows.


Our first step in overcoming these difficulties is to characterize the geometry of the marginal log-likelihood as the parameter dimension diverges. We identify the directions along which the curvature vanishes as the dimension grows, even at the true parameter; see Lemma~\ref{lemma_convexity} below. This implies that the marginal log-likelihood is not uniformly strongly convex in the high-dimensional regime, a feature fundamentally different from the low-dimensional settings considered in much of the existing literature~\citep{huber2004estimation,ma2010explicit}. At the same time, the problematic directions identified by this analysis provide crucial geometric insight that guides our subsequent arguments establishing consistency and asymptotic normality.



     We establish consistency of the marginal maximum likelihood estimator via a construction of an oracle auxiliary estimator. This estimator borrows a carefully controlled amount of information from the true parameter, which is designed based on the flat directions identified in the previous analysis. This auxiliary estimator restores well-behaved curvature along the path to the true parameter while remaining tightly aligned with the original marginal estimator. We thereby first establish consistency of the auxiliary estimator and then show that it is sufficiently close to the marginal maximum likelihood estimator. This proximity, in turn, implies consistency of the latter.
    
   We further derive asymptotic normality of the marginal maximum likelihood estimator. The major obstacle is that the marginal Hessian $\partial_{\btheta}^2\cL(\btheta\mid\bX)$ is dense, ill-conditioned, and does not admit a tractable inverse. We show that, to leading order, the dense component can be expressed as a Schur-complement term arising from the Hessian of the conditional likelihood, with the latent variables treated as parameters as well.
   This algebraic representation makes clear that the dense off-diagonal structure arises from the complex coupling between the model parameters and latent variables within each integral. 
   This also suggests a solution that rather than expanding the marginal score directly, we expand the scores of the conditional likelihood for the model parameters and latent variables together. This yields the estimator’s asymptotic distribution. Moreover, its asymptotic covariance matrix coincides with the inverse Fisher information in the oracle setting where the latent variables are observed, thereby establishing asymptotic efficiency.
   
   For latent-variable recovery, we establish statistical guarantees from frequentist and Bayesian perspectives. From the frequentist perspective, we study the mode of the plug-in posterior, constructed using the marginal maximum likelihood estimator, as a point estimator of the realized latent variable and establish its asymptotic normality. From the Bayesian perspective, we prove a Bernstein–von Mises-type result for the entire plug-in posterior. Together, these results provide the first theoretical guarantee for these two commonly used approaches to recovering the latent variables following marginal estimation  under practical assumptions.

The rest of the paper is organized as follows. Section~\ref{sec_setup} introduces the setup for the marginal maximum likelihood estimation. Section~\ref{sec_res} presents the main theoretical results, and Section~\ref{sec_proof} outlines the proof strategies. Sections~\ref{sec_simu} and~\ref{sec_data} present simulation studies and a data analysis, respectively, while Section~\ref{sec_discussion} concludes. The Supplementary Material contains additional numerical studies and complete proofs.

\noindent {\bf Notation.} We first define notation. For any positive integer $N$, let $[N]=\{1,\ldots,N\}$. For $a,b\in\RR$, let $a\vee b=\max(a,b)$ and $a\wedge b=\min(a,b)$. For sequences $a_n,b_n$, $a_n\asymp b_n$ means that their ratio is bounded above and away from zero, and $a_n\ll b_n$ means $a_n/b_n\to0$.
For a vector $\bx$, $\|\bx\|$ and $\|\bx\|_\infty$ denote its Euclidean and maximum norms. For a matrix $\bM$, $\|\bM\|$ and $\|\bM\|_{\Fn}$ denote its spectral and Frobenius norms, respectively. We write $\mathrm{col}(\bM)$ for its column space. For a square matrix $\bM$, write $\lambda_{\min}(\bM)$ and $\lambda_{\max}(\bM)$ for its smallest and largest eigenvalues. Write $\zero_d$ for the $d$-vectors of zeros and $\bI_d$ for the $d\times d$ identity matrix. Unless explicitly qualified, $\cOp(\cdot)$, $o_p(\cdot)$, and convergence in distribution refer to the joint law of the latent variables and responses.

\section{Marginal Maximum Likelihood Estimation}\label{sec_setup}

Throughout, $\bU_i\in\RR^K$ denotes the latent random variable associated with subject $i$, $\bu_i^*$ denotes its realized value, and $\bu=(u_1,\ldots,u_K)^\T\in\RR^K$ denotes a generic argument in a function or integral. We assume that $K$ is fixed and known. Write the linear predictor in \eqref{eq_glm_link} as
{\begin{equation*}
    \eta(\bu;\btheta_j) \; = \; \beta_j+\sum_{r=1}^K\gamma_{jr}u_r \; = \; \beta_j+\bgamma_j^\T\bu,
\end{equation*}}
where $\bgamma_j=(\gamma_{j1},\ldots,\gamma_{jK})^\T$, and $\btheta_j=(\beta_j,\bgamma_j^\T)^\T$. We collect the parameters as $\bbeta=(\beta_1,\ldots,\beta_q)^\T$, $\bGamma=(\bgamma_1,\ldots,\bgamma_q)^\T$, $\bTheta=(\bbeta,\bGamma)=(\btheta_1,\ldots,\btheta_q)^\T$, and $\btheta=(\btheta_1^\T,\ldots,\btheta_q^\T)^\T$.

For subject $i$, define the conditional log-likelihood at a variable value $\bu\in\RR^K$ as 
\begin{equation*}\ell_i(\bu;\btheta) \; = \; \sum_{j=1}^q\log p_{j}\{X_{ij}\mid\eta(\bu;\btheta_j)\}.\end{equation*} We suppress the dependence of $\ell_i(\cdot\, ;\, \cdot)$ on the observed data $\bX_i$ when it is clear from the context and use the subscript $i$ to emphasize this subject-specific conditional likelihood. 
Constructing the marginal likelihood requires specifying the latent distribution $\PP_{U}$, which is generally unknown in practice.
As noted by~\citet{ma2010explicit}, nonparametric estimators of $\PP_U$ converge slowly, and $\PP_U$ itself is not the primary object of inference. Marginal estimation therefore usually proceeds using a working prior, most commonly taken to be the standard Gaussian density~\citep{skrondal2004generalized,mcculloch2008generalized,bartholomew2011latent}.
We adopt this convention, denote the standard Gaussian density by $\pwk(\cdot)$, and consider the working marginal log-likelihood
\begin{equation*}
    \cL(\btheta) \; = \; -\sum_{i=1}^n\log\int_{\bu\in\RR^K}\exp\big\{\ell_i(\bu;\btheta)\big\}\pwk(\bu)d\bu .
\end{equation*}
Our theory leaves the true latent distribution $\PP_U$ unspecified, apart from the mild regularity conditions in Assumption~\ref{assump_standard_id} below. This relaxes the assumption adopted in much of the random-effects literature that restricts $\PP_U$ to be Gaussian~\citep[e.g.][]{bock1981marginal,huber2004estimation,jiang2022usable}. Our results thereby provide theoretical support for the practice of using a Gaussian working prior in the large-scale settings.
For notational simplicity, we continue to write $\cL(\btheta)$ for the working marginal likelihood.
We study the marginal estimator given by
\begin{align}
    \hat\btheta \;\in\; \mathop{\argmin}_{\btheta\in\Xi}\;\cL(\btheta),\label{eq_mmle_IC1}
\end{align}
where, with $\bu_i^{\star}(\btheta)=\mathop{\argmax}_{\bu\in\RR^K}\ell_i(\bu;\btheta)$, the feasible set $\Xi$ is defined as
\begin{align*}
    \Xi\; =\; \left\{\btheta\,:\,\|\btheta\|_{\infty}\le C_{\theta},\quad q^{-1/2}\sigma_{K+1}(\bTheta)\ge c_{\theta},\quad\max_{i\in[n]}\big\|\bu_i^{\star}(\btheta)\big\| \le C_u\sqrt{\log n}\,\right\},
\end{align*} for some constants $C_{\theta},c_{\theta},C_u$.
Here, $\sigma_{K+1}(\bTheta)$ denotes the $(K+1)$th largest singular value of $\bTheta$; $\bu_i^{\star}(\btheta)$ is the conditional mode around which $\exp\{\ell_i(\bu;\btheta)\}$ concentrates its mass in each integral of the marginal likelihood. 

The constraints defining $\Xi$ are imposed to exclude pathological parameters.
In nonlinear latent variable models, the log-likelihood can become nearly flat in the tails. For example, under the logistic link $p_j(X_{ij}\mid \eta) = \exp(X_{ij}\eta)/\{1+\exp(\eta)\}$, we have $-\partial_{\eta}^2\log p_j(X_{ij}\mid \eta) = \exp(\eta)/\{1+\exp(\eta)\}^2$, which decays exponentially as $|\eta|\to\infty$. The constraints $\|\btheta\|_\infty\le C_{\theta}$ and $\max_{i\in[n]}\|\bu_i^\star(\btheta)\|\le C_u\sqrt{\log n}$ therefore keep the predictors $\eta(\bu;\btheta_j)$ within a controlled region in which the objective can have informative curvature.
In addition, the constraint $\sigma_{K+1}(\bTheta)\ge c_\theta\sqrt q$ rules out nearly rank-deficient parameters. For theoretical purposes, the constants $C_{\theta}$, $c_{\theta}$, and $C_u$ may be chosen to satisfy the thresholds imposed by the regularity conditions; see Remark~\ref{remark_explain_xi} below. In practice, these constraints need not be explicitly enforced or, when imposed, can be chosen sufficiently loosely to remain inactive for any consistent estimator.
In what follows, we present some technical conditions.


\begin{assumption}\label{assump_standard_id}
The random variables $\bU_i$ are independently and identically distributed with sub-Gaussian norm bounded by a constant $M_u$. Moreover, $\EE[\bU_i] = \zero_K$ and $\bSigma_u^*:=\EE[\bU_i\bU_i{}^\T]=\bI_K$.
\end{assumption}
Assumption~\ref{assump_standard_id} imposes the standard location and scale normalizations for identification. In Section~\ref{sec_res_pract_id} of the Supplementary Material, we further extend our results to the setting with correlated latent variables by relaxing the constraint $\bSigma_u^* = \bI_K$. The sub-Gaussian requirement covers the Gaussian latent distributions~\citep{bock1981marginal,huber2004estimation,jiang2022usable} and bounded-support latent distributions~\citep{chen2019joint,chen2020structured}.   Under Assumption~\ref{assump_standard_id}, for some constant $C>0$,
\begin{equation*}
     \PP\big\{\max_{i\in[n]} \|\bU_i\|\le CM_u\sqrt{\log n}\big\} \to 1,\quad\text{ as }\quad n\to\infty.
\end{equation*}
We denote this event by $\mathcal E_\eta$ and conduct all subsequent analysis on $\mathcal E_\eta$, whose probability tends to one. 
Denote the true parameters by $\btheta^* = (\btheta_1^*{}^{\T},\dots,\btheta_q^*{}^\T)^\T$, where $\btheta_j^* = (\beta_j^*,\bgamma_j^*{}^\T)^\T$ for $j\in[q]$. The following assumption imposes boundedness and nondegeneracy conditions standard in the latent-variable literature~\citep{bai2003inferential,bai2012statistical,chen2019joint,Wang2018Maximum}.

\begin{assumption}
	\label{assumption_psd_covariance} There exists a constant $M_{\theta}$ such that $\|\btheta^*\|_{\infty}\le M_{\theta}$. The limit $\bSigma_{\theta}^* = \lim_{q \rightarrow \infty}q^{-1}\sum_{j=1}^q\btheta_j^*{\btheta_j^*}^\T$ exists and is positive definite.
\end{assumption}
Assumptions~\ref{assump_standard_id} and~\ref{assumption_psd_covariance} together imply that under $\cE_{\eta}$, $\eta_{ij}^* := \eta(\bu_i^*;\btheta_j^*)$ satisfy
\begin{equation*}
    \max_{i\in[n],j\in[q]}|\eta_{ij}^*|\; \le \; C_{\eta}\sqrt{\log n},
\end{equation*}for some constant $C_{\eta}$.
Here $\bu_i^*$ denotes the unobserved realization of $\bU_i$ underlying the observed response vector $\bX_i$. The superscript $^*$ marks either a realized latent value, as in $\bu_i^*$, or the true value of a model parameter, as in $\btheta^*$.
Let $\ell_{ij}(\eta) = \log p_j(X_{ij}\mid \eta)$, where we again suppress its dependence on $X_{ij}$ when it is clear from context. We introduce the following condition on $\ell_{ij}(\eta)$.

\begin{assumption}
\label{assumption_smoothness}
{Each $\ell_{ij}(\cdot)$ is concave on $\mathbb R$. There exist constants $c_0,C_0,b_{L},b_{U}>0$ such that, for every $\eta\in\mathbb R$, $-\ell_{ij}''(\eta)\ge c_0\exp(-b_{L}|\eta|) $ and $|\ell_{ij}^{(s)}(\eta)|
\le C_0\exp(b_{U}|\eta|)$ for $s=2,3,4$. Moreover, conditionally on $\{U_i\}_{i=1}^n$, the variables
$\ell_{ij}'(\eta_{ij}^*)$ are independent and mean zero, and have sub-exponential norms bounded by $C_0\exp\{b_{U}|\eta_{ij}^*|/2\}$.}
\end{assumption}

Assumption~\ref{assumption_smoothness} allows the curvature and higher-order derivatives of $\ell_{ij}(\eta)$ to diverge as $|\eta|$ increases, but restricts the rate to be at most exponential. It is satisfied by many commonly used response models, including Gaussian responses, Bernoulli responses with a logistic link, and Poisson responses with a log link, among others. To better characterize the exponential tail, define the deterministic sequences
\begin{equation*}
    B_n \; = \;1\vee C_0\exp(b_U C_\eta\sqrt{\log n}),\qquad b_n \; = \;1\wedge c_0\exp(-b_L C_\eta\sqrt{\log n}).
\end{equation*}
Then, uniformly over $|\eta|\le C_\eta\sqrt{\log n}$, we have $b_n\le -\ell_{ij}''(\eta)\le B_n$, $\max_{s=3,4}|\ell_{ij}^{(s)}(\eta)|\le B_n$, and the conditional sub-exponential norm of $\ell_{ij}'(\eta_{ij}^*)$ is of order $\sqrt{B_n}$. 
We summarize the combined effect of $B_n$ and $b_n$ through the effective noise-to-information ratio $\kappa_n := {\sqrt{B_n}}/{b_n}$. 
Because $\exp\{C\sqrt{\log n}\}=n^{C/\sqrt{\log n}}=n^{o(1)}$, $B_n$, $b_n^{-1}$, and $\kappa_n$ grow more slowly than any fixed polynomial power of $n$. To ensure that these quantities do not alter the polynomial orders of the rates expressed in terms of both $n$ and $q$, we impose the following condition.
\begin{assumption}\label{assump_scaling}
     For fixed constants $c,C>0$, assume $n^c\le q\le n^C$.
\end{assumption}
Assumption~\ref{assump_scaling} implies that $B_n,\; b_n^{-1},\;\kappa_n = o\{ (n\wedge q)^a\}$ for every fixed $a>0$. 
This assumption specifies the high-dimensional regime considered in this paper, allowing a broad range of polynomial growth rates. Notably, it overlaps the region $q=o(n^{1/3})$, a sufficient growth condition under which the estimation error is asymptotically normal~\citep{portnoy1985asymptotic,he2000parameters}.



   \begin{remark}\label{remark_explain_xi}
       For theoretical analysis, we need the constraints defining the feasible set $\Xi$ to satisfy
    \[C_{\theta}>M_{\theta},\quad 
    c_{\theta}<\sqrt{\lambda_{\min}(\bSigma_{\theta}^*)},\quad
   \text{ and }\quad C_u>CM_u .\]
In practice, these constants may be taken sufficiently loose. 
We emphasize that, to the best of our knowledge, comparable theoretical guarantees for marginal likelihood estimation in the considered high-dimensional setup are not available. In light of the preceding discussion, the proposed feasible set provides a reasonable parameter space and, as established below, defines a region on which asymptotic theory can be developed for the resulting estimator. 
\end{remark}

We present two useful preliminary results. The first shows that $\Xi$ contains a local neighborhood of $\btheta^*$, and the second characterizes the local curvature of the marginal log-likelihood function.

\medskip
\noindent\textbf{Localization. } Define a local region around the true parameter $\btheta^*$ by
\begin{equation}\label{eq_define_loca_region}\cB_{\epsilon}(\btheta^*) \; :=\; \big\{\btheta:\|\btheta\|_{\infty}\le C_{\theta},\;q^{-1/2}\|\btheta-\btheta^*\|\le \epsilon\, \big\}.\end{equation} 
The next result shows that, for sufficiently small $\epsilon$, every $\btheta\in\cB_{\epsilon}(\btheta^*)$ also belongs to $\Xi$ with probability tending to one.
   \begin{lemma}\label{lemma_neighbour_control}
    Under Assumptions~\ref{assump_standard_id}--\ref{assumption_smoothness} and $\kappa_n^2\log n = o(q)$, suppose that $\btheta$ satisfies $\|\btheta\|_{\infty}\le C_{\theta}$ and $q^{-1/2}\|\btheta-\btheta^*\|\le \epsilon$ for some $\epsilon\ll b_n/(B_n\sqrt{\log n})$. Then $\bu_i^{\star}(\btheta)= \arg\max_{\bu\in\RR^K}\ell_i(\bu;\btheta)$ exists for each $i\in[n]$ and
    \begin{align*}
        \Big(n^{-1}\sum_{i=1}^n\big\|\bu_i^{\star}(\btheta) - \bu_i^*\big\|^2\Big)^{1/2} &= \cOp\big(\epsilon B_n/b_n + \kappa_n/\sqrt{q}\big), \\
        \max_{i\in[n]}\big\|\bu_i^{\star}(\btheta) - \bu_i^*\big\|_{\infty} &= \cOp\big(\sqrt{\log n}\{\epsilon B_n/b_n + \kappa_n/\sqrt{q}\}\big).
    \end{align*}
    Thus, $\btheta\in\Xi$ with probability tending to one as $n,q\to\infty$. 
\end{lemma}
Lemma~\ref{lemma_neighbour_control} implies that, for sufficiently small $\epsilon$, within the local neighborhood $\cB_{\epsilon}(\btheta^*)$, the third constraint defining $\Xi$ is inactive with probability tending to one. Thus, if the estimator is shown to lie within this neighborhood, the constraint can be omitted in the subsequent analysis. Indeed, the role of this constraint is precisely to enable this localization. Restriction of this form is also essential because, outside the controlled region, the likelihood curvature may decay rapidly, potentially leading to many candidate solutions with little informative local curvature.

\medskip
\noindent\textbf{Local Geometry.}
Let $\cH_{\theta\theta}(\btheta) := \partial_{\btheta}^2 \cL(\btheta)$ be the Hessian matrix of the negative marginal log-likelihood.
The following lemma shows that $n^{-1}\cH_{\theta\theta}(\btheta)$ has vanishing eigenvalues at $\btheta^*$, and characterizes the associated weak-curvature directions. Let $\mathrm{vec}(\bM)$ denote the vector obtained by stacking the columns of $\bM$.
\begin{lemma}\label{lemma_convexity}
    Suppose Assumptions~\ref{assump_standard_id}--\ref{assump_scaling} hold. Then there exist some constants $\gamma,c>0$ such that, for any $\btheta\in\cB_{\epsilon}(\btheta^*)$ with $\epsilon\ll (\log n)^{-1}\wedge b_n^3/(B_n^3\sqrt{\log n})$ and any $\bV  \in\RR^{(K+1)\times q}$ with $\bv = \mathrm{vec}(\bV)$, it holds that
    \begin{equation}\label{eq_hessian_lower_bound}
        \PP\left\{n^{-1}\bv^\T\cH_{\theta\theta}(\btheta)\bv \; \ge \; 
        \gamma b_n\|\bV\|_{\Fn}^2 - cb_nq^{-1}\big\| \bV\bGamma^*  \big\|_{\Fn}^2\right\}\; \to \; 1,\; \text{ as }\; n,q\to\infty.
    \end{equation}Furthermore,  $\big|\lambda_{\min}\{n^{-1}\cH_{\theta\theta}(\btheta^*)\}\big| = \cOp\big(B_n^{5/2}b_n^{-2}/\sqrt{n\wedge q}\big)$.
\end{lemma}
To establish Lemma~\ref{lemma_convexity}, we carefully calculate higher-order terms in the Laplace approximation. In particular, for $j\neq j'$, each entry in the off-diagonal block $\partial_{\btheta_j\btheta_{j'}}^2\cL(\btheta)$ is of constant order, both for the exact Hessian and for its approximation. Collectively, these blocks form a $q(K+1)\times q(K+1)$ dense matrix and are therefore not negligible relative to the diagonal blocks. To derive the lower bound, we show that several terms in the Laplace approximation cancel, yielding an approximate decomposition into a block-diagonal component and a Schur-complement component; see Section~\ref{sec_prove_norma} for some discussion. This decomposition allows us to identify and separate the problematic directions, which leads to the lower bound in \eqref{eq_hessian_lower_bound}.

The result reveals a geometry of the marginal likelihood different from that in low-dimensional problems, where the likelihood is often assumed to be uniformly strongly convex over a local neighborhood~\citep{huber2004estimation,ma2010explicit,Bianconcini2012AsymptoticPO}.
In the high-dimensional setting, despite the dense off-diagonal blocks of the Hessian, the quadratic form admits a lower bound comprising an isotropic curvature term and a correction term depending on $\bV\bGamma^*$, implying that the off-diagonal part does not affect the curvature equally in all directions. In particular, the isotropic curvature can be significantly offset only along directions aligned with $\mathrm{col}(\bGamma^*)$, and the bound on the minimum eigenvalue confirms this loss of curvature. Together, these results characterize directions along which the marginal log-likelihood can become nearly flat and show that the problematic directions span only a low-dimensional subspace of the parameter space.


\section{Main Results}\label{sec_res}

\subsection{Estimation and Inference for Model Parameters}

We study the marginal maximum likelihood estimator $\hat\btheta$ obtained in \eqref{eq_mmle_IC1}. Note that the likelihood is unchanged under any orthogonal transformation of $\bgamma_j$s. Specifically, for any orthogonal matrix $\bQ\in\RR^{K\times K}$ and $\btheta^{\prime} =(\beta_1, (\bQ\bgamma_1)^\T,\dots, \beta_q, (\bQ\bgamma_q)^\T)^\T$, we have $\cL(\btheta)=\cL(\btheta^{\prime})$.
To address this indeterminacy, we adopt the standard Procrustes alignment~\citep{ten1977orthogonal} obtained by{
\begin{equation}\label{eq_q_transf}
    \hat\bQ^{*} \;\in \argmin_{\bQ\in\RR^{K\times K}:\,\bQ\bQ^\T=\bI_K}
    \sum_{j=1}^q\big\|\bQ\hat\bgamma_j - \bgamma_j^*\|^2 \quad= \argmin_{\bQ\in\RR^{K\times K}:\,\bQ\bQ^\T=\bI_K}
    \big\|\hat\bGamma\bQ^\T - \bGamma^*\big\|_{\Fn}^2 ,
\end{equation}
}where $\hat\bGamma = (\hat\bgamma_1,\dots,\hat\bgamma_q)^\T$.
Define $\varepsilon_{nq} = \exp\{\log^{\frac12 + c}(n\vee q)\}$ where $c\in(0,1/2)$ is a fixed constant allowed to be arbitrarily small. This quantity is used to control the tail contributions in the Laplace approximations to the $q(K+1)$-dimensional marginal score. By definition, $\varepsilon_{nq}$ asymptotically dominates both $B_n$ and $b_n^{-1}$, while remaining subpolynomial in $n\wedge q$ under Assumption~\ref{assump_scaling}. Define
\begin{equation}
    r_{\theta,nq} := \sqrt{B_n} + \frac{B_n\kappa_n}{\sqrt{q}} + \frac{B_n\varepsilon_{nq}\sqrt{n\log n}}{q}.
    \label{eq_define_r_theta_nq}
\end{equation}
Lemma~\ref{lemma_first_order_concern} in the Supplementary Material shows that the $\ell_2$ norm of the marginal score at the true parameter is $\cOp(r_{\theta,nq}\sqrt{nq})$, where the third term in $r_{\theta,nq}$ accounts for the error introduced by the Laplace approximation to the score. The following theorem establishes consistency for $\hat\btheta$.
\begin{theorem}[Consistency]\label{thm_consis}
    Under Assumptions~\ref{assump_standard_id}--\ref{assump_scaling}, for $\hat\btheta$ in \eqref{eq_mmle_IC1} and $\hat\bgamma_j^Q := \hat\bQ^*\hat\bgamma_j$, we have
    \begin{align*}
         &q^{-1/2}\Big(\sum_{j=1}^q\Big\{\big\|\hat\bgamma_j^Q - \bgamma_j^*\big\|^2 + \big|\hat\beta_j - \beta_j^*\big|^2\Big\}\Big)^{1/2} = \cOp \; \Big(\frac{\kappa_n^2r_{\theta,nq}}{\sqrt{n}} + \frac{\kappa_n}{\sqrt{q}}\Big),\\
         &\max_{j\in[q]}\left(\big\|\hat\bgamma_j^Q - \bgamma_j^*\big\| +\big|\hat\beta_j - \beta^*_j\big|\right) \; = \; \cOp\Big(\frac{r_{\theta,nq}}{\sqrt{n}}\Big\{\frac{\sqrt{\log q}}{b_n} + \frac{B_n^2}{b_n^3}\Big\} + \frac{\kappa_n}{\sqrt{q}}\Big).
    \end{align*}
\end{theorem}
The terms involving $r_{\theta,nq}$ reflect the contribution of the marginal score, including the Laplace approximation error. The remaining term, $\kappa_n/\sqrt q$, arises from estimating the subject-specific conditional modes and corresponds to the second term in Lemma~\ref{lemma_neighbour_control}. Assumption~\ref{assump_scaling} ensures that both displayed rates converge to zero as $n,q\to\infty$.

It is worth noting that Theorem~\ref{thm_consis} does not conflict with the low-dimensional theory. Our result concerns the high-dimensional regime in which both $n$ and $q$ diverge at comparable polynomial rates. It is not intended to be sharp when $q$ is fixed or grows only subpolynomially in $n$, such as when $q\le \kappa_n^2$. Our proof relies on Laplace approximations to the derivatives of the marginal likelihood with respect to the $q(K+1)$-dimensional parameter, where the residual term scales with $q^{-1}$. Thus, for small $q$, this approximation may be too coarse, making the resulting convergence rates conservative. In the finite-$q$ case, the model reduces to a conventional finite-dimensional estimation problem, for which standard $M$-estimation theory applies under additional regularity conditions~\citep{van2000asymptotic}. %


The primary difficulty in establishing Theorem~\ref{thm_consis} is to control the curvature of the marginal likelihood $\cL(\cdot)$ along the path between $\hat\btheta$ and $\btheta^*$. To apply Lemma~\ref{lemma_convexity}, after localizing $\hat\btheta$ within the neighborhood $\cB_{\epsilon}(\btheta^*)$, we must show that, for some orthogonal matrix $\bQ$, the rows of  error matrix $\hat\bV := (\hat\bbeta-\bbeta^*,\hat\bGamma\bQ-\bGamma^*)^\T$ are nearly orthogonal to $\mathrm{col}(\bGamma^*)$. 
Directly establishing this directional control for the
high-dimensional estimator can be hard, explained later in Section~\ref{sec_prove_consis}.
To address this difficulty, we construct a novel auxiliary estimator. Roughly speaking, we design an {\it oracle} auxiliary estimator, denoted by $\hat\btheta^*$, with a carefully controlled dependence on the true parameters. This oracle dependence allows us to show 
\begin{equation}\label{eq_desired_direction}
    \big(\hat\bbeta^* - \bbeta^*,\hat\bGamma^* - \bGamma^*\big)^\T\bGamma^*=\zero_{(K+1)\times K},
\end{equation}
where $\hat\bGamma^*$ and $\hat\bbeta^*$ denote the corresponding components of $\hat\btheta^*$.
These constraints thereby restrict the auxiliary estimator to directions in which the isotropic curvature term in~\eqref{eq_hessian_lower_bound} dominates, which ensures nonvanishing curvature in the path between $\btheta^*$ and $\hat\btheta^*$. The consistency results follow by first establishing the $\ell_2$ and $\ell_{\infty}$ consistency of $\hat\btheta^*$ and then showing that $\hat\btheta$ is also close to this oracle auxiliary estimator. The technical details are given in Section~\ref{sec_proof}.

We now establish the asymptotic distribution of $\hat\btheta$. Let $\bar\bmu_u = n^{-1}\sum_{i=1}^n\bu_i^*$ and $\bar\bSigma_u = n^{-1}\sum_{i=1}^n(\bu_i^* - \bar\bmu_u)(\bu_i^* - \bar\bmu_u)^\T$. Under Assumption~\ref{assump_standard_id}, one can verify $\|\bar\bmu_u \| = \cOp(n^{-1/2})$ and $\|\bar\bSigma_u - \bI_K\| = \cOp(n^{-1/2})$. We then define the {\it canonical latent variables and parameters} by
\begin{equation}\label{eq_canp_trans}
    \bu_i^0\ =\ \bar\bSigma_u^{-1/2}(\bu_i^*-\bar\bmu_u),\qquad \bgamma_j^0 \ =\  \bar\bSigma_u^{1/2}\bgamma_j^*, \qquad \beta_j^0 \ =\ \beta_j^*+(\bgamma_j^*)^\T\bar\bmu_u .
\end{equation}Note that the canonical representation yields the same linear predictor $\eta_{ij}^* = (\bgamma_j^0)^\T\bu_i^0 + \beta_j^0 = (\bgamma_j^*)^\T\bu_i^* + \beta_j^*$.
The canonical parameters serve as the centers of the asymptotic distributions below. Under the canonical representation, define Procrustes alignment to $\bgamma_j^0$s as $\hat\bQ^0\in\argmin_{\bQ\bQ^\T=\bI_K}\sum_{j=1}^q\|\bQ\hat\bgamma_j-\bgamma_j^0\|^2$.
\begin{theorem}[Asymptotic Distribution]\label{thm_asymp}
    Suppose Assumptions~\ref{assump_standard_id}--\ref{assump_scaling} hold. Let
\begingroup
\setlength{\abovedisplayskip}{5pt}
\setlength{\belowdisplayskip}{5pt}\begin{equation}
    \label{eq_define_tau_nq} \tau_{nq} = \kappa_n^3B_n^2\Big\{\frac{\kappa_n^6B_nr_{\theta,nq}^2}{n} + \frac{\varepsilon_{nq}}{q}\Big\}. 
\end{equation}\endgroup If $\sqrt{n\log n}\,\tau_{nq}\to 0$ as $n,q\to\infty$, then, for $\hat\btheta$ obtained in \eqref{eq_mmle_IC1}, conditioning on $\{\bU_i\}_{i\le n}$, it holds for each $j\in[q]$ that
    \begin{equation*}
        (\bPhi_{j}^*)^{-1/2} \begingroup
    \renewcommand{\arraystretch}{0.6}\begin{pmatrix}
            \hat\beta_j - \beta_j^0\\ \hat\bQ^0 \hat\bgamma_j - \bgamma_j^0
        \end{pmatrix}\endgroup \; \overset{d}{\to} \; \cN(\zero_{K+1},\bI_{K+1}), \; \text{ as }\;  n,q\to\infty,
    \end{equation*}
    where $\bu_i^{c0} = \big(1,(\bu_i^0)^\T\big)^\T$ and $\bPhi^*_j$ is given by
    \begin{equation*}
        \bPhi^*_j = \Big(-\sum_{i=1}^n\ell_{ij}^{\prime\prime}(\eta_{ij}^*)\bu_i^{c0}(\bu_i^{c0})^\T\Big)^{-1}\Big(\sum_{i=1}^n\big\{\ell_{ij}^{\prime}(\eta_{ij}^*)\big\}^2\bu_i^{c0}(\bu_i^{c0})^\T\Big)\Big(-\sum_{i=1}^n\ell_{ij}^{\prime\prime}(\eta_{ij}^*)\bu_i^{c0}(\bu_i^{c0})^\T\Big)^{-1}.
    \end{equation*}
    Moreover, $\bPhi^*_j $ can be consistently estimated by the plug-in estimator $\hat\bPhi_j$ obtained by replacing $\eta_{ij}^*$ and $\bu_i^0$ with $\hat\eta_{ij} = \hat\bgamma_j^\T\bu_i^{\star}(\hat\btheta) + \hat\beta_j$ and $\hat\bQ^0\bu_i^{\star}(\hat\btheta)$, respectively, where $\bu_i^{\star}(\hat\btheta) = \argmax_{\bu\in\RR^K}\ell_i(\bu;\hat\btheta)$. The asymptotic normality result remains valid with $\bPhi_j^*$ replaced by $\hat\bPhi_j$.
\end{theorem}


The asymptotic distribution is centered at the canonical representative rather than at the original true parameters, reflecting the {\it implicit normalization} induced by the working prior $\pwk(\cdot)$. To see this, define the likelihood mode at marginal estimator as $\bu_i^{\star}(\hat\btheta) = \argmax_{\bu\in\RR^K}\ell_i(\bu; \hat\btheta)$. In the proof of Lemma~\ref{lemma_bound_aug_via_id} in the Supplementary Material, we show  
\begin{equation}
    \big\|\frac{1}{n}\sum_{i=1}^n \bu_i^{\star}(\hat\btheta) \big\| = \cOp(\varepsilon_{nq}q^{-1}), 
    \qquad \big\| \frac{1}{n}\sum_{i=1}^n \bu_i^{\star}(\hat\btheta)\{\bu_i^{\star}(\hat\btheta)\}^\T - \bI_K\big\| =    \cOp(\varepsilon_{nq}q^{-1}).\label{eq_implicit_reg_1}
\end{equation} Thus, the likelihood modes are centered and standardized up to an error of order $\varepsilon_{nq}q^{-1}$. These errors are smaller than those of $\|n^{-1}\sum_{i=1}^n\bu_i^*\|$ and $\|n^{-1}\sum_{i=1}^n\bu_i^*(\bu_i^*)^\T - \bI_K\|$ under the scaling condition in Theorem~\ref{thm_asymp}. This sharper normalization \eqref{eq_implicit_reg_1} explains why the exactly normalized canonical representative, rather than the original representative, serves as the center of the asymptotic distribution. We call it implicit because $\pwk(\cdot)$ is asymptotically negligible relative to the conditional log-likelihood, and its normalizing effect is manifested in \eqref{eq_implicit_reg_1} through the likelihood modes $\bu_i^{\star}(\hat\btheta)$, rather than directly through the marginal estimator.

Importantly, this phenomenon is not an artifact of the specification of $\pwk(\cdot)$. It occurs for any latent distribution satisfying Assumption~\ref{assump_standard_id}, including the standard Gaussian distribution. On the other side, it does not affect the practical utility of the inference theory, because the canonical and original representatives generate exactly the same linear predictors and, for each $j$, their parameters differ by only $\cOp(n^{-1/2})$.

Finally, we note that under correct parametric specification for the observed response, the information identity $\EE[\{\ell_{ij}'(\eta_{ij}^*)\}^2\mid\bU_i] = -\EE[\ell_{ij}''(\eta_{ij}^*)\mid\bU_i]$ implies that the sandwich covariance $\bPhi_j^*$ is asymptotically equivalent to the usual inverse Fisher information in the oracle setting where the latent variables are observed. Thus, marginal maximum likelihood estimation attains the same asymptotic efficiency as an oracle estimator, despite the unknown latent variables. In particular, the implicit normalization does not entail a first-order efficiency loss, supporting the use of the canonical representative as the appropriate target of asymptotic inference.

\begin{remark}
The scaling condition $\sqrt{n\log n}\tau_{nq}\to0$ ensures that the second-order remainder is asymptotically negligible relative to the asymptotically normal term. This condition is satisfied when $n^{1/2+c}\le q$ for arbitrarily small constant $c>0$, which ensures the remainder terms in the asymptotic expansion are negligible. Similar conditions appear in~\citet{bai2003inferential}, \citet{Wang2018Maximum}, and \citet{cui2025identifiability} for joint estimation that treat the latent variables as parameters and therefore do not involve marginalization over latent variables. In the present setting, the remainder is substantially harder to control because it combines errors from integral approximation, plug-in estimation of latent modes, and Hessian perturbations. 
\end{remark}

In practice, identifiability conditions are typically imposed to make estimation and inference interpretable and uniquely defined~\citep{bai2012statistical,cui2025identifiability}. Marginal estimation and our asymptotic theory can be extended to commonly used practical identifiability conditions, accommodating both uncorrelated and correlated latent variables. To streamline the presentation, we defer the details to Section~\ref{sec_res_pract_id} of the Supplementary Material.

\subsection{Latent Variable Recovery}\label{subsec_latent_recover}

We next turn to inference for the latent variables. Although the latent variables are treated as random in the marginal likelihood formulation, inference on their individual realizations is important in practice. We consider two approaches: a frequentist approach based on point estimation and a Bayesian approach based on the posterior distributions~\citep{skrondal2004generalized,mcculloch2008generalized}. Both involve the plug-in posterior obtained from the marginal maximum likelihood estimator. For notational simplicity, we work directly with the optimally aligned estimator $\hat\btheta^Q = \big(\hat\beta_1,\hat\bgamma_1^\T(\hat\bQ^0)^\T,\dots,\hat\beta_q,\hat\bgamma_q^\T(\hat\bQ^0)^\T\big)^\T$. The plug-in posterior is given by
\begin{equation}
    \hat p_i(\bu\mid \bX) \; = \; \frac{\exp\{\ell_i(\bu;\hat\btheta^Q)\}\pwk(\bu)}{\int_{\bv\in\RR^{K}}\exp\{\ell_i(\bv;\hat\btheta^Q)\}\pwk(\bv)d\bv}.\label{eq_poster_estim}
\end{equation}
In what follows, we derive the asymptotic sampling distribution of the posterior mode as a point estimator from the frequentist perspective and establish a Bernstein–von Mises-type approximation to the plug-in posterior from the Bayesian perspective.

\medskip
\noindent{\bfseries Maximum a Posteriori: A Frequentist View. }
We first consider the frequentist view and study the maximum a posteriori (MAP) estimator
\begin{equation}
    \hat\bu_i \; = \; \argmax_{\bu\in\RR^K}\big\{\ell_i(\bu;\hat\btheta^Q) + \log\pwk(\bu)\big\}.\label{eq_define_map}
\end{equation}
Based on Theorem~\ref{thm_asymp}, we obtain the following asymptotic distribution for the point estimate.


\begin{corollary}[Asymptotic Distribution of Point Estimator]\label{thm_asymp_bu}
    Suppose Assumptions~\ref{assump_standard_id}--\ref{assump_scaling} hold. If $\sqrt{q\log n}\tau_{nq}\to 0$ as $n,q\to\infty$ with $\tau_{nq}$ given in \eqref{eq_define_tau_nq}, conditioning on $\mathscr{U}_n = \{\bU_i\}_{i\le n}$, it holds for each $i\in[n]$ that
    \begin{equation*}
        (\bPsi_{i}^*)^{-1/2} \big(\hat\bu_i - \bu_i^0)\; \overset{d}{\to}\; \cN(\zero_K,\bI_K), \; \text{ as } \; n,q\to\infty,
    \end{equation*}
   where $\bu_i^0$ is given in \eqref{eq_canp_trans} and the conditional law is taken with respect to the responses given $\mathscr{U}_n$, and $\bPsi^*_i$ is given by
    \begin{equation*}
        \bPsi^*_i = \Big(-\sum_{j=1}^q\ell_{ij}^{\prime\prime}(\eta_{ij}^*)\bgamma_j^0(\bgamma_j^0)^\T\Big)^{-1}\Big(\sum_{j=1}^q\big\{\ell_{ij}^{\prime}(\eta_{ij}^*)\big\}^2\bgamma_j^0(\bgamma_j^0)^\T\Big)\Big(-\sum_{j=1}^q\ell_{ij}^{\prime\prime}(\eta_{ij}^*)\bgamma_j^0(\bgamma_j^0)^\T\Big)^{-1}.
    \end{equation*}
        Moreover, $\bPsi^*_i $ can be consistently estimated by the plug-in estimator $\hat\bPsi_i$ obtained by replacing $\eta_{ij}^*$ and $\bgamma_j^0$ with $\hat\eta_{ij} = \hat\bgamma_j^\T\bu_i^{\star}(\hat\btheta) + \hat\beta_j$ and $\hat\bQ^0\hat\bgamma_j$, respectively. The asymptotic normality result remains valid with $\bPsi_i^*$ replaced by $\hat\bPsi_i$.
\end{corollary}
 Corollary~\ref{thm_asymp_bu} provides a frequentist asymptotic theory for the point estimator and therefore permits inference on individual latent variables when $n$ and $q$ are large. Similar to Theorem~\ref{thm_asymp}, the canonical latent variables $\{\bu_i^0\}_{i=1}^n$ serve as the centers of the asymptotic distributions, which are measurable with respect to $\mathscr{U}_n$.  Under correct model specification, the sandwich covariance $\bPsi_i^*$ is asymptotically equivalent to the inverse conditional Fisher information $\big[\EE\{-\sum_{j=1}^q\ell_{ij}^{\prime\prime}(\eta_{ij}^*)\bgamma_j^0(\bgamma_j^0)^\T\mid \mathscr{U}_n\}\big]^{-1}$. Corollary~\ref{thm_asymp_bu} also holds when $\hat\bu_i$ is replaced by the likelihood mode $\bu_i^{\star}(\hat\btheta^Q) = \argmax_{\bu\in\RR^K}\ell_i(\bu;\hat\btheta^Q)$, because $\log\pwk(\bu)$ contributes only a negligible term compared with the $\ell_i(\cdot;\cdot)$ in \eqref{eq_define_map}. Both the MAP estimator and the likelihood mode are therefore consistent and asymptotically normal.
 



\medskip
\noindent{\bfseries Posterior Asymptotic Normality: Bayesian View.} We next characterize uncertainty in each latent variable through its estimated posterior distribution.  For $\hat p_i(\cdot\mid \bX)$ defined in \eqref{eq_poster_estim}, we establish the following Bernstein–von Mises-type result.
\begin{corollary}[Posterior Asymptotic Normality]
    \label{thm_bvm_irt}
    Let $\hat\bu_i^{\star} = \argmax_{\bu\in\RR^K}\ell_i(\bu;\hat\btheta^Q)$ and $\bar\bPsi_{i}^* = \big\{\sum_{j=1}^q-\ell_{ij}^{\prime\prime}(\eta_{ij}^*)\bgamma_j^0(\bgamma_j^0)^\T\big\}^{-1}$. Under Assumptions~\ref{assump_standard_id}--\ref{assump_scaling}, for each $i\in[n]$, we have
    \begin{equation*}
          \int_{\RR^{K}} \, \Big|\hat p_i\big(\bu \mid \bX\big) - \pwk\big(\bu\mid \hat{\bu}_i^{\star}, \bar\bPsi^*_i\big) \Big|\, d\bu  \;\overset{p}{\to}\; 0, \; \text{ as } n,q \to \infty,
    \end{equation*}
    where $\pwk\big(\bu\mid \hat{\bu}_i^{\star}, \bar\bPsi^*_i\big)$ denotes normal density with mean $ \hat{\bu}_i^{\star}$ and covariance $\bar\bPsi^*_i$, and the convergence in probability is under the joint distribution of
the latent variables and responses.
\end{corollary}

Corollary~\ref{thm_bvm_irt} shows that the plug-in posterior is asymptotically Gaussian, centered at $\hat\bu_i^{\star}$ with covariance $\bar\bPsi_i^*$. By the preceding frequentist result, $\hat\bu_i^{\star}$ is itself asymptotically normal around $\bu_i^0$, conditional on $\mathscr{U}_n$, and $\bar\bPsi_i^*$ is asymptotically equivalent to the covariance matrix $\bPsi_i^*$ in Corollary~\ref{thm_asymp_bu}. Thus, the posterior uncertainty asymptotically agrees with the frequentist sampling uncertainty of its center. This provides a rigorous justification for the longstanding practice of summarizing latent variables or random effects by their posterior distributions, whose credible sets and covariance matrices are asymptotically valid for frequentist inference on $\bu_i^0$~\citep{bock1981marginal,skrondal2004generalized}.



The Gaussian approximation of the latent-variable posterior with plugged-in item parameter estimation has long been conjectured in the psychometrics literature, remaining unsolved since~\citet{holland1990dutch}.
Existing work has largely considered oracle settings in which the item parameters are known~\citep[e.g.][]{chang1993asymptotic,chang1996asymptotic,kornely2022asymptotic}. 
These results provide an important theoretical basis in applications such as educational assessment, where practitioners report individual ability estimates and quantify their uncertainty via posterior distribution. Their applicability is nevertheless limited by the assumption that the model parameters are known.
Corollary~\ref{thm_bvm_irt} closes this gap by showing that the same Gaussian approximation remains valid for the plug-in posterior based on the marginal maximum likelihood estimator. To the best of our knowledge, this provides one of the first Bernstein–von Mises-type results for posterior distributions of latent variables in a nonlinear high-dimensional regime.


\section{Overview of Proof Techniques}\label{sec_proof}
In this section, we outline the key technical arguments underlying the theoretical results for $\hat\btheta$ presented in Section~\ref{sec_res}. Complete proofs of the results and technical details are provided in the Supplementary Material.


We begin with a critical but often overlooked obstacle: establishing the first-order optimality condition for $\hat\btheta$. This condition is the starting point for asymptotic analysis, including the expansion in \eqref{eq_expand_informal}, but it is nontrivial to obtain because $\hat\btheta$ is obtained by optimizing over the constrained feasible set $\Xi$. 
In what follows, we first explain how to prove the first order condition and then outline the arguments for establishing consistency and asymptotic normality.

\begin{remark}[Comparison with Low-Dimensional Setting]\label{remark_comparison}
    Establishing the first-order optimality condition is a major technical challenge in the considered setting, in contrast to fixed- or low-dimensional problems. In the latter case, one can first establish $\ell_2$-consistency using standard $M$-estimation theory under regularity conditions on the likelihood, such as uniform strong convexity in a local region~\citep{huber2004estimation,ma2010explicit,Bianconcini2012AsymptoticPO}. This consistency can then be used to show that the estimator lies in the interior of the parameter space, yielding the first-order condition. This strategy fails in our high-dimensional setting for two reasons. First, the required strong convexity does not hold (Lemma~\ref{lemma_convexity}) and, more importantly, $\ell_2$-consistency alone does not preclude large entrywise deviations, making the interior point condition nontrivial to prove.
\end{remark}

\subsection{Establishing the First-Order Condition}\label{subsec_first_order}
Remark~\ref{remark_comparison} identifies two reasons the conventional argument for the first-order condition fails. We now elaborate on the subtle difficulty implied by the second reason: establishing a uniform error bound for the estimator. A natural approach is a Taylor expansion of the marginal likelihood along the path connecting $\hat\btheta$ and $\btheta^*$. Lemma~\ref{lemma_convexity} tells that controlling the curvature of $\cL(\cdot)$ along this path requires the error direction $\hat\btheta-\btheta^*$ to have little overlap with $\mathrm{col}(\bGamma^*)$. This directional information, however, is unavailable at this stage, where even the first-order condition itself remains unestablished. 


To resolve this dilemma, we construct an auxiliary estimator that is allowed to borrow a controlled amount of information on the true parameters $\btheta^*$ and thereby exhibits more tractable behavior in its path to $\btheta^*$. Let $\pwk(\cdot\mid\bmu,\bSigma)$ denote the density of a multivariate Gaussian distribution with mean $\bmu\in\RR^{K}$ and covariance matrix $\bSigma\in\RR^{K\times K}$. Treating $\bmu$ and $\bSigma$ as additional arguments, define marginal likelihood function
\begin{equation}
    \cL(\btheta,\bmu,\bSigma)\;  = \; -\sum_{i=1}^n\log\int_{\bu_i\in\RR^K}\exp\big\{\ell_i(\bu_i;\btheta)\big\}\pwk(\bu_i\mid\bmu,\bSigma)d\bu_i .\label{eq_full_marginal_likelihood}
\end{equation}
The auxiliary estimator is then defined by
\begin{equation}
(\hat\btheta^*,\hat\bmu^*,\hat\bSigma^*)\; \;  \in \; \mathop{\arg\min}_{\substack{\btheta\in \cB_{\epsilon}(\btheta^*),(\bmu,\bSigma)
    \in\cK_C}}\cL(\btheta,\bmu,\bSigma) + P^*(\btheta),\label{eq_aux_estimator_illus}
\end{equation}
where $\cB_{\epsilon}(\btheta^*)$ is defined in \eqref{eq_define_loca_region}, $\cK_{C}$ is a regularized parameter set defined in \eqref{eq_define_KC} of the Supplementary Material and
\begin{equation}
    P^* ( \btheta) \; = \; c_Pnq\left\| q^{-1}\big(\bbeta-\bbeta^*,\bGamma-\bGamma^*\big)^\T\bGamma^*\right\|_{\Fn}^2.\label{eq_define_p_star_illus}
\end{equation}
Here $c_P\asymp b_n$ and $\epsilon$ is specified in Section~\ref{supp_sec_prove_lemma_key_first_order} of the Supplementary Material. The superscript $^*$ emphasizes that $P^*(\btheta)$ depends on the true parameters and is thus not computable from the data. This poses no practical limitation because the $P^*(\cdot)$ is introduced solely for analytical purposes. 

In \eqref{eq_aux_estimator_illus}, the auxiliary estimator $\hat\btheta^*$ incorporates the true parameter information via $P^*(\btheta)$.
To understand the influence of $P^*(\cdot)$, we note that $P^*(\btheta) = c_Pnq\left\|q^{-1}\bV\bGamma^*\right\|_{\Fn}^2$ for $\bV=(\bbeta-\bbeta^*,\bGamma-\bGamma^*)^\T$, equaling the correction term in the curvature bound \eqref{eq_hessian_lower_bound} up to scaling. Intuitively, $P^*(\cdot)$ discourages directions for which $\bV\bGamma^*$ is large, thereby ensuring well-behaved curvature along the path between $\hat\btheta^*$ and $\btheta^*$.


The introduction of arguments $\bmu$ and $\bSigma$ in the marginal likelihood contributes $K+K(K+1)/2$ degrees of freedom. Together with the $K(K-1)/2$ degrees of freedom associated with orthogonal transformations of $\bgamma_j$s, this gives $K+\frac{K(K+1)}{2}+\frac{K(K-1)}{2} = K^2 + K$, exactly matching the number of scalar entries in $\bV\bGamma^*\in\RR^{(K+1)\times K}$ measured by $P^*(\cdot)$. These additional degrees of freedom allow us to establish $P^*(\hat\btheta^*)=0$. Because $P^*(\cdot)$ is nonnegative, we can then establish that the auxiliary minimizer of $\cL+P^*$ also minimizes the negative marginal log-likelihood $\cL$ within the local region. In this sense, the auxiliary estimator borrows precisely the amount of true-parameter information needed to regularize the curvature while retaining a bridge to the original estimator.


The first-order optimality conditions can then be established in three steps, which we outline here. The technical details are deferred to Section~\ref{supp_sec_prove_lemma_key_first_order} of the Supplementary Material.
\begin{enumerate}[label=\arabic*.]
\item\; We first establish the theoretical properties of the auxiliary estimator $(\hat\btheta^*,\hat\bmu^*,\hat\bSigma^*)$ defined by the local optimization problem in \eqref{eq_aux_estimator_illus}, including the first-order conditions and uniform error bounds. Notably, $(\hat\btheta^*,\hat\bmu^*,\hat\bSigma^*)$ is a consistent estimator for $(\btheta^*,\zero_K,\bI_K)$.


\item\; We then relate $\hat\btheta^*$ to $\hat\btheta$ by showing that both estimators minimize the appropriate negative marginal log-likelihood and are asymptotically equivalent under a suitable transformation. 

\item\; Finally, using the transformation identified in Step~2, a change of variables expresses the score at $(\hat\btheta,\zero_K,\bI_K)$ as a linear transformation of the score at $(\hat\btheta^*,\hat\bmu^*,\hat\bSigma^*)$. The latter vanishes by Step~1, and hence so does the former, yielding the desired first-order optimality conditions for $\hat\btheta$.

 \end{enumerate}
Following the above procedure, we establish:
\begin{lemma}[First-order Condition] \label{lemma_key_first_order}
Define $\cS_{\theta}=\partial_{\btheta}\cL$, $\cS_{\mu}=\partial_{\bmu}\cL$, and $\cS_{\Sigma}=\partial_{\bSigma}\cL$ for $\cL(\btheta,\bmu,\bSigma)$ in \eqref{eq_full_marginal_likelihood}. Under Assumptions~\ref{assump_standard_id}--\ref{assump_scaling}, the following holds with probability tending to one:
    \begin{equation*}
        \cS_{\theta}(\hat\btheta,\zero_K,\bI_K) = \zero,\qquad\cS_{\mu}(\hat\btheta,\zero_K,\bI_K) = \zero,\quad\text{ and }\quad\cS_{\Sigma}(\hat\btheta,\zero_K,\bI_K) = \zero.
    \end{equation*}
\end{lemma}

\subsection{Consistency}\label{sec_prove_consis}
The consistency result follows from the properties of the auxiliary estimator. Specifically, Step~2 above establishes, up to an orthogonal transformation of $\hat\bgamma_j$s, that
\begin{equation*}
    \hat\bgamma_j = (\hat\bSigma^*)^{1/2}\hat\bgamma_j^* \;\text{ and }\; \hat\beta_j = \hat\beta_j^* + (\hat\bgamma_j^*)^\T\hat\bmu^*\;\text{with probability tending to one.}
\end{equation*}
In Step~1, we also bound $\|\hat\bmu^*\|\text{ and }\|\hat\bSigma^* - \bI_K\|$. Combining the convergence rate of $\hat\btheta^*$ established in Step~1 above proves Theorem~\ref{thm_consis}.

Quantities $\hat\bmu^*$ and $\hat\bSigma^* - \bI_K$ are closely related to $\bu_i^{\star}(\hat\btheta^*) = \argmax_{\bu\in\RR^K}\ell_i(\bu;\hat\btheta^*)$. Indeed, the first-order conditions for $\bmu$ and $\bSigma$ in Lemma~\ref{lemma_key_first_order} together with Laplace approximations yield 
\begin{equation*}
    \hat\bmu^* \;\approx\;\frac{1}{n}\sum_{i=1}^n\bu_i^{\star}(\hat\btheta^*)\quad\text{and}\quad\hat\bSigma^* \;\approx\;\frac{1}{n}\sum_{i=1}^n\big\{\bu_i^{\star}(\hat\btheta^*) - \hat\bmu^*\big\}\big\{\bu_i^{\star}(\hat\btheta^*) - \hat\bmu^*\big\}^\T.
\end{equation*}
Together with Lemma~\ref{lemma_neighbour_control} and consistency of $\hat\btheta^*$, these approximations yield the error bounds for $\|\hat\bmu^*\|$ and $\|\hat\bSigma^* - \bI_K\|$. The approximations manifest the implicit normalization in marginal estimation. When the mean and covariance parameters of $\pwk$ are set free, the estimated $\hat\bmu^*$ and $\hat\bSigma^*$ agree with the empirical mean and covariance of the likelihood modes up to small order, even though $\log\pwk(\cdot)$ is asymptotically negligible to the marginal likelihood. 

\subsection{Asymptotic Normality}\label{sec_prove_norma}
For simplicity, in this subsection we write $(\btheta^*,\bu^*)$ for the canonical representative and $\hat\btheta$ for the canonically aligned estimator $\hat\btheta^Q$. With $\cS_{\theta}(\hat\btheta,\zero_K,\bI_K) = \zero$, a natural starting point for deriving the asymptotic distribution of $\hat\btheta$ is expansion \eqref{eq_expand_informal}, which involves the inverse of $\partial_{\btheta}^2\cL(\btheta^*) = \cH_{\theta\theta}(\btheta^*)$. In particular, the discussion following Lemma~\ref{lemma_convexity} shows that it consists of a block-diagonal leading term and a dense, non-negligible correction arising from a Schur complement. This makes an explicit calculation of $\cH_{\theta\theta}^{-1}$ almost intractable.

To address this, we first observe that the dense component is the price paid for profiling out the latent variables. To make this connection precise, we consider the conditional negative log-likelihood $L(\btheta,\bu) = -\sum_{i=1}^n\ell_i(\bu_i;\btheta)$ with $\bu=(\bu_1^\T,\ldots,\bu_n^\T)^\T$ as the argument. Here, the stacked latent variables $\bu$ and the model parameter $\btheta$ are both arguments of $L(\btheta,\bu)$. Define the first- and second-order derivatives by
\begin{equation*}
    \bS_{L\theta}(\btheta,\bu) :=\, \frac{\partial L(\btheta,\bu)}{\partial\btheta},
    \quad\quad\quad
    \bS_{Lu}(\btheta,\bu)    :=  \frac{\partial L(\btheta,\bu)}{\partial\bu},\end{equation*}\begin{equation*}
    \bH_L(\btheta,\bu) :=\, \frac{\partial^2 L(\btheta,\bu) }{\partial^2(\btheta^\T,\bu^\T)^\T} = \begin{pmatrix}
        \bH_{L\theta\theta}(\btheta,\bu)&\bH_{L\theta u}(\btheta,\bu)  \\   \bH_{Lu\theta}(\btheta,\bu) & \bH_{Luu}(\btheta,\bu) 
        \end{pmatrix},
\end{equation*}
where, omitting the argument $(\btheta,\bu)$, the blocks in $\bH_L(\btheta,\bu)$ are given by
\begin{align*}
    \bH_{L\theta\theta} :=\frac{\partial^2L}{\partial\btheta\,\partial\btheta^\T},\quad
    \bH_{L\theta u}:= \frac{\partial^2L}{\partial\btheta\,\partial\bu^\T},\quad
    \bH_{Lu\theta}:=\frac{\partial^2L}{\partial\bu\,\partial\btheta^\T},\; \text{and}\;
    \bH_{Luu}:= \frac{\partial^2L}{\partial\bu\,\partial\bu^\T}.
\end{align*}
Lemma~\ref{lemma_laplace_approx_second_order} in the Supplementary Material gives the following Laplace approximation to $\cH_{\theta\theta}(\cdot)$:{
\setlength{\abovedisplayskip}{4pt}
\setlength{\belowdisplayskip}{4pt}
\begin{equation*}
    \cH_{\theta\theta}(\btheta) \;\approx\;\bH_{L\theta\theta}(\btheta,\bu^{\star}(\btheta)) \,- \,\bH_{L\theta u}\big(\btheta,\bu^{\star}(\btheta)\big)\,\left\{\bH_{Lu u}\big(\btheta,\bu^{\star}(\btheta)\big)\right\}^{-1}\,\bH_{Lu\theta}\big(\btheta,\bu^{\star}(\btheta)\big),
\end{equation*}}
where $\bu^{\star}(\btheta) = \big(\bu_1^{\star}(\btheta)^\T,\cdots,\bu_n^{\star}(\btheta)^\T\big)^\T$, with $\bu_i^{\star}(\btheta)$ denoting the likelihood mode at $\btheta$. The right-hand side is the Schur complement of the lower-right block $\bH_{Luu}$ in $\bH_L$.

What does this Schur-complement representation tell us? It suggests that the dense component in $\cH_{\theta\theta}(\btheta)$ arises from integrating out the latent variables, folding their contribution into the Schur complement $\bH_{L\theta u}\bH_{Luu}^{-1}\bH_{Lu\theta}$. To retain the latent-variable contribution explicitly, let $\bu^* = ((\bu_1^*)^\T,\dots,(\bu_n^*)^\T)^\T$ and consider the following score expansion:
\begin{equation}\label{eq_augmented_score_expansion}
    \begingroup
    \renewcommand{\arraystretch}{0.8}\begin{pmatrix}\bS_{L\theta}\big(\hat\btheta,\bu^{\star}(\hat\btheta)\big)\\\bS_{Lu}\big(\hat\btheta,\bu^{\star}(\hat\btheta)\big)\end{pmatrix} \,-\, \begin{pmatrix}\bS_{L\theta}(\btheta^*,\bu^*)\\\bS_{Lu}(\btheta^*,\bu^*)\end{pmatrix} \;=\; \bH_L(\btheta^*,\bu^*)\!\begin{pmatrix}
        \hat\btheta - \btheta^*\\\bu^{\star}(\hat\btheta) - \bu^*
    \end{pmatrix} \;+\; \mathsf{remainder}\endgroup.
\end{equation}
 The following features make this augmented expansion particularly tractable: 
\begin{itemize}
    \item \; On the left side, the score at $(\hat\btheta,\bu^\star(\hat\btheta))$ is asymptotically negligible, since (1) $\bS_{L\theta}(\hat\btheta,\bu^{\star}(\hat\btheta))\approx \cS_{\theta}(\hat\btheta,\zero_K,\bI_K) = \zero$ by Laplace approximation and Lemma~\ref{lemma_key_first_order}; and (2)  $\bS_{Lu}(\hat\btheta,\bu^{\star}(\hat\btheta)) = \zero$ follows directly from the definition of $\bu^\star(\hat\btheta)$. In addition, the score at $(\btheta^*,\bu^*)$ satisfies a conditional central limit theorem given $\{\bU_i\}_{i\in[n]}$.
    \item \; The Hessian $\bH_L(\btheta^*,\bu^*)$ can be inverted explicitly (after a correction detailed in Lemmas~\ref{lemma_approx_hessian_inverse} and~\ref{lemma_bound_aug_via_id} in the Supplementary Material). Its diagonal blocks, $\bH_{L\theta\theta}(\btheta^*,\bu^*)$ and $\bH_{Luu}(\btheta^*,\bu^*)$, are block-diagonal and thus admit direct inverses. We can then apply a standard $2\times 2$ block-matrix inversion formula to obtain a tractable expansion for $\hat\btheta - \btheta^*$. 
    This tractability arises because the dependence between the model parameters and latent variables remains isolated in the off-diagonal blocks $\bH_{L\theta u}(\btheta^*,\bu^*)$ and $\bH_{Lu\theta}(\btheta^*,\bu^*)$. The Schur complement folds this dependence into the dense term in $\cH_{\theta\theta}(\btheta^*)$, whereas the augmented score expansion keeps it explicit. 
\end{itemize}

Combining these analyses yields componentwise asymptotic normality for $\hat\btheta$ and $\hat\bu^\star$, which proves Theorem~\ref{thm_asymp}. Corollary~\ref{thm_asymp_bu} then follows by showing that $\hat\bu_i^\star-\hat\bu_i$ is asymptotically negligible, and Corollary~\ref{thm_bvm_irt} follows from a standard Bernstein–von Mises argument. Full details are provided in Sections~\ref{supp_sec_prove_asym_0}--\ref{supp_sec_prove_thm_bvm_irt} of the Supplementary Material.

\section{Simulation Studies}\label{sec_simu}
We assess the finite-sample coverage of the inference procedures developed in Section~\ref{sec_res}. We consider three latent distributions, with each component in $\bU_i$ generated independently from a Gaussian distribution, $\cN(0,1)$; a uniform distribution,  $\mathrm{Unif}(-2,2)$; or the Gaussian mixture $
\frac{1}{2}\cN(0.5,0.5^2)
+
\frac{1}{2}\cN(-0.5,0.5^2)$.
Within each replication, the generated latent variables are centered and normalized to obtain the canonical representation. The latent dimension is fixed at $K=3$. We vary $n\in \{300,600,1000\}$ and $q\in\{150,300,600\}$. For each $j\in[q]$, $\bgamma_j$s are generated independently from $\cN(\zero_K,\bSigma_\gamma)$ and then truncated entrywise to $[-2,2]$, where $\bSigma_\gamma\in\RR^{K\times K}$ has its $(l,h)$th entry being $0.3^{|l-h|}$. The intercepts are generated independently as $\beta_j\sim\cN(0,1)$ and truncated to $[-2,2]$.

\begin{table}[htbp]
\centering
\caption{Empirical coverage probabilities for model parameters under the Bernoulli model. Parentheses report Monte Carlo standard errors multiplied by \(1000\). The nominal coverage level is \(0.95\).}
\label{tab:coverage-gamma-beta-bernoulli}
\begin{tabular}{lccccccccc}
\toprule
 & \multicolumn{3}{c}{Gaussian} & \multicolumn{3}{c}{Uniform} & \multicolumn{3}{c}{Gaussian Mixture} \\
\cmidrule(lr){2-4}\cmidrule(lr){5-7}\cmidrule(lr){8-10}
$n$ & $q=150$ & $q=300$ & $q=600$ & $q=150$ & $q=300$ & $q=600$ & $q=150$ & $q=300$ & $q=600$ \\
\midrule
300 & 0.930 & 0.941 & 0.945 & 0.932 & 0.941 & 0.947 & 0.931 & 0.941 & 0.945 \\
 & (1.55) & (0.70) & (0.47) & (1.56) & (0.69) & (0.46) & (1.57) & (0.72) & (0.47) \\
600 & 0.931 & 0.942 & 0.946 & 0.933 & 0.942 & 0.946 & 0.931 & 0.942 & 0.946 \\
 & (1.53) & (0.69) & (0.47) & (1.56) & (0.67) & (0.45) & (1.53) & (0.69) & (0.46) \\
1000 & 0.937 & 0.943 & 0.945 & 0.930 & 0.941 & 0.946 & 0.932 & 0.943 & 0.946 \\
 & (1.50) & (0.71) & (0.46) & (1.67) & (0.69) & (0.46) & (1.53) & (0.67) & (0.46) \\
\bottomrule
\end{tabular}
\end{table}

We consider three observation models: a linear Gaussian model, a Bernoulli model with a logistic link, and a Poisson model with a log link. For each configuration, we conduct $100$ independent replications. Estimation proceeds in two stages. We first adopt the variational approximation~\citep{hui2017variational} to obtain a computationally efficient initial estimate. Then we apply an EM algorithm in which the required integrals are evaluated using adaptive Gauss--Hermite quadrature with seven nodes per dimension, giving $7^K=343$ quadrature nodes in total. This refinement targets the marginal maximum likelihood estimator up to small numerical optimization
and quadrature errors.
We report empirical coverage using the sandwich-form asymptotic variance estimator for the model parameters (Theorem~\ref{thm_asymp}) and the sandwich-form asymptotic variance estimator for the posterior-mode point estimate of $\bU_i$ (Corollary~\ref{thm_asymp_bu}).
 The empirical coverage results for the Bernoulli model are reported in Tables~\ref{tab:coverage-gamma-beta-bernoulli} and~\ref{tab:coverage-u-bernoulli}. The results under Gaussian and Poisson models exhibit similar patterns and are reported in Section~\ref{supp_Sec_simu} of the Supplementary Material.

\begin{table}[htbp]
\centering
\caption{Empirical coverage probabilities for latent variables $\bU_i$ under the Bernoulli model. Parentheses report Monte Carlo standard errors multiplied by \(1000\). The nominal coverage level is \(0.95\).}
\label{tab:coverage-u-bernoulli}
\begin{tabular}{lccccccccc}
\toprule
 & \multicolumn{3}{c}{Gaussian} & \multicolumn{3}{c}{Uniform} & \multicolumn{3}{c}{Gaussian Mixture} \\
\cmidrule(lr){2-4}\cmidrule(lr){5-7}\cmidrule(lr){8-10}
$n$ & $q=150$ & $q=300$ & $q=600$ & $q=150$ & $q=300$ & $q=600$ & $q=150$ & $q=300$ & $q=600$ \\
\midrule
300 & 0.928 & 0.933 & 0.937 & 0.926 & 0.931 & 0.936 & 0.927 & 0.932 & 0.935 \\
 & (1.05) & (0.99) & (0.90) & (1.30) & (0.86) & (0.89) & (1.06) & (0.92) & (0.79) \\
600 & 0.934 & 0.938 & 0.942 & 0.932 & 0.939 & 0.941 & 0.935 & 0.939 & 0.942 \\
 & (0.86) & (0.60) & (0.54) & (0.70) & (0.56) & (0.61) & (0.70) & (0.57) & (0.60) \\
1000 & 0.938 & 0.941 & 0.944 & 0.934 & 0.941 & 0.943 & 0.937 & 0.942 & 0.943 \\
 & (0.63) & (0.47) & (0.44) & (0.55) & (0.48) & (0.44) & (0.71) & (0.44) & (0.40) \\
\bottomrule
\end{tabular}
\end{table}

The results show that confidence intervals based on our asymptotic theory attain empirical coverage close to the nominal 95\% level, with little variation across replications. The results are similar under Gaussian and non-Gaussian latent distributions, providing numerical support that the result does not depend on the specific choice of latent distribution. Inference for each $\btheta_j$ requires inversion of one $(K+1)\times(K+1)$ matrix, which is computationally efficient. As a comparison, the conventional method inverts the full Hessian with respect to entire parameter $\btheta\in\RR^{q(K+1)}$. We do not carry out this comparison over the full 100 replications for several reasons. First, the computational cost is substantial. With $S$ Monte Carlo draws, forming this Hessian costs $O(nSq^2(K+1)^2)$, in addition to the cost of inverting a $q(K+1)$-dimensional matrix. Moreover, the estimated Hessian was indefinite in every replication we examined. The diminishing curvature established in Lemma~\ref{lemma_convexity} makes its smallest eigenvalues particularly sensitive to Monte Carlo error. Finally, using a Moore--Penrose pseudoinverse to approximate the Hessian inverse can produce negative variance estimates for many entries and substantial undercoverage for the other parameters with positive estimated variances. These findings make the conventional approach unreliable to use as a benchmark in high dimensions.


\section{Real Data Analysis}\label{sec_data}
    We analyze data from the 2026 release of the North American Breeding Bird Survey (BBS)~\citep{pardieck2016north}\footnote{The data can be found in \href{https://doi.org/10.5066/P144YU3S}{doi:10.5066/P144YU3S}}. The BBS is a continent-wide monitoring program designed to assess the status and long-term trends of North American bird populations. Along survey routes distributed across North America, observers record all birds seen or heard. The resulting data play an important role in documenting changes in avian populations and informing conservation and wildlife-management decisions. 

We focus on the 2024 survey year in the continental United States and organize the data into a route-by-species count matrix $\bX\in\{0,1,\cdots,50\}^{2124\times 595}$, whose rows correspond to survey routes and whose columns correspond to bird species. Each entry records the number of the $50$ roadside stops along a route at which the species was detected, so the responses are bounded detection counts rather than numbers of individuals. For each route, the latent variable $\bU_i$ provides a low-dimensional representation of avian assemblage structure. For each species $j$, $\btheta_j=(\beta_j,\bgamma_j^\T)^\T$ characterizes its abundance pattern across these latent dimensions. We fit the model with a Poisson response distribution with dispersion $\phi_j$ estimated during fitting from the Pearson residuals. We set the number of latent variables at $K=5$. To address rotational indeterminacy and facilitate interpretation, we apply an oblique Geomin rotation~\citep{yates1987multivariate} to the estimated loading matrix and transform the latent scores accordingly. This rotation yields a more interpretable loading structure while allowing the rotated latent variables to be correlated.


\begin{figure}
    \centering
    \includegraphics[width=0.9\linewidth]{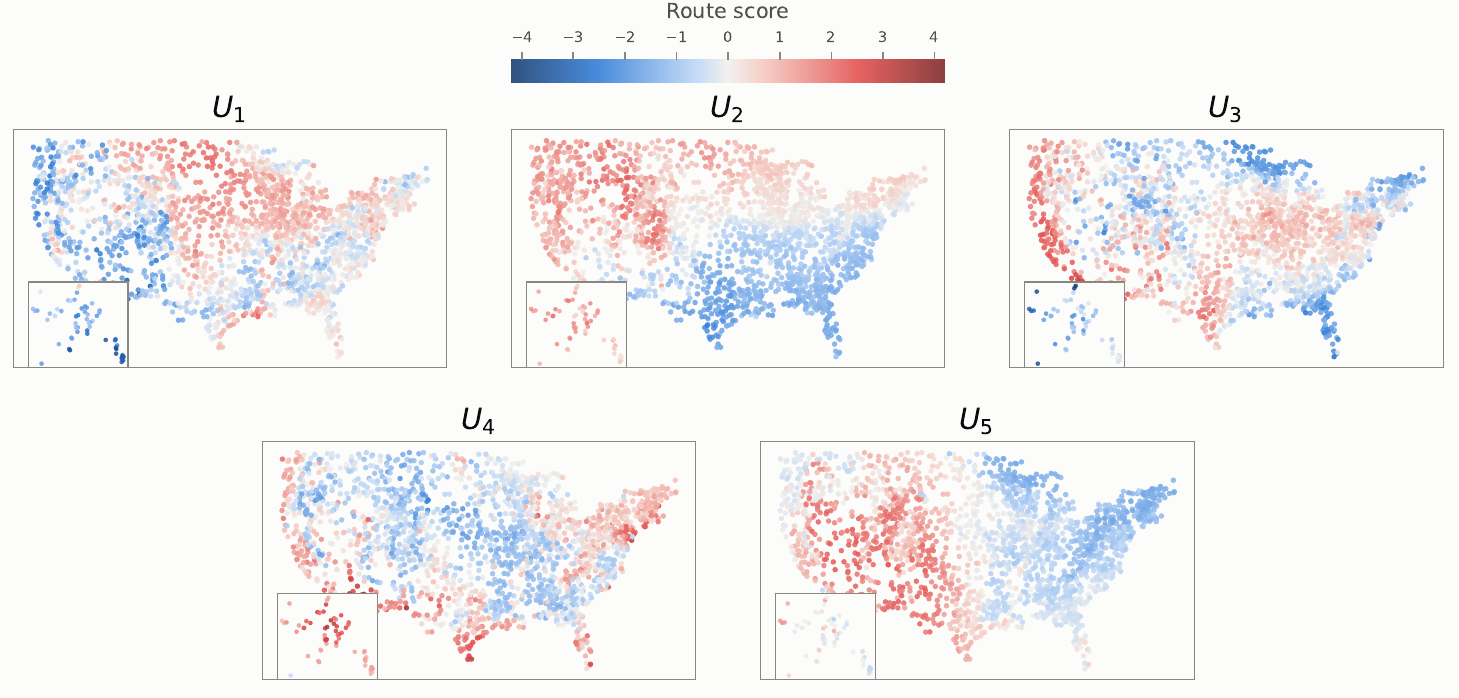}
    \caption{Estimated route scores for the five latent variables across the continental United States.}
    \label{fig_route1}
\end{figure}

Figure~\ref{fig_route1} maps the estimated route scores for the five latent variables. These latent variables align with recognizable geographic gradients. Specifically, $U_1$ highlights the grassland and prairie routes of the Great Plains, contrasting them primarily with routes in the western mountains and along the Pacific coast; $U_2$ separates northern and northwestern routes from those in the Southeast and exhibits a clear latitudinal gradient; $U_3$ separates Californian routes most strongly, with the far north and the Southeast at the opposite pole; and $U_4$ captures a broader coastal–inland separation, with the coastal pattern particularly visible in New England, Florida, Texas, and Alaska.


The fifth variable, $U_5$, exhibits a pronounced east--west separation. Its spatial pattern follows approximately the 100th-meridian transition, a major climatic and biogeographic boundary that also separates much of the eastern and western avifauna. This interpretation is further supported by the species-specific parameter $\gamma_{j5}$. Wood Thrush and Ovenbird, both characteristic of eastern deciduous forests, have the largest estimates of one sign, $-4.09$ and $-3.91$. By contrast, Western Meadowlark and Western Kingbird, which are predominantly associated with western and southwestern habitats, have positive estimates of $+1.50$ and $+1.44$, respectively. All four entries are highly significant, with p-value below $10^{-40}$.
The inference results obtained from our asymptotic theory can further support various downstream tasks, including comparing association strengths across species or ecological groups and testing their relationships with habitat, migration, or taxonomy.


\begin{figure}[hbtp]
    \centering
    \includegraphics[width=0.75\linewidth]{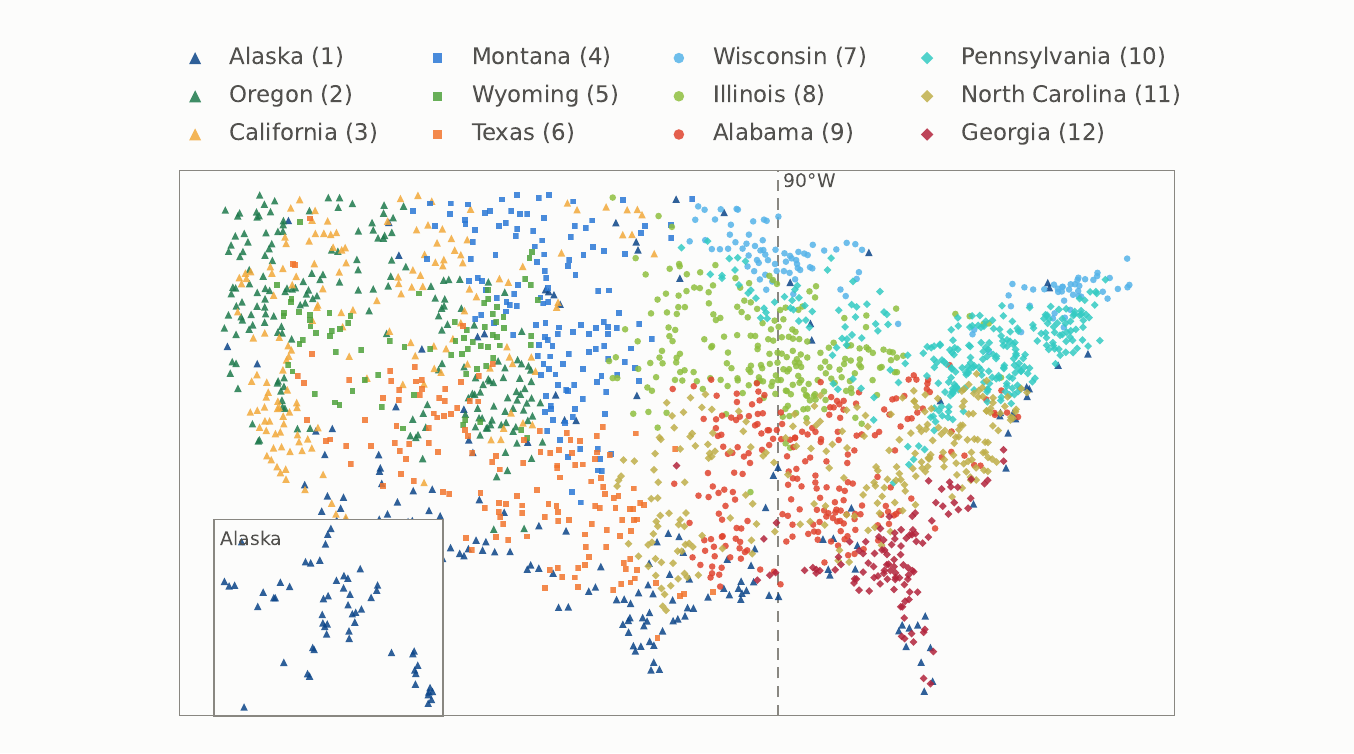}
    \caption{Geographic distribution of the twelve route clusters based on the estimated latent variables. Each cluster is named after the state most frequently represented among its assigned routes.}
    \label{fig_xd_clusters}
\end{figure}

We apply extreme deconvolution~\citep{bovy2011extreme} to cluster the routes, which fits a Gaussian mixture to the underlying latent distribution while accounting for heteroskedastic measurement error at each route. Motivated by Corollaries~\ref{thm_asymp_bu} and~\ref{thm_bvm_irt}, we use $\hat\bu_i$ as a point estimate of $\bU_i$ and its estimated asymptotic covariance as route-specific variance. Fitting mixtures with $G=1,\ldots,20$ components, the BIC falls steeply from ${\rm BIC}(G=1)=28606.9$ for a single Gaussian to a minimum ${\rm BIC}(G=12)=20463.5$.  The discrepancy in BIC indicates that the latent distribution is not adequately described by a single multivariate Gaussian.
As shown in Figure~\ref{fig_xd_clusters}, the resulting clusters are largely geographically coherent. For example, east of $90^\circ$W the clusters form an orderly latitudinal sequence, from Cluster~7 in the upper Midwest and northern New England, through Cluster~10 centred on Pennsylvania and the Northeast and Cluster~8 in the Corn Belt, to Cluster~11 across the central Appalachians and the Carolinas, Cluster~9 in the interior Deep South and Cluster~12 on the Gulf and South Atlantic coastal plain.


\section{Discussion}\label{sec_discussion}
This paper develops asymptotic theory for marginal maximum likelihood estimation in high-dimensional latent variable models. Theoretically, we establish the first rigorous statistical guarantees for marginal estimation in this regime, including consistency and asymptotic normality for model parameter estimation, as well as latent-variable recovery, closing a long-standing gap in the literature. Our analysis reveals several fundamental features of the problem, including the local geometry of the marginal likelihood, the implicit normalization of the conditional likelihood modes, and how integrating out the latent variables shapes the curvature of the marginal likelihood. We assess the finite-sample performance of the resulting inferential procedures through simulations and illustrate their practical utility in an analysis of North American Breeding Bird Survey data.

Our work opens several directions for future research. First, our analysis assumes the standard conditionally independent measurement model. More complex dependence structures arise in hierarchical, longitudinal, spatial, and network latent-variable models, and whether our techniques can be adapted to these settings is an interesting question~\citep{hoff2002latent,skrondal2004generalized,mcculloch2008generalized}. Second, we focus on continuous latent variables. Models with discrete and mixed-type latent variables are also widely used and raise additional identifiability issues~\citep{allman2009identifiability,xu2017identifiability}. Establishing statistical guarantees for their marginal estimators is an important direction. Third, our results provide a likelihood-based starting point for studying latent representation learning. Machine learning tools employ increasingly complex representations, such as discrete and hierarchical latent codes and the autoencoded latent spaces used by diffusion models~\citep{van2017neural,vahdat2020nvae,rombach2022high}. Extending our analysis to these settings could bring a statistical perspective to the properties and uncertainty of learned representations.

{\spacingset{1.2}
\bibliographystyle{agsm}
\bibliography{reference}}

\newpage

\newtheorem{suppTheorem}{Theorem}[section]
\renewcommand{\thesuppTheorem}{\thesection}
\renewcommand{\theequation}{S\arabic{equation}}
\setcounter{equation}{0}
\renewcommand{\thefigure}{S\arabic{figure}}
\setcounter{figure}{0}
\ifx\lemma\undefined
\newtheorem{lemma}{Lemma}[section]
\fi
\renewcommand{\thelemma}{\thesection\arabic{lemma}}
\appendix

\addtocontents{toc}{\protect\setcounter{tocdepth}{1}}
\setcounter{tocdepth}{1}
\tableofcontents

\section{Notation and Approximation to the Derivatives}\label{supp_sec_prelim}

We introduce notation used throughout the supplement.  For a vector $\bx\in\RR^d$ and an index set $\cS\subseteq[d]$, $\bx_{\cS}$ is the subvector indexed by $\cS$, and $|\cS|$ is the cardinality of $\cS$. For a matrix $\bM=(M_{rs})\in\RR^{d_1\times d_2}$ and index sets $\cS_1\subseteq[d_1]$ and $\cS_2\subseteq[d_2]$, $\bM_{\cS_1,\cS_2}$ is the corresponding submatrix. We abbreviate $\bM_{[d_1],\cS_2}$ and $\bM_{\cS_1,[d_2]}$ as $\bM_{,\cS_2}$ and $\bM_{\cS_1,}$, respectively. For a matrix $\bM$, let $\|\bM\|_\infty$ and $\|\bM\|_{*}$ denote its maximum-row-sum and nuclear norms, respectively. Let $\mathrm{diag}(\bM)$ retain only the diagonal of $\bM$. The operator $\mathrm{vech}(\bM)$ stacks the lower-triangular entries of a symmetric $\bM$. Write $\be_r^{(d)}$ for the $r$th standard basis vector of $\RR^d$  and $\one_d$ for the $d$-vector of ones. We write $\cO^K=\{\bQ\in\RR^{K\times K}:\bQ\bQ^\T=\bI_K\}$ for the orthogonal group. For a random variable $Z$, $\|Z\|_{\psi_1}$ denotes its sub-exponential Orlicz norm, and $\psi_1\mid\mathcal F$ denotes its conditional version. We write $\operatorname{rank}(\bM)$ and $\|\bM\|_*$ for the rank and nuclear norm of a matrix, $\bP_{\cS}$ for the orthogonal projector onto a subspace $\cS$, and $\bP_{\cS}^{\perp}=\bI-\bP_{\cS}$. Let
\begin{equation}
    \cK_i=\{(i-1)K+1,\ldots,iK\},\qquad
    \cK_j^+=\{(j-1)(K+1)+1,\ldots,j(K+1)\}.\label{eq_define_Kij}
\end{equation} We use $\bu=(\bu_1^\T,\ldots,\bu_n^\T)^\T$, $\bbeta=(\beta_1,\ldots,\beta_q)^\T$, $\bGamma=(\bgamma_1,\ldots,\bgamma_q)^\T$, and $\bTheta=(\bbeta,\bGamma)$. Superscript $^*$ denotes the true parameter or realized latent value (e.g. $\bu^*$) and $\bu_i^\star(\btheta)$ denotes the conditional-likelihood mode at $\btheta$. Let $\bSigma_\gamma^*:=\lim_{q\to\infty}q^{-1}\sum_{j=1}^q\bgamma_j^*(\bgamma_j^*)^\T$, which exists and is positive definite by Assumption~\ref{assumption_psd_covariance}. Denote the joint law of latent variables and responses by $\PP$.

We further clarify scaling conditions implied by Assumption~\ref{assump_scaling}. Because $n^c\le q\le n^C$ for constants $c,C>0$, we have 
\begin{equation*}
    \log q\asymp \log n.
\end{equation*}Moreover, $B_n$, $b_n^{-1}$, and $\kappa_n$ are all of order at most $\exp\{O(\sqrt{\log n})\} = n^{o(1)}$. For $\varepsilon_{nq} = \exp\{\log^{1/2+c}(n\vee q)\} = (n\vee q)^{o(1)}$, it dominates any fixed product of $B_n$, $b_n^{-1}$, $\kappa_n$, and $\log(n\vee q)$. It is worth mentioning that any fixed product of powers of these quantities remains subpolynomial, and multiplication by any fixed negative power of $n\wedge q$ makes it converge to zero.



Recall that $\{\bu_i^*\}_{i=1}^n$ denotes the realized values of the random variables $\{\bU_i\}_{i=1}^n$. Another quantity that plays an important role in the subsequent analysis is $ \bu_i^{\star} 
    = \argmax_{\bu\in\RR^K} \ell_i(\bu;\btheta^*)$. The following result characterizes this oracle estimator.
\begin{lemma}
    \label{lemma_bu_star_theta_true}\em At the true parameter $\btheta^*$, let $\bu_i^{\star}$ be the solution to $\arg\max_{\bu\in\RR^{K}}\ell_i(\bu;\btheta^*)$. Under Assumptions~\ref{assump_standard_id}--\ref{assump_scaling},
    \begin{equation*}
        \big\|\bu_i^{\star} - \bu_i^*\big\| = \cOp\big(\tfrac{\kappa_n}{\sqrt{q}}\big),\quad n^{-1}\sum_{i=1}^n\big\|\bu_i^{\star} - \bu_i^*\big\|^2 = \cOp\big(\tfrac{\kappa_n^2}{q}\big),\quad \max_{i\in[n]}\big\|\bu_i^{\star} - \bu_i^*\big\| = \cOp\big(\kappa_n\sqrt{\tfrac{\log n}{q}}\big).
    \end{equation*}
    Furthermore,
    \begin{equation}
        \bu_i^{\star} - \bu_i^* = -\big\{\sum_{j=1}^q\ell_{ij}^{\prime\prime}(\eta_{ij}^*)\bgamma_j^*(\bgamma_j^*)^\T\big\}^{-1}\sum_{j=1}^q\ell_{ij}^{\prime}(\eta_{ij}^*)\bgamma_j^* + \cR_i^{\star},\label{eq_expansion_u_star}
    \end{equation}
    where $\cR_i^{\star}$ satisfies $\|\cR_i^{\star}\| = \cOp\big(\kappa_n^2B_n/(qb_n)\big)$, $\big(n^{-1}\sum_{i=1}^n\|\cR_i^{\star}\|^2\big)^{1/2} = \cOp\big(\kappa_n^2B_n/(qb_n)\big)$ and $\max_{i\in[n]}\|\cR_i^{\star}\| = \cOp(\kappa_n^2B_n\log n/(qb_n))$.
   Furthermore,
    \begin{equation*}
        \Big\|\sum_{i=1}^n\big(\bu_i^\star - \bu_i^*\big)\Big\| = \cOp\big(\kappa_n\sqrt{n/q} + \kappa_n^3\sqrt{B_n}n/q\big),\end{equation*}
        and
        \begin{equation*}\Big\|\sum_{i=1}^n\big(\bu_i^\star{\bu_i^\star}^\T - \bu_i^*{\bu_i^*}^\T\big)\Big\| = \cOp\big(\kappa_n\sqrt{n/q} + \kappa_n^3\sqrt{B_n}n/q\big).
    \end{equation*}
\end{lemma}
\begin{proof}
    See Section~\ref{supp_sec_prove_lemma_bu_star_theta_true}.
\end{proof}
\subsection{Derivatives of the Marginal Log-likelihood}\label{supp_sec_deriva_form}
Recall that $\cL(\btheta,\bmu,\bSigma)$ is the negative marginal log-likelihood
\begin{equation*}
    \cL(\btheta,\bmu,\bSigma) = -\sum_{i=1}^n\log\int_{\bu_i\in\RR^K}\prod_{j=1}^qp_j(X_{ij}\mid\bgamma_j^\T\bu_i + \beta_j)\pwk(\bu_i\mid\bmu,\bSigma)d\bu_i .
\end{equation*}
Its first-order derivatives are defined as
\begin{equation*}
\cS_{\theta}(\btheta,\bmu,\bSigma) = \frac{\partial \cL(\btheta,\bmu,\bSigma)}{\partial\btheta},\; \cS_{\mu}(\btheta,\bmu,\bSigma) = \frac{\partial \cL(\btheta,\bmu,\bSigma)}{\partial\bmu}\text{ and }\cS_{\Sigma}(\btheta,\bmu,\bSigma) = \frac{\partial \cL(\btheta,\bmu,\bSigma)}{\partial\bSigma},
\end{equation*}
Here, $\cS_\Sigma$ denotes the symmetric matrix gradient satisfying
$d\cL=\tr(\cS_\Sigma^\T d\bSigma)$ and let $\cS_\Sigma^{\mathrm{vec}}=\mathrm{vec}(\cS_\Sigma)$ denote its fully vectorized form. The second-order derivative with respect to $\btheta$ is
\begin{equation*}
    \cH_{\theta\theta}(\btheta,\bmu,\bSigma) = \frac{\partial^2\cL(\btheta,\bmu,\bSigma)}{\partial\btheta^2}.
\end{equation*}
For any of these quantities, we suppress the arguments $(\bmu,\bSigma)$ when they are fixed at $(\zero_K,\bI_K)$.

Next, we introduce derivatives of the conditional negative log-likelihood $L(\btheta,\bu)=-\sum_{i=1}^n\sum_{j=1}^q\ell_{ij}(\bgamma_j^\T\bu_i+\beta_j)$, treating the $\bu_i$ as arguments. Scores and Hessians carrying a subscript $L$ refer to the conditional negative log-likelihood $L(\btheta,\bu)=-\sum_{i=1}^n\ell_i(\bu_i;\btheta)$, while calligraphic scores ($\cS$) and Hessians ($\cH$) refer to the negative marginal log-likelihood.
Define $\bS_L(\btheta,\bu) = \big(\{\bS_{L\theta}(\btheta,\bu)\}^\T,\{\bS_{Lu}(\btheta,\bu)\}^\T\big)^\T$ with
\begin{equation*}
    \big\{\bS_{L\theta}(\btheta,\bu)\big\}_{\cK_j^+} =-\sum_{i=1}^n\ell^{\prime}_{ij}\big(\bgamma_j^\T\bu_i + \beta_j\big)\begingroup \renewcommand{\arraystretch}{0.7}\begin{pmatrix}
        1\\\bu_i
    \end{pmatrix}\endgroup\;\text{ for }\;j\in[q],
\end{equation*}
and 
\begin{equation*}
    \big\{\bS_{Lu}(\btheta,\bu)\big\}_{\cK_i} = -\sum_{j=1}^q\ell^{\prime}_{ij}\big(\bgamma_j^\T\bu_i + \beta_j\big)\bgamma_j\;\text{ for }\;i\in[n].
\end{equation*}$\cK_i$ and $\cK_j^+$ have been defined in \eqref{eq_define_Kij}.

Define 
\begin{equation*}
    \bH_L(\btheta,\bu) = \begin{pmatrix}
        \bH_{L\theta\theta}(\btheta,\bu) & \bH_{L\theta u}(\btheta,\bu) \\ \bH_{Lu\theta}(\btheta,\bu) & \bH_{Luu}(\btheta,\bu)
    \end{pmatrix}
\end{equation*}
with 
\begin{equation*}
    \bH_{L\theta_j\theta_j}(\btheta,\bu):=\big\{\bH_{L\theta\theta}(\btheta,\bu) \big\}_{\cK_j^+,\cK_j^+} = -\sum_{i=1}^n\ell_{ij}^{\prime\prime}(\eta_{ij})\begingroup \renewcommand{\arraystretch}{0.5}\begin{pmatrix}
        1&\bu_i^\T\\\bu_i&\bu_i\bu_i^\T
    \end{pmatrix}\endgroup\text{ for }j\in[q]
\end{equation*}
and $\big\{\bH_{L\theta\theta}(\btheta,\bu) \big\}_{\cK_{j_1}^+,\cK_{j_2}^+} = \zero$ if $j_1\neq j_2$;\begin{equation*}
    \bH_{Lu_iu_i}(\btheta,\bu):=\big\{\bH_{Luu}(\btheta,\bu) \big\}_{\cK_i,\cK_i} = -\sum_{j=1}^q\ell_{ij}^{\prime\prime}(\eta_{ij})\bgamma_j\bgamma_j^\T\text{ for }i\in[n],
\end{equation*}
and $\big\{\bH_{Luu}(\btheta,\bu) \big\}_{\cK_{i_1},\cK_{i_2}} = \zero$ if $i_1\neq i_2$; 
\begin{equation*}
    \bH_{L\theta_ju_i}(\btheta,\bu):=\big\{\bH_{L\theta u}(\btheta,\bu) \big\}_{\cK_j^+,\cK_i} = -\ell_{ij}^{\prime\prime}(\eta_{ij})\begingroup \renewcommand{\arraystretch}{0.7}\begin{pmatrix}
        1\\\bu_i
    \end{pmatrix}\bgamma_j^\T - \begin{pmatrix}
        \zero_K^\T\\\ell_{ij}^{\prime}(\eta_{ij})\bI_K
    \end{pmatrix}\endgroup\text{ for }i\in[n], j\in[q],
\end{equation*}
and $\bH_{Lu\theta}(\btheta,\bu) = \big\{\bH_{L\theta u}(\btheta,\bu)\big\}^\T$. 
Here $\eta_{ij} = \bgamma_j^\T\bu_i + \beta_j$. Additionally, let $\bJ_{L\theta u}(\btheta,\bu) = \{\bJ_{Lu\theta }(\btheta,\bu)\}^\T$ be the matrix containing only the first-order terms in $\bH_{L\theta u}(\btheta,\bu)$. Specifically,
\begin{equation*}
    \bJ_{L\theta_ju_i}(\btheta,\bu):=\big\{\bJ_{L\theta u}(\btheta,\bu)\big\}_{\cK_j^+,\cK_i} = -\begingroup \renewcommand{\arraystretch}{0.7}\begin{pmatrix}
        \zero_K^\T\\\ell_{ij}^{\prime}(\eta_{ij})\bI_K
    \end{pmatrix}\endgroup\text{ for }i\in[n], j\in[q].
\end{equation*} Let $\bar\bH_{L\theta u}(\btheta,\bu) = \{\bar\bH_{Lu\theta }(\btheta,\bu)\}^\T = \bH_{L\theta u}(\btheta,\bu) - \bJ_{L\theta u}(\btheta,\bu)$.

Next, we use Laplace approximation to show that the derivatives for the marginal likelihood can be approximated by the derivatives of the conditional log-likelihood function displayed above.
For fixed constants $\rho\in(0,1)$ and $C>0$, define the parameter sets
\begin{equation*}
\bar\cK_\rho = \left\{ (\bmu,\bSigma): \|\bmu\|\le \rho,\quad\|\bSigma-\bI_K\|\le \rho \right\}.
\end{equation*}and
\begin{equation}\label{eq_define_KC}
    \cK_{C} := \{ (\bmu,\bSigma): \|\bmu\|\le C,\quad C^{-1}\le\lambda_{\min}(\bSigma)\le\lambda_{\max}(\bSigma)\le C\}.
\end{equation}
The approximations to the first-order derivatives are given in the following lemma.
\begin{lemma}[Approximation to the First-order Derivatives]
    \label{lemma_first_order_base}\em Suppose that Assumptions~\ref{assump_standard_id}--\ref{assump_scaling} hold. Given any $\btheta\in\Xi$, let $\bu_i^{\star}(\btheta) = \arg\max_{\bu\in\RR^{K}}\ell_i(\bu;\btheta)$ and $\bu^{\star}(\btheta) = \big({\bu_1^{\star}(\btheta)}^\T,\dots,{\bu_n^{\star}(\btheta)}^\T\big)^\T$. For any fixed $\rho\in(0,1)$, the following hold uniformly over $\btheta\in\Xi$ and $(\bmu,\bSigma)\in\bar\cK_\rho$:
    \begin{equation*}
        \big\|\cS_{\theta}(\btheta,\bmu,\bSigma) - \bS_{L\theta}\big(\btheta,\bu^{\star}(\btheta)\big)\big\|_{\infty} = \cOp(nB_n\sqrt{\log n}\varepsilon_{nq}/q),
    \end{equation*}
    and
    \begin{equation*}
       \big\|\cS_{\theta}(\btheta,\bmu,\bSigma) - \bS_{L\theta}\big(\btheta,\bu^{\star}(\btheta)\big)\big\| = \cOp(nB_n\varepsilon_{nq}\sqrt{\log n}/\sqrt{q}).
    \end{equation*}
    Moreover, the following hold uniformly over $\btheta\in\Xi$ and $(\bmu,\bSigma)\in\bar\cK_\rho$:
    \begin{equation*}
        \big\|\bSigma \cS_{\mu}(\btheta,\bmu,\bSigma) + \sum_{i=1}^n(\bu_i^\star(\btheta) - \bmu)\big\| = \cOp\big(\|\bSigma\|n\varepsilon_{nq}/q\big),
    \end{equation*}
    and \begin{equation*}
       \left\|\bSigma \cS_{\Sigma}(\btheta,\bmu,\bSigma) \bSigma+\frac{1}{2} \sum_{i=1}^n\Big\{(\bu_i^\star(\btheta) - \bmu)(\bu_i^\star(\btheta) - \bmu)^\T - \bSigma\Big\}\right\| = \cOp\big(\|\bSigma\|^2n\varepsilon_{nq}/q\big).
    \end{equation*}
    The same conclusions remain valid uniformly over $(\bmu,\bSigma)\in\cK_C$ for any fixed $C>0$.
\end{lemma}
\begin{proof}See Section~\ref{supp_sec_prove_lemma_first_order_base}.\end{proof}
 Together with the first-order conditions in Lemma~\ref{lemma_key_first_order}, the last two conclusions imply that the conditional-likelihood modes $\bu^{\star}_i(\hat\btheta)$ satisfy
\begin{equation*}
    n^{-1}\sum_{i=1}^n\bu^{\star}_i(\hat\btheta) = \cOp(q^{-1}\varepsilon_{nq})\quad\text{ and }\quad n^{-1}\sum_{i=1}^n\bu^{\star}_i(\hat\btheta)\bu^{\star}_i(\hat\btheta)^\T = \bI_K+ \cOp(q^{-1}\varepsilon_{nq}).
\end{equation*}
This motivates the canonical realization-level normalization \eqref{eq_id_st} used later for asymptotic distribution statements.

Assumption~\ref{assumption_smoothness} states that the conditional scores $\ell_{ij}'(\eta_{ij}^*)$ are sub-exponential given $\{\bu_i^*\}_{i=1}^n$, and consequently we have the following lemma.
\begin{lemma}\label{coro_first_order_concern}\em
    Suppose Assumptions~\ref{assump_standard_id}--\ref{assump_scaling} hold. Denote by $\btheta^{*}$ the true model parameters and by $\bu^*$ the realized latent variables. Then
    \begin{equation*}
        \frac{1}{\sqrt{n}}\big\|\bS_{L\theta}(\btheta^*, \bu^*)\big\| = \cOp\big(\sqrt{qB_n}\big)\text{, }\frac{1}{\sqrt{q}}\big\|\bS_{Lu}(\btheta^*, \bu^*)\big\| = \cOp\big(\sqrt{nB_n}\big),
    \end{equation*}
    and
    \begin{equation*}
        \frac{1}{\sqrt{n}}\big\|\bS_{L\theta}(\btheta^*, \bu^*)\big\|_{\infty} = \cOp\left(\sqrt{B_n\log q}\right)\text{, }\frac{1}{\sqrt{q}}\big\|\bS_{Lu}(\btheta^*, \bu^*)\big\|_{\infty} = \cOp\left(\sqrt{B_n\log n} \right),
    \end{equation*}
\end{lemma}
\begin{proof}
    See Section~\ref{sub_prove_coro_first_order_concern}.
\end{proof}
Based on Lemmas~\ref{lemma_bu_star_theta_true}, \ref{lemma_first_order_base} and~\ref{coro_first_order_concern}, we arrive at the following result on the first-order derivative $\cS_{\theta}$ at the true parameters. 
\begin{lemma}
    \label{lemma_first_order_concern}\em Suppose Assumptions~\ref{assump_standard_id}--\ref{assump_scaling} hold. Denote by $\btheta^{*}$ the true model parameters. Given any $0<\rho < 1$, uniformly over $(\bmu,\bSigma)\in\bar\cK_\rho$, it holds that \begin{enumerate}[label=$(\roman*)$]
        \item 
        \begin{equation*}
        \big\|\cS_{\theta}(\btheta^*,\bmu,\bSigma)\big\| = \cOp\left(\sqrt{nq}r_{\theta,nq}\right); 
    \end{equation*}
    \item 
    \begin{equation*}
        \big\|\cS_{\theta}(\btheta^*,\bmu,\bSigma)\big\|_{\infty} = \cOp\left(\sqrt{n\log q}r_{\theta,nq}\right),
    \end{equation*} where $r_{\theta,nq}$ is defined in Equation~\eqref{eq_define_r_theta_nq}.
    \end{enumerate}
\end{lemma}
\begin{proof}See Section~\ref{supp_sec_prove_lemma_first_order_concern}.\end{proof}

The next lemma regards the approximation to the second-order derivatives $\cH_{\theta\theta}(\btheta,\bmu,\bSigma)$. 
\begin{lemma}[Approximation to the Second-order Derivatives]\label{lemma_laplace_approx_second_order}\em
   Suppose that Assumptions~\ref{assump_standard_id}--\ref{assump_scaling} hold. Given any $\btheta\in\Xi$, let $\bu_i^{\star}(\btheta) = \arg\max_{\bu\in\RR^{K}}\ell_i(\bu;\btheta)$ and $\bu^{\star}(\btheta) = \big(\bu_1^{\star}(\btheta)^\T, \dots, \bu_n^{\star}(\btheta)^\T\big)^\T$. Define 
    \begin{equation*}
        \bH_{-\theta}^{\star}(\btheta):= \bH_{L\theta\theta}(\btheta,\bu^{\star}(\btheta)) - \bH_{L\theta u}\big(\btheta,\bu^{\star}(\btheta)\big)\,\left\{\bH_{Lu u}\big(\btheta,\bu^{\star}(\btheta)\big)\right\}^{-1}\,\bH_{Lu\theta}\big(\btheta,\bu^{\star}(\btheta)\big).
    \end{equation*}
    Then, for any fixed $\rho\in(0,1)$, the following hold uniformly over $\btheta\in\Xi$ and $(\bmu,\bSigma)\in\bar\cK_\rho$:
    \begin{equation*}
        \Big\|\cH_{\theta\theta}(\btheta,\bmu,\bSigma) - \bH_{-\theta}^{\star}(\btheta)\Big\| = \cOp(n\varepsilon_{nq}/q),
    \end{equation*}
   \begin{equation*}
        \Big\|\cH_{\theta\theta}(\btheta,\bmu,\bSigma) - \bH_{-\theta}^{\star}(\btheta)\Big\|_{\infty} = \cOp(n\varepsilon_{nq}/q).
    \end{equation*}
    
\end{lemma}
\begin{proof}See Section~\ref{supp_sec_prove_lemma_laplace_approx_second_order}.\end{proof}

\section{Proof of Lemmas in the Main Text}\label{supp_sec_main_lemma}
\subsection{Proof of Lemma~\ref{lemma_neighbour_control}}\label{supp_sec_prove_lemma_neighbour_control}
In what follows, we establish that
\begin{align}
        \big\|\bu_i^{\star}(\btheta) - \bu_i^*\big\|& = \cOp\big(\sqrt{B_n}\kappa_n\epsilon\vee \kappa_n/\sqrt{q}\big),\text{ for each }i\in[n];\label{eq_lemma_neigh_average_control}\\
         \max_{i\in[n]}\big\|\bu_i^{\star}(\btheta) - \bu_i^*\big\| &= \cOp\Big(\sqrt{\log n}\big\{\sqrt{B_n}\kappa_n\epsilon\vee \kappa_n/\sqrt{q}\big\}\Big);\label{eq_lemma_neigh_unif_control}\\
         n^{-1}\sum_{i=1}^n\big\|\bu_i^{\star}(\btheta) - \bu_i^*\big\|^2 &=  \cOp\big(B_n\kappa^2_n\epsilon^2\vee \kappa_n^2/q\big).\label{eq_lemma_neigh_aver_control}
    \end{align} under the setting of Lemma~\ref{lemma_neighbour_control}. Note that $\ell_{ij}''(\eta)< 0$ by Assumption~\ref{assumption_smoothness} and that 
    \begin{equation}\begin{aligned}
        \lambda_{\min}\big\{\sum_{j=1}^q\bgamma_j\bgamma_j^\T\big\}\ge &\, \lambda_{\min}\big\{\sum_{j=1}^q\bgamma_j^*(\bgamma_j^*)^\T\big\} - \big\|\sum_{j=1}^q(\bgamma_j - \bgamma_j^*)(\bgamma_j^*)^\T\big\|-\big\|\sum_{j=1}^q\bgamma_j(\bgamma_j - \bgamma_j^*)^\T\big\| \\\ge &\,\lambda_{\min}(\bSigma_{\gamma}^*)q/2 - Cq\epsilon,
    \end{aligned}\label{eq_condition_gamma}\end{equation} when $q$ is large enough. Here, the first inequality follows from Weyl's inequality, and the second follows from Assumption~\ref{assumption_psd_covariance}, $\|\btheta\|_\infty\le C$, and $\|\btheta - \btheta^*\|\le \sqrt{q}\epsilon$. Within the controlled predictor region, \eqref{eq_condition_gamma} and Assumption~\ref{assumption_smoothness} make $\partial_{\bu}^2\ell_i(\bu;\btheta)$ negative definite, so any maximizer in that region is unique. The existence of the maximizer can be justified by the fact that $\ell_i(\bu;\btheta)$ is continuous and concave for $\btheta\in\Xi$.
    Next, we introduce the following lemma.
    \begin{lemma}\em
    Let $\Lb_{\eta} = \big\{\ell_{ij}^{\prime}(\eta_{ij}^*)\big\}_{n\times q}$ be an $n\times q$ matrix with the $(i,j)$th entry being $\ell_{ij}^{\prime}(\eta_{ij}^*)$. Denote its $i$th row and $j$th column as $\Lb_{\eta,i} = (\Lb_{\eta})_{i,}$ and $\Lb_{\eta,j} = (\Lb_{\eta})_{,j}$ , respectively. Under Assumptions\ref{assump_standard_id}--~\ref{assumption_smoothness}, on $\cE_{\eta}$, we have
    \begin{enumerate}[label=$(\roman*)$]
    \item for any $i\in[n]$, $\big\|(\Lb_{\eta})_{i,}\big\|=\cOp\big(\sqrt{qB_n}\big)$;
     $\max_{1\le i\le n}\big\|(\Lb_{\eta})_{i,}\big\| = \cOp(\sqrt{qB_n}+\sqrt{B_n}{\log n})$;
    \item for any $j\in[q]$, $\big\|(\Lb_{\eta})_{,j}\big\| = \cOp\big(\sqrt{nB_n}\big)$; $\max_{1\le j\le q}\big\|(\Lb_{\eta})_{,j}\big\| = \cOp\big(\sqrt{nB_n}+\sqrt{B_n}{\log q}\big)$;
        \item $\big\|\Lb_{\eta}\big\| = \cOp\big(\sqrt{(n\vee q)B_n}\big)$.
    \end{enumerate}
    Furthermore, consider any $(\btheta,\bu)$ satisfying, for some absolute constant $C$ and a sequence $\epsilon\to0$,
    \begin{equation}
       \|\btheta\|_{\infty}\le C,\;\|\bu\|_{\infty}\le C\sqrt{\log n} \text{ and }q^{-1/2}\big\|\btheta - \btheta^* \big\| + n^{-1/2}\big\|\bu - \bu^*\big\|\le \epsilon,\label{eq_feasible_theta}
    \end{equation} define $\eta_{ij} = \bgamma_j^\T\bu_i + \beta_j$ for each $i\in[n]$ and $j\in[q]$. Let $\Lb_{\eta}(\btheta,\bu) = \big\{\ell_{ij}^{\prime}(\eta_{ij})\big\}_{n\times q}$ be an $n\times q$ matrix with the $(i,j)$th entry being $\ell_{ij}^{\prime}(\eta_{ij})$. Then the following hold uniformly over $(\btheta,\bu)$ defined by \eqref{eq_feasible_theta}:
     \begin{enumerate}[label=$(\roman*)$, start=4]
        \item $\big\|\Lb_{\eta}(\btheta,\bu)\big\| = \cOp\big(\sqrt{(n\vee q)B_n} + B_n\sqrt{nq\log n}\epsilon\big)$.
\end{enumerate}
    \label{lemma_concentration_l_mat}
\end{lemma}
\begin{proof}
See Section~\ref{supp_sec_prove_lemma_concentration_l_mat}.
\end{proof}
Expanding $\sum_{j=1}^q\ell_{ij}^{\prime}(\bgamma_j^\T\bu_i^* + \beta_j)\bgamma_j$ around $\sum_{j=1}^q\ell_{ij}^{\prime}\{(\bgamma_j^*)^\T\bu_i^* + \beta_j^*\}\bgamma_j^*$ gives, for any $\btheta$ satisfying \eqref{eq_feasible_theta},
\begin{align}
    \sum_{j=1}^q\ell_{ij}^{\prime}\big(\bgamma_j^\T\bu_i^{*} + \beta_j\big)\bgamma_j = &\,\sum_{j=1}^q\ell_{ij}^{\prime}\big((\bgamma_j^*)^\T\bu_i^* + \beta_j^*\big)\bgamma_j^*\label{eq_bound_neighb_theta0}\\&\,+\sum_{j=1}^q\left\{\ell_{ij}^{\prime}\big(\bgamma_j^\T\bu_i^* + \beta_j\big) - \ell_{ij}^{\prime}\big((\bgamma_j^*)^\T\bu_i^* + \beta_j^*\big)\right\}\bgamma_j^*\label{eq_bound_neighb_theta1}\\&\,+\sum_{j=1}^q\ell_{ij}^{\prime}\big(\bgamma_j^\T\bu_i^* + \beta_j\big)\big(\bgamma_j - \bgamma_j^*\big).\label{eq_bound_neighb_theta2}
\end{align}
The argument used to prove Lemma~\ref{coro_first_order_concern}, based on Lemma~\ref{lemma_concentration}, gives $\|\eqref{eq_bound_neighb_theta0}\| = \cOp(\sqrt{qB_n})$. For \eqref{eq_bound_neighb_theta1}, we have
\begin{equation*}
    \big\|\eqref{eq_bound_neighb_theta1}\big\| = \Big\|\sum_{j=1}^q\ell_{ij}^{\prime\prime}(\tilde\bgamma_j^\T\bu_i^* + \tilde\beta_j )\big\{(\bgamma_j-\bgamma_j^*)^\T\bu_i^* + (\beta_j - \beta_j^*)\big\}\bgamma_j^*\Big\|
\end{equation*}
where $\tilde\bgamma_j$ lies on the segment between $\bgamma_j^*$ and $\bgamma_j$, and $\tilde\beta_j$ lies on the segment between $\beta_j^*$ and $\beta_j$. The parameter bounds imply $\max_{j\in[q]}\|\tilde\bgamma_j\|_{\infty}\le C$ and $\max_{j\in[q]}|\tilde\beta_j|\le C$. By Assumption~\ref{assumption_smoothness}, $|\ell_{ij}^{\prime\prime}(\tilde\bgamma_j^\T\bu_i^* + \tilde\beta_j)|\le B_n$, and therefore
\begin{align*}
    \big\|\eqref{eq_bound_neighb_theta1}\big\| \le &\,C\sqrt{q}B_n\left(\sum_{j=1}^q\|\bgamma_j-\bgamma_j^*\|^2\|\bu_i^*\|^2\right)^{1/2} + C\sqrt{q}B_n\left(\sum_{j=1}^q(\beta_j-\beta_j^*)^2\right)^{1/2}\\\le &\, \cOp\big(qB_n\epsilon\big).
\end{align*}
when $\|\btheta - \btheta^*\| \le \sqrt{q}\epsilon$. The first inequality follows from $\|\bgamma_j^*\|\le \sqrt{K}\|\bgamma_j^*\|_{\infty}\le C$ uniformly for $j\in[q]$ by Assumption~\ref{assumption_psd_covariance}; the second inequality follows from $\|\bu_i^*\|\le \cOp(1)$ by the sub-Gaussianity Assumption~\ref{assump_standard_id} and $\|\btheta - \btheta^*\|\le \sqrt{q}\epsilon$. For \eqref{eq_bound_neighb_theta2},
\begin{align*}
    \big\|\eqref{eq_bound_neighb_theta2}\big\| \le&\, \Big\|\sum_{j=1}^q\ell_{ij}^{\prime}\big((\bgamma_j^*)^\T\bu_i^* + \beta_j^*\big)\big(\bgamma_j - \bgamma_j^*\big)\Big\| \\&\,+\Big\| \sum_{j=1}^q\Big\{\ell_{ij}^{\prime}\big((\bgamma_j^*)^\T\bu_i^* + \beta_j^*\big)-\ell_{ij}^{\prime}\big(\bgamma_j^\T\bu_i^* + \beta_j\big)\Big\}\big(\bgamma_j - \bgamma_j^*\big)\Big\|\\\le &\,\big\|(\Lb_{\eta})_{i,}\big\|\big\|\btheta-\btheta^*\big\| + \cO(qB_n\epsilon).
\end{align*} The second inequality follows similarly from the bound derived above for \eqref{eq_bound_neighb_theta1}. 
Here $\Lb_{\eta}$ is as defined in Lemma~\ref{lemma_concentration_l_mat}, and by $(i)$ of this lemma and $\|\btheta - \btheta^*\|\le \sqrt{q}\epsilon$, we know the first term can be bounded by $\cOp(q\sqrt{B_n}\epsilon)$. Taking the above estimates together with $B_n\ge 1$, we have
$\big\|\sum_{j=1}^q\ell_{ij}^{\prime}\big(\bgamma_j^\T\bu_i^{*} + \beta_j\big)\bgamma_j\big\| = \cOp\big(qB_n\epsilon + \sqrt{qB_n}\big)$.
Next, since $\bu_i^{\star}(\btheta) = \arg\max_{\bu\in\RR^K}\ell_i(\bu;\btheta)$, we have $\sum_{j=1}^q\ell_{ij}^{\prime}(\bgamma_j^\T\bu_i^{\star}(\btheta)+\beta_j)\bgamma_j = \zero_K$ and consequently,
\begin{equation}
    \Big\|\sum_{j=1}^q\ell_{ij}^{\prime}\big(\bgamma_j^\T\bu_i^{*} + \beta_j\big)\bgamma_j - \sum_{j=1}^q\ell_{ij}^{\prime}(\bgamma_j^\T\bu_i^{\star}(\btheta)+\beta_j)\bgamma_j\Big\| = \cOp\big(qB_n\epsilon + \sqrt{qB_n}\big).\label{eq_bound_neighb_theta_der_diff}
\end{equation}
Next, apply the integral mean value theorem, and we have the expansion
\begin{align}
    &\,\sum_{j=1}^q\ell_{ij}^{\prime}\big(\bgamma_j^\T\bu_i^{*} + \beta_j\big)\bgamma_j - \sum_{j=1}^q\ell_{ij}^{\prime}(\bgamma_j^\T\bu_i^{\star}(\btheta)+\beta_j)\bgamma_j\nonumber \\&\,= \int_{0}^1\sum_{j=1}^q-\ell_{ij}^{\prime\prime}\big(\bgamma_j^\T[\bu_i^{*} + s\{\bu_i^{\star}(\btheta) - \bu_i^*\}] + \beta_j\big)\bgamma_j\bgamma_j^\T(\bu_i^{\star}(\btheta) - \bu_i^*)\,ds.\label{eq_bound_neighb_theta_2der_diff}
\end{align}
If $\|\bu_i^{\star}(\btheta)\|_{\infty}\le C\sqrt{\log n}$, Assumption~\ref{assumption_smoothness} implies that the negative second derivative along the segment in \eqref{eq_bound_neighb_theta_2der_diff} is bounded below by $b_n$. Hence
\begin{align}
    \|\eqref{eq_bound_neighb_theta_2der_diff}\|\ge &\,\lambda_{\min}\left[\int_{0}^1\sum_{j=1}^q-\ell_{ij}^{\prime\prime}\big(\bgamma_j^\T[\bu_i^{*} + s\{\bu_i^{\star}(\btheta) - \bu_i^*\}] + \beta_j\big)\bgamma_j\bgamma_j^\T\,ds\right]\big\|\bu_i^{\star}(\btheta) - \bu_i^*\big\|\nonumber\\
    \ge &\,b_nq\lambda_{\min}\big\{q^{-1}\sum_{j=1}^q\bgamma_j\bgamma_j^\T\big\} \big\|\bu_i^{\star}(\btheta) - \bu_i^*\big\|\nonumber\\
    \ge &\,b_nq\left[\lambda_{\min}\Big\{q^{-1}\sum_{j=1}^q\bgamma_j^*(\bgamma_j^*)^\T\Big\} - q^{-1}\Big\|\sum_{j=1}^q\bgamma_j^*(\bgamma_j^*)^\T-\sum_{j=1}^q\bgamma_j\bgamma_j^\T\Big\|\right] \big\|\bu_i^{\star}(\btheta) - \bu_i^*\big\|\nonumber\\
    \ge &\,\big\{b_nq\lambda_{\min}(\bSigma_{\gamma}^*)/4\big\}\big\|\bu_i^{\star}(\btheta) - \bu_i^*\big\|.\label{eq_bound_neighb_theta_2der_diff_derive}
\end{align}
for all sufficiently large $n$ and $q$. The third inequality follows because Assumption~\ref{assumption_psd_covariance} gives $\lambda_{\min}\big\{q^{-1}\sum_{j=1}^q\bgamma_j^*(\bgamma_j^*)^\T\big\}\ge\lambda_{\min}(\bSigma_{\gamma}^*)/2$ for sufficiently large $q$, while $q^{-1}\big\|\sum_{j=1}^q\bgamma_j^*(\bgamma_j^*)^\T-\sum_{j=1}^q\bgamma_j\bgamma_j^\T\big\| \le C\epsilon\le \lambda_{\min}(\bSigma_{\gamma}^*)/4$ follows from $\|\btheta-\btheta^*\|\le \sqrt{q}\epsilon$ and $\epsilon=o(1)$ for some absolute constant $C > 0$. Combining this result with \eqref{eq_bound_neighb_theta_der_diff} gives
\begin{equation*}
    \big\|\bu_i^{\star}(\btheta) - \bu_i^*\big\|\le \cOp\big((qB_n\epsilon + \sqrt{qB_n})/{(qb_n)}\big),
\end{equation*}
which proves \eqref{eq_lemma_neigh_average_control}.
We next establish the required uniform bound on $\|\bu_i^{\star}(\btheta)\|_{\infty}$. Let $C_0>0$ satisfy $\max_{i\in[n]}\|\bu_i^*\|_{\infty}\le C_0\sqrt{\log n}$ with probability tending to one, and fix $M>2C_0$. Suppose, to obtain a contradiction, that $\|\bu_i^{\star}(\btheta)\|_{\infty}>M\sqrt{\log n}$ for some $i$. Let $\bu_i^{\star,M}(\btheta)$ be the unique intersection of the boundary $\{\bu:\|\bu\|_{\infty}=M\sqrt{\log n}\}$ with the segment from $\bu_i^*$ to $\bu_i^{\star}(\btheta)$. Define the unit vector
\begin{equation*}
    \bv_i=
    \frac{\bu_i^{\star,M}(\btheta)-\bu_i^*}{\|\bu_i^{\star,M}(\btheta)-\bu_i^*\|}.
\end{equation*} Left-multiplying the integral in \eqref{eq_bound_neighb_theta_2der_diff} by $\bv_i^\T$, we split the result into two parts:
\begin{align*}
    \bv_i^\T\eqref{eq_bound_neighb_theta_2der_diff} &=\bv_i^\T\int_{0}^1\sum_{j=1}^q-\ell_{ij}^{\prime\prime}\big(\bgamma_j^\T\bu_i^{*} + s\bgamma_j^\T\{\bu_i^{\star,M}(\btheta) - \bu_i^*\} + \beta_j\big)\bgamma_j\bgamma_j^\T(\bu_i^{\star,M}(\btheta) - \bu_i^*)\,ds \\&+ \bv_i^\T\int_{0}^1\sum_{j=1}^q-\ell_{ij}^{\prime\prime}\big(\bgamma_j^\T\bu_i^{\star,M}(\btheta) + s\bgamma_j^\T\{\bu_i^{\star}(\btheta) - \bu_i^{\star,M}(\btheta)\} + \beta_j\big)\bgamma_j\bgamma_j^\T(\bu_i^{\star}(\btheta) - \bu_i^{\star,M}(\btheta))\,ds,
\end{align*}
where $\bu_i^{\star,M}(\theta)$ is the intersection of $\{\bu:\|\bu\|_{\infty}\le M\sqrt{\log n}\}$ and the segment between $\bu_i^{\star}(\theta)$ and $\bu_i^*$. Since the second term is the quadratic form of a positive definite matrix and $\bv_i$ is unit vector, we know
\begin{equation*}
    \|\eqref{eq_bound_neighb_theta_2der_diff}\|\ge \bv_i^\T\int_{0}^1\sum_{j=1}^q-\ell_{ij}^{\prime\prime}\big(\bgamma_j^\T[\bu_i^{*} + s\{\bu_i^{\star,M}(\btheta) - \bu_i^*\}] + \beta_j\big)\bgamma_j\bgamma_j^\T(\bu_i^{\star,M}(\btheta) - \bu_i^*)\,ds.
\end{equation*}
Since $\big\|\bu_i^{\star,M}(\btheta)-\bu_i^*\big\|_{\infty} \ge (M-C_0)\sqrt{\log n}\ge (M/2)\sqrt{\log n}$. An argument similar to the derivation of \eqref{eq_bound_neighb_theta_2der_diff_derive} therefore gives
\begin{equation*}
    \|\eqref{eq_bound_neighb_theta_2der_diff}\|\ge Cb_nq\lambda_{\min}(\bSigma_{\gamma}^*)\sqrt{\log n},
\end{equation*}
which contradicts $\|\eqref{eq_bound_neighb_theta_2der_diff}\| = \|\eqref{eq_bound_neighb_theta_der_diff}\|$ as $\epsilon \ll b_n/(B_n\sqrt{\log n})$.

To establish \eqref{eq_lemma_neigh_unif_control}, similar to deriving \eqref{eq_bound_neighb_theta_der_diff}, one can prove with Lemma~\ref{coro_first_order_concern} that
\begin{align}
    \max_{1\le i\le n}\Big\|\sum_{j=1}^q\ell_{ij}^{\prime}\big(\bgamma_j^\T\bu_i^{*} + \beta_j\big)\bgamma_j - \sum_{j=1}^q\ell_{ij}^{\prime}(\bgamma_j^\T\bu_i^{\star}(\btheta)+\beta_j)\bgamma_j\Big\| = \cOp\big(qB_n\epsilon\sqrt{\log n} + \sqrt{qB_n\log n}\big).
\end{align}
The integral mean value theorem gives
\begin{align}
    &\,\max_{1\le i\le n}\Big\|\sum_{j=1}^q\ell_{ij}^{\prime}\big(\bgamma_j^\T\bu_i^{*} + \beta_j\big)\bgamma_j - \sum_{j=1}^q\ell_{ij}^{\prime}(\bgamma_j^\T\bu_i^{\star}(\btheta)+\beta_j)\bgamma_j\Big\|\nonumber \\&\,= \max_{1\le i\le n}\left\|\int_{0}^1\sum_{j=1}^q-\ell_{ij}^{\prime\prime}\big(\bgamma_j^\T[\bu_i^{*} + s\{\bu_i^{\star}(\btheta) - \bu_i^*\}] + \beta_j\big)\bgamma_j\bgamma_j^\T(\bu_i^{\star}(\btheta) - \bu_i^*)\,ds\right\|.
\end{align}
Using the same argument as for \eqref{eq_bound_neighb_theta_2der_diff}, we obtain
\begin{equation}\label{eq_condition_u}
\max_{1\le i\le n}\big\|\bu_i^{\star}(\btheta) - \bu_i^*\big\| = \cOp\Big(\sqrt{\log n}\Big\{\frac{B_n\epsilon}{b_n}\vee \frac{\sqrt{B_n}}{\sqrt{q}b_n}\Big\}\Big)=o_{\PP}(1),
\end{equation}
where the last inequality follows from $\sqrt{\log n}\big\{\frac{B_n\epsilon}{b_n}\vee \frac{\sqrt{B_n}}{\sqrt{q}b_n}\big\}\ll 1$ by Assumption~\ref{assump_scaling}. Note that the first inequality, appearing also in Lemma~\ref{lemma_neighbour_control}, does not require Assumption~\ref{assump_scaling}.

    Finally, to bound \eqref{eq_lemma_neigh_aver_control}, Lemma~\ref{coro_first_order_concern} gives $\sum_{i=1}^n\|\sum_{j=1}^q\ell_{ij}^{\prime}\{(\bgamma_j^*)^\T\bu_i^* + \beta_j^*\}\bgamma_j^*\|^2 = \cOp(nqB_n)$. As in the derivation of \eqref{eq_bound_neighb_theta_der_diff}, we can bound $\sum_{i=1}^n\|\sum_{j=1}^q\ell_{ij}^{\prime}(\bgamma_j^\T\bu_i^{*} + \beta_j)\bgamma_j\|^2$ by $\cOp(nqB_n\vee nq^2B_n^2\epsilon^2)$. Moreover, \eqref{eq_condition_u} and Assumptions~\ref{assumption_psd_covariance}--\ref{assumption_smoothness} imply that $-\ell_{ij}^{\prime\prime}(\bgamma_j^\T[\bu_i^{*} + s\{\bu_i^{\star}(\btheta) - \bu_i^*\}] + \beta_j)$ is bounded below by $b_n$. Therefore, \eqref{eq_bound_neighb_theta_2der_diff} and the argument for \eqref{eq_bound_neighb_theta_2der_diff_derive} yield
\begin{equation*}
    n^{-1}\sum_{i=1}^n\|\bu_i^{\star}(\btheta) - \bu_i^*\|^2\le \cOp\big(n^{-1}\times q^{-2}\times b_n^{-2}(nqB_n\vee nq^2B_n^2\epsilon^2)\big) = \cOp\Big(\frac{\epsilon^2B_n^2}{b_n^2}\vee \frac{B_n}{qb_n^2}\Big).
\end{equation*}
    Finally, Weyl's inequality gives
    \begin{equation*}
        q^{-1/2}\sigma_{K+1}(\bTheta)
        \ge q^{-1/2}\sigma_{K+1}(\bTheta^*)-q^{-1/2}\|\bTheta-\bTheta^*\|_{\Fn}
        \ge \sqrt{\lambda_{\min}(\bSigma_\theta^*)}-\epsilon+o(1).
    \end{equation*}
    Combining this bound with \eqref{eq_condition_u}, the assumed $\ell_\infty$ bound on $\btheta$, and the choices in Remark~\ref{remark_explain_xi} proves $\btheta\in\Xi$ with probability tending to one.

\subsection{Proof of Lemma~\ref{lemma_convexity}}

    We establish the upper bound for $\lambda_{\min}\big\{n^{-1}\cH_{\theta\theta}(\btheta^*)\big\} $ first. For any $r\in\{0\}\cup [K]$ and $l\in[K]$, let
    \begin{equation}
        \bv_{rl}(\btheta) = \left(\gamma_{1l}\big(\be_{r+1}^{(K+1)}\big)^\T, \dots, \gamma_{ql}\big(\be_{r+1}^{(K+1)}\big)^\T\right)^\T \in \RR^{qK+q}.\label{eq_define_v_pq}
    \end{equation}
   Let $\{\bu_i^{\star}(\btheta)\}_{i=1}^n$ and $\bH_{-\theta}^{\star}(\btheta)$ be as in Lemma~\ref{lemma_laplace_approx_second_order}, and set $\eta_{ij}^{\star}(\btheta) = \bgamma_j^\T\bu_i^{\star}(\btheta) + \beta_j$. Simple algebra yields
    \begin{align*}
       &\,\big\{\bv_{rl}(\btheta)\big\}^\T \bH_{-\theta}^{\star}(\btheta) \bv_{rl}(\btheta) \\= &\,\sum_{j=1}^q\gamma_{jl}\big(\be_{r+1}^{(K+1)}\big)^\T\Big\{-\sum_{i=1}^n\ell_{ij}^{\prime\prime}(\eta_{ij}^{\star}(\btheta))\begingroup \renewcommand{\arraystretch}{0.5}\begin{pmatrix}
        1&\big\{\bu_i^{\star}(\btheta)\big\}^\T\\\bu_i^{\star}(\btheta)&\bu_i^{\star}(\btheta)\big\{\bu_i^{\star}(\btheta)\big\}^\T
    \end{pmatrix}\endgroup\Big\}\be_{r+1}^{(K+1)}\gamma_{jl}\\&\, -\left[\sum_{j=1}^q\gamma_{jl}\left\{\bH_{Lu\theta}\big(\btheta,\bu^{\star}(\btheta)\big)\right\}_{,(j-1)(K+1) + r+1}\right]^\T\left\{\bH_{Lu u}\big(\btheta,\bu^{\star}(\btheta)\big)\right\}^{-1}\\&\qquad\times\left[\sum_{j=1}^q\gamma_{jl}\left\{\bH_{Lu\theta}\big(\btheta,\bu^{\star}(\btheta)\big)\right\}_{,(j-1)(K+1) + r+1}\right]\\
    = &\,-\sum_{j=1}^q\sum_{i=1}^n\gamma_{jl}^2\ell_{ij}^{\prime\prime}(\eta_{ij}^{\star}(\btheta))\big\{u_{ir}^{\star}(\btheta)\big\}^2 \quad (:=\alpha_1)\\&\, + \sum_{i=1}^n\left\{\Big(\sum_{j=1}^q\gamma_{jl}\ell_{ij}^{\prime\prime}(\eta_{ij}^{\star}(\btheta))u_{ir}^{\star}(\btheta)\bgamma_j + \gamma_{jl}\ell_{ij}^{\prime}(\eta_{ij}^{\star}(\btheta))1_{(r\ge 1)}\be_{r}^{(K)}\Big)\right\}^\T\left\{\sum_{s=1}^q\ell_{is}^{\prime\prime}(\eta_{is}^{\star}(\btheta))\bgamma_s\bgamma_s^\T\right\}^{-1}\\&\,\qquad\qquad\times \left\{\sum_{j=1}^q\Big(\gamma_{jl}\ell_{ij}^{\prime\prime}(\eta_{ij}^{\star}(\btheta))u_{ir}^{\star}(\btheta)\bgamma_j + \gamma_{jl}\ell_{ij}^{\prime}(\eta_{ij}^{\star}(\btheta))1_{(r\ge 1)}\be_{r}^{(K)}\Big)\right\}\quad (:=\alpha_2).
    \end{align*}
    Here we abuse the notation to denote $u_{i0}^{\star}(\btheta) = 1$. For $\alpha_2$, since
    \begin{align*}
        &\,\left\{\sum_{s=1}^q\ell_{is}^{\prime\prime}(\eta_{is}^{\star}(\btheta))\bgamma_s\bgamma_s^\T\right\}^{-1} \sum_{j=1}^q\gamma_{jl}\ell_{ij}^{\prime\prime}(\eta_{ij}^{\star}(\btheta))u_{ir}^{\star}(\btheta)\bgamma_j\\ =&\, \left\{\sum_{j=1}^q\ell_{ij}^{\prime\prime}(\eta_{ij}^{\star}(\btheta))\bgamma_j\bgamma_j^\T\right\}^{-1} \left\{\sum_{j=1}^q\bgamma_{j}\bgamma_j^\T\ell_{ij}^{\prime\prime}(\eta_{ij}^{\star}(\btheta))\right\}\be_{l}^{(K)}u_{ir}^{\star}(\btheta) =\be_l^{(K)} u_{ir}^{\star}(\btheta),
    \end{align*}
    it can be simplified to
    \begin{equation}
    \begin{aligned}
        \alpha_2 = &\,\underbrace{\sum_{i=1}^n\sum_{j=1}^q\gamma_{jl}\ell_{ij}^{\prime\prime}(\eta_{ij}^{\star}(\btheta))u_{ir}^{\star}(\btheta)\bgamma_j^\T\be_l^{(K)} u_{ir}^{\star}(\btheta)}_{\beta_1} + \underbrace{2\sum_{i=1}^n\sum_{j=1}^q\gamma_{jl}\ell_{ij}^{\prime}(\eta_{ij}^{\star}(\btheta))u_{ir}^{\star}(\btheta)1_{(r\ge 1)}(\be_r^{(K)})^\T\be_l^{(K)}}_{\beta_2} \\&\, + \underbrace{\sum_{i=1}^n\left\{\sum_{j=1}^q\gamma_{jl}\ell_{ij}^{\prime}(\eta_{ij}^{\star}(\btheta))1_{(r\ge 1)}(\be_{r}^{(K)})^\T\right\}\left\{\sum_{s=1}^q\ell_{is}^{\prime\prime}(\eta_{is}^{\star}(\btheta))\bgamma_s\bgamma_s^\T\right\}^{-1}\left\{\sum_{j=1}^q\gamma_{jl}\ell_{ij}^{\prime}(\eta_{ij}^{\star}(\btheta))1_{(r\ge 1)}\be_{r}^{(K)}\right\}}_{\beta_3}.
    \end{aligned}\label{eq_aux_bound_imply_null_space}\end{equation}
    Observe that $\beta_1 +\alpha_1 =0$ and $\beta_2 = \beta_3 = 0$ when $r = 0$. Then together with Lemma~\ref{lemma_laplace_approx_second_order}, we know
    \begin{equation}
        \lambda_{\min}\big\{n^{-1}\cH_{\theta\theta}(\btheta^*)\big\} \le \lambda_{\min}\big\{n^{-1}\bH_{-\theta}^{\star}(\btheta^*)\big\} + \big\|n^{-1}\cH_{\theta\theta}(\btheta^*) - n^{-1}\bH_{-\theta}^{\star}(\btheta^*)\big\|\le \cOp\big(\varepsilon_{nq}q^{-1}\big).\label{eq_hessian_true_upper}
    \end{equation}
    When $r\ge 1$ and $\btheta \in \cB_{\epsilon}(\btheta^*)$ for some sufficiently small $\epsilon>0$, it can be shown that $|\beta_2| = \cOp\big(\sqrt{nqB_n} + \sqrt{\epsilon}\big)$ and $|\beta_3|=\cOp(nB_n/b_n)$. These rates are no smaller than those in the case $r=0$. They reflect the diminishing directions of $\cH_{\theta\theta}(\btheta^*)$, which also appear in the lower bound in \eqref{eq_hessian_lower_bound}.
    We next introduce the following lemma, which establishes two lower bounds.


    \begin{lemma}\em\label{lemma_convexity2}
        Suppose $\bv_{rl}(\btheta)$ is as defined in \eqref{eq_define_v_pq} and $\bH_{-\theta}^{\star}(\btheta)$ is as defined in Lemma~\ref{lemma_laplace_approx_second_order}. Under Assumptions~\ref{assump_standard_id}--\ref{assump_scaling}, for some $\epsilon\ll (\log n)^{-1}\wedge b_n^3/(B_n^3\sqrt{\log n})$, there exist absolute constants $\gamma,c>0$ such that
        \begin{equation*}
            \PP\left(\min_{\btheta\in\cB_{\epsilon}(\btheta^*)}\lambda_{\min}\Big[n^{-1}\bH_{-\theta}^{\star}(\btheta) + cb_nq^{-1}\sum_{r=0}^{K}\sum_{l=1}^K\bv_{rl}(\btheta)\big\{\bv_{rl}(\btheta)\big\}^\T\Big]\ge \gamma b_n\right)\to 1,\text{ as }n,q\to\infty.
        \end{equation*}
        In addition, the result holds when replacing $\bv_{rl}(\btheta)\big\{\bv_{rl}(\btheta)\big\}^\T$ with $\bv_{rl}(\btheta^*)\big\{\bv_{rl}(\btheta^*)\big\}^\T$. Moreover, 
         $\lambda_{\min}\{n^{-1}\cH_{\theta\theta}(\btheta^*)\}\ge -\cOp(B_n^{5/2}b_n^{-2}(n\wedge q)^{-1/2})$. 
    \end{lemma}
    \begin{proof}See Section~\ref{supp_sec_prove_lemma_convexity2}.\end{proof}
    The last part of Lemma~\ref{lemma_convexity2} together with \eqref{eq_hessian_true_upper} establishes $\big|\lambda_{\min}\{n^{-1}\cH_{\theta\theta}(\btheta^*)\}\big| = \cOp(B_n^{5/2}b_n^{-2}(n\wedge q)^{-1/2})$.
    Now we establish the lower bound \eqref{eq_hessian_lower_bound}. Let $\bV =(\bbeta^v,\bGamma^v)^\T \in\RR^{(K+1)\times q}$ and $\bv = \mathrm{vec}\big((\bbeta^v,\bGamma^v)^\T\big)$. Simple algebra gives
    \begin{align*}
        \bv^\T\sum_{r=0}^{K}\sum_{l=1}^K\bv_{rl}(\btheta^*)\big\{\bv_{rl}(\btheta^*)\big\}^\T\bv =&\, \sum_{l=1}^K\left(\sum_{j=1}^q\beta_j^v\gamma_{jl}^*\right)^2 + \sum_{r=1}^{K}\sum_{l=1}^K\left(\sum_{j=1}^q\gamma_{jr}^v\gamma_{jl}^*\right)^2\\=&\, \big\|(\bbeta^v)^\T\bGamma^*\big\|^2 + \big\|(\bGamma^v)^\T\bGamma^*\big\|_{\Fn}^2.
    \end{align*}
    Consequently, replacing $\bv_{rl}(\btheta)\big\{\bv_{rl}(\btheta)\big\}^\T$ with $\bv_{rl}(\btheta^*)\big\{\bv_{rl}(\btheta^*)\big\}^\T$ in Lemma~\ref{lemma_convexity2} yields that, when $n,q\to\infty$,
\begin{equation*}
            \PP\left(\min_{\substack{\btheta\in\cB_{\epsilon}(\btheta^*)\\\bv\in\RR^{qK+q},\|\bv\|=1}}n^{-1}\bv^\T\bH_{-\theta}^{\star}(\btheta) \bv + cb_nq^{-1}\big\|(\bbeta^v)^\T\bGamma^*\big\|^2 + cb_nq^{-1}\big\|(\bGamma^v)^\T\bGamma^*\big\|_{\Fn}^2\ge \gamma b_n \right)\to 1,
        \end{equation*}
        for any $\bV$ given in Lemma~\ref{lemma_convexity}. Next, invoke Lemma~\ref{lemma_laplace_approx_second_order} and Weyl's theorem to obtain that for any $\bv\in\RR^{qK+q}$, 
        \begin{equation*}
            n^{-1}\left|\bv^\T\cH_{\theta\theta}(\btheta) \bv- \bv^\T\bH_{-\theta}^{\star}(\btheta)\bv\right| = \cOp(\varepsilon_{nq}q^{-1}\|\bv\|^2).
        \end{equation*}
        The proof follows because the approximation error is $o_p(b_n\|\bv\|^2)$ under Assumption~\ref{assump_scaling}.

\subsection{Proof of Lemma~\ref{lemma_key_first_order}}\label{supp_sec_prove_lemma_key_first_order}
The following result allows us to localize the analysis but does not provide uniform control of the estimator.
\begin{lemma}
    \em\label{lemma_average_converge_eta} Suppose Assumptions~\ref{assump_standard_id}--\ref{assump_scaling} hold. For any $\hat\btheta\in\Xi$ and $(\bmu,\bSigma)\in\bar\cK_\rho$ for $\rho<1$, if $\cL(\hat\btheta,\bmu,\bSigma) \le \cL(\btheta^*,\bmu,\bSigma)$, then it holds that
    \begin{equation*}
        \sum_{i=1}^n\sum_{j=1}^q\left(\hat\bgamma_j^\T\bu_i^{\star}(\hat\btheta) + \hat\beta_j - (\bgamma_j^*)^\T\bu_i^{\star}(\btheta^*) - \beta_j^*\right)^2 = \cOp\big\{(n\vee q)\log (n\vee q)\kappa_n^2B_n^2/b_n^2 \big\},
    \end{equation*}
    where $\bu_i^{\star}(\btheta) = \argmax_{\bu\in\RR^{K}}\ell_i(\bu;\btheta)$.
\end{lemma}
\begin{proof}See Section~\ref{supp_sec_prove_lemma_average_converge_eta}.\end{proof}

\noindent\underline{\large \it Step 0: Setup. }
Let $\cK_C = \{(\bmu,\bSigma):\|\bmu\|\le C,\; C^{-1}\bI_K\preceq \bSigma\preceq C\bI_K\}$ for $C > 1$ as defined in \eqref{eq_define_KC}. Recall that in Section~\ref{subsec_first_order}, we have constructed
\begin{equation}
    (\hat\btheta^*,\hat\bmu^*,\hat\bSigma^*) \in \mathop{\arg\min}_{\btheta\in \cB_{\epsilon}(\btheta^*),(\bmu,\bSigma)\in\cK_C}\cL(\btheta,\bmu,\bSigma) + P^*(\btheta),\label{eq_aux_estimator}
\end{equation}
for some $\epsilon$ specified below in \eqref{eq_key_scaling_equation_epsilon}, where $P^*(\btheta)$ in \eqref{eq_define_p_star_illus} can be equivalently written as
\begin{equation}
    P^* ( \btheta) = c_P n q\Big\|q^{-1}\sum_{j=1}^q\bgamma_j^*\btheta_j^\T - q^{-1}\sum_{j=1}^q\bgamma_j^*(\btheta^*_j)^\T\Big\|_{\Fn}^2,\label{eq_define_p_star}
\end{equation}where $c_P\asymp b_n$ is chosen to match the penalty scale in Lemma~\ref{lemma_convexity2}. We do not index the estimator by $\epsilon$ because, as shown in Lemma~\ref{lemma_aux_estimator_property}, it does not depend on the choice of $\epsilon$ as long as
\begin{equation}
    \frac{\varepsilon_{nq}}{\sqrt{B_n(n\wedge q)}}\ll\epsilon\ll (\log n)^{-1}\wedge\frac{b_n^3}{B_n^3\sqrt{\log n}}.\label{eq_key_scaling_equation_epsilon}
\end{equation}
This interval is nonempty under Assumption~\ref{assump_scaling}.
Write $\cQ^*(\btheta,\bmu,\bSigma) = \cL(\btheta,\bmu,\bSigma) + P^*(\btheta)$. 
The derivatives of $P^*(\btheta)$ are
\begin{align}
    &\partial_{\btheta}P^*(\btheta) = {2c_Pn}\sum_{l=1}^K\big(q^{-1}\sum_{j=1}^q\gamma_{jl}^*\beta_j-M_{l0}^*\big)\bv_{0l}^*+{2c_Pn}\sum_{r=1}^K\sum_{l=1}^K\big(q^{-1}\sum_{j=1}^q\gamma_{jl}^*\gamma_{jr}-M_{lr}^*\big)\bv_{rl}^*;\label{eq_first_order_p_star}\\
    &\partial_{\btheta\btheta}^2P^*(\btheta) = 2c_Pnq^{-1}\sum_{r=0}^K\sum_{l=1}^K\bv_{rl}^*(\bv_{rl}^*)^\T,\nonumber
\end{align}
where $\bM^* =(M_{rl}^*)_{K\times(K+1)}= q^{-1}\sum_{j=1}^q\bgamma_j^*(\btheta^*_j)^\T$ and $\bv^{*}_{rl} = \bv_{rl}(\btheta^*)$ is given as
\begin{equation*}
    \bv_{rl}(\btheta) = \left(\gamma_{1l}\big(\be_{r+1}^{(K+1)}\big)^\T, \dots, \gamma_{ql}\big(\be_{r+1}^{(K+1)}\big)^\T\right)^\T \in \RR^{qK+q}.
\end{equation*}
\noindent\underline{\large \it Step 1: Theoretical Analysis of $\hat\btheta^*$. }
The first advantage of the constructed objective function follows directly from Lemmas~\ref{lemma_laplace_approx_second_order} and~\ref{lemma_convexity2}. Specifically, whenever $\btheta \in\cB_{\epsilon}(\btheta^*)$, for $c$ specified in Lemma~\ref{lemma_convexity2}, it holds with probability approaching $1$ that
\begin{equation}\begin{aligned}
    \lambda_{\min}\left\{\partial_{\btheta\btheta}^2\cQ^*(\btheta,\bmu,\bSigma) \right\}\ge &\, \lambda_{\min}\left\{\bH_{-\theta}^{\star}(\btheta) + \partial_{\btheta\btheta}^2 P^*(\btheta)\right\} -\left\|\cH_{\theta\theta}(\btheta,\bmu,\bSigma) - \bH_{-\theta}^{\star}(\btheta)\right\|\\\ge &\,nb_n\gamma - o_p(nb_n) \ge \gamma nb_n/2,
\end{aligned}\label{eq_convex_aux_obj}
\end{equation}
where the second inequality follows from Lemmas~\ref{lemma_laplace_approx_second_order} and~\ref{lemma_convexity2}.
The next lemma establishes that $\hat{\btheta}^*$ converges to $\btheta^*$ at a near-optimal rate, ensuring that it lies in the interior of the feasible set of problem \eqref{eq_aux_estimator}.
\begin{lemma}\label{lemma_aux_estimator_property}\em
    {For $r_{\theta,nq}$ defined in Equation~\eqref{eq_define_r_theta_nq}, choose $\epsilon$ to satisfy \eqref{eq_key_scaling_equation_epsilon} and let}
    \begin{equation*}
        \delta_{nq} \; := \; \frac{\kappa_n^2r_{\theta,nq}}{\sqrt{n}} \; + \; \frac{\kappa_n}{\sqrt{q}}. 
    \end{equation*}
    Given $(\hat\btheta^*, \hat\bmu^*,\hat\bSigma^*)$ in \eqref{eq_aux_estimator} under Assumptions~\ref{assump_standard_id}--\ref{assump_scaling}, one has
    \begin{enumerate}[label=$(\roman*)$]
        \item average consistency for $\hat\btheta^*$:
        \begin{equation*}
        q^{-1/2}\big\|\hat\btheta^* - \btheta^*\big\| = \cOp(b_n^{-1}n^{-1/2} r_{\theta,nq});\end{equation*}
        \item consistency for $\hat\bmu^*$ and $\hat\bSigma^*$:\begin{equation*}\big\|\hat\bmu^*\big\| = \cOp(\delta_{nq})\quad\text{ and }\quad\big\|\hat\bSigma^* - \bSigma_u^*\big\| = \cOp(\delta_{nq});
        \end{equation*}
        where $\delta_{nq} = \frac{\kappa_n^2r_{\theta,nq}}{\sqrt{n}} + \frac{\kappa_n}{\sqrt{q}}$;
        \item uniform consistency for $\hat\btheta^*$: \begin{equation*} \big\|\hat\btheta^* - \btheta^*\big\|_{\infty} = \cOp\left[\frac{r_{\theta,nq}}{\sqrt{n}}\left\{\frac{\sqrt{\log q}}{b_n} + \frac{B_n^2}{b_n^3}\right\}\right];\end{equation*}
        \item first-order conditions: $\cS_{\theta}(\hat\btheta^*, \hat\bmu^*,\hat\bSigma^*) = \zero$, $\cS_{\mu}(\hat\btheta^*, \hat\bmu^*,\hat\bSigma^*) = \zero$, and $\cS_{\Sigma}(\hat\btheta^*, \hat\bmu^*,\hat\bSigma^*) = \zero$ with probability approaching $1$ as $n,q\to\infty$.
    \end{enumerate}
\end{lemma}
\begin{proof}
    The proof of this lemma does not require the empirical conditions $n^{-1}\sum_{i=1}^n\bu_i^* = \zero_K$ or $n^{-1}\sum_{i=1}^n\bu_i^*(\bu_i^*)^\T  = \bI_K$. Instead, it uses $\EE[\bU_i] = \zero_K$ and $\EE[\bU_i\bU_i^\T] = \bSigma_u^*$. Thus, the result extends to estimation under the practical identifiability conditions. See the complete proof in Section~\ref{supp_sec_prove_lemma_aux_estimator_property}.
\end{proof}
\noindent\underline{\large \it Step 2: Link $\hat\btheta^*$ with $\hat\btheta$. }In this step and Step~3, we work with the canonical representative introduced before Theorem~\ref{thm_asymp}. This does not affect the consistency result because the difference between the canonical and original true parameter sets is of order $n^{-1/2}$. For simplicity, we use $(\btheta^*,\bu^*)$ to denote the canonical parameters and assume that
\begin{equation}
    \label{eq_id_st}n^{-1}\sum_{i=1}^n\bu_i^*= \zero_K,\qquad n^{-1}\sum_{i=1}^n\bu_i^*(\bu_i^*)^\T=\bI_K.
\end{equation}
This does not affect our analysis because we can always perform an additional linear whitening transformation, which does not change the derivation of the first-order condition, as will be shown in Step 3 below. The error introduced by the transformation is of order $\cOp(n^{-1/2})$ and is therefore negligible in establishing consistency.
Now we establish the relationship between $(\hat\btheta^*, \hat\bmu^*,\hat\bSigma^*)$ and $\hat\btheta$. First, we show that $P^*(\hat\btheta^*) = 0$. By the average consistency result $(i)$ in Lemma~\ref{lemma_aux_estimator_property} and $\|\btheta^*\|\le C\sqrt{q}$, one has
\begin{equation}
    \Big\|q^{-1}\sum_{j=1}^q\bgamma_j^*(\hat\btheta_j^*)^\T - q^{-1}\sum_{j=1}^q\bgamma_j^*(\btheta_j^*)^\T\Big\|  = \cOp( r_{\theta,nq}b_n^{-1}n^{-1/2}).\label{eq_small_p_Star_hat_theta}
\end{equation}
Suppose, for contradiction, that $P^*(\hat\btheta^*) = \epsilon_P>0$. We show that there exists $\tilde\btheta^*\in\cB_{\epsilon}(\btheta^*)$ attaining a smaller value of $\cQ^*(\tilde\btheta^*,\tilde\bmu^*,\tilde\bSigma^*)$ for suitable $\tilde\bmu^*$ and $\tilde\bSigma^*$.

Let \begin{equation*}
    \tilde\bG^* = \Big\{\sum_{j=1}^q\bgamma_j^*(\hat\bgamma_j^*)^\T\Big\}^{-1}\Big\{\sum_{j=1}^q\bgamma_j^*(\bgamma_j^*)^\T\Big\},\;\tilde\bd^* = -\Big\{\sum_{j=1}^q\bgamma_j^*(\hat\bgamma_j^*)^\T\Big\}^{-1}\Big\{\sum_{j=1}^q\bgamma_j^*(\hat\beta_j^* - \beta_j^*)\Big\},
\end{equation*}
    and $\tilde\btheta^* = (\tilde\beta_1^*, (\tilde\bgamma_1^*)^\T,\dots,\tilde\beta_q^*, (\tilde\bgamma_q^*)^\T)^\T$ with
\begin{equation*}
    \tilde\bgamma_j^* = (\tilde\bG^*)^\T\hat\bgamma_j^*\qquad\text{ and }\qquad\tilde\beta_j^* = \hat\beta_j^* + (\hat\bgamma_j^*)^\T\tilde\bd^*.
\end{equation*}
When $P^*(\hat\btheta^*)\neq 0$, $\tilde\bG^*\neq \bI_K$ or $\tilde\bd^*\neq \zero_K$, implying that $\hat\btheta^* \neq \tilde\btheta^*$. By construction,
\begin{equation}
    P^*(\tilde\btheta^*) = 0\label{eq_P_star_exact_0}.
\end{equation}
By Assumption~\ref{assumption_psd_covariance}, $\lambda_{\min}\big\{q^{-1}\sum_{j=1}^q\bgamma_j^*(\bgamma_j^*)^\T\big\}\ge \lambda_{\min}(\bSigma_\gamma^*)/2>0$ when $q$ is large enough. Then by \eqref{eq_small_p_Star_hat_theta} and Weyl's theorem, one has $\sigma_{K}\big\{q^{-1}\sum_{j=1}^q\bgamma_j^*(\hat\bgamma_j^*)^\T\big\}\ge \lambda_{\min}(\bSigma_\gamma^*)/2 - \cOp( r_{\theta,nq}b_n^{-1}n^{-1/2})$, and thus $\tilde\bG^*$ is well defined with probability approaching $1$, with
\begin{align*}
    \big\|\tilde\bG^* - \bI_K\big\| \le \Big\|\Big\{q^{-1}\sum_{j=1}^q\bgamma_j^*(\hat\bgamma_j^*)^\T\Big\}^{-1}\Big\|\Big\|\Big\{q^{-1}\sum_{j=1}^q\bgamma_j^*(\bgamma_j^* - \hat\bgamma_j^*)^\T\Big\}\Big\| = \cOp( r_{\theta,nq}/(b_n\sqrt{n})),
\end{align*}
where the equality follows from Cauchy's inequality, $\|\btheta^*\|\le C\sqrt{q}$, and \eqref{eq_small_p_Star_hat_theta}. Similarly, $\big\|\tilde\bd^*\big\| = \cOp\{r_{\theta,nq}/(b_n\sqrt{n})\}$.
Because $\hat\btheta^*\in\Xi$, direct calculation gives
\begin{align*}
    q^{-1}\sum_{j=1}^q\big\|\tilde\btheta_j^* - \btheta_j^*\big\|^2 &\le 2q^{-1}\sum_{j=1}^q\big\|\hat\btheta_j^* - \btheta_j^*\big\|^2 + 2\left(\big\|\tilde\bG^* - \bI_K\big\|^2 + \big\|\tilde\bd^*\big\|^2\right)q^{-1}\sum_{j=1}^q\big\|\hat\bgamma_j^*\big\|^2  \\&=\cOp\big\{r_{\theta,nq}^2/(b_n^2n)\big\},
\end{align*}
and 
\begin{align*}
    \|\tilde\btheta^* \|_{\infty} &\le \|\tilde\btheta^* - \hat\btheta^*\|_{\infty} + \|\hat\btheta^* - \btheta^*\|_{\infty} + \|\btheta^*\|_{\infty}\\&\le \left(\|\tilde\bG^* -\bI_K\| +\|\tilde\bd^*\| + 1\right)\left(\big\|\hat\btheta^* - \btheta^*\big\|_{\infty} + \big\|\btheta^*\big\|_{\infty}\right) \\&\le (1+o_p(1))(C_{\theta}/2 + o_p(1))\le C_{\theta}.
\end{align*}
For the second inequality, uniformly over $j\in[q]$,
\begin{equation*}
    \|\tilde\bgamma_j^*-\hat\bgamma_j^*\|_\infty
    \le \|\tilde\bG^*-\bI_K\|\,\|\hat\bgamma_j^*\|,
    \qquad
    |\tilde\beta_j^*-\hat\beta_j^*|
    \le \|\hat\bgamma_j^*\|\,\|\tilde\bd^*\|,
\end{equation*}
and $\max_{j\in[q]}\|\hat\bgamma_j^*\|=\cOp(1)$ by the uniform-consistency result. For the third inequality, $C_\theta$ can be enlarged so that $\|\btheta^*\|_\infty\le C_\theta/2$ under Assumption~\ref{assumption_psd_covariance}. Lemma~\ref{lemma_neighbour_control} then gives $\tilde\btheta^*\in\cB_\epsilon(\btheta^*)$.

Next, we construct 
\begin{equation*}
    \tilde\bmu^* = (\tilde\bG^*)^{-1}(\hat\bmu^* - \tilde\bd^*)\text{ and }\tilde\bSigma^*= (\tilde\bG^*)^{-1}\hat\bSigma^*(\tilde\bG^*)^{-\T},
\end{equation*}
such that $\cL(\hat\btheta^*, \hat\bmu^*,\hat\bSigma^*) = \cL(\tilde\btheta^*, \tilde\bmu^*,\tilde\bSigma^*)$. By $(ii)$ of Lemma~\ref{lemma_aux_estimator_property} and $\bSigma_u^* = \bI_K$ under the considered identifiability condition,
\begin{align*}
    \big\|\tilde\bSigma^* - \bI_K\big\|
    \le \big\|(\tilde\bG^*)^{-1}\big\|^2\big\|\hat\bSigma^* - \bI_K\big\|
    +\big\|(\tilde\bG^*)^{-1}(\tilde\bG^*)^{-\T}-\bI_K\big\|=o_p(1).
\end{align*}
Also, $\|\tilde\bmu^*\|\le \|(\tilde\bG^*)^{-1}\|(\|\hat\bmu^*\| + \|\tilde\bd^*\|)=o_p(1)$. Thus $(\tilde\bmu^*,\tilde\bSigma^*)\in\cK_C$. Combining $\cL(\hat\btheta^*, \hat\bmu^*,\hat\bSigma^*) = \cL(\tilde\btheta^*, \tilde\bmu^*,\tilde\bSigma^*)$, $P^*(\hat\btheta^*)>0$, and \eqref{eq_P_star_exact_0}, $(\tilde\btheta^*,\tilde\bmu^*,\tilde\bSigma^*)$ attains a smaller value of $\cQ^* = \cL + P^*$, a contradiction. Therefore $P^*(\hat\btheta^*)=0$. Applying the same locally valid affine reparameterization to any parameter set produces a representative with the same value of $\cL$ and zero penalty. Hence the penalized minimizer also minimizes $\cL(\btheta,\bmu,\bSigma)$ within $\cB_{\epsilon}(\btheta^*)$ and $(\bmu,\bSigma)\in\bar\cK_\rho$.

Now, to relate this auxiliary estimator $(\hat\btheta^*,\hat\bmu^*,\hat\bSigma^*)$ with $(\hat\btheta,\zero,\bI_K)$, we construct\begin{equation*}
    \hat\bG^{\dagger} = \Big\{\sum_{s=1}^q\bgamma_s^*(\hat\bgamma_s)^\T\Big\}^{-1}\Big\{\sum_{s=1}^q\bgamma_s^*(\bgamma_s^*)^\T\Big\}\text{, }\hat\bd^{\dagger} = -\Big\{\sum_{s=1}^q\bgamma_s^*\hat\bgamma_s^\T\Big\}^{-1}\Big\{\sum_{s=1}^q\bgamma_s^*(\hat\beta_s - \beta_s^*)\Big\},
\end{equation*}
and $\hat\btheta^{\dagger} = (\hat\beta_1^{\dagger}, (\hat\bgamma_1^{\dagger})^\T,\dots,\hat\beta_q^{\dagger}, (\hat\bgamma_q^{\dagger})^\T)^\T$ with
\begin{equation}
    \hat\bgamma_j^{\dagger} = (\hat\bG^{\dagger})^\T\hat\bgamma_j\text{ and }\hat\beta_j^{\dagger} = \hat\beta_j + \hat\bgamma_j^\T\hat\bd^{\dagger}\text{ for each }j\in[q].\label{eq_transform_from_hat_theta_star}
\end{equation}
We aim to show that $\hat\btheta^{\dagger}\in \cB_{\epsilon}(\btheta^*)$. To establish this, we first show for some constant $c_0$,
\begin{equation}
    \sigma_{K}\Big\{q^{-1}\sum_{s=1}^q\hat\bgamma_s(\bgamma_s^*)^\T\Big\}\ge c_0.\label{eq_non_degenerate_loading}
\end{equation}
 By Lemmas~\ref{lemma_bu_star_theta_true} and~\ref{lemma_average_converge_eta}, we know 
 \begin{align*}
        &\, \sum_{i=1}^n\sum_{j=1}^q\left\{\hat\bgamma_j^\T\bu_i^{\star}(\hat\btheta) + \hat\beta_j - (\bgamma_j^*)^\T\bu_i^* - \beta_j^*\right\}^2\\
        \le &\, 2\sum_{i=1}^n\sum_{j=1}^q\left\{\hat\bgamma_j^\T\bu_i^{\star}(\hat\btheta) + \hat\beta_j - (\bgamma_j^*)^\T\bu_i^{\star}(\btheta^*) - \beta_j^*\right\}^2
        +2\sum_{i=1}^n\sum_{j=1}^q\left\{(\bgamma_j^*)^\T\bu_i^{\star}(\btheta^*) - (\bgamma_j^*)^\T\bu_i^* \right\}^2\\
        = &\, \cOp\big\{(n\vee q)\log (n\vee q)\kappa_n^2B_n^2/b_n^2+n\kappa_n^2\big\}\\
        = &\, \cOp\big\{(n\vee q)\log (n\vee q)\kappa_n^2B_n^2/b_n^2\big\},
    \end{align*}
    or, in matrix form:
    \begin{equation}
        \Big\|\underbrace{\big(\hat\bbeta\;\hat\bGamma\big) \big(\one_n\;\bU^{\star}(\hat\btheta)\big)^\T}_{\hat \bM} - \underbrace{\big(\bbeta^*\; \bGamma^*\big) \big(\one_n\;\bU^*\big)^\T}_{\bM^*}\Big\|_{\Fn}^2
        = \cOp\big\{(n\vee q)\log (n\vee q)\kappa_n^2B_n^2/b_n^2\big\}.\label{eq_average_eta_u_true_mat}
    \end{equation}
    where $\bU^{\star}(\hat\btheta) = \big(\bu_1^{\star}(\hat\btheta),\dots,\bu_n^{\star}(\hat\btheta)\big)^\T$ and $\bU^* = (\bu_1^*,\dots,\bu_n^*)^\T$. Here, $\bU^*$ denotes the matrix of realized latent variables and should not be confused with the generic random vector $\bU_i$. 
    Denote the normalized $(\hat\bbeta,\hat\bGamma)$ as $(\hat\bbeta^{n},\hat\bGamma^{n})$, i.e., 
    \begin{equation*}
        (\hat\bbeta^{n},\hat\bGamma^{n})^\T (\hat\bbeta^{n},\hat\bGamma^{n}) = \bI_{K+1}.
    \end{equation*}
By the definition of $\Xi$, there exists an invertible matrix $\bR^n\in\RR^{(K+1) \times (K+1)}$ such that $(\hat\bbeta^{n},\hat\bGamma^{n})=q^{-1/2}(\hat\bbeta,\hat\bGamma)\bR^n$ and $c\le \sigma_{K+1}(\bR^n)\le \sigma_{1}(\bR^n)\le C$ for some absolute constants $c,C>0$. For the true parameters, first note that our setup~\eqref{eq_id_st} gives
    \begin{equation*}
n^{-1}\big(\one_n\;\bU^*\big)^\T\big(\one_n\;\bU^*\big) = \bI_{K+1}.
    \end{equation*}
    Then $n^{-1/2}\big(\one_n\;\bU^*\big)$ is column-wise orthogonal spanning the right singular space of $\bM^*$. 
    
    Let $\bP = \bI_n - n^{-1}\one_n\one_n^\T$. By \eqref{eq_id_st}, $\one_n^\T\bU^* = \zero$ and $(\bU^*)^\T\bU^* = n\bI_K$, so $\bM^*\bP = \bGamma^*(\bU^*)^\T$ and, since $\|\bP\| = 1$, \eqref{eq_average_eta_u_true_mat} gives
\begin{equation*}
    \big\|\hat\bM\bP - \bM^*\bP\big\|_{\Fn}^2 = \cOp\big\{(n\vee q)\log (n\vee q)\kappa_n^2B_n^2/b_n^2\big\}.
\end{equation*}
Because $q^{-1}(\bGamma^*)^\T\bGamma^*\to\bSigma_{\gamma}^*$ is positive definite by Assumption~\ref{assumption_psd_covariance}, we have $\sigma_{K}(\bM^*\bP) = \sqrt{n}\,\sigma_K(\bGamma^*)\ge c\sqrt{nq}$ and $\sigma_{K+1}(\bM^*\bP) = 0$. Weyl's inequality then shows that $\hat\bM\bP = \hat\bGamma\big\{\bU^{\star}(\hat\btheta)\big\}^\T\bP$ has rank exactly $K$ with probability tending to one, so that its top-$K$ left singular subspace is $\mathrm{col}(\hat\bGamma)$, and the Davis--Kahan theorem applied to the pair $(\hat\bM\bP,\bM^*\bP)$ yields
\begin{equation*}
    \sin\Theta\big\{\mathrm{col}(\hat\bGamma),\mathrm{col}(\bGamma^*)\big\} = \cOp\left\{\frac{\kappa_nB_n\log^{1/2} (n\vee q)}{b_n\sqrt{n\wedge q}} \right\} = o_p(1),
\end{equation*}
the last step by Assumption~\ref{assump_scaling}. Set $\bR^{\prime} = \big\{q^{-1}(\bGamma^*)^\T\bGamma^*\big\}^{-1}q^{-1}(\bGamma^*)^\T\hat\bGamma\in\RR^{K\times K}$, so that $q^{-1}(\bGamma^*)^\T\hat\bGamma = \big\{q^{-1}(\bGamma^*)^\T\bGamma^*\big\}\bR^{\prime}$. With $\sigma_1(q^{-1/2}\hat\bGamma)\le C$, we know $q^{-1/2}\big\|\hat\bGamma - \bGamma^*\bR^{\prime}\big\|_{\Fn} = o_p(1)$. In addition, because $\hat\btheta\in\Xi$, by singular value interlacing, we have $\sigma_K(q^{-1/2}\hat\bGamma)\ge \sigma_{K+1}(q^{-1/2}\hat\bTheta)\ge c$. Hence $\sigma_K(q^{-1/2}\bGamma^*\bR^{\prime})\ge c - o_p(1)$ and, since $\sigma_1(q^{-1/2}\bGamma^*)\le C$, also $\sigma_K(\bR^{\prime})\ge c/(2C)$. Consequently,
\begin{equation*}
    \sigma_K\Big\{q^{-1}\sum_{s=1}^q\hat\bgamma_s(\bgamma_s^*)^\T\Big\} = \sigma_K\big[\big\{q^{-1}(\bGamma^*)^\T\bGamma^*\big\}\bR^{\prime}\big]\ge \tfrac12\lambda_{\min}(\bSigma_{\gamma}^*)\,\sigma_K(\bR^{\prime})\ge c_0,
\end{equation*}
    This proves \eqref{eq_non_degenerate_loading}, and we conclude that $\big\|\hat\btheta^{\dagger}\big\|_{\infty}\le C$ for some absolute constant $C$. 

Next, we show $q^{-1/2}\|\hat\btheta^{\dagger} - \btheta^*\|\ll \epsilon$. We introduce the following lemma.  
\begin{lemma}
    \em\label{lemma_average_eta} For any $(\btheta,\bu)$ such that
    \begin{equation*}
        (nq)^{-1}\sum_{i=1}^n\sum_{j=1}^q\big(\bgamma_j^\T\bu_i + \beta_j - (\bgamma_j^*)^\T\bu_i^{\star}(\btheta^*)  -\beta_j^*\big)^2 = o_p(1),
    \end{equation*}
    if $P^*(\btheta) = c_Pnq\big\|q^{-1}\sum_{j=1}^q\bgamma_j^*(\btheta_j - \btheta_j^*)^\T\|_{\Fn}^2= 0$, one has
    \begin{equation*}
        q^{-1/2}\big\|\btheta - \btheta^*\big\| = \cOp\Big\{(nq)^{-1/2}{\Big(\sum_{i=1}^n\sum_{j=1}^q\big(\bgamma_j^\T\bu_i + \beta_j - (\bgamma_j^*)^\T\bu_i^{\star}(\btheta^*)  -\beta_j^*\big)^2\Big)^{1/2}}\Big\}.
    \end{equation*}
\end{lemma}
\begin{proof}See Section~\ref{supp_sec_prove_lemma_average_eta}.\end{proof}

By construction, $P^*(\hat\btheta^{\dagger}) = 0$. Then \eqref{eq_transform_from_hat_theta_star}, \eqref{eq_average_eta_u_true_mat}, Lemma~\ref{lemma_average_eta}, and Assumption~\ref{assump_scaling} give
\begin{equation*}
    q^{-1/2}\big\|\hat\btheta^{\dagger} - \btheta^*\big\| = \cOp\Big\{\frac{\kappa_nB_n}{b_n\sqrt{n\wedge q}} \Big\} =  o_p\Big(\frac{\varepsilon_{nq}}{\sqrt{B_n(n\wedge q)}}\Big).
\end{equation*}The last inequality follows from Assumption~\ref{assump_scaling}.
Hence $\hat\btheta^{\dagger}\in\cB_{\epsilon}(\btheta^*)$ for a properly selected $\epsilon$ satisfying the requirement in Lemma~\ref{lemma_aux_estimator_property}, and the associated mean and variance parameters $\big(-(\hat\bG^{\dagger})^{-1}\hat\bd^{\dagger}, \big\{(\hat\bG^{\dagger})^\T\hat\bG^{\dagger}\big\}^{-1}\big)$ can be verified to lie within $\cK_C$. Therefore, we know
\begin{equation}
    \cL(\hat\btheta,\zero_K,\bI_K) = \cL\left(\hat\btheta^{\dagger}, -(\hat\bG^{\dagger})^{-1}\hat\bd^{\dagger}, \big\{(\hat\bG^{\dagger})^\T\hat\bG^{\dagger}\big\}^{-1}\right)\ge \cL(\hat\btheta^{*},\hat\bmu^*,\hat\bSigma^*).\label{eq_theta_ge_theta_star}
\end{equation}

Conversely, we construct $\hat\btheta^{\ddagger}$ transformed from $\hat\btheta^*$ with 
\begin{equation*}
    \hat\bgamma_j^{\ddagger} = (\hat\bSigma^*)^{1/2}\hat\bgamma_j^*,\text{ and }\hat\beta_j^{\ddagger} = \hat\beta_j^* + (\hat\bgamma_j^*)^\T\hat\bmu^*\text{ for each }j\in[q].
\end{equation*}
Then $\cL(\hat\btheta^{\ddagger},\zero_K,\bI_K) = \cL(\hat\btheta^*,\hat\bmu^*,\hat\bSigma^*)$. Since $\|(\hat\bSigma^*)^{1/2} - \bI_K\|+\|\hat\bmu^*\| = \cOp(\delta_{nq}) = o_p(1)$,
the same argument used for $\tilde\btheta^*$ gives $\hat\btheta^{\ddagger} \in \cB_{\epsilon}(\btheta^*) \subseteq\Xi$. By definition,
\begin{equation}
\cL(\hat\btheta^*,\hat\bmu^*,\hat\bSigma^*)=\cL(\hat\btheta^{\ddagger},\zero_K,\bI_K)\ge \cL(\hat\btheta,\zero_K,\bI_K).\label{eq_theta_star_theta_ge}
\end{equation}

Finally, \eqref{eq_theta_ge_theta_star} and \eqref{eq_theta_star_theta_ge} give
\begin{equation*}
    \cL(\hat\btheta^*,\hat\bmu^*,\hat\bSigma^*) = \cL(\hat\btheta,\zero_K,\bI_K).
\end{equation*}
Both corresponding parameter sets have zero penalty and the same value of $\cL$, hence the same value of $\cQ^*=\cL+P^*$. Thus, the two local argmin sets correspond under the displayed affine transformations. We may choose aligned representatives satisfying $\hat\btheta^* = \hat\btheta^{\dagger}$, or equivalently $\hat\btheta = \hat\btheta^{\ddagger}$ up to an orthogonal transformation. For simplicity, we suppress this transformation and write $\hat\btheta = \hat\btheta^{\ddagger}$.

\noindent\underline{\large \it Step 3: Establish the First-order Conditions. }
By the conclusion of Step 2, we may take $\hat\btheta=\hat\btheta^{\ddagger}$ up to an
orthogonal transformation. Hence, for each $j\in[q]$, $\hat\bgamma_j = (\hat\bSigma^*)^{1/2}\hat\bgamma_j^*$ and $\hat\beta_j = \hat\beta_j^* + (\hat\bgamma_j^*)^\T\hat\bmu^*$, and therefore
\begin{equation*}
    \ell_i(\bu;\hat\btheta)=\ell_i\big((\hat\bSigma^*)^{1/2}\bu+\hat\bmu^* ;\hat\btheta^*\big).
\end{equation*}
For notational simplicity, write
\begin{equation*}
    \hat\bA=(\hat\bSigma^*)^{1/2}, \qquad \hat\bR =\begin{pmatrix}
        1&\zero_K^\T\\\hat\bmu^*&\hat\bA
    \end{pmatrix}.
\end{equation*}
Then, for $\bv=\hat\bA\bu+\hat\bmu^*$, we have the change-of-variable form
\begin{equation*}
    \pwk(\bu)d\bu=\pwk(\bv|\hat\bmu^*,\hat\bSigma^*)d\bv.
\end{equation*}
For each $j\in[q]$, we have
\begin{align*}
   \cS_{\theta_j}(\hat\btheta) &= \partial_{\btheta_j}\cL(\hat\btheta) \\
   &= -\sum_{i=1}^n \int_{\bu_i\in\RR^{K}}\ell_{ij}^{\prime}(\hat\bgamma_j^\T\bu_i + \hat\beta_j)
   \begin{pmatrix}
        1\\\bu_i
        \end{pmatrix}
   \frac{\exp\big\{\ell_i(\bu_i;\hat\btheta)\big\}\pwk(\bu_i)d\bu_i}{\int_{\bu_i\in\RR^{K}}\exp\big\{\ell_i(\bu_i;\hat\btheta)\big\}\pwk(\bu_i)d\bu_i}\\
   &=-\sum_{i=1}^n\int_{\bu_i\in\RR^{K}}\ell_{ij}^{\prime}\left[(\hat\bgamma_j^*)^\T\big\{\hat\bA\bu_i+\hat\bmu^*\big\}     + \hat\beta_j^*\right]
   \begin{pmatrix}1\\\bu_i\end{pmatrix}\\&\qquad\qquad\times
   \frac{\exp\big\{\ell_i(\hat\bA\bu_i+\hat\bmu^*;\hat\btheta^*)\big\}\pwk(\bu_i)d\bu_i}{\int_{\bu_i\in\RR^{K}}\exp\big\{\ell_i(\hat\bA\bu_i+\hat\bmu^*;\hat\btheta^*)\big\}\pwk(\bu_i)d\bu_i}\\
   &=-\hat\bR^{-1}\sum_{i=1}^n\int_{\bv_i\in\RR^{K}}\ell_{ij}^{\prime}
   \big((\hat\bgamma_j^*)^\T\bv_i + \hat\beta_j^*\big)
   \begin{pmatrix}
        1\\\bv_i
   \end{pmatrix}
   \frac{\exp\big\{\ell_i(\bv_i;\hat\btheta^*)\big\}\pwk(\bv_i|\hat\bmu^*,\hat\bSigma^*)d\bv_i}{\int_{\bv_i\in\RR^{K}}\exp\big\{\ell_i(\bv_i;\hat\btheta^*)\big\} \pwk(\bv_i|\hat\bmu^*,\hat\bSigma^*)d\bv_i}\\
   &=\hat\bR^{-1}\cS_{\theta_j}(\hat\btheta^*,\hat\bmu^*,\hat\bSigma^*)=
   \zero .
\end{align*}
Similarly,
\begin{align*}
    \cS_{\mu}(\hat\btheta)&= \partial_{\bmu}\cL(\hat\btheta) \\
    &=\sum_{i=1}^n\frac{\int_{\bu_i}\exp\{\ell_i(\bu_i;\hat\btheta)\}(-\bu_i)\pwk(\bu_i)d\bu_i}{\int_{\bu_i} \exp\{\ell_i(\bu_i;\hat\btheta)\}\pwk(\bu_i)d\bu_i}\\
    &=\hat\bA\sum_{i=1}^n\frac{\int_{\bv_i}\exp\{\ell_i(\bv_i;\hat\btheta^*)\}\big[-(\hat\bSigma^*)^{-1}(\bv_i-\hat\bmu^*)\big]\pwk(\bv_i|\hat\bmu^*,\hat\bSigma^*)d\bv_i}{\int_{\bv_i}\exp\{\ell_i(\bv_i;\hat\btheta^*)\}\pwk(\bv_i|\hat\bmu^*,\hat\bSigma^*)d\bv_i}\\
    &=\hat\bA\,\cS_{\mu}(\hat\btheta^*,\hat\bmu^*,\hat\bSigma^*)=
    \zero .
\end{align*}
For the covariance score, we first write the fully vectorized form; the $\mathrm{vech}$ form then follows directly. Specifically,
\begin{align*}
    \cS_{\Sigma}^{\mathrm{vec}}(\hat\btheta)&=\partial_{\mathrm{vec}(\bSigma)}\cL(\hat\btheta)\\
    &=-\sum_{i=1}^n\frac{\int_{\bu_i}\exp\{\ell_i(\bu_i;\hat\btheta)\}\frac{1}{2}\mathrm{vec}\big\{\bu_i\bu_i^\T-\bI_K\big\}    \pwk(\bu_i)d\bu_i}{\int_{\bu_i}\exp\{\ell_i(\bu_i;\hat\btheta)\}\pwk(\bu_i)d\bu_i}\\
    &=\big(\hat\bA\otimes\hat\bA\big)\sum_{i=1}^n\Big\{\frac{1}{\int_{\bv_i} \exp\{\ell_i(\bv_i;\hat\btheta^*)\}\pwk(\bv_i|\hat\bmu^*,\hat\bSigma^*)d\bv_i}\\
    &\qquad\qquad\times\int_{\bv_i}-\exp\{\ell_i(\bv_i;\hat\btheta^*)\}\frac{1}{2}\mathrm{vec}\big[(\hat\bSigma^*)^{-1}(\bv_i-\hat\bmu^*)(\bv_i-\hat\bmu^*)^\T(\hat\bSigma^*)^{-1} - (\hat\bSigma^*)^{-1}\big]\\
    &\qquad\qquad\qquad\times \pwk(\bv_i|\hat\bmu^*,\hat\bSigma^*)d\bv_i\Bigg\}\\
    &= \big(\hat\bA\otimes\hat\bA\big) \cS_{\Sigma}^{\mathrm{vec}}(\hat\btheta^*,\hat\bmu^*,\hat\bSigma^*) = \zero .
\end{align*}
Consequently, $\cS_{\Sigma}(\hat\btheta)=\zero_{K\times K}$, and hence the score under the $\mathrm{vech}(\bSigma)$ parameterization is also zero. Here the preceding transformation follows from
$\mathrm{vec}(\bA\bX\bB)=(\bB^\T\otimes \bA)\mathrm{vec}(\bX)$ and the identity
\[
    \bu_i\bu_i^\T-\bI_K = \hat\bA \left[ (\hat\bSigma^*)^{-1} (\bv_i-\hat\bmu^*)(\bv_i-\hat\bmu^*)^\T (\hat\bSigma^*)^{-1} - (\hat\bSigma^*)^{-1} \right] \hat\bA^\T ,
    \qquad
    \bv_i=\hat\bA\bu_i+\hat\bmu^* .
\]
This finishes the proof.


\section{Proof of Main Results}\label{supp_sec_main_proof}

\subsection{Proof of Theorem~\ref{thm_consis}}\label{supp_sec_prove_consis}
In this subsection, starred true parameters refer to the original representative, and $\hat\bQ^*$ denotes the original-target rotation in \eqref{eq_q_transf} of the main text. Applying the affine equivalence from Step~2 of the proof of Lemma~\ref{lemma_key_first_order}  to this target, take the unrotated representative $\hat\bgamma_j=(\hat\bSigma^*)^{1/2}\hat\bgamma_j^*$. Then
\begin{equation*}
    \hat\bQ^* = {\arg\min}_{\bQ\in\cO^{K}}\sum_{j=1}^q\big\|\bQ(\hat\bSigma^*)^{1/2}\hat\bgamma_j^* - \bgamma_j^*\big\|^2,
\end{equation*}
Then
\begin{equation*}
    \hat\bgamma_j^Q = \hat\bQ^*(\hat\bSigma^*)^{1/2}\hat\bgamma_j^*,\;\text{ and }\;\hat\beta_j^Q = \hat\beta_j^* + (\hat\bgamma_j^*)^\T\hat\bmu^*\;\text{ for each }j\in[q],
\end{equation*} where $\{\hat\bgamma_j^*,\hat\beta_j^*\}_{j\in[q]}$, $\hat\bmu^*$ and $\hat\bSigma^*$ are obtained in 
\eqref{eq_aux_estimator}. 
By Theorem~2 of \citet{ten1977orthogonal}, $\hat\bQ^*$ solves the Procrustes problem if and only if
\begin{equation}
    \hat\bQ^*(\hat\bSigma^*)^{1/2}\sum_{j=1}^q\hat\bgamma_j^*(\bgamma_j^*)^\T
    =\sum_{j=1}^q\bgamma_j^*(\hat\bgamma_j^*)^\T(\hat\bSigma^*)^{1/2}(\hat\bQ^*)^\T
    \label{eq_solve_symm_align}
\end{equation}
is positive definite. The proof of Lemma~\ref{lemma_key_first_order} gives
\begin{equation*}
    q^{-1}\sum_{j=1}^q\hat\bgamma_j^*(\bgamma_j^*)^\T
    =q^{-1}\sum_{j=1}^q\bgamma_j^*(\bgamma_j^*)^\T.
\end{equation*}
Let $\bM_\gamma^*=q^{-1}\sum_{j=1}^q\bgamma_j^*(\bgamma_j^*)^\T$ and $\hat\bA=(\hat\bSigma^*)^{1/2}$. The orthogonal polar factor satisfying \eqref{eq_solve_symm_align} is
\begin{equation*}
    \hat\bQ^*
    =\big(\bM_\gamma^*\hat\bSigma^*\bM_\gamma^*\big)^{-1/2}\bM_\gamma^*\hat\bA.
\end{equation*}
Indeed, $\hat\bQ^*\hat\bA\bM_\gamma^*=(\bM_\gamma^*\hat\bSigma^*\bM_\gamma^*)^{1/2}$ is symmetric positive definite. By conclusion~$(ii)$ of Lemma~\ref{lemma_aux_estimator_property},
$\|\hat\bSigma^*-\bI_K\|=\cOp(\delta_{nq})$, while $\bM_\gamma^*\to\bSigma_\gamma^*$ with eigenvalues bounded above and away from zero. The matrix square root $\bA^{1/2}$ and inverse square root $\bA^{-1/2}$ mappings are locally Lipschitz. Thus by telescoping, we have
\begin{equation*}
    \|\hat\bQ^*-\bI_K\|=\cOp(\delta_{nq})
    =\cOp\left(\frac{\kappa_n^2r_{\theta,nq}}{\sqrt n}+\frac{\kappa_n}{\sqrt q}\right).
\end{equation*}
Theorem~\ref{thm_consis} now follows from conclusions~$(i)$--$(iii)$ of Lemma~\ref{lemma_aux_estimator_property}, together with the boundedness of $\|\hat\btheta^*\|_{\infty}$. For later use, we also derive the consistency rate for $\hat\bu^{\star} = \left((\hat\bu_{1}^{\star})^\T,\dots, (\hat\bu_{n}^{\star})^\T\right)^\T\in\RR^{nK}$, where $\hat\bu_{i}^{\star} = \bu_i^{\star}(\hat\btheta) = \argmax_{\bu}\ell_i(\bu;\hat\btheta)$. Let $\bu^{\star}_i(\hat\btheta^*) = \argmax_{\bu}\ell_i(\bu;\hat\btheta^*)$ for $\hat\btheta^*$ in \eqref{eq_aux_estimator}. Conclusion~$(i)$ of Lemma~\ref{lemma_aux_estimator_property} and Lemma~\ref{lemma_neighbour_control} imply
\begin{equation*}
    \Big(n^{-1}\sum_{i=1}^n\big\|\bu^{\star}_i(\hat\btheta^*) - \bu_i^*\big\|^2\Big)^{1/2} = \cOp(\delta_{nq}).
\end{equation*}
Next, with $\hat\bgamma_j = \hat\bQ^*(\hat\bSigma^*)^{1/2}\hat\bgamma_j^*$ and $\hat\beta_j = \hat\beta_j^* + (\hat\bgamma_j^*)^\T\hat\bmu^*$, simple algebra yields $\hat\bu_i^{\star} = \hat\bQ^*(\hat\bSigma^*)^{-1/2} \left\{\bu_i^{\star}(\hat\btheta^*) - \hat\bmu^*\right\}$.
With the rates for $\|\hat\bmu^*\|$, $ \|\hat\bSigma^* - \bI_K\|$, and $\|\hat\bQ^* - \bI_K\|$ given above, one can show that 
\begin{equation}
    \Big(n^{-1}\sum_{i=1}^n\big\|\hat\bu^{\star}_i - \bu_i^*\big\|^2\Big)^{1/2} = \cOp(\delta_{nq}).\label{eq_consis_u_star_hat}
\end{equation}

\subsection{Proof of Theorem~\ref{thm_asymp}}\label{supp_sec_prove_asym_0}
Throughout Sections~\ref{supp_sec_prove_asym_0}--\ref{supp_sec_prove_thm_bvm_irt} and the proofs of Lemmas~\ref{lemma_residual_in_asym_expand}--\ref{lemma_bound_aug_via_id}, write $(\btheta^*,\bu^*)$ for the canonical representative in \eqref{eq_canp_trans}, and $\hat\btheta$ for the estimator aligned by $\hat\bQ^0$ as defined before Theorem~\ref{thm_asymp} . In this notation,
\begin{equation*}
    n^{-1}\sum_{i=1}^n\bu_i^*= \zero_K,\qquad n^{-1}\sum_{i=1}^n\bu_i^*(\bu_i^*)^\T=\bI_K,
\end{equation*} and $\hat\btheta$ obeys
\begin{equation*}
    \bI_K = \argmin_{\bQ\in\cO^K}\sum_{j=1}^q\big\|\bQ\hat\bgamma_j - \bgamma_j^*\big\|^2.
\end{equation*}
Then Theorem~2 of \citet{ten1977orthogonal} implies that
\begin{equation}
    \sum_{j=1}^q\hat\bgamma_j (\bgamma_j^*)^\T = \sum_{j=1}^q\bgamma_j^*\hat\bgamma_j^\T\label{eq_id_gamma_rotate}
\end{equation}
is positive definite. We briefly explain how to transfer the consistency bounds from Section~\ref{supp_sec_prove_consis}. Set $\bD=\hat\bQ^0(\hat\bQ^*)^\T$, and one can easily verify by $\|\bar\bmu_u \| = \cOp(n^{-1/2})$ and $\|\bar\bSigma_u - \bI_K\| = \cOp(n^{-1/2})$  that $\|\bD-\bI_K\|=\cOp(n^{-1/2})$. 
The canonical transformation changes each $\bgamma_j$, $\bu_i$, and $\beta_j$ uniformly by $\cOp(n^{-1/2})$. Thus their consistency bounds are preserved. 
Recall that $\delta_{nq} = \kappa_n^2r_{\theta,nq}/\sqrt{n}+\kappa_n/\sqrt{q}$. We have
\begin{equation}
    \label{eq_consis_ave_all}\Big(q^{-1}\sum_{j=1}^q\big\|\hat\btheta_j - \btheta_j^*\big\|^2\Big)^{1/2} = \cOp(\delta_{nq})\quad\text{and}\quad\Big(n^{-1}\sum_{i=1}^n\big\|\hat\bu^{\star}_i - \bu_i^*\big\|^2\Big)^{1/2} = \cOp(\delta_{nq}).
\end{equation}
Here $\hat\bu_i^{\star}=\bu_i^{\star}(\hat\btheta)$ is evaluated at the canonically aligned estimator. This convention is also used later in Sections~\ref{supp_sec_prove_thm_asymp_bu}--\ref{supp_sec_prove_thm_bvm_irt}. In what follows, the derivations are conducted conditional on $\mathscr U_n$, on latent events whose $\PP_U$-probabilities tend to one. Accordingly, all conditional stochastic orders and limits in this proof hold in $\PP_U$-probability.

In Lemma~\ref{lemma_key_first_order}, we have established that, with probability approaching $1$, one has
    \begin{equation*}
        \cS_{\theta}(\hat\btheta) = \zero,\;\cS_{\mu}(\hat\btheta) = \zero,\text{ and }\cS_{\Sigma}(\hat\btheta) = \zero.
    \end{equation*}
    A standard route to asymptotic normality is to expand $\cS_{\theta}(\hat\btheta)$ around $\cS_{\theta}(\btheta^*)$. However, Lemma~\ref{lemma_laplace_approx_second_order} gives an approximation to $\partial_{\btheta}\cS_{\theta}$ containing both a block-diagonal term $\bH_{L\theta\theta}$ and a dense term, so no convenient inversion formula is available. To overcome this problem, we use the derivatives $\{\bS_{Lu}(\btheta,\bu)\}_{\cK_i} = \partial_{\bu_i}L(\btheta,\bu)$. 
    By definition, one has $\bS_{Lu}(\hat\btheta,\hat\bu^{\star}) = \zero$. By Lemma~\ref{lemma_first_order_base}, we know that for each $j\in[q]$,
    \begin{equation}
        \big\|\bS_{L\theta_j}(\hat\btheta,\hat\bu^{\star}) \big\| = \big\|\cS_{\theta_j}(\hat\btheta) \big\| + \cOp(nB_n\varepsilon_{nq}\sqrt{\log n}/q) = \cOp(nB_n\varepsilon_{nq}\sqrt{\log n}/q).\label{eq_first_order_near_zero}
    \end{equation}
This is of smaller order than $\|\bS_{L\theta_j}(\btheta^*,\bu^*)\| = \cOp(\sqrt{nB_n})$ under the scaling condition in Theorem~\ref{thm_asymp}. We subsequently expand the first-order derivatives $\bS_L(\btheta,\bu) = \begin{pmatrix}
        \bS_{L\theta}(\btheta,\bu)\\
        \bS_{Lu}(\btheta,\bu)
    \end{pmatrix} $    as follows
    \begin{equation}
        \bS_L\big(\hat\btheta,\hat\bu^{\star}\big) - \bS_{L}(\btheta^*,\bu^*) = \bH_L(\btheta^*,\bu^*) \begin{pmatrix}
            \hat\btheta - \btheta^*\\ \hat\bu^{\star} - \bu^*
        \end{pmatrix} + \frac{1}{2}\begin{pmatrix}\bR_{L\theta}\\\bR_{Lu}\end{pmatrix},\label{eq_asymp_expand}
    \end{equation}
    where $\bR_{L\theta}\in\RR^{qK+q}$ and $\bR_{Lu}\in\RR^{nK}$ are the residuals bounded by the following lemma. Let
    \begin{equation*}
        \bar\delta_{nq} = \frac{B_n}{b_n}\delta_{nq} = \frac{B_n}{b_n}\Big\{\frac{\kappa_n^2r_{\theta,nq}}{\sqrt{n}} \; + \; \frac{\kappa_n}{\sqrt{q}}\Big\}.
    \end{equation*}
    \begin{lemma}\em\label{lemma_residual_in_asym_expand}
        Under Assumptions~\ref{assump_standard_id}--\ref{assump_scaling}, for each $i\in[n]$ and $j\in[q]$
        \begin{equation*}
            \big\|\hat\btheta_j - \btheta^*_j\big\| = \cOp(\bar\delta_{nq})\quad\text{ and }\quad\big\|\hat\bu_i^{\star} - \bu_i^*\big\| = \cOp(\bar\delta_{nq}).
        \end{equation*}
        Moreover, for the expansion in \eqref{eq_asymp_expand}, it holds that, for each $j\in[q]$,
        \begin{equation*}
            \big\|(\bR_{L\theta})_{\cK_j^+}\big\| = \cOp\big(nB_n\bar\delta_{nq}^2\big);\quad \big\|\bR_{L\theta}\big\| = \cOp\big(n\sqrt{q}B_n\bar\delta_{nq}^2\big);\quad \max_{j\in[q]}\big\|(\bR_{L\theta})_{\cK_j^+}\big\| = \cOp\big(nB_n\bar\delta_{nq}^2\log q\big),
        \end{equation*}
        and for each $i\in[n]$,
        \begin{equation*}
            \big\|(\bR_{Lu})_{\cK_i}\big\| = \cOp\big(qB_n\bar\delta_{nq}^2\big);\quad \big\|\bR_{Lu}\big\| = \cOp\big(\sqrt{n}qB_n\bar\delta_{nq}^2\big);\quad \max_{i\in[n]}\big\|(\bR_{Lu})_{\cK_i}\big\| = \cOp\big(qB_n\bar\delta_{nq}^2{\log n}\big).
        \end{equation*}
    \end{lemma}
\begin{proof}
See Section~\ref{sec_prove_lemma_residual_in_asym_expand}.
\end{proof}
We next consider the first term on the right-hand side. Our goal is to invert $\bH_{L}(\btheta^*,\bu^*)$. Direct inversion of this matrix, however, does not yield a tractable form. Specifically, for any $r\in\{0\}\cup[K]$ and $l\in[K]$, define
\begin{equation*}
    \bphi_{rl}(\btheta,\bu) = \left(q^{-1/2}\gamma_{1l}\big(\be_{r+1}^{(K+1)}\big)^\T,\dots, q^{-1/2}\gamma_{ql}\big(\be_{r+1}^{(K+1)}\big)^\T, -n^{-1/2}u_{1r}(\be_l^{(K)})^\T,\dots, -n^{-1/2}u_{nr}(\be_l^{(K)})^\T\right)^\T.
\end{equation*}
Here we abuse the notation to denote $u_{i0} = 1$ for $i\in[n]$. Then similar to the proof of Lemma~\ref{lemma_convexity}, one can check that $\{\bphi_{rl}(\btheta^*,\bu^*)\}^\T\bS_m\bH_L(\btheta^*,\bu^*)\bS_m\bphi_{rl}(\btheta^*,\bu^*) = \cOp(\sqrt{B_n/(nq)})$ with $\|\bphi_{rl}(\btheta^*,\bu^*)\| = \cOp(1)$ and $\big\|\bS_m\bH_L(\btheta^*,\bu^*)\bS_m\big\| = \cOp(B_n)$, which follows from $|\ell_{ij}^{\prime\prime}|\le B_n$, $\max_{i\in[n]}\|\bu_i^*\|_{\infty}\le C\sqrt{\log n}$, $n^{-1}\sum_{i=1}^n\|\bu_i^*\|^2 = \cOp(1)$ and $q^{-1}\|\bGamma^*\|^2 = \cOp(1)$ after bounding the off-diagonal block in Frobenius norm. Here,  \begin{equation}
    \bS_m = \begin{pmatrix}
            n^{-1/2}\bI_{qK+q}&\zero\\\zero&q^{-1/2}\bI_{nK}
        \end{pmatrix}.\label{eq_define_sm1}
    \end{equation}
    These vectors span a subspace along which the matrix $\bH_{L}(\btheta^*,\bu^*)$ is nearly flat. As a result, directly inverting $\bH_{L}(\btheta^*,\bu^*)$ yields expressions that are difficult to control. To address this issue, we construct the following augmentation $\bH_P^*$:
    \begin{equation*}
        \bH_P^* = cb_n\sum_{r=0}^K\sum_{l=1}^K\bpsi_{rl}^*(\bpsi_{rl}^*)^\T,
    \end{equation*}
    where $c<1$ is some positive constant specified later in Lemma~\ref{lemma_approx_hessian_inverse}, and
    \begin{enumerate}[label=$(\roman*)$]
        \item $\bpsi_{rl}^* = \sqrt{n/q}\big(\gamma_{1l}^*(\be_{r+1}^{(K+1)})^\T - \gamma_{1r}^*(\be_{l+1}^{(K+1)})^\T,\dots, \gamma_{ql}^*(\be_{r+1}^{(K+1)})^\T - \gamma_{qr}^*(\be_{l+1}^{(K+1)})^\T,\zero_{nK}^\T\big)^\T$ when $1\le l<r\le K$;
        \item $\bpsi_{0l}^* = \sqrt{q/n}\big(\zero_{qK+q}^\T, (\be_l^{(K)})^\T,\dots, (\be_l^{(K)})^\T\big)^\T$ when $1\le l\le K$;
        \item $\bpsi_{rl}^* = \sqrt{q/n}\big(\zero_{qK+q}^\T, u_{1r}^*(\be_l^{(K)})^\T,\dots, u_{nr}^*(\be_l^{(K)})^\T\big)^\T$ when $1\le r=l\le K$;
        \item $\bpsi_{rl}^* = \sqrt{q/n}\big(\zero_{qK+q}^\T, u_{1r}^*(\be_l^{(K)})^\T + u_{1l}^*(\be_r^{(K)})^\T,\dots, u_{nr}^*(\be_l^{(K)})^\T+ u_{nl}^*(\be_r^{(K)})^\T\big)^\T$ when $1\le r<l\le K$.
    \end{enumerate}
The augmentation $\bH_P^*$ is constructed so that the inverse of $\bH_{L}(\btheta^*, \bu^*) + \bH_P^*$ can be accurately approximated by a block-diagonal matrix, as shown in Lemma~\ref{lemma_approx_hessian_inverse}. Moreover, the term $\bH_P^*\big((\hat\btheta - \btheta^*)^\T,(\hat\bu^{\star} - \bu^*)^\T\big)^\T$ is of the same order as the residual terms $\bR_{L\theta}$ and $\bR_{Lu}$ in expansion \eqref{eq_asymp_expand}, as shown in Lemma~\ref{lemma_bound_aug_via_id}.
\begin{lemma}
    \em\label{lemma_approx_hessian_inverse} Suppose Assumptions~\ref{assump_standard_id}--\ref{assumption_smoothness} hold and $c$ is some positive constant. Denote $\bH_{L}^* = \bH_{L}(\btheta^*, \bu^*)$, $\bH_{L\theta\theta}^* = \bH_{L\theta\theta}(\btheta^*, \bu^*)$, and $\bH_{Luu}^* = \bH_{Luu}(\btheta^*, \bu^*)$. For each $j\in[q]$,
    \begin{align*}
        (a).&\quad\Big\|\bS_m^{-1}\big\{(\bH_{L}^* + \bH_P^*)^{-1}\bS_m^{-1}\big\}_{\cK_j^+,} -\big\{ n(\bH_{L\theta\theta}^*)^{-1}\big\}_{\cK_j^+,}\Big\| = \cOp\big(\kappa_n^3\sqrt{B_n}q^{-1/2}\big);\\
        (b).&\quad\Big\|\Big[\big\{(\bH_{L}^* + \bH_P^*)^{-1}\big\}_{\cK_j^+,} - \big\{(\bH_{L\theta\theta}^*)^{-1}\big\}_{\cK_j^+,}\Big]\bS_L(\btheta^*,\bu^*)\Big\| = \cOp\Big(\frac{\kappa_n^2B_n}{\sqrt{nq}} + \frac{\kappa_n^2}{n\wedge q}\Big);
    \end{align*}
    and, for each $i\in[n]$,
    \begin{align*}
        (c)&\quad\Big\|\big\{\bS_m^{-1}(\bH_{L}^* + \bH_P^*)^{-1}\bS_m^{-1}\big\}_{qK+q+\cK_i,} - q\big\{(\bH_{Luu}^*)^{-1}\big\}_{\cK_i,}\Big\| = \cOp\big(\kappa_n^3\sqrt{B_n}n^{-1/2}\big);\\
        (d)&\quad\left\|\Big[\big\{(\bH_{L}^* + \bH_P^*)^{-1}\big\}_{qK+q+\cK_i,} - \big\{(\bH_{Luu}^*)^{-1}\big\}_{\cK_i,}\Big]\bS_L(\btheta^*,\bu^*)\right\| = \cOp\Big(\frac{\kappa_n^2B_n}{\sqrt{nq}} + \frac{\kappa_n^2}{n\wedge q}\Big).
    \end{align*}
\end{lemma}
\begin{proof}
    See Section~\ref{sec_prove_lemma_approx_hessian_inverse}.
\end{proof}


\begin{lemma}
    \label{lemma_bound_aug_via_id}\em Given \eqref{eq_id_gamma_rotate}, $\cS_{\mu}(\hat\btheta) = 0$ and $\cS_{\Sigma}(\hat\btheta) = 0$, under Assumptions~\ref{assump_standard_id}--\ref{assump_scaling},
    \begin{equation*}
        \left\|\bS_m^{2}\bH_P^*\begingroup\renewcommand{\arraystretch}{0.6}\begin{pmatrix}
            \hat\btheta - \btheta^*\\\hat\bu^{\star} - \bu^{*}
        \end{pmatrix}\endgroup\right\|_{\infty}= \cOp\big(b_n\sqrt{\log n}(\delta_{nq}^2 + \varepsilon_{nq}q^{-1})\big).
    \end{equation*}
\end{lemma}
\begin{proof}
See Section~\ref{sec_prove_lemma_bound_aug_via_id}. The key technical input is that
\begin{equation*}
    \big\|n^{-1}\sum_{i=1}^n\hat\bu_i^{\star}\big\| = \cOp(\varepsilon_{nq}q^{-1})\text{ and } \big\|n^{-1}\sum_{i=1}^n\hat\bu_i^{\star}(\hat\bu_i^{\star})^\T - \bI_K\big\| = \cOp(\varepsilon_{nq}q^{-1}).
\end{equation*}
by $\cS_{\mu}(\hat\btheta) = 0$, $\cS_{\Sigma}(\hat\btheta) = 0$ from 
Lemma~\ref{lemma_key_first_order} and Lemma~\ref{lemma_first_order_base}. At this point, we can explain why the canonical parameters must serve as the target of the asymptotic distribution for the marginal likelihood estimator. The canonical representations, now denoted by $\bu^*$ and $\btheta^*$ in this proof for notational convenience, satisfy
\begin{equation}\label{eq_canproperty}
    n^{-1}\sum_{i=1}^n\bu_i^*=\zero_K, \qquad n^{-1}\sum_{i=1}^n\bu_i^*(\bu_i^*)^\T=\bI_K.
\end{equation}
In the proof detailed in Section~\ref{sec_prove_lemma_bound_aug_via_id}, this is crucial for showing the infinity norm of $\bS_m^{2}\bH_P^*\begin{pmatrix}
            \hat\btheta - \btheta^*\\\hat\bu^{\star} - \bu^{*}
        \end{pmatrix}$ can be bounded in a smaller order than $n^{-1/2}$ and $q^{-1/2}$. With the random variables generated under Assumption~\ref{assump_standard_id}, the rate can at best be $\cOp(n^{-1/2})$, causing a non-negligible bias in the subsequent analysis. Intuitively, this normalization ensures a clear target for distributional analysis and that the conditioning distribution is not affected by the additional randomness in the latent variables.
\end{proof}
Lemma~\ref{lemma_bound_aug_via_id} plays a crucial role for two reasons. First, the corresponding perturbation term $\bH_P^*$ gives the augmented Hessian $\bH_L^*+\bH_P^*$ a favorable inversion structure, as shown in Lemma~\ref{lemma_approx_hessian_inverse}. Second, the same lemma shows that the magnitude of this perturbation is sufficiently small.
With the augmentation $\bH_P^*$, we can re-write the expansion \eqref{eq_asymp_expand} as 
\begin{equation}
         - \bS_{L}(\btheta^*,\bu^*) = \big\{\bH_L(\btheta^*,\bu^*) + \bH_P^*\big\} \begingroup\renewcommand{\arraystretch}{0.8}\begin{pmatrix}
            \hat\btheta - \btheta^*\\ \hat\bu^{\star} - \bu^*
        \end{pmatrix}\endgroup + \frac{1}{2}\begingroup\renewcommand{\arraystretch}{0.8}\begin{pmatrix}\bR_{L\theta}^{\prime}\\\bR_{Lu}^{\prime}\end{pmatrix}\endgroup,\label{eq_asymp_expand_2}
    \end{equation}
with the residual given as
\begin{equation*}
    \begingroup\renewcommand{\arraystretch}{0.8}\begin{pmatrix}\bR_{L\theta}^{\prime}\\\bR_{Lu}^{\prime}\end{pmatrix}\endgroup = \begingroup\renewcommand{\arraystretch}{0.8}\begin{pmatrix}\bR_{L\theta}\\\bR_{Lu}\end{pmatrix}\endgroup  - 2\bS_L(\hat\btheta,\hat\bu^{\star}) - 2\bH_P^*\begingroup\renewcommand{\arraystretch}{0.6}\begin{pmatrix}
            \hat\btheta - \btheta^*\\\hat\bu^{\star} - \bu^{*}
        \end{pmatrix}\endgroup.
\end{equation*}
By definition, one has $\bS_{Lu}(\hat\btheta, \hat\bu^{\star}) = 0$. 
Then together with \eqref{eq_first_order_near_zero}, Lemmas~\ref{lemma_residual_in_asym_expand} and~\ref{lemma_bound_aug_via_id}, we know that, for each $j\in[q]$,
        \begin{equation}
            \big\|(\bR^{\prime}_{L\theta})_{\cK_j^+}\big\| = \cOp\Big(\frac{nB_n\sqrt{\log n}\tau_{nq}}{\kappa_n^3B_n}\Big),\; \big\|\bR^{\prime}_{L\theta}\big\| = \cOp\Big(\frac{n\sqrt{q}B_n\sqrt{\log n}\tau_{nq}}{\kappa_n^3B_n}\Big),\label{eq_residual_prime_j}
        \end{equation}
        and for each $i\in[n]$,
        \begin{equation}
            \big\|(\bR^{\prime}_{Lu})_{\cK_i}\big\| = \cOp\Big(\frac{qB_n\sqrt{\log n}\tau_{nq}}{\kappa_n^3B_n}\Big),\; \big\|\bR^{\prime}_{Lu}\big\| = \cOp\Big(\frac{\sqrt{n}qB_n\sqrt{\log n}\tau_{nq}}{\kappa_n^3B_n}\Big).\label{eq_residual_prime_i}
        \end{equation}
         Here, we use $\bar\delta_{nq}^2=\cO\{\tau_{nq}/(\kappa_n^3B_n^2)\}$ and $b_n\{\delta_{nq}^2+\varepsilon_{nq}/q\}=\cO\{\tau_{nq}/(\kappa_n^3B_n)\}$.
Now we compute the asymptotic error of $\hat\btheta_j - \btheta_j^* + \{\bH_{L\theta_j\theta_j}(\btheta^*,\bu^*)\}^{-1}\bS_{L\theta_j}(\btheta^*,\bu^*)$ via the expansion \eqref{eq_asymp_expand_2}. Write for short $\bH_{L}^* = \bH_{L}(\btheta^*,\bu^*)$, $\bS_{L}^* = \bS_{L}(\btheta^*,\bu^*)$, $\bH_{L\theta_j\theta_j}^* = \bH_{L\theta_j\theta_j}(\btheta^*,\bu^*)$, $\bS_{L\theta_j}^* = \bS_{L\theta_j}(\btheta^*,\bu^*)$, and $\bdelta = \big((\hat\btheta - \btheta^*)^\T,(\hat\bu^{\star} - \bu^*)^\T\big)^\T$. Then \eqref{eq_asymp_expand_2} is
\begin{equation*}
    (\bH_L^* + \bH_P^*)\bdelta = -\bS_L^* - \frac{1}{2}\begingroup\renewcommand{\arraystretch}{0.8}\begin{pmatrix}\bR_{L\theta}^{\prime}\\\bR_{Lu}^{\prime}\end{pmatrix}\endgroup \tag{\ref{eq_asymp_expand_2}$'$}
\end{equation*}
We then compute
\begin{align*}
    &\,\Big\|\hat\btheta_j - \btheta_j^* + (\bH_{L\theta_j\theta_j}^*)^{-1}\bS_{L\theta_j}^*\Big\|\\
= &\,\Big\|\big[\big\{\bH_L^* + \bH_P^*\big\}^{-1}\big]_{\cK_j^+,}\big\{\bH_L^* + \bH_P^*\big\} \bdelta + (\bH_{L\theta_j\theta_j}^*)^{-1}\bS_{L\theta_j}^*\Big\|\\
\overset{(i)}{\le} &\,\Big\|\big[\big(\bH_L^* + \bH_P^*\big)^{-1} - \big(\bH_{L\theta\theta}^*\big)^{-1}\big]_{\cK_j^+,}\Big\{-\bS_L^* - \frac{1}{2}\begingroup\renewcommand{\arraystretch}{0.8}\begin{pmatrix}\bR_{L\theta}^{\prime}\\
\bR_{Lu}^{\prime}\end{pmatrix}\endgroup\Big\}\Big\| \\
&\, + \Big\|-\big(\bH_{L\theta\theta}^*\big)^{-1}_{\cK_j^+,}\bS_L^* + (\bH_{L\theta_j\theta_j}^*)^{-1}\bS_{L\theta_j}^* \Big\| + \Big\|\frac{1}{2}\big(\bH_{L\theta\theta}^*\big)^{-1}_{\cK_j^+,}\begingroup\renewcommand{\arraystretch}{0.8}\begin{pmatrix}\bR_{L\theta}^{\prime}\\
\bR_{Lu}^{\prime}\end{pmatrix}\endgroup\Big\|\\
\overset{(ii)}{ \le}&\,  \Big\|\big[\big(\bH_L^* + \bH_P^*\big)^{-1} - \big(\bH_{L\theta\theta}^*\big)^{-1}\big]_{\cK_j^+,}\bS_L^* \Big\|\\
&\,+ \frac{1}{2}\Big\|\big[\big(\bH_L^* + \bH_P^*\big)^{-1} - \big(\bH_{L\theta\theta}^*\big)^{-1}\big]_{\cK_j^+,} \begingroup\renewcommand{\arraystretch}{0.8}\begin{pmatrix}\bR_{L\theta}^{\prime}\\
\bR_{Lu}^{\prime}\end{pmatrix}\endgroup\Big\|+ \frac{1}{2} \Big\|\big(\bH_{L\theta_j\theta_j}^*\big)^{-1}\Big\|\big\|(\bR_{L\theta}^{\prime})_{\cK_j^+}\big\|\\
\overset{(iii)}{=} &\,\cOp\Big(\frac{\kappa_n^2B_n^2}{\sqrt{n q}} + \frac{\kappa_n^2}{n\wedge q} + \frac{B_n}{b_n}\tau_{nq}\sqrt{\log n}+\frac{\tau_{nq}\sqrt{\log n}}{\sqrt{B_n}}\Big) = \cOp\Big(\frac{\tau_{nq}\sqrt{\log n}}{\sqrt{B_n}}\Big).
\end{align*}
Here, $(i)$ follows from expansion \eqref{eq_asymp_expand_2}; $(ii)$ follows from the block diagonal structure of $\bH_{L\theta\theta}^*$; and $(iii)$ follows from Lemma~\ref{lemma_approx_hessian_inverse}, $\|(\bH_{L\theta_j\theta_j}^*)^{-1}\| = \cOp((nb_n)^{-1})$ from the proof of Lemma~\ref{lemma_approx_hessian_inverse}, \eqref{eq_residual_prime_j}, \eqref{eq_residual_prime_i} and that 
\begin{align*}
    &\,\Big\|\big[\big(\bH_L^* + \bH_P^*\big)^{-1} - \big(\bH_{L\theta\theta}^*\big)^{-1}\big]_{\cK_j^+,} \begingroup\renewcommand{\arraystretch}{0.8}\begin{pmatrix}\bR_{L\theta}^{\prime}\\
\bR_{Lu}^{\prime}\end{pmatrix}\endgroup\Big\|\\
\le &\,\frac{1}{2}\Big\|\big[\big(\bH_L^* + \bH_P^*\big)^{-1}\bS_m^{-1} - \sqrt{n}\big(\bH_{L\theta\theta}^*\big)^{-1}\big]_{\cK_j^+,}\Big\|\Big\| \bS_m\begingroup\renewcommand{\arraystretch}{0.8}\begin{pmatrix}\bR_{L\theta}^{\prime}\\
\bR_{Lu}^{\prime}\end{pmatrix}\endgroup\Big\|\\
\le &\,\frac{1}{\sqrt{n}}\Big\|\big[\bS_m\big(\bH_L^* + \bH_P^*\big)^{-1}\bS_m^{-1} - n\big(\bH_{L\theta\theta}^*\big)^{-1}\big]_{\cK_j^+,}\Big\|\Big\| \bS_m\begingroup\renewcommand{\arraystretch}{0.8}\begin{pmatrix}\bR_{L\theta}^{\prime}\\
\bR_{Lu}^{\prime}\end{pmatrix}\endgroup\Big\|\\
= &\,\cOp\Big(n^{-1/2}\times \kappa_n^3\sqrt{B_n}q^{-1/2}\times \frac{\sqrt{nq}B_n\tau_{nq}\sqrt{\log n}}{\kappa_n^3B_n}\Big)\\
 = &\,\cOp\Big(\frac{\tau_{nq}\sqrt{\log n}}{\sqrt{B_n}}\Big).
\end{align*}
Here the second step inserts $\bS_m^{-1}\bS_m$, the third uses that the rows indexed by $\cK_j^+$ lie in the $\btheta$-block on which $\bS_m^{-1} = \sqrt{n}\,\bI$, and  the last equality follows from Lemma~\ref{lemma_approx_hessian_inverse}, \eqref{eq_residual_prime_j}, and \eqref{eq_residual_prime_i}. This gives
\begin{equation}
    \hat\btheta_j - \btheta_j^* = -\big(\bH_{L\theta_j\theta_j}^*\big)^{-1}\bS_{L\theta_j}^* +\cOp\Big(\frac{\tau_{nq}\sqrt{\log n}}{\sqrt{B_n}}\Big).\label{eq_asymp_expand_2_bgamma}
\end{equation}
Similarly, one can show from the expansion \eqref{eq_asymp_expand_2} that
\begin{equation}
    \hat\bu_i^{\star} - \bu_i^* = -\big(\bH_{Lu_iu_i}^*\big)^{-1}\bS_{Lu_i}^* +\cOp\Big(\frac{\tau_{nq}\sqrt{\log n}}{\sqrt{B_n}}\Big).\label{eq_asymp_expand_2_bu}
\end{equation}
Subsequently, with $\sqrt{n\log n}\tau_{nq}\to 0$, to prove Theorem~\ref{thm_asymp}, it suffices to show that, for each $j\in[q]$
\begin{equation}
    (\bPhi^*_j)^{-1/2}(\bH_{L\theta_j\theta_j}^*)^{-1}\bS_{L\theta_j}^* \overset{d}{\to}\cN\big(\zero_{K+1},\bI_{K+1}\big)\text{, as }n,q\to\infty.\label{eq_clt}
\end{equation}Here, we note that $\|(\bPhi_j^*)^{-1/2}\| = \cOp(\sqrt{nB_n})$ by Assumptions~\ref{assumption_psd_covariance} and~\ref{assumption_smoothness}.
To verify \eqref{eq_clt}, write $\bu_i^{c*}=\big(1,(\bu_i^*)^\T\big)^\T\in\RR^{K+1}$ and set
\begin{equation*}
    \bH_j=-\sum_{i=1}^n\ell_{ij}''(\eta_{ij}^*)\bu_i^{c*}(\bu_i^{c*})^\T,
    \qquad
    \bV_j=\sum_{i=1}^n\{\ell_{ij}'(\eta_{ij}^*)\}^2\bu_i^{c*}(\bu_i^{c*})^\T,
\end{equation*}
so that $\bPhi_j^*=\bH_j^{-1}\bV_j\bH_j^{-1}$. Conditional on $\{\bu_i^*\}_{i=1}^n$, the score $\bS_{L\theta_j}^*= -\sum_{i=1}^n \ell_{ij}'(\eta_{ij}^*)\bu_i^{c*}$ is a sum of independent mean-zero random vectors. Let $\bar\bV_j=\sum_{i=1}^n\EE\!\left[\{\ell_{ij}'(\eta_{ij}^*)\}^2\mid\bu_i^*\right]\bu_i^{c*}(\bu_i^{c*})^\T$.
The curvature bounds, the information identity, and the empirical moment bounds for $\{\bu_i^{c*}\}$ imply that the eigenvalues of $\bH_j$ and $\bar\bV_j$ lie between constant multiples of $nb_n$ and $nB_n$ with probability tending to one. Moreover, the conditional sub-exponential bound and $\max_i\|\bu_i^{c*}\|=\cOp(\sqrt{\log n})$, together with Assumption~\ref{assump_scaling}, yield the conditional Lindeberg condition. Hence the multivariate Lindeberg--Feller theorem gives
\[
    \bar\bV_j^{-1/2}\bS_{L\theta_j}^*\ \overset{d}{\longrightarrow}\ \cN(\zero_{K+1},\bI_{K+1}).
\]
It remains to show that the above holds when replacing $\bar\bV_j$ with $\bV_j$. Conditionally on $\{\bu_i^*\}_{i\in[n]}$, $\bV_j-\bar\bV_j = \sum_{i=1}^n\big[\{\ell_{ij}'(\eta_{ij}^*)\}^2 - \EE\big[\big\{\ell_{ij}'(\eta_{ij}^*)\}^2\mid\bu_i^*]\big]\bu_i^{c*}(\bu_i^{c*})^\T$. The conditional sub-exponential bound implies $\EE[\{\ell_{ij}'(\eta_{ij}^*)\}^4\mid\mathscr U_n]\le CB_n^2$. Truncate the scores at $C\sqrt{B_n}\log(nq)$ and center their squared values conditionally on $\mathscr U_n$. The truncation probability and bias are negligible for sufficiently large $C$. Using $\max_i\|\bu_i^{c*}\|^2=\cOp(\log n)$ and $n^{-1}\sum_i\|\bu_i^{c*}\|^4=\cOp(1)$, Bernstein's inequality for these bounded, independent variables and a union bound over $j$ give $\max_{j\in[q]}\|\bV_j-\bar\bV_j\|=\cOp[B_n\{\sqrt{n\log q}+\log^4(nq)\}]=\cOp(B_n\sqrt{n\log q})$, where the last equality uses Assumption~\ref{assump_scaling}. Assumption~\ref{assump_scaling}, $\lambda_{\min}(\bar\bV_j)\ge cnb_n$, and telescoping yield $\big\|\bar\bV_j^{-1/2}(\bV_j-\bar\bV_j)\bar\bV_j^{-1/2}\big\|=\cOp\big(B_n\sqrt{\log q}/(b_n\sqrt{n})\big)=o_p(1)$.
Slutsky's theorem and $\bPhi_j^*=\bH_j^{-1}\bV_j\bH_j^{-1}$ now yield \eqref{eq_clt}. In particular, $\|(\bPhi_j^*)^{-1/2}\|=\cOp(\sqrt{nB_n})$. The asymptotic distribution with $\bPhi_j^*$ replaced by the empirical counterpart can be similarly verified.

Similarly, one can establish that
\begin{equation}
    (\bPsi^*_i)^{-1/2}(\bH_{Lu_iu_i}^*)^{-1}\bS_{Lu_i}^*\overset{d}{\to}\cN(\zero_K,\bI_K)\text{, as }n,q\to\infty.\label{eq_asymp_expand_2_bu_2}
\end{equation}
This will be used to establish Corollary~\ref{thm_asymp_bu}.

\subsection{Proof of Corollary~\ref{thm_asymp_bu}}\label{supp_sec_prove_thm_asymp_bu}
With Equations~\eqref{eq_asymp_expand_2_bu} and~\eqref{eq_asymp_expand_2_bu_2}, it suffices to show that $\hat\bu_i^{\star} - \hat\bu_i$ is negligible in the asymptotic distribution, where
\begin{equation*}
    \hat\bu_i^{\star} = \argmax_{\bu\in\RR^K}\ell_i(\bu;\hat\btheta) \quad\text{ and }\quad\hat\bu_i = \argmax_{\bu\in\RR^K}\big\{\ell_i(\bu;\hat\btheta) - \frac{1}{2}\|\bu\|^2\big\}
\end{equation*}
By the first-order optimality conditions, we know
\begin{align*}
    \zero = &\, \partial_{\bu}\ell_i(\hat\bu_i^\star;\hat\btheta) - \big\{\partial_{\bu}\ell_i(\hat\bu_i;\hat\btheta)-\hat\bu_i\big\} \\ = &\,
\hat\bu_i-\int_0^1 \partial_{\bu\bu}^2
\ell_i(\hat\bu_i^\star+t\bd_i;\hat\btheta) dt\big(\hat\bu_i-\hat\bu_i^\star\big).
\end{align*}
Here, the second equality follows from the mean-value expansion of $\partial_{\bu}\ell_i(\cdot;\hat\btheta)$ between $\hat\bu_i^\star$ and $\hat\bu_i$, with $\bd_i=\hat\bu_i-\hat\bu_i^\star$. Letting $\bar\bH_i = -\int_0^1 \nabla_{\bu\bu}^2 \ell_i(\hat\bu_i^\star+t\bd_i;\hat\btheta) dt$, we obtain
\begin{equation*}
    \hat\bu_i-\hat\bu_i^\star = -(\bar\bH_i+\bI_K)^{-1}\hat\bu_i^\star.
\end{equation*}
Assumptions~\ref{assumption_psd_covariance} and~\ref{assumption_smoothness} give $\lambda_{\min}(\bar\bH_i)\ge cb_nq\gg 1$. Since $\hat\bu_i^{\star} = \cOp(\sqrt{\log n})$ by the latent-variable tail bound and Lemma~\ref{lemma_residual_in_asym_expand}, we conclude that \[\big\|\hat\bu_i-\hat\bu_i^\star\big\| = \cOp\big({\sqrt{\log n}}/{(b_nq)}\big).\] The definition of $\bPsi_i^*$ and the curvature bounds give $\|(\bPsi_i^*)^{-1/2}\| = \cOp(\sqrt {qB_n})$, which yields
\begin{equation*}
    \big\|\big(\bPsi_i^*)^{-1/2}\big(\hat\bu_i-\hat\bu_i^\star\big)\big\| = \cOp\Big(\frac{\kappa_n\sqrt{\log n}}{\sqrt{q}}\Big).
\end{equation*}
Together with Equations~\eqref{eq_asymp_expand_2_bu} and~\eqref{eq_asymp_expand_2_bu_2}, Assumption~\ref{assump_scaling}, and $\sqrt{q\log n}\tau_{nq}\to 0$, this completes the proof.

\subsection{Proof of Corollary~\ref{thm_bvm_irt}}\label{supp_sec_prove_thm_bvm_irt}
Because $\hat\btheta$ is consistent, the argument follows the classical Bayesian asymptotic framework (see Chapter 10 of van der Vaart \cite{van2000asymptotic}). For completeness, we present the full proof below.

For notational simplicity, fix $i\in[n]$ throughout the proof and write $\hat\bJ_i = -\sum_{j=1}^q\ell_{ij}^{\prime\prime}(\hat\eta_{ij}^{\star}) \hat\bgamma_j\hat\bgamma_j^\T$ and $\hat\bPsi_i^{\star}=\hat\bJ_i^{-1}$. For each $j\in[q]$, let $\hat\eta_{ij}^{\star} = \hat\bgamma_j^\T\hat\bu_i^{\star} + \hat\beta_j$. At the true parameters, let $\bJ_i^*=-\sum_{j=1}^q\ell_{ij}^{\prime\prime}(\eta_{ij}^*)
    \bgamma_j^*(\bgamma_j^*)^\T$ and $\bar\bPsi_i^*=(\bJ_i^*)^{-1}.$ By Assumptions~\ref{assumption_psd_covariance} and~\ref{assumption_smoothness}, together with \eqref{eq_consis_ave_all}, Lemma~\ref{lemma_residual_in_asym_expand}, and $\delta_{nq}\ll b_n^2/B_n^2$, we obtain with probability approaching one that
\begin{equation}\label{eq_bvm_curvature_bounds}
    cqb_n\le \lambda_{\min}(\hat\bJ_i)\le \lambda_{\max}(\hat\bJ_i)\le CqB_n,\qquad
    cqb_n\le \lambda_{\min}(\bJ_i^*)\le \lambda_{\max}(\bJ_i^*)\le CqB_n,
\end{equation} 
for some constants $c,C$. This is similar to the derivation of \eqref{eq_reference_eigen_lowerbound} in Section~\ref{supp_sec_prove_lemma_convexity2}. A telescoping argument and a Taylor expansion of $\ell''_{ij}(\cdot)$, with $|\ell^{(3)}_{ij}(\eta)|\le B_n$ for $|\eta|\le C\sqrt{\log n}$, yield
\begin{equation*}
    \big\|(\bJ_i^*)^{-1/2}(\hat\bJ_i-\bJ_i^*)(\bJ_i^*)^{-1/2}\big\| = \cOp(b_n^{-1}B_n\bar\delta_{nq})
    =\cOp\big(B_n^2b_n^{-2}\delta_{nq}\big)=o_p(1).
\end{equation*}
Consequently, by the continuity of Gaussian densities in total variation,
\begin{equation*}
    \int_{\RR^K}\left|
    \pwk(\bu\mid \hat\bu_i^\star,\hat\bPsi_i^\star)
    -\pwk(\bu\mid \hat\bu_i^\star,\bar\bPsi_i^*)
    \right|d\bu=o_p(1).
\end{equation*}
It then suffices to prove the approximation with the Gaussian density with covariance $\hat\bPsi_i^\star$ as
\begin{equation}\label{eq_bvm_fit_cov_goal}
    \int_{\RR^K}\left|\hat p_i(\bu\mid\bX)-
    \pwk(\bu\mid \hat\bu_i^\star,\hat\bPsi_i^\star)\right|d\bu=o_p(1).
\end{equation}
Recall that $\pwk(\bx)=(2\pi)^{-K/2}\exp(-\|\bx\|^2/2)$ is the  standard Gaussian density. Let $\bu_{\bx}=\hat\bu_i^\star+(\hat\bPsi_i^\star)^{1/2}\bx$ and define
\begin{equation*}
    R_{nq}(\bx)=
    \exp\{\ell_i(\bu_{\bx};\hat\btheta)-\ell_i(\hat\bu_i^\star;\hat\btheta)\}\frac{\pwk(\bu_{\bx})}{\pwk(\hat\bu_i^\star)},
\end{equation*}
with $Z_{nq} = \int_{\RR^K}R_{nq}(\bz)d\bz$. Then $\det(\hat\bPsi_i^\star)^{1/2}\hat p_i(\bu_{\bx}\mid\bX)=R_{nq}(\bx)/Z_{nq}$. After the change of variables $\bx=(\hat\bPsi_i^\star)^{-1/2}(\bu-\hat\bu_i^\star)$, \eqref{eq_bvm_fit_cov_goal} is equivalent to
\begin{equation}\label{eq_bvm_rescaled_goal}
    \int_{\RR^K}\left|\frac{R_{nq}(\bx)}{Z_{nq}}-\pwk(\bx)\right|d\bx=o_p(1).
\end{equation}

By Assumption~\ref{assump_scaling}, we can choose deterministic sequences $m_q\to\infty$ and $M_q\to\infty$, with $m_q=o(M_q)$, such that
\begin{equation*}
    \frac{B_nm_q^3}{b_n^{3/2}\sqrt q}\to0,\qquad
    \frac{M_q\sqrt{\log n}}{\sqrt{qb_n}}\to0, \qquad
    \frac{b_nm_q^2}{B_n\log(n\vee q)}\to\infty.
\end{equation*}
In particular, $m_q\sqrt{\log n}/ \sqrt{qb_n}\to 0$ and $b_nM_q^2/\{B_n\log(n\wedge q)\}\to\infty$. Then we partition $\RR^{K}$ into
\begin{equation*}
    A_x^{(1)} = \{\bx:\|\bx\|\le m_q\},
    \quad
    A_x^{(2)} = \{\bx:m_q<\|\bx\|\le M_q\},
    \quad
    A_x^{(3)} = \{\bx:\|\bx\|>M_q\},
\end{equation*}
and bound the integral \eqref{eq_bvm_rescaled_goal} separately over these regions.

We first consider $A_x^{(1)}$. Since $\hat\bu_i^\star$ maximizes $\ell_i(\bu;\hat\btheta)$, it holds that $\nabla_{\bu}\ell_i(\hat\bu_i^\star;\hat\btheta)=\zero $. For $\bx\in A_x^{(1)}$, Taylor's theorem gives
\begin{equation}\label{eq_talor_expand_ell_u}
\begin{aligned}
    \ell_i(\bu_{\bx};\hat\btheta)  - \ell_i(\hat\bu_i^\star;\hat\btheta)
    &= \frac12 \sum_{j=1}^q \ell_{ij}^{\prime\prime}(\hat\eta_{ij}^\star) \bx^\T (\hat\bPsi_i^\star)^{1/2} \hat\bgamma_j\hat\bgamma_j^\T (\hat\bPsi_i^\star)^{1/2} \bx + r_{nq}(\bx)  \\
    &= -\frac12\|\bx\|^2+r_{nq}(\bx),
\end{aligned}
\end{equation}
where the last equality follows from $(\hat\bPsi_i^\star)^{1/2}\hat\bJ_i (\hat\bPsi_i^\star)^{1/2} = \bI_K$, and
\begin{equation*}
    \sup_{\bx\in A_x^{(1)}}|r_{nq}(\bx)| \le C B_n q \|\hat\bPsi_i^\star\|^{3/2}m_q^3 = \cOp\Big( \frac{B_nm_q^3}{b_n^{3/2}\sqrt q} \Big) = o_p(1).
\end{equation*}
Since $\|\hat\bu_i^\star\|=\cOp(\sqrt{\log n})$ and $\sup_{\bx\in A_x^{(1)}}\|\bu_{\bx}-\hat\bu_i^\star\| \le \|\hat\bPsi_i^\star\|^{1/2}m_q = \cOp\big({m_q}/{\sqrt{qb_n}}\big)$, we obtain
\begin{equation*}
    \sup_{\bx\in A_x^{(1)}} \left| \frac{\pwk(\bu_{\bx})}{\pwk(\hat\bu_i^\star)}-1\right| =\sup_{\bx\in A_x^{(1)}} \left| \exp\big\{ - \hat\bu_i^{\star\T}(\bu_{\bx}-\hat\bu_i^\star) - \frac12\|\bu_{\bx}-\hat\bu_i^\star\|^2\big\}- 1 \right|=
    o_p(1).
\end{equation*}
Hence we have \[\sup_{\bx\in A_x^{(1)}}\big|R_{nq}(\bx)\exp(\|\bx\|^2/2) - 1\big| = o_p(1),\] which gives
\begin{equation*}
    \int_{A_x^{(1)}}\left| R_{nq}(\bx)-\exp(-\|\bx\|^2/2) \right|d\bx = o_p(1).
\end{equation*}

Next consider $A_x^{(2)}$. For $\bx\in A_x^{(2)}$, the segment between $\hat\bu_i^\star$ and $\bu_{\bx}$ is bounded within $\{\bu:\|\bu\|\le C\sqrt{\log n}\}$
with probability tending to one as $M_q\sqrt{\log n}/\sqrt{qb_n}\to 0$. Using \eqref{eq_talor_expand_ell_u}, we then have by Assumption~\ref{assumption_smoothness} that
\begin{equation*}
    \ell_i(\bu_{\bx};\hat\btheta) - \ell_i(\hat\bu_i^\star;\hat\btheta) \le -c\frac{b_n}{B_n}\|\bx\|^2.
\end{equation*}
For the ratio ${\pwk(\bu_{\bx})}/{\pwk(\hat\bu_i^\star)}$, similarly one has
\begin{equation}\label{eq_prior_diff}
    \frac{\pwk(\bu_{\bx})}{\pwk(\hat\bu_i^\star)} \le \exp\left\{ C\frac{\sqrt{\log n}}{\sqrt{qb_n}}\|\bx\| \right\}.
\end{equation}
Since $m_q\sqrt{\log n}\ll\sqrt{qb_n}$, $b_nm_q^2\gg B_n$, and $\|\bx\|>m_q$ on $A_x^{(2)}$, the chosen sequences satisfy
\begin{equation*}
    C\frac{\sqrt{\log n}}{\sqrt{qb_n}}\le \frac{c}{2}\frac{b_n}{B_n}m_q,
\end{equation*}
for all sufficiently large $n$ and $q$. In this case, with probability tending to one,
\begin{equation*}
    R_{nq}(\bx)\le \exp\left\{ -c\frac{b_n}{2B_n}\|\bx\|^2\right\},
\end{equation*}
for $\bx\in A_x^{(2)}$.
Thus, because $K$ is fixed, one can check that 
\begin{align*}
    \int_{A_x^{(2)}}R_{nq}(\bx)d\bx \le &\,\int_{\|\bx\|>m_q}\exp\left\{-c\frac{b_n}{2B_n}\|\bx\|^2\right\}d\bx \\=&\, \left(\frac{B_n}{cb_n}\right)^{K/2}\int_{\|\by\|>\frac{m_q\sqrt{cb_n}}{\sqrt{B_n}}}e^{-\|\by\|^2/2}d\by \le C\left(\frac{B_n}{b_n}\right)^{K/2}\Big(\frac{m_q\sqrt{cb_n}}{\sqrt{B_n}}\Big)^{K-2}e^{-\big(m_q\sqrt{cb_n/B_n}\big)^2/2},
\end{align*}
Since $s_q^2 = cb_nm_q^2/B_n \gg \log(n\vee q)$, $\int_{A_x^{(2)}}R_{nq}(\bx)d\bx = o_p(1)$.
With $\int_{A_x^{(2)}}\exp(-\|\bx\|^2/2)d\bx=o_p(1)$ as $m_q\to\infty$, we get
\begin{equation*}
    \int_{A_x^{(2)}} \left| R_{nq}(\bx)-\exp(-\|\bx\|^2/2) \right|d\bx = o_p(1).\end{equation*}

It remains to handle $A_x^{(3)}$. For any unit vector $\bv\in\mathbb S^{K-1}$, define $h_{\bv}(t) = \ell_i\big( \hat\bu_i^\star + (\hat\bPsi_i^\star)^{1/2}t\bv ;\hat\btheta \big) - \ell_i(\hat\bu_i^\star;\hat\btheta)$.
For $0\le t\le M_q$, the path remains inside $\{\bu:\|\bu\|\le C\sqrt{\log n}\}$ with probability tending to one. Hence $h_{\bv}''(t)\le-cb_n/B_n$ for $0\le t\le M_q$ uniformly over $\bv\in\mathbb S^{K-1}$. Since $h_{\bv}'(0)=0$, it follows that
\begin{equation*}
    h_{\bv}(M_q) \le -c\frac{b_n}{B_n}M_q^2,
    \qquad
    h_{\bv}'(M_q) \le -c\frac{b_n}{B_n}M_q.
\end{equation*}
For $t>M_q$, $h_{\bv}$ remains concave by the global concavity clause of Assumption~\ref{assumption_smoothness}, although its curvature may diminish. Thus,
\begin{equation*}    h_{\bv}(t) \le h_{\bv}(M_q)+h_{\bv}'(M_q)(t-M_q) \le -c\frac{b_n}{B_n}M_qt,
\end{equation*}
for $t\ge M_q$.
Subsequently, for $\bx=t\bv\in A_x^{(3)}$, we can bound $\ell_i(\bu_{\bx};\hat\btheta) - \ell_i(\hat\bu_i^\star;\hat\btheta)$ by 
\begin{equation*} 
    \ell_i(\bu_{\bx};\hat\btheta) - \ell_i(\hat\bu_i^\star;\hat\btheta) \le -c\frac{b_n}{B_n}M_q\|\bx\|.
\end{equation*}
Combining this with  \eqref{eq_prior_diff}, we know
\begin{equation*}
    R_{nq}(\bx)
    \le \exp\Big\{ - \Big( c\frac{b_n}{B_n}M_q - C\frac{\sqrt{\log n}}{\sqrt{qb_n}} \Big) \|\bx\| \Big\}.
\end{equation*}
Since $\frac{\sqrt{\log n}/\sqrt{qb_n}} {(b_n/B_n)M_q} = o(1)$, with probability tending to one, it holds that
\begin{equation*}
    R_{nq}(\bx)
    \le \exp\left\{ -c\frac{b_n}{B_n}M_q\|\bx\| \right\},
\end{equation*}
for $\bx\in A_x^{(3)}$.
Then one can also show $\int_{A_x^{(3)}}R_{nq}(\bx)d\bx = o_p(1)$ under $b_nM_q^2/B_n\gg\log(n\vee q)$. With $\int_{A_x^{(3)}}\exp(-\|\bx\|^2/2)d\bx=o(1)$, we arrive at
\begin{equation*}
    \int_{A_x^{(3)}} \left| R_{nq}(\bx)-\exp(-\|\bx\|^2/2) \right|d\bx =o_p(1).
\end{equation*}

Combining the bounds over $A_{x}^{(1)}$--$A_{x}^{(3)}$ gives
\begin{equation*}
    \int_{\RR^K} \left| R_{nq}(\bx)-\exp(-\|\bx\|^2/2) \right|d\bx = o_p(1).
\end{equation*}
We can also derive that $Z_{nq} = \int_{\RR^K}R_{nq}(\bx)d\bx = \int_{\RR^K}\exp(-\|\bx\|^2/2)d\bx + o_p(1) = (2\pi)^{K/2}+o_p(1).$ Therefore, we conclude that
\[
\begin{aligned}
    \int_{\RR^K} \left| \frac{R_{nq}(\bx)}{Z_{nq}} - \pwk(\bx) \right|d\bx &\le Z_{nq}^{-1} \int_{\RR^K} \left| R_{nq}(\bx)-\exp(-\|\bx\|^2/2) \right|d\bx  \\ &\quad+ \left| Z_{nq}^{-1} - (2\pi)^{-K/2} \right| \int_{\RR^K}\exp(-\|\bx\|^2/2)d\bx  \\ &= o_p(1).
\end{aligned}
\] This establishes \eqref{eq_bvm_rescaled_goal} and completes the proof.

 \section{Estimation under Practical Identifiability Conditions}\label{sec_res_pract_id}

In the main text, we proved the asymptotic theory for marginal maximum likelihood estimation up to transformations $\hat\bQ^*$ and $\hat\bQ^0$, which depend on the true parameters. The result, while statistically significant, is not readily applicable in practice. To remove this ambiguity, identifiability conditions are typically imposed, making the theoretical results practically meaningful.
In this section, we further establish asymptotic theory for marginal likelihood estimation under practical identifiability conditions. We consider the uncorrelated case in Section~\ref{subsec_uncor} and the correlated case in Section~\ref{subsec_cor}.
The practically identified estimators and their covariance corrections below are defined directly by Conditions~\ref{condition1}--\ref{condition2} and do not use either Procrustes rotation.

\subsection{Uncorrelated Case}\label{subsec_uncor}
We introduce the following condition. 
\begin{cond}\label{condition1}
\begin{enumerate}[label=$(\roman*)$]
    \item $n^{-1}\sum_{i=1}^n\bu_i^*  = \zero_K$;
    \item $n^{-1}\sum_{i=1}^n\bu_i^*(\bu_i^*)^\T = \bI_K$; 
    \item up to a column permutation of $\bGamma$, for each $r\in\{2,\dots,K\}$, there is some $\cA_r\subseteq[q]$ of size $r-1$ such that $\bGamma_{\cA_r,r} = \zero_{r-1}$ and ${\rm rank}(\bGamma_{\cA_r,1:r}) = r-1$.
\end{enumerate}
\end{cond}
Conditions~\ref{condition1}$(i)$--$(ii)$ retain the centering and scaling normalizations introduced in Section~\ref{sec_res}, but impose them directly on the realized latent variables. They can be understood as the canonical representative of the realized latent configuration used in the asymptotic theory of Section~\ref{sec_res}. For notational simplicity, we continue to denote the normalized representatives by $\{\bu_i^*\}_{i\in[n]}$. A closely related treatment appears in \citet{bai2012statistical}, who study a quasi-likelihood for linear factor models. Although the factors are formally treated as fixed (when they are random, the analysis is interpreted conditionally on their realized values), their estimation procedure nevertheless follows the random-effects likelihood formulation of classical factor analysis, using a covariance-based criterion in which the factors enter only through their empirical second moment~\citep{andersonintroduction}. For identification and inference, the realized factors are centered, and one of their identification schemes normalizes their empirical second moment to the identity matrix, paralleling Conditions~\ref{condition1}$(i)$--$(ii)$.

To remove the remaining orthogonal indeterminacy, part~$(iii)$ imposes $K(K-1)/2$ constraints on $\bGamma$, matching the number of degrees of freedom of a $K$-dimensional orthogonal matrix. This condition corresponds to IC$^{(1)}$ of~\citet{cui2025identifiability}.
These constraints encode prior structural information: for each known $j\in\cA_r$, the response $X_{ij}$ does not depend directly on $U_{ir}$ after conditioning on the remaining latent variables. A common example is a lower-triangular design for $\bGamma$, which is widely used in factor analysis~\citep{huber2004estimation,bai2012statistical,cui2025identifiability}.

We consider the following marginal estimator subject to this identifiability condition.
\begin{equation}
    \begin{aligned}
    \hat\btheta^{(1)} = &\,\argmin_{\btheta\in\Xi}\cL(\btheta)\\&\,\text{ subject to } \gamma_{j,r} = 0\text{ when }j\in\cA_r\text{ for }r=2,\ldots,K.
    \end{aligned}\label{eq_mmle_IC2}
\end{equation}
Parts $(i)$ and $(ii)$ of Condition~\ref{condition1} are implemented in \eqref{eq_mmle_IC2} implicitly via working prior $\pwk(\cdot)$ under its implicit normalizing effect. 
In what follows, we present the asymptotic theory for the estimator $\hat\btheta^{(1)}$.

\begin{theorem}[Asymptotic Theory under Condition~\ref{condition1}]\label{thm_asymp_id1}
\it
    Suppose Assumptions~\ref{assump_standard_id}--\ref{assump_scaling} hold and $(\btheta^*,\bu^*)$ satisfy Condition~\ref{condition1}. Let $\cS_j:=\{r\in[K]:j\notin\cA_r\}$ be the indices of the free entries of $\bgamma_j$, with $\cA_1:=\varnothing$, and let $\bE_{(j)}=(\bI_{K+1})_{\{1\}\cup\{r+1:r\in\cS_j\},}$. Then, for $\hat\btheta^{(1)}$ obtained in \eqref{eq_mmle_IC2}, the following conclusions hold conditionally on $\{\bU_i\}_{i\in[n]}$.
\begin{enumerate}[label=$(\roman*)$]
    \item If $\sqrt{n\log n}\tau_{nq}\to 0$ as $n,q\to\infty$, then
\begin{equation*}
    \Big(\bE_{(j)}\bPhi_{j}^{(1)}\bE_{(j)}^\T\Big)^{-1/2}\begin{pmatrix}
        \hat\beta_j^{(1)} - \beta_j^*\\\hat\bgamma_{j,\cS_j}^{(1)} - \bgamma_{j,\cS_j}^*\end{pmatrix}\overset{d}{\to}\cN\left(\zero_{|\cS_j|+1},\bI_{|\cS_j|+1}\right),
\end{equation*}
where $\bPhi^{(1)}_j$ is given in \eqref{eq_define_cov_ic1};
\item Let $\hat\bu_i^{(1)} = \mathop{\argmax}_{\bu\in\RR^{K}}\big\{\ell_i\big(\bu;\hat\btheta^{(1)}\big) + \log \pwk(\bu)\big\}$ be the empirical Bayes estimator. If $\sqrt{q\log n}\tau_{nq}\to 0$ as $n,q\to\infty$, then
    \begin{equation*}
    \big(\bPsi_i^{(1)}\big)^{-1/2}\big(\hat\bu_i^{(1)} - \bu_i^*\big)\overset{d}{\to}\cN(\zero_K,\bI_K), \text{ as }n,q\to \infty,
\end{equation*}
where $\bPsi^{(1)}_i$ is given in \eqref{eq_define_cov_psi_i_1};
    \item Let $\hat\bu_i^{(1)\star} = \argmax_{\bu\in\RR^{K}}\ell_i(\bu;\hat\btheta^{(1)})$, $\bar\bPsi_{i}^* = \big\{-\sum_{j=1}^q\ell_{ij}^{\prime\prime}(\eta_{ij}^*)\bgamma_j^*(\bgamma_j^*)^\T\big\}^{-1}$, and
    \begin{equation*}
    \hat p_i^{(1)}(\bu\mid \bX) = \frac{\exp\{\ell_i(\bu;\hat\btheta^{(1)})\}\pwk(\bu)}{\int_{\bv\in\RR^{K}}\exp\{\ell_i(\bv;\hat\btheta^{(1)})\}\pwk(\bv)d\bv};
\end{equation*}
then
    \begin{equation*}
          \int_{\RR^{K}} \Big|\hat p_i^{(1)}(\bu \mid \bX) - \pwk(\bu\mid \hat{\bu}_i^{(1)\star}, \bar\bPsi^*_i) \Big|d\bu  \overset{p}{\to} 0 \quad \text{as } n,q \to \infty.
    \end{equation*}
\end{enumerate}
\end{theorem}
\begin{proof}
    See Section~\ref{supp_sec_prove_thm_asymp_id1}.
\end{proof}
Theorem~\ref{thm_asymp_id1} establishes the asymptotic theory for marginal likelihood estimation under Condition~\ref{condition1}. This result is practically useful because, unlike the estimator studied in the preceding section, the constrained estimator $\hat\btheta^{(1)}$ is directly computable from the data.

A related setting was studied by \citet{huber2004estimation}, but their analysis focuses on the fixed-$q$ regime. In that setting, the number of model parameters remains fixed, so the asymptotic theory can be developed using classical finite-dimensional $M$-estimation arguments. By contrast, our regime allows $q$ and hence the number of model parameters to diverge. More importantly, the high-dimensional marginal likelihood does not have uniformly regular curvature in all directions, as established in Lemma~\ref{lemma_convexity}. Therefore, even after imposing identifiability constraints, the asymptotic analysis is quite different from the classical fixed-dimensional theory.

The theorem also reveals an interesting phenomenon: under practical identifiability constraints, the asymptotic covariance of the constrained marginal estimator $\hat\btheta^{(1)}$ can be larger than that of $\hat\btheta^Q$ studied in the previous section. A similar phenomenon was observed in the fixed-effect analysis of \citet{cui2025identifiability}, where the latent variables are treated as fixed parameters and estimated jointly with the model parameters. Intuitively, this occurs because the identifiability constraints fix certain entries of $\bGamma$ at zero to determine a representative within the rotational equivalence class. Although these entries are not estimated as free parameters, their corresponding uncertainty is not eliminated. Instead, it is propagated through the data-dependent transformation needed to enforce the constraints, thereby enlarging the uncertainty of the remaining free coordinates.

Another feature concerns posterior asymptotic normality in part~$(iii)$, which describes a different notion of uncertainty from part~$(ii)$. The posterior covariance is $\bar\bPsi_i^*$, the same as in Corollary~\ref{thm_bvm_irt}. Under Condition~\ref{condition1}, the additional uncertainty from estimating $\btheta$ is asymptotically negligible for posterior asymptotic normality but can still contribute to the frequentist variance of the point estimator. Consequently, posterior credible regions should be constructed using $\bar\bPsi_i^*$, whereas frequentist confidence regions for $\bu_i^*$ should use $\bPsi_i^{(1)}$, and these two can differ significantly. We illustrate this distinction through simulations in Section~\ref{supp_simu_2}.

\subsection{Correlated Case}\label{subsec_cor}
We next turn to settings with correlated latent variables. We consider the following identifiability condition adapted from IC$^{(4)}$ of~\citet{cui2025identifiability}.
\begin{cond}\label{condition2}
\begin{enumerate}[label= $(\roman*)$]
    \item $n^{-1}\sum_{i=1}^n\bu_i^*  = \zero_K$;
    \item $\mathrm{diag}\big(n^{-1}\sum_{i=1}^n\bu_i^*(\bu_i^*)^\T\big) = \bI_K$; 
    \item for each $r\in\{1,\dots,K\}$, there is some $\cA_r\subseteq[q]$ of size $K-1$ such that $\bGamma_{\cA_r,r} = \zero_{K-1}$ and ${\rm rank}(\bGamma_{\cA_r,}) = K-1$.
\end{enumerate}
\end{cond}
Compared with Condition~\ref{condition1}, Condition~\ref{condition2} relaxes the constraints on the off-diagonal entries of $n^{-1}\sum_{i=1}^n\bu_i^*(\bu_i^*)^\T$, thereby allowing the latent variables to be correlated. As before, the normalization is imposed directly on the realized latent variables for notational convenience. Since part~$(ii)$ now consists of $K$ constraints rather than the $K(K+1)/2$ constraints in Condition~\ref{condition1}, part~$(iii)$ compensates by imposing additional restrictions on $\bGamma$, similar to IC$^{(4)}$ in~\citet{cui2025identifiability}.

Throughout this subsection, we replace the identity-covariance clause of Assumption~\ref{assump_standard_id} by the requirement that $\bSigma_u^*=\EE(\bU_i\bU_i^\T)$ be positive definite with $\mathrm{diag}(\bSigma_u^*)=\bI_K$ and eigenvalues bounded above and away from zero. The other parts of Assumption~\ref{assump_standard_id} are unchanged.

For estimation, recall that the working prior $\pwk(\cdot)$ imposes an implicit normalization on the latent variables. To accommodate correlated latent variables, use $\cL(\btheta,\bmu,\bSigma)$ as defined in \eqref{eq_full_marginal_likelihood}. Specifically, we define
\begin{equation}
    \begin{aligned}
    (\hat\btheta^{(2)},\hat\bSigma^{(2)}) = &\,\argmin_{\btheta\in\Xi,\bSigma\in \SSS_{++}^K, }\cL(\btheta,\zero_K,\bSigma)\\&\,\text{ subject to }\gamma_{j,r} = 0\text{ when }j\in\cA_r\text{ for }r\in[K],\; \text{ and }\mathrm{diag}(\bSigma) = \bI_K\\&\qquad\qquad\quad\text{and } \lambda_{\min}(\bSigma)>c\text{ for some constant }c
    \end{aligned}\label{eq_mmle_IC3}
\end{equation}
Part~$(i)$ of Condition~\ref{condition2} is imposed by setting $\bmu = \zero_K$, and part~$(ii)$ is implemented in \eqref{eq_mmle_IC3} through the constraint $\mathrm{diag}(\bSigma) = \bI_K$.
In what follows, we present the limiting distributions for the estimator $\hat\btheta^{(2)}$. 
\begin{theorem}[Asymptotic Theory under Condition~\ref{condition2}]\label{thm_asymp_id2}
\it
     Suppose the modified version of Assumption~\ref{assump_standard_id} stated above and Assumptions~\ref{assumption_psd_covariance}--\ref{assump_scaling} hold, and suppose $(\btheta^*,\bu^*)$ satisfies Condition~\ref{condition2}. Let $\cS_j:=\{r\in[K]:j\notin\cA_r\}$ be the indices of the free entries of $\bgamma_j$, and let 
     \(\bE_{(j)} = \begin{pmatrix}
         1 \\ & (\bI_{K})_{\cS_j,}
     \end{pmatrix}.\)
     Then, for $\hat\btheta^{(2)}$ obtained in \eqref{eq_mmle_IC3}, the following conclusions hold conditionally on $\{\bU_i\}_{i\in[n]}$.
\begin{enumerate}[label=$(\roman*)$]
    \item If $\sqrt{n\log n}\tau_{nq}\to 0$ as $n,q\to\infty$, then
\begin{equation*}
    \Big(\bE_{(j)}\bPhi_{j}^{(2)}\bE_{(j)}^\T\Big)^{-1/2}\begin{pmatrix}
        \hat\beta_j^{(2)} - \beta_j^*\\\hat\bgamma_{j,\cS_j}^{(2)} - \bgamma_{j,\cS_j}^*\end{pmatrix}\overset{d}{\to}\cN\left(\zero_{|\cS_j|+1},\bI_{|\cS_j|+1}\right),
\end{equation*}
where $\bPhi^{(2)}_j$ is given in \eqref{eq_define_cov_ic2};
\item Let $\hat\bu_i^{(2)} = \argmax_{\bu\in\RR^{K}}\big\{\ell_i\big(\bu;\hat\btheta^{(2)}\big) + \log\pwk(\bu\mid \zero_K,\hat\bSigma^{(2)})\big\}$ be the empirical Bayes estimator. If $\sqrt{q\log n}\tau_{nq}\to 0$ as $n,q\to\infty$, then
    \begin{equation}
    \big(\bPsi_i^{(2)}\big)^{-1/2}\big(\hat\bu_i^{(2)} - \bu_i^*\big)\overset{d}{\to}\cN(\zero_K,\bI_K), \text{ as }n,q\to\infty,
\end{equation}
where $\bPsi^{(2)}_i$ is given in \eqref{eq_define_cov_psi_i_2};
    \item Let $\hat\bu_i^{\star(2)} = \argmax_{\bu\in\RR^K}\ell_i(\bu;\hat\btheta^{(2)})$, $\bar\bPsi_{i}^* = \big\{-\sum_{j=1}^q\ell_{ij}^{\prime\prime}(\eta_{ij}^*)\bgamma_j^*(\bgamma_j^*)^\T\big\}^{-1}$, and
        \begin{equation*}
    \hat p_i^{(2)}(\bu\mid \bX) = \frac{\exp\{\ell_i(\bu;\hat\btheta^{(2)})\}\pwk(\bu\mid\zero_K,\hat\bSigma^{(2)})}{\int_{\bv\in\RR^{K}}\exp\{\ell_i(\bv;\hat\btheta^{(2)})\}\pwk(\bv\mid\zero_K,\hat\bSigma^{(2)})d\bv};
\end{equation*}
Then
    \begin{equation*}
          \int_{\RR^{K}} \Big|\hat p_i^{(2)}(\bu \mid \bX) - \pwk(\bu\mid \hat{\bu}_i^{\star(2)}, \bar\bPsi^*_i) \Big|d\bu  \overset{p}{\to} 0 \quad \text{as } n,q \to \infty.
    \end{equation*}
\end{enumerate}
\end{theorem}
\begin{proof}
    See Section~\ref{supp_sec_prove_thm_asymp_id2}.
\end{proof}
Theorem~\ref{thm_asymp_id2} establishes the asymptotic theory for estimation under correlated latent variables. Similarly, the asymptotic covariances $\bPhi_j^{(2)}$ and $\bPsi_i^{(2)}$ generally contain additional variability relative to $\bPhi_j^*$ and $\bPsi_i^*$. Still, the covariance in posterior asymptotic normality under this setup matches that in Corollary~\ref{thm_bvm_irt}. Together with Theorem~\ref{thm_asymp_id1}, it provides a comprehensive asymptotic framework for inference under practically implementable identification constraints. We further examine these theoretical findings through empirical studies in Section~\ref{supp_simu_2}.

\subsection{Proof of Theorem~\ref{thm_asymp_id1}}\label{supp_sec_prove_thm_asymp_id1}
Theorem~\ref{thm_asymp_id1} consists of three parts, whose proofs are given in Sections~\ref{thm_asymp_id1_part_1}, \ref{thm_asymp_id1_part_2}, and~\ref{thm_asymp_id1_part_3}, respectively.

\subsubsection{Asymptotic Distributions for Item Parameters}\label{thm_asymp_id1_part_1}
Following the idea in proving Lemma~\ref{lemma_key_first_order}, we start with studying the estimator $(\hat\btheta^*,\hat\bmu^*,\hat\bSigma^*)$ in \eqref{eq_aux_estimator}. Under the identifiability Conditions~\ref{condition1} or~\ref{condition2}, the estimator remains consistent. In particular, we introduce the following result.
\begin{lemma}\em\label{lemma_asymptotic_normality_star}
    Under Assumptions~\ref{assump_standard_id}--\ref{assump_scaling}, we have
   \begin{equation*}
    \big\|\hat\btheta_j^* - \btheta^*_j\big\| = \cOp(\bar\delta_{nq})\quad\text{ and }\quad\big\|\hat\bu_i^{\star}(\hat\btheta^*) - \bu_i^*\big\| = \cOp(\bar\delta_{nq})\text{ for }i\in[n],j\in[q].
\end{equation*} Furthermore, if $\sqrt{n\log n}\tau_{nq}\to 0$ as $n,q\to\infty$, for $\hat\btheta^*$ obtained in \eqref{eq_aux_estimator}, we have
    \begin{equation*}
        (\bPhi_{j}^*)^{-1/2} \begin{pmatrix}
            \hat\beta_j^* - \beta_j^*\\  \hat\bgamma_j^* - \bgamma_j^*
        \end{pmatrix}\overset{d}{\to}\cN(\zero_{K+1},\bI_{K+1}),\text{ as }n,q\to\infty,
    \end{equation*}
    where $\bPhi_j^*$ is given in Theorem~\ref{thm_asymp}. For $\bu_i^{\star}(\hat\btheta^*) = \argmax_{\bu\in\RR^{K}}\ell_i(\bu;\hat\btheta^*)$, suppose Assumptions~\ref{assump_standard_id}--\ref{assump_scaling} hold and $\sqrt{q\log n}\tau_{nq}\to 0$ as $n,q\to\infty$. Then for each $i\in[n]$,
    \begin{equation*}
        (\bPsi_{i}^*)^{-1/2} (\bu_i^{\star}(\hat\btheta^*) - \bu_i^*)\overset{d}{\to}\cN(\zero_K,\bI_K),\text{ as }n,q\to\infty,
    \end{equation*}
    where $\bPsi^*_i$ is given in Corollary~\ref{thm_asymp_bu}. Furthermore, for $\hat\bmu^*$ and $\hat\bSigma^*$ obtained in \eqref{eq_aux_estimator}, we have
    \begin{equation*}
        \|\hat\bmu^*\|=  \cOp(\delta_{nq})\quad\text{ and }\quad\|\hat\bSigma^* - \bSigma_u^*\|=  \cOp(\delta_{nq})
    \end{equation*}
\end{lemma}
\begin{proof}
    See Section~\ref{supp_sec_prove_lemma_asymptotic_normality_star}. The oracle penalty $P^*(\btheta)$ makes the result independent of any identifiability condition; in particular, no constraints on $n^{-1}\sum_{i=1}^n\bu_i^*$ or $n^{-1}\sum_{i=1}^n\bu_i^*(\bu_i^*)^\T$ are needed. Under the whitening condition $n^{-1}\sum_{i=1}^n\bu_i^* = \zero$ and $n^{-1}\sum_{i=1}^n\bu_i^*(\bu_i^*)^\T = \bI_K$ (Condition~\ref{condition1}), the proof further yields
\begin{equation}
    \label{eq_accurate_control_hsigma_star_2}\big\|\hat\bmu^*\big\| = \cOp\big(\tau_{nq}\sqrt{\log n/B_n}\big);\quad \big\|\hat\bSigma^* - \bI_K\big\| = \cOp\big(\tau_{nq}\sqrt{\log n/B_n}\big).
\end{equation}
    Under Condition~\ref{condition2}, which imposes $n^{-1}\sum_{i=1}^n\bu_i^* = \zero$ and ${\rm diag}\big\{n^{-1}\sum_{i=1}^n\bu_i^*(\bu_i^*)^\T \big\}=\bI_K$, the same argument gives
    \begin{equation}
        \label{eq_accurate_control_hsigma_star_3}\big\|\hat\bmu^*\big\| = \cOp\big(\tau_{nq}\sqrt{\log n/B_n}\big);\quad \big\|{\rm diag}(\hat\bSigma^* - \bI_K)\big\| = \cOp \big(\tau_{nq}\sqrt{\log n/B_n}\big).
    \end{equation}
    It is worth mentioning that, if the identifiability constraints are instead imposed on the population quantities $\EE[\bU_i]$ and $\EE[\bU_i\bU_i^\T]$, such rates would no longer be available, and the subsequent derivations of the asymptotic distribution would not proceed as stated.
\end{proof}

In the following, we show that $\hat\btheta^{(1)}$ can be obtained from $\hat\btheta^*$. Similar to Step 2 in the proof of Lemma~\ref{lemma_key_first_order}, we show that $\hat\btheta^{(1)}$ and $\hat\btheta^*$ are equivalent up to some linear transformation.
First, we construct 
\begin{align*}
    \hat\bG^{(1)\dagger} = &\,\Big\{\sum_{s=1}^q\bgamma_s^*(\hat\bgamma_s^{(1)})^\T\Big\}^{-1}\Big\{\sum_{s=1}^q\bgamma_s^*(\bgamma_s^*)^\T\Big\}\\\hat\bd^{(1)\dagger} = &\,-\Big\{\sum_{s=1}^q\bgamma_s^*(\hat\bgamma_s^{(1)})^\T\Big\}^{-1}\Big\{\sum_{s=1}^q\bgamma_s^*(\hat\beta_s^{(1)} - \beta_s^*)\Big\},
\end{align*}
and $\hat\btheta^{(1)\dagger} = (\hat\beta_1^{(1)\dagger}, (\hat\bgamma_1^{(1)\dagger})^\T,\dots,\hat\beta_q^{(1)\dagger}, (\hat\bgamma_q^{(1)\dagger})^\T)^\T$ with
\begin{equation}
    \hat\bgamma_j^{(1)\dagger} = (\hat\bG^{(1)\dagger})^\T\hat\bgamma_j^{(1)}\quad\text{ and }\quad\hat\beta_j^{(1)\dagger} = \hat\beta_j^{(1)} + (\hat\bgamma_j^{(1)})^\T\hat\bd^{(1)\dagger}\text{ for each }j\in[q].
\end{equation}
Following a similar procedure of Step~2 in the proof of Lemma~\ref{lemma_key_first_order}, one can show that $\hat\btheta^{(1)\dagger}\in \cB_{\epsilon}(\btheta^*)$, that is, $\cL(\hat\btheta^{(1)})\ge \cL(\hat\btheta^{*})$. 

Next, we construct $\hat\btheta^{(1)\ddagger}$ transformed from $\hat\btheta^*$ with 
\begin{equation*}
    \hat\bgamma_j^{(1)\ddagger} = \hat\bR^{(1)}(\hat\bSigma^*)^{1/2}\hat\bgamma_j^*\text{ and }\hat\beta_j^{(1)\ddagger} = \hat\beta_j^* + (\hat\bgamma_j^*)^\T\hat\bmu^*\text{ for each }j\in[q],
\end{equation*}
where $\hat\bR^{(1)} = (\hat\br_1^{(1)},\dots,\hat\br_K^{(1)})^\T$ is an orthogonal matrix constructed as follows. For each $l\in[K]$, $\hat\br_l^{(1)}$ obeys
\begin{equation*}
    \hat\br_l^{(1)}\perp \mathrm{span}\Big\{\{\hat\br_{h}^{(1)}\}_{h>l}\cup\mathrm{row}\big(\hat\bGamma^*_{\cA_l,}(\hat\bSigma^*)^{1/2}\big)\Big\}\text{ and }\|\hat\br_l^{(1)}\| = 1.
\end{equation*}
Because $|\cA_l|=l-1$, a vector satisfying these requirements exists. By construction, $\cL(\hat\btheta^{(1)\ddagger},\zero_K,\bI_K) = \cL(\hat\btheta^*,\hat\bmu^*,\hat\bSigma^*)$, and $\hat\btheta^{(1)\ddagger}$ satisfies Condition~\ref{condition1}. As in the proof of Lemma~\ref{lemma_key_first_order},
\begin{equation}
\cL(\hat\btheta^*,\hat\bmu^*,\hat\bSigma^*)=\cL(\hat\btheta^{(1)\ddagger},\zero_K,\bI_K)\ge \cL(\hat\btheta^{(1)},\zero_K,\bI_K).
\end{equation}
Therefore we conclude that $\hat\btheta^{(1)\ddagger} = \hat\btheta^{(1)}$, up to sign changes. Next, we first show the consistency of $\hat\btheta^{(1)}$ obtained as follows
\begin{equation}\label{eq_transform_to_est_1}
    \hat\bgamma_j^{(1)} = \hat\bR^{(1)}(\hat\bSigma^*)^{1/2}\hat\bgamma_j^*,\text{ and }\hat\beta_j^{(1)} = \hat\beta_j^* + (\hat\bgamma_j^*)^\T\hat\bmu^*\text{ for each }j\in[q].
\end{equation}
We introduce the following lemma.
 \begin{lemma}\em \label{lemma_linear_equation_perturb}
        Suppose $\bH\in\RR^{(K-1)\times K}$ has rank $K-1$ and its nonzero singular values are bounded above and away from zero. Let $\bx_h$ be a unit vector spanning $\ker(\bH)$. Suppose $\tilde\bH\in\RR^{(K-1)\times K}$ satisfies $\|\bH - \tilde\bH\|\le \epsilon$, and let $\tilde\bx_h$ be a consistently oriented unit vector spanning $\ker(\tilde\bH)$. For sufficiently small $\epsilon$,
        \begin{equation*}
            \sin\angle(\bx_h,\tilde\bx_h) \lesssim \epsilon.
        \end{equation*}
    \end{lemma}
    \begin{proof}
        See Section~\ref{prove_lemma_linear_equation_perturb}.
    \end{proof}
By Lemma~\ref{lemma_asymptotic_normality_star} and $\bSigma_u^* = \bI_K$, we know that \begin{equation}\|\hat\bSigma^* - \bI_K\| = \cOp(\delta_{nq})\quad \text{ and }\quad \|\hat\bmu^*\| = \cOp(\delta_{nq}).\label{eq_accurate_control_hsigma_star_1}\end{equation} 
For $\hat\bR^{(1)}$, we first show that $\|\hat\br_h^{(1)} - \be_h\| = \cOp(\bar\delta_{nq})$ for each $h\in[K]$. Since $\{\hat\br_h^{(1)}\}_{h\in[K]}$ are constructed consecutively, we begin with $\hat\br_{K}^{(1)}$. By Lemma~\ref{lemma_asymptotic_normality_star}, we have $\|\hat\bGamma_{\cA_K,}^* - \bGamma^*_{\cA_K,}\| = \cOp\big(\bar\delta_{nq}\big)$. Note that $\bGamma^*_{\cA_K,}\be_K = \bGamma^*_{\cA_K,K} = \zero$ by Condition~\ref{condition1} and $\hat\bGamma^*_{\cA_K,}(\hat\bSigma^*)^{1/2}\hat\br_K^{(1)} = \zero$ by definition.
Note that $\sigma_{h-1}(\bGamma_{\cA_h,1:h}^*)$ is bounded below under the rank condition in Condition~\ref{condition1} and $K$ is finite. Combined with \eqref{eq_accurate_control_hsigma_star_1}, Lemma~\ref{lemma_linear_equation_perturb}, and $\|\hat\br_K^{(1)}\| = \|\be_K\| = 1$, this gives
\begin{equation*}
    \big\|\hat\br_K^{(1)} - \be_K\big\| = \cOp(\bar\delta_{nq}).
\end{equation*}
Suppose $\big\|\hat\br_k^{(1)} - \be_k\big\| = \cOp(\bar\delta_{nq})$ for all $k>h$. For $\hat\br_h^{(1)}$, the defining equations give
\begin{equation*}
    \begin{pmatrix}
        (\hat\bSigma^*)^{1/2}\big(\hat\bGamma^*_{\cA_h,}\big)^\T&\hat\br_{h+1}^{(1)}&\cdots&\hat\br_{K}^{(1)}
    \end{pmatrix}^\T\hat\br_h^{(1)} = \zero,
\end{equation*}
by definition. One can also verify that
\begin{equation*}
    \begin{pmatrix}
        \big(\bGamma^*_{\cA_h,}\big)^\T&\be_{h+1}&\cdots&\be_{K}
    \end{pmatrix}^\T\be_h = \zero.
\end{equation*}
Lemma~\ref{lemma_asymptotic_normality_star} gives $\|\hat\bGamma^*_{\cA_h,} - \bGamma^*_{\cA_h,}\| = \cOp(\bar\delta_{nq})$. Lemma~\ref{lemma_linear_equation_perturb} and the induction hypothesis therefore yield $\big\|\hat\br_h^{(1)} - \be_h\big\| = \cOp(\bar\delta_{nq})$, and hence $\|\hat\bR^{(1)} - \bI_K\| = \cOp(\bar\delta_{nq})$. Therefore, with \eqref{eq_transform_to_est_1}, Lemma~\ref{lemma_aux_estimator_property}, and $n^{-1/2}r_{\theta,nq}\ll \delta_{nq}\le\bar\delta_{nq}$, one has
\begin{equation*}
    q^{-1/2}\big\|\hat\btheta^{(1)} - \btheta^*\big\| = \cOp(\bar\delta_{nq}).
\end{equation*}

For the asymptotic distribution, the current convergence rate implies that we have to take $\hat\bR^{(1)}$ into account. Specifically, let $\hat\bX^{(1)} = \hat\bR^{(1)}(\hat\bSigma^*)^{1/2} - \bI_K$. The bound $\|\hat\bR^{(1)} - \bI_K\| = \cOp(\bar\delta_{nq})$ and \eqref{eq_accurate_control_hsigma_star_1} imply
\begin{equation}
    \big\|\hat\bX^{(1)}\big\| = \cOp(\bar\delta_{nq}),\label{eq_bound_x_1}
\end{equation}
by telescoping.
Recall that $\hat\bGamma^{(1)} = \hat\bGamma^*(\hat\bSigma^*)^{1/2}(\hat\bR^{(1)})^\T = \hat\bGamma^*(\hat\bX^{(1)}+\bI_K)^\T$. The identifiability condition $\hat\bGamma_{\cA_r,r}^{(1)} = \bGamma^*_{\cA_r,r} = \zero_{r-1}$ yields
\begin{equation*}
    \zero_{r-1} = \hat\bGamma^{(1)}_{\cA_r,r} - \bGamma_{\cA_r,r}^*=\hat\bGamma^*_{\cA_r,}\big(\hat\bX^{(1)}_{r,}\big)^\T + \hat\bGamma^{*}_{\cA_r,r}- \bGamma^*_{\cA_r,r}.
\end{equation*}
By Lemma~\ref{lemma_asymptotic_normality_star}, $\big\|\hat\bGamma^*_{\cA_r,} - \bGamma^*_{\cA_r,}\big\| = \cOp(\bar\delta_{nq})$. Together with \eqref{eq_bound_x_1}, we arrive at
\begin{equation}\begin{aligned}
    \zero_{r-1} =&\, \bGamma^*_{\cA_r,}\big(\hat\bX^{(1)}_{r,}\big)^\T + \hat\bGamma^{*}_{\cA_r,r}- \bGamma^*_{\cA_r,r} + \left\{\hat\bGamma^*_{\cA_r,} - \bGamma^*_{\cA_r,}\right\}\big(\hat\bX^{(1)}_{r,}\big)^\T\\ = &\,\bGamma^*_{\cA_r,}\big(\hat\bX^{(1)}_{r,}\big)^\T + \hat\bGamma^{*}_{\cA_r,r}- \bGamma^*_{\cA_r,r} + \cOp\big(\bar\delta_{nq}^2\big),\end{aligned}\label{eq1_linear_eq_x1}
\end{equation}
for each $r\in[K]$. Here $\cOp(\bar\delta_{nq}^2)$ is interpreted as an elementwise error bound. Next, since $\hat\bR^{(1)}$ is orthogonal, one has
\begin{equation*}
    \Big\|\big(\hat \bX^{(1)} + \bI_K\big)^\T\big(\hat \bX^{(1)} + \bI_K\big) - \bI_K\Big\|  = \big\|\hat\bSigma^* - \bI_K\big\| =  \cOp\big(\tau_{nq}\sqrt{\log n/B_n}\big),
\end{equation*}
due to \eqref{eq_accurate_control_hsigma_star_2}. Subsequently, with $\big\|\hat\bX^{(1)}\big\| = \cOp\big(\bar\delta_{nq}\big)$, one has
\begin{equation}\begin{aligned}
    \big(\hat \bX^{(1)}\big)^\T + \hat \bX^{(1)} =   \big(\hat \bX^{(1)} + \bI_K\big)^\T\big(\hat \bX^{(1)} + \bI_K\big) - \bI_K  -\big(\hat \bX^{(1)}\big)^\T\hat \bX^{(1)} =\cOp\big(\tau_{nq}\sqrt{\log n/B_n}\big).\end{aligned}\label{eq2_linear_eq_x1}
\end{equation}
Again, the $\cOp(\cdot)$ term is an elementwise error bound.
Equations \eqref{eq1_linear_eq_x1} and \eqref{eq2_linear_eq_x1} together form a system of $K^2$ linearly independent equations with random errors of order $\cOp\big(\tau_{nq}\sqrt{\log n/B_n}\big)$. In particular, \eqref{eq2_linear_eq_x1} gives $\hat X_{lh}^{(1)} = -\hat X_{hl}^{(1)} + \cOp\big(\tau_{nq}\sqrt{\log n/B_n}\big)$ for $l\neq h$ and $\hat X_{ll}^{(1)} = \cOp\big(\tau_{nq}\sqrt{\log n/B_n}\big)$ for $l\in[K]$. Subsequently, it suffices to solve the lower triangular part of $\hat\bX^{(1)}$. For each $2\le r\le K$, define $\cK_r^{tri} = \{(r-1)(r-2)/2 + 1,\dots,r(r-1)/2\}$ and let $\hat\bx^{(1)}$ be the vectorized lower-triangular part of $\hat\bX^{(1)}$, i.e., $\hat\bx^{(1)}_{\cK_r^{tri}} = \hat\bX^{(1)}_{r,1:(r-1)}$ for each $2\le r\le K$. Then \eqref{eq1_linear_eq_x1} can be expressed as
\begin{equation*}
    \bLambda_{X}^{(1)}\hat\bx^{(1)} = -\begin{pmatrix}
        \hat\bGamma^{*}_{\cA_2,2}- \bGamma^*_{\cA_2,2}\\\vdots\\\hat\bGamma^{*}_{\cA_K,K}- \bGamma^*_{\cA_K,K}
    \end{pmatrix} + \cOp\big(\tau_{nq}\sqrt{\log n/B_n}\big),
\end{equation*}
where $\bLambda_{X}^{(1)}$ is the coefficient matrix given as
\begin{equation*}
    \big(\bLambda_{X}^{(1)}\big)_{\cK_l^{tri}, \cK_h^{tri}} = 1_{(l=h)}\bGamma^*_{\cA_l,1:(l-1)} - 1_{(l<h)}\bGamma^*_{\cA_l,h}\big(\be_l^{(h-1)}\big)^\T\text{ for each }l,h\in[K]
\end{equation*}
The upper-triangular structure guarantees that the solution $\hat\bX^{(1)}$ can be expressed uniquely in the following form:
\begin{equation}\label{eq_linear_solver_X1}
    \hat X_{lh}^{(1)} = \left\{\begin{aligned}& -\big(\be_h^{(l-1)}\big)^\T \sum_{k=l}^K\big(\bLambda_{X}^{(1)}\big)^{-1}_{\cK_{l}^{tri},\cK_{k}^{tri}}\big(\hat\bGamma^{*}_{\cA_k,k}- \bGamma^*_{\cA_k,k}\big)+ \cOp\big(\tau_{nq}\sqrt{\log n/B_n}\big)\text{, }l > h,\\& \cOp\big(\tau_{nq}\sqrt{\log n/B_n}\big)\text{, }l = h,\\& -\hat X_{hl}^{(1)}+ \cOp\big(\tau_{nq}\sqrt{\log n/B_n}\big)\text{, }l < h.\end{aligned}\right.
\end{equation}
By Lemma~\ref{lemma_asymptotic_normality_star}, we know that $\hat\bGamma_{\cA_k,k}^* - \bGamma_{\cA_k,k}^*$ is asymptotically normal as $n,q\to\infty$. Define $\bPhi_{\Gamma}^{(1)} \in \RR^{\frac{K(K-1)}{2}\times \frac{K(K-1)}{2}}$ as
\begin{equation*}
    \left(\bPhi_{\Gamma}^{(1)}\right)_{\cK_l^{tri},\cK_h^{tri}} = \left\{1_{(j_1=j_2)}\be_l^\T\bPhi_{\gamma,j_1}^*\be_h\right\}_{j_1\in \cA_l,j_2\in \cA_h}\text{ for each }l,h\in[K]\setminus\{1\},
\end{equation*}
where $\bPhi_{\gamma,j_1}^* = (\bPhi_{j_1}^*)_{2:(K+1),2:(K+1)}$. Subsequently, 
\begin{equation*}
    \left(\bPhi_{\Gamma}^{(1)}\right)^{-1/2}\left\{\left(\hat\bGamma_{\cA_2,2}^* - \bGamma_{\cA_2,2}^*\right)^\T,\dots, \left(\hat\bGamma_{\cA_K,K}^* - \bGamma_{\cA_K,K}^*\right)^\T\right\}^\T \overset{d}{\to } \cN\left(\zero,\bI_{\frac{K(K-1)}{2}}\right).
\end{equation*}
Because the first-order part of $\hat\bX^{(1)}$ is skew-symmetric, its covariance is singular on $\RR^{K^2}$. Accordingly, the Gaussian limit is stated for nondegenerate projections:
\begin{equation}
    \frac{\ba^\T\mathrm{vec}(\hat\bX^{(1)})}
    {\{\ba^\T\bPhi_X^{(1)}\ba\}^{1/2}}
    \overset{d}{\to}\cN(0,1)
    \quad\text{for every fixed $\ba$ such that $\ba^\T\bPhi_X^{(1)}\ba>0$}.\label{eq_asymp_dis_X_1}
\end{equation}
To define $\bPhi^{(1)}_{X}$, let $\iota(l,h) = {(l\vee h-1)(l\vee h-2)}/{2} + (l\wedge h)\in\cK^{tri}_{l\vee h}$ and $\bPhi^{(1)}_{X}$ is given as
\begin{equation}
    \big(\bPhi^{(1)}_{X}\big)_{\cK_{l}, \cK_{h}} = \left\{\mathrm{sgn}(l - i)\mathrm{sgn}(h - j)\be_{\iota(i,l)}^\T\big(\bLambda_{X}^{(1)}\big)^{-1}\bPhi_{\Gamma}^{(1)}\big(\bLambda_{X}^{(1)}\big)^{-\T}\be_{\iota(j,h)}\right\}_{i\in[K],j\in[K]},\label{eq_asymp_cov_X_1}
\end{equation}
where the formula holds for each $l,h\in[K]$, and $\cK_l=\{(l-1)K+1,\ldots,lK\}$ indexes the $l$th column in $\mathrm{vec}(\bX)$.

We now derive the asymptotic distributions for $\hat\btheta_j^{(1)} - \btheta_j^* = \big((\hat\beta_j^{(1)}-\beta_j^*), (\hat\bgamma_j^{(1)}-\bgamma_j^*)^\T\big)^\T$. For $\hat\beta_j^{(1)} - \beta_j^*$, \eqref{eq_accurate_control_hsigma_star_2}, \eqref{eq_transform_to_est_1}, and $\|\hat\btheta^*\|_{\infty}=\cO(1)$ give
\begin{equation*}
    \hat\beta_j^{(1)} - \beta_j^* = \hat\beta_j^* - \beta_j^* + \cOp\big(\tau_{nq}\sqrt{\log n/B_n}\big).
\end{equation*}
When $\sqrt{n\log n}\tau_{nq}\to 0$, $ \hat\beta_j^{(1)} - \beta_j^*$ has the same asymptotic distribution as $\hat\beta_j^* - \beta_j^*$. For $\hat\bgamma_j^{(1)} - \bgamma_j^*$, note that
\begin{equation}\hat\bgamma_j^{(1)} - \bgamma_j^* = \big(\hat\bX^{(1)} + \bI_K\big)\hat\bgamma_j^* - \bgamma_j^*= \hat\bX^{(1)}\bgamma_j^* +  \hat\bgamma_j^*- \bgamma_j^* + \hat\bX^{(1)}(\hat\bgamma_j^*- \bgamma_j^*)\end{equation}
The first term $\hat\bX^{(1)}\bgamma_j^*$ has a tractable asymptotic distribution given in \eqref{eq_asymp_dis_X_1}; the second term $\hat\bgamma_j^* - \bgamma_j^*$ has an asymptotic distribution given in Lemma~\ref{lemma_asymptotic_normality_star}; and the third term is $\cOp\big(\bar\delta_{nq}^2\big)$ by \eqref{eq_bound_x_1} and $\|\hat\bgamma_j^* - \bgamma_j^*\| = \cOp(\bar\delta_{nq})$. It remains to calculate the covariance between $\hat\bX^{(1)}\bgamma_j^*$ and $\hat\btheta_j^* - \btheta_j^*$.

When $j\notin \cA_r$ for all $r\in[K]$, by Lemma~\ref{lemma_asymptotic_normality_star}, one can check that $\hat\btheta_j^* - \btheta_j^*$ and $\hat\bGamma_{\cA_k,k}^* - \bGamma_{\cA_k,k}^*$ are asymptotically independent pairwise for $k\in[K]$. Hence $\hat\btheta_j^* - \btheta_j^*$ and $\hat\bX^{(1)}$ are asymptotically independent, and we have
\begin{equation}\begin{aligned}
    &\,\left(\bPhi_j^* + \begin{pmatrix}
        0&\zero\\\zero & \big\{(\bgamma_j^*)^\T\otimes \bI_K\big\}\bPhi^{(1)}_X\big\{\bgamma_j^*\otimes \bI_K\big\}
    \end{pmatrix}\right)^{-1/2}\begin{pmatrix}\hat\beta_j^{(1)} - \beta_j^*\\\hat\bgamma_j^{(1)} - \bgamma_j^*\end{pmatrix}\\
    =&\,\left(\bPhi_j^* + \begin{pmatrix}
        0&\zero\\\zero & \big\{(\bgamma_j^*)^\T\otimes \bI_K\big\}\bPhi^{(1)}_X\big\{\bgamma_j^*\otimes \bI_K\big\}
    \end{pmatrix}\right)^{-1/2}\begin{pmatrix}\hat\beta_j^{*} - \beta_j^*\\\hat\bgamma_j^{*} - \bgamma_j^* +\hat\bX^{(1)}\bgamma_j^*\end{pmatrix} \\&\quad+ \cOp\big(\sqrt{n}\tau_{nq}\sqrt{\log n/B_n}\big)\\\overset{d}{\to}&\,\cN\big(\zero_{K+1},\bI_{K+1}\big),\end{aligned}\label{eq_example_asym_X1_ind}
\end{equation}
as $n,q\to\infty$, provided that $\sqrt{n\log n}\tau_{nq}\to0$. Here, we use $\|(\bPhi_j^*)^{-1}\| = \cOp(nB_n)$ as established in Section~\ref{supp_sec_prove_asym_0}.

When $j\in\cA_k$ for at least one $k$, the covariance between $\hat\btheta_j^* - \btheta_j^*$ and $\hat\bX^{(1)}\bgamma_j^*$ must also be included. Let $j_{kr}=\cA_k(r)$ for $k=2,\ldots,K$ and $r\in[k-1]$. From \eqref{eq_linear_solver_X1}, define the $(K+1)\times K$ matrix
\begin{equation}
\begin{aligned}
    \bPhi_{X,kr}^{(1)}
    =\sum_{l=1}^K\Bigg[&-\sum_{h<l}\be_{k+1}^{(K+1)}(\be_r^{(k-1)})^\T
        \left\{\big((\bLambda_X^{(1)})^{-1}\big)_{\cK_l^{tri},\cK_k^{tri}}\right\}^\T
        \be_h^{(l-1)}\gamma_{j_{kr}h}^*\\
    &+\sum_{h>l}\be_{k+1}^{(K+1)}(\be_r^{(k-1)})^\T
        \left\{\big((\bLambda_X^{(1)})^{-1}\big)_{\cK_h^{tri},\cK_k^{tri}}\right\}^\T
        \be_l^{(h-1)}\gamma_{j_{kr}h}^*\Bigg]\be_l^\T.
\end{aligned}\label{eq_define_phi_X_kr}
\end{equation}
The use of $\be_{k+1}^{(K+1)}$ reflects that $\gamma_{jk}$ is the $(k+1)$st coordinate of $\btheta_j$. Hence we have
\begin{equation*}
    \mathrm{Cov}\left(\hat\btheta_j^* - \btheta_j^*,\hat\bX^{(1)}\bgamma_j^*\right)
    =\sum_{k=2}^K\sum_{r=1}^{k-1}1_{\{j=j_{kr}\}}\bPhi_j^*\bPhi_{X,kr}^{(1)}
    +\cOp\big(\tau_{nq}\bar\delta_{nq}\sqrt{\log n/B_n}\big).
\end{equation*}
Similar to the derivation of \eqref{eq_example_asym_X1_ind}, it holds that
\begin{equation*}
    \Big(\bE_{(j)}\bPhi_{j}^{(1)}\bE_{(j)}^\T\Big)^{-1/2}\begin{pmatrix}
        \hat\beta_j^{(1)} - \beta_j^*\\\hat\bgamma_{j,\cS_j}^{(1)} - \bgamma_{j,\cS_j}^*\end{pmatrix}\overset{d}{\to}\cN\big(\zero_{K_j+1},\bI_{K_j+1}\big)
\end{equation*}
where $K_j=|\cS_j|$ and
\begin{equation}
\begin{aligned}
    \bPhi_{j}^{(1)}
    =&\,\bPhi_j^*
    +\bJ_\gamma\big\{(\bgamma_j^*)^\T\otimes \bI_K\big\}\bPhi_X^{(1)}
        \big\{\bgamma_j^*\otimes \bI_K\big\}\bJ_\gamma^\T\\
    &\,+\Big\{\sum_{k=2}^K\sum_{r=1}^{k-1}1_{\{j=j_{kr}\}}\bPhi_j^*\bPhi_{X,kr}^{(1)}\Big\}\bJ_\gamma^\T+\bJ_\gamma\Big\{\sum_{k=2}^K\sum_{r=1}^{k-1}1_{\{j=j_{kr}\}}\bPhi_j^*\bPhi_{X,kr}^{(1)}\Big\}^\T.
\end{aligned}\label{eq_define_cov_ic1}
\end{equation}
with $\bPhi_X^{(1)}$ defined in \eqref{eq_asymp_cov_X_1}, $\bPhi_{X,kr}^{(1)}$ defined in \eqref{eq_define_phi_X_kr}, and
\begin{equation*}
    \bJ_\gamma := \begin{pmatrix}\zero_{1\times K}\\\bI_K\end{pmatrix}\in\RR^{(K+1)\times K}.
\end{equation*}

\subsubsection{Asymptotic Distributions for Latent Variables}\label{thm_asymp_id1_part_2}
For the asymptotic distribution of $\hat\bu_i^{(1)}$, following a similar argument in the proof of Corollary~\ref{thm_asymp_bu}, we first establish the asymptotic distribution of $\hat\bu_i^{\star(1)}$. We note that $\hat\bu_i^{\star(1)}$ solves
\begin{align*}
    \hat\bu^{\star(1)}_i &= \mathop{\arg\max}_{\bu\in\RR^{K}}\sum_{j=1}^q\ell_{ij}\left(\big(\hat\bgamma_j^{(1)}\big)^\T\bu + \hat\beta_j^{(1)}\right) \\ &= \mathop{\arg\max}_{\bu\in\RR^{K}}\sum_{j=1}^q\ell_{ij}\left(\big(\hat\bgamma_j^{*})^\T\big(\hat\bX^{(1)} + \bI_K\big)^\T\bu + \hat\beta_j^{*} + (\hat\bgamma_j^*)^\T\hat\bmu^*\right).
\end{align*}
Recall that $\bu_i^{\star}(\hat\btheta^*)$ is the unique maximizer of $\sum_{j=1}^q\ell_{ij}\big((\hat\bgamma_j^*)^\T\bu + \hat\beta_j^*\big)$ as given in Lemma~\ref{lemma_asymptotic_normality_star}. Then we know 
\begin{align*}
    \hat\bu^{\star(1)}_i - \bu_i^*& = \big(\hat\bX^{(1)} + \bI_K\big)^{-\T}\big\{\bu_i^{\star}(\hat\btheta^*) -\hat\bmu^*\big\} - \bu_i^*\\& = \bu_i^{\star}(\hat\btheta^*)  - \bu_i^* - \big(\hat\bX^{(1)} \big)^\T\bu_i^* -  \big(\hat\bX^{(1)} \big)^\T\big(\bu_i^{\star}(\hat\btheta^*) - \bu_i^*\big) - \hat\bmu^* + \cOp(\bar\delta_{nq}^2).
\end{align*}
The remainder in the second equality follows from \eqref{eq_bound_x_1} and Lemma~\ref{lemma_asymptotic_normality_star}. Together with \eqref{eq_asymp_dis_X_1} and $\sqrt{q\log n}\tau_{nq}\to 0$, this gives
\begin{equation*}
    \Big(\bPsi_i^{(1)}\Big)^{-1/2}\big(\hat\bu^{\star(1)}_i - \bu_i^*\big) \overset{d}{\to} \cN\big(\zero_K,\bI_K\big),
\end{equation*}
with 
\begin{equation}
    \bPsi_i^{(1)} = \bPsi_i^* + \big\{\bI_K\otimes (\bu_i^*)^\T\big\}\bPhi^{(1)}_X\big\{\bI_K\otimes \bu_i^*\big\}.\label{eq_define_cov_psi_i_1}
\end{equation}
Here, we use the fact that $(\hat\bX^{(1)})^\T\bu_i^*$ is asymptotically independent of $\bu_i^{\star}(\hat\btheta^*)$, because Lemma~\ref{lemma_asymptotic_normality_star} makes each $\hat\btheta_j^*$ asymptotically independent of $\bu_i^{\star}(\hat\btheta^*)$. As in Section~\ref{supp_sec_prove_thm_asymp_bu}, $\big(\bPsi_i^{(1)}\big)^{-1/2}(\hat\bu_i^{\star(1)} - \hat\bu_i^{(1)}) = o_{\PP}(1)$, which finishes the proof.

\subsubsection{Posterior Asymptotic Normality}\label{thm_asymp_id1_part_3}
We only sketch the proof, since it follows the same Bernstein--von Mises argument as in the proof of Corollary~\ref{thm_bvm_irt}. Fix $i\in[n]$ and write
\begin{equation*}
    \hat\bJ_i^{(1)}=-\sum_{j=1}^q\ell_{ij}^{\prime\prime}(\hat\eta_{ij}^{\star(1)})
    \hat\bgamma_j^{(1)}(\hat\bgamma_j^{(1)})^\T,
    \qquad
    \hat\bPsi_i^{\star(1)}=(\hat\bJ_i^{(1)})^{-1},
\end{equation*}
where $\hat\eta_{ij}^{\star(1)}=(\hat\bgamma_j^{(1)})^\T\hat\bu_i^{\star(1)}+\hat\beta_j^{(1)}$. By the consistency of $\hat\btheta^{(1)}$ and the arguments already used in Section~\ref{supp_sec_prove_thm_bvm_irt}, the same curvature bounds as in \eqref{eq_bvm_curvature_bounds} hold for $\hat\bJ_i^{(1)}$. Moreover, the Gaussian approximation with covariance $\bar\bPsi_i^*$ can first be reduced to the one with covariance $\hat\bPsi_i^{\star(1)}$, by the same covariance-comparison argument used at the beginning of the proof of Corollary~\ref{thm_bvm_irt}. Hence it remains to prove the local normal approximation after centering at $\hat\bu_i^{\star(1)}$ and scaling by $\hat\bPsi_i^{\star(1)}$.

Let $\bu_{\bx}^{(1)}=\hat\bu_i^{\star(1)}+(\hat\bPsi_i^{\star(1)})^{1/2}\bx$ and define
\begin{equation*}
    \widetilde R_{nq}^{(1)}(\bx)=
    \exp\{\ell_i(\bu_{\bx}^{(1)};\hat\btheta^{(1)})
    -\ell_i(\hat\bu_i^{\star(1)};\hat\btheta^{(1)})\}
    \frac{\pwk(\bu_{\bx}^{(1)})}{\pwk(\hat\bu_i^{\star(1)})},
    \qquad
    \widetilde Z_{nq}^{(1)}=\int_{\RR^K}\widetilde R_{nq}^{(1)}(\bz)d\bz.
\end{equation*}
After the change of variables $\bx=(\hat\bPsi_i^{\star(1)})^{-1/2}(\bu-\hat\bu_i^{\star(1)})$, it suffices to show
\begin{equation}
    \int_{\RR^K}\left|\frac{\widetilde R_{nq}^{(1)}(\bx)}{\widetilde Z_{nq}^{(1)}}-\pwk(\bx)\right|d\bx=o_p(1).
    \label{eq_posterior_asymp_2_IC1}
\end{equation}
Take the same sequences $m_q$ and $M_q$ as in Section~\ref{supp_sec_prove_thm_bvm_irt} and decompose $\RR^K$ into
\begin{equation*}
    A_x^{(1)}=\{\|\bx\|\le m_q\},\qquad
    A_x^{(2)}=\{m_q<\|\bx\|\le M_q\},\qquad
    A_x^{(3)}=\{\|\bx\|>M_q\}.
\end{equation*}
On $A_x^{(1)}$, since $\hat\bu_i^{\star(1)}$ maximizes $\ell_i(\bu;\hat\btheta^{(1)})$, the linear term in the Taylor expansion vanishes and
\begin{equation*}
\begin{aligned}
    \ell_i(\bu_{\bx}^{(1)};\hat\btheta^{(1)})
    -\ell_i(\hat\bu_i^{\star(1)};\hat\btheta^{(1)})
    &= -\frac12\|\bx\|^2+r_{nq}^{(1)}(\bx),\\
    \sup_{\bx\in A_x^{(1)}}|r_{nq}^{(1)}(\bx)|
    &\le C B_n q\|\hat\bPsi_i^{\star(1)}\|^{3/2}m_q^3=o_p(1).
\end{aligned}
\end{equation*}
Also, uniformly on $A_x^{(1)}$, we know ${\pwk(\bu_{\bx}^{(1)})}/{\pwk(\hat\bu_i^{\star(1)})}=1+o_p(1)$ because $\|\hat\bu_i^{\star(1)}\|=\cOp(\sqrt{\log n})$ and $\sup_{\bx\in A_x^{(1)}}\|\bu_{\bx}^{(1)}-\hat\bu_i^{\star(1)}\|=\cOp\{m_q/(\sqrt{qb_n})\}=o_p(1/\sqrt{\log n})$. Therefore,
\begin{equation*}
    \int_{A_x^{(1)}}\left|\widetilde R_{nq}^{(1)}(\bx)-\exp(-\|\bx\|^2/2)\right|d\bx=o_p(1).
\end{equation*}
On $A_x^{(2)}$ and $A_x^{(3)}$, the same concavity and prior-ratio bounds as in the proof of Corollary~\ref{thm_bvm_irt} give
\begin{equation*}
    \int_{A_x^{(2)}\cup A_x^{(3)}}\widetilde R_{nq}^{(1)}(\bx)d\bx=o_p(1),
    \qquad
    \int_{A_x^{(2)}\cup A_x^{(3)}}\exp(-\|\bx\|^2/2)d\bx=o(1).
\end{equation*}
Combining the three regions yields
\begin{equation*}
    \int_{\RR^K}\left|\widetilde R_{nq}^{(1)}(\bx)-\exp(-\|\bx\|^2/2)\right|d\bx=o_p(1),
\end{equation*}
and hence $\widetilde Z_{nq}^{(1)}=(2\pi)^{K/2}+o_p(1)$. This proves \eqref{eq_posterior_asymp_2_IC1}. Together with the covariance-comparison step, the desired posterior asymptotic normality follows.

\subsection{Proof of Theorem~\ref{thm_asymp_id2}}\label{supp_sec_prove_thm_asymp_id2}
Theorem~\ref{thm_asymp_id2} consists of three parts, whose proofs are given in Sections~\ref{thm_asymp_id2_part_1}, \ref{thm_asymp_id2_part_2}, and~\ref{thm_asymp_id2_part_3}, respectively.
\subsubsection{Asymptotic Distributions for Item Parameters}\label{thm_asymp_id2_part_1}
Similar to the derivation in Section~\ref{thm_asymp_id1_part_1}, we can show that $\hat\btheta^{(2)}$ in \eqref{eq_mmle_IC3} can be obtained from \eqref{eq_aux_estimator} via the following linear transformation:
\begin{equation}
    \hat\bgamma_j^{(2)} = \hat\bR^{(2)}_1\hat\bR^{(2)}_2(\hat\bSigma^*)^{1/2}\hat\bgamma_j^*\text{ and }\hat\beta_j^{(2)} = \hat\beta_j^* + (\hat\bgamma_j^*)^\T \hat\bmu^*\text{ for each }j\in[q],\label{eq_transform_theta_Star_2}
\end{equation}
where $\hat\bR^{(2)}_1 = \mathrm{diag}(\tilde r_1^{(2)},\dots,\tilde r_K^{(2)})$ and $\hat\bR^{(2)}_2 = \big(\hat\br_1^{(2)},\dots,\hat\br_K^{(2)}\big)^\T$ are constructed as follows. For each $l\in[K]$, $\hat\br_l^{(2)}$ obeys
\begin{equation*}
    \hat\br_l^{(2)}\perp\mathrm{row}\left\{\hat\bGamma_{\cA_l,}^*(\hat\bSigma^*)^{1/2}\right\}\text{ and }\big\|\hat\br_l^{(2)}\big\| = 1,
\end{equation*}
and we choose the diagonal matrix $\hat\bR_1^{(2)}$ with positive diagonal entries to satisfy \begin{equation}\mathrm{diag}\big\{\big(\hat\bR_1^{(2)}\big)^{-\T}\big(\hat\bR_2^{(2)}\big)^{-\T}\big(\hat\bR_2^{(2)}\big)^{-1}\big(\hat\bR_1^{(2)}\big)^{-1}\big\} = \bI_K.\label{eq_constraint_br_12}\end{equation}
Here, the first requirement determines each direction up to sign; we choose the orientation satisfying $(\hat\br_l^{(2)})^\T\be_l\geq0$. The second requirement determines the scale. In addition, one can check that $\hat\bSigma^{(2)}$ from \eqref{eq_mmle_IC3} satisfies
\begin{equation}
    \hat\bSigma^{(2)} = \big(\hat\bR_1^{(2)}\big)^{-\T}\big(\hat\bR_2^{(2)}\big)^{-\T}\big(\hat\bR_2^{(2)}\big)^{-1}\big(\hat\bR_1^{(2)}\big)^{-1},\label{eq_transform_theta_Star_2_sigma}
\end{equation}
 and by Lemma~\ref{lemma_asymptotic_normality_star}, we know 
\begin{equation}
    \big\|\hat\bSigma^* - \bSigma_u^*\big\| = \cOp(\delta_{nq})\label{eq_ic4_sigma_consistency}.
\end{equation}

By Condition~\ref{condition2}, for each $r\in[K]$, we have $\mathrm{rank}(\bGamma_{\cA_r,}^*) = K-1$ and $\bGamma_{\cA_r,}^*\be_r = \zero_{K-1}$. By Lemma~\ref{lemma_asymptotic_normality_star},
\begin{equation*}
    \begin{aligned}
\|\hat\bGamma^*_{\cA_r,}(\hat\bSigma^*)^{1/2}-\bGamma^*_{\cA_r,}(\bSigma_u^*)^{1/2}\|
&\le\|\hat\bGamma^*_{\cA_r,}-\bGamma^*_{\cA_r,}\|\,\|(\bSigma_u^*)^{1/2}\|\\
&\quad+\|\hat\bGamma^*_{\cA_r,}\|\,\|(\hat\bSigma^*)^{1/2}-(\bSigma_u^*)^{1/2}\|
=\cOp(\bar\delta_{nq}).
\end{aligned}
\end{equation*} where the second term uses the perturbation bound for matrix square root operator together with the positive definiteness of $\bSigma_u^*$. Since $\bGamma_{\cA_r,}^*\be_r = \zero_{K-1}$, the null space of $\bGamma_{\cA_r,}^*(\bSigma_u^*)^{1/2}$ is spanned by $\br_r^{\circ} =  c_r^{-1}(\bSigma_u^*)^{-1/2}\be_r$ for $ c_r = \|(\bSigma_u^*)^{-1/2}\be_r\|$.
Because $\hat\bGamma_{\cA_r,}^*(\hat\bSigma^*)^{1/2}\hat\br_r^{(2)} = \zero_{K-1}$, Lemma~\ref{lemma_linear_equation_perturb} and $\|\hat\br_r^{(2)}\| = \|\br_r^{\circ}\| = 1$ give
\begin{equation*}
    \big\|\hat\br_r^{(2)} - \br_r^o\big\| = \cOp(\bar\delta_{nq}).
\end{equation*}
Writing $\bC = \mathrm{diag}(c_1,\dots,c_K)$, this says $\big\|\hat\bR^{(2)}_2 - \bC^{-1}(\bSigma_u^*)^{-1/2}\big\| = \cOp(\bar\delta_{nq})$. By telescoping, one can obtain $\big\|(\hat\bR^{(2)}_2)^{-\T}(\hat\bR^{(2)}_2)^{-1} - \bC\bSigma_u^*\bC\big\| = \cOp(\bar\delta_{nq})$. Since $\hat\bR^{(2)}_1$ is diagonal with positive entries and satisfies \eqref{eq_constraint_br_12}, and since $\mathrm{diag}(\bC\bSigma_u^*\bC) = \bC^2$ by $\mathrm{diag}(\bSigma_u^*) = \bI_K$, it follows that $\big\|\hat\bR_1^{(2)} - \bC\big\| = \cOp(\bar\delta_{nq})$. Thus, we have $\bC\cdot\bC^{-1}(\bSigma_u^*)^{-1/2}\cdot(\bSigma_u^*)^{1/2} = \bI_K$, which yields \[\big\|\hat\bR_1^{(2)}\hat\bR_2^{(2)}(\hat\bSigma^*)^{1/2}-\bI_K\big\| = \cOp(\bar\delta_{nq}).\] Together with \eqref{eq_transform_theta_Star_2} and $\|\hat\bmu^*\| = \cOp\big(\delta_{nq}\big)$, this yields
\begin{equation*}
    q^{-1/2}\big\|\hat\btheta^{(2)} - \btheta^*\big\| = \cOp(\bar\delta_{nq}).
\end{equation*}

Now we follow a similar procedure in Section~\ref{thm_asymp_id1_part_1} to derive the asymptotic distribution for $\hat\btheta^{(2)}$. In particular, define $\hat\bX^{(2)} = \hat\bR^{(2)}_1\hat\bR_2^{(2)}(\hat\bSigma^*)^{1/2} - \bI_K$. First, we have
\begin{equation}
    \big\|\hat\bX^{(2)}\big\| = \cOp(\bar\delta_{nq}),\label{eq_bound_x_2}
\end{equation}
by the above discussion. Subsequently, $(\bI_K + \hat\bX^{(2)})^{-1}$ can be approximated by
\begin{equation*}
    (\bI_K + \hat\bX^{(2)})^{-1} = \bI_K - \hat\bX^{(2)} + \cOp(\bar\delta_{nq}^2).
\end{equation*}
Now we construct linear equations for $\hat\bX^{(2)}$. By \eqref{eq_bound_x_2}, one has
\begin{align*}
    \bI_K &= \mathrm{diag}\left\{ \big(\hat\bX^{(2)} + \bI_K\big)^{-\T} \hat\bSigma^*\big(\hat\bX^{(2)} + \bI_K\big)^{-1}\right\} \\&= \mathrm{diag}\left\{  \hat\bSigma^* - \big(\hat\bX^{(2)}\big)^\T\hat\bSigma^* - \hat\bSigma^*\hat\bX^{(2)}\right\} + \cOp(\bar\delta_{nq}^2)
\end{align*}
With $\|{\rm diag}(\hat\bSigma^* - \bI_K)\| = \cOp\big(\tau_{nq}\sqrt{\log n/B_n}\big)$ by \eqref{eq_accurate_control_hsigma_star_3}, we arrive at 
\begin{equation*}
    \mathrm{diag}\left\{\big(\hat\bX^{(2)}\big)^\T\bSigma_u^* + \bSigma_u^*\hat\bX^{(2)}\right\}=\cOp\big(\tau_{nq}\sqrt{\log n/B_n}\big).
\end{equation*}
Since the diagonal of $\bSigma_u^*$ is one and the matrix is symmetric, the component form is
\begin{equation}\label{eq_correlated_diagonal_correction}
\hat X^{(2)}_{ll}=-\sum_{h\ne l}(\bSigma_u^*)_{lh}\hat X^{(2)}_{hl}
+O_p\big(\tau_{nq}\sqrt{\log n/B_n}\big).
\end{equation}
Next, we invoke the identifiability condition 
\begin{equation*}
    \zero_{K-1} = \hat\bGamma^{(2)}_{\cA_r,r} - \bGamma_{\cA_r,r}^* = \Big\{\hat\bGamma^*\big(\hat\bX^{(2)} + \bI_K\big)^\T\Big\}_{\cA_r,r} - \bGamma_{\cA_r,r}^*.
\end{equation*}
The second equality follows from $\hat\bGamma^{(2)} = \hat\bGamma^*(\hat\bSigma^*)^{1/2}\big(\hat\bR_2^{(2)}\big)^\T\big(\hat\bR_1^{(2)}\big)^\T$ and \eqref{eq_transform_theta_Star_2}. Since $\|\hat\bGamma_{\cA_r,}^{*} - \bGamma_{\cA_r,}^*\| = \cOp\big(\bar\delta_{nq}\big)$ by Lemma~\ref{lemma_asymptotic_normality_star},
\begin{equation}
    \zero_{K-1} = \bGamma^*_{\cA_r,}\big(\hat\bX_{r,}^{(2)}\big)^\T + \hat\bGamma^*_{\cA_r,r} - \bGamma_{\cA_r,r}^* + \cOp(\bar\delta_{nq}^2).\label{eq_solve_linear_eq2_con2}
\end{equation}
Similar to the calculation in Section~\ref{thm_asymp_id1_part_1}, we can determine $\hat\bX^{(2)}$ from \eqref{eq_correlated_diagonal_correction} and \eqref{eq_solve_linear_eq2_con2}. 
Let $\ba_l=\hat\bGamma^*_{\cA_l,l}-\bGamma^*_{\cA_l,l}$, $\bE_l^{(2)}=(\bI_K)_{,-l}$, and $\bD_l^{(2)}=\bE_l^{(2)}(\bGamma^*_{\cA_l,-l})^{-1}$. Define $\bY$ with zero diagonal by $(\bY_{l,})^\T=-\bD_l^{(2)}\ba_l$. Then $(\hat\bX^{(2)}_{l,-l})^\T=(\bY_{l,-l})^\T+O_p\big(\tau_{nq}\sqrt{\log n/B_n}\big)$ for $l\in[K]$.
By \eqref{eq_correlated_diagonal_correction}, we can recover the diagonal entries in $\hat\bX^{(2)}$. Therefore, we arrive at 
\begin{equation}\label{eq_correlated_linear_map}
\mathrm{vec}(\hat\bX^{(2)})=\Big[\bI_{K^2}-\sum_{l=1}^K
(\be_l\otimes\be_l)\{\be_l^\T\otimes(\be_l^\T\bSigma_u^*)\}\Big]\mathrm{vec}(\bY)+O_p\big(\tau_{nq}\sqrt{\log n/B_n}\big).
\end{equation}
Denote $\bT_\Sigma=\bI_{K^2}-\sum_{l=1}^K
(\be_l\otimes\be_l)\{\be_l^\T\otimes(\be_l^\T\bSigma_u^*)\}$. The asymptotic expansion in Lemma~\ref{lemma_asymptotic_normality_star} gives
\begin{equation*}
    \left(\bPhi_{\Gamma}^{(2)}\right)^{-1/2}\left\{\left(\hat\bGamma_{\cA_1,1}^* - \bGamma_{\cA_1,1}^*\right)^\T,\dots, \left(\hat\bGamma_{\cA_K,K}^* - \bGamma_{\cA_K,K}^*\right)^\T\right\}^\T \overset{d}{\to } \cN\left(\zero,\bI_{{K(K-1)}}\right),
\end{equation*}
where $\bPhi_{\Gamma}^{(2)} \in \RR^{{K(K-1)}\times {K(K-1)}}$ is given as
\begin{equation*}
    \left(\bPhi_{\Gamma}^{(2)}\right)_{\cK_l^-,\cK_h^-} = \left\{1_{(j_1=j_2)}\be_l^\T\bPhi_{\gamma,j_1}^*\be_h\right\}_{j_1\in \cA_l,j_2\in \cA_h}\text{ for each }l,h\in[K],
\end{equation*}
where $\cK_l^- = \{(l-1)(K-1)+1,\dots,l(K-1)\}$. Let $\bOmega_X^{(2)}$ be the covariance matrix for the row-stacked representation, with $K\times K$ blocks
\begin{equation}
    \big(\bOmega_X^{(2)}\big)_{\cK_l,\cK_h}
    =\bD_l^{(2)}\big(\bPhi_{\Gamma}^{(2)}\big)_{\cK_l^-,\cK_h^-}(\bD_h^{(2)})^\T,
    \qquad l,h\in[K].
    \label{eq_asymp_cov_X_2_row}
\end{equation}
Let $\bP_K$ be the commutation matrix satisfying $\bP_K\mathrm{vec}(\bM^\T)=\mathrm{vec}(\bM)$ for every $\bM\in\RR^{K\times K}$, and define
\begin{equation}
    \bPhi_X^{(2)}=\bT_\Sigma\bP_K\bOmega_X^{(2)}\bP_K^\T\bT_\Sigma^\T.\label{eq_asymp_cov_X_2}
\end{equation}
The distributional limit of $\hat\bX^{(2)}$ is therefore written as
\begin{equation}
    \frac{\ba^\T\mathrm{vec}(\hat\bX^{(2)})}
    {\{\ba^\T\bPhi_X^{(2)}\ba\}^{1/2}}
    \overset{d}{\to}\cN(0,1)
    \quad\text{for every fixed $\ba$ such that $\ba^\T\bPhi_X^{(2)}\ba>0$}.\label{eq_asymp_dis_X_2}
\end{equation}
Now we derive the asymptotic distributions for $\hat\btheta_j^{(2)} - \btheta_j^* = \big((\hat\beta_j^{(2)}  -\beta_j^*), (\hat\bgamma_j^{(2)} - \bgamma_j^*)^\T\big)^\T$. For $\hat\beta_j^{(2)}$, by \eqref{eq_accurate_control_hsigma_star_3} and \eqref{eq_transform_theta_Star_2}, one has
\begin{equation*}
    \hat\beta_j^{(2)} - \beta_j^* = \hat\beta_j^* - \beta_j^* + \cOp\big(\tau_{nq}\sqrt{\log n/B_n}\big).
\end{equation*}
When $\sqrt{n\log n}\tau_{nq}\to 0$, $ \hat\beta_j^{(2)} - \beta_j^*$ has the same asymptotic distribution as $\hat\beta_j^* - \beta_j^*$. For $\hat\bgamma_j^{(2)} - \bgamma_j^*$, one notes
\begin{equation*}
    \hat\bgamma_j^{(2)} - \bgamma_j^*
    = \hat\bX^{(2)}\bgamma_j^* + \hat\bgamma_j^*- \bgamma_j^*
    + \hat\bX^{(2)}(\hat\bgamma_j^*- \bgamma_j^*).
\end{equation*}
The first term $\hat\bX^{(2)}\bgamma_j^*$ has the projected asymptotic distribution in \eqref{eq_asymp_dis_X_2}; the second term has the asymptotic distribution given in Lemma~\ref{lemma_asymptotic_normality_star}; and the third term is $\cOp(\bar\delta_{nq}^2)$ by \eqref{eq_bound_x_2} and $\|\hat\bgamma_j^* - \bgamma_j^*\| = \cOp(\bar\delta_{nq})$. It remains to calculate the asymptotic covariance between $\hat\bX^{(2)}\bgamma_j^*$ and $\hat\btheta_j^* - \btheta_j^*$.

When $j\notin \cA_r$ for all $r\in[K]$, Lemma~\ref{lemma_asymptotic_normality_star} shows that $\hat\btheta_j^* - \btheta_j^*$ is asymptotically independent of $\hat\bGamma_{\cA_k,k}^* - \bGamma_{\cA_k,k}^*$ for all $k\in[K]$, and hence of $\hat\bX^{(2)}$. Therefore,
\begin{equation}\begin{aligned}
    &\,\left(\bPhi_j^* + \begin{pmatrix}
        0&\zero\\\zero & \big\{(\bgamma_j^*)^\T\otimes \bI_K\big\}\bPhi^{(2)}_X\big\{\bgamma_j^*\otimes \bI_K\big\}
    \end{pmatrix}\right)^{-1/2}\begin{pmatrix}\hat\beta_j^{(2)} - \beta_j^*\\\hat\bgamma_j^{(2)} - \bgamma_j^*\end{pmatrix}\\
    =&\,\left(\bPhi_j^* + \begin{pmatrix}
        0&\zero\\\zero & \big\{(\bgamma_j^*)^\T\otimes \bI_K\big\}\bPhi^{(2)}_X\big\{\bgamma_j^*\otimes \bI_K\big\}
    \end{pmatrix}\right)^{-1/2}\begin{pmatrix}\hat\beta_j^{*} - \beta_j^*\\\hat\bgamma_j^{*} - \bgamma_j^* +\hat\bX^{(2)}\bgamma_j^*\end{pmatrix} \\&\quad+ \cOp\big(\sqrt{n}\tau_{nq}\sqrt{\log n/B_n}\big)\\\overset{d}{\to}&\,\cN\big(\zero_{K+1},\bI_{K+1}\big),\end{aligned}\label{eq_example_asym_X2_ind}
\end{equation}
as $n,q\to\infty$, provided that $\sqrt {n\log n}\,\tau_{nq}\to0$. When $j\in\cA_k$ for at least one $k$, the covariance between $\hat\btheta_j^* - \btheta_j^*$ and $\hat\bX^{(2)}\bgamma_j^*$ must also be included. Write $j_{kr}=\cA_k(r)$ and $\bM_k=\bT_\Sigma(\bD_k^{(2)}\otimes\be_k)$. Then $\mathrm{vec}(\hat\bX^{(2)})=-\sum_k\bM_k\ba_k$ to first order. Define 
\begin{equation}
    \bPhi^{(2)}_{X,kr}
    =-\be_{k+1}^{(K+1)}(\be_r^{(K-1)})^\T(\bM_k)^\T
    (\bgamma_{j_{kr}}^*\otimes\bI_K),
    \qquad k\in[K],\quad r\in[K-1].\label{eq_define_phi_X2_kr}
\end{equation}
Similarly, we know
\begin{equation*}
    \mathrm{Cov}\left(\hat\btheta_j^* - \btheta_j^*,\hat\bX^{(2)}\bgamma_j^*\right)
    =\sum_{k=1}^K\sum_{r=1}^{K-1}1_{\{j=j_{kr}\}}\bPhi_j^*\bPhi_{X,kr}^{(2)}
    +\cOp\big(\bar\delta_{nq}\tau_{nq}\sqrt{\log n/B_n}\big).
\end{equation*}
Subject to the joint limit noted above, the same calculation as in Section~\ref{thm_asymp_id1_part_1} yields
\begin{equation*}
    \Big(\bE_{(j)}\bPhi_{j}^{(2)}\bE_{(j)}^\T\Big)^{-1/2}\begin{pmatrix}
        \hat\beta_j^{(2)} - \beta_j^*\\\hat\bgamma_{j,\cS_j}^{(2)} - \bgamma_{j,\cS_j}^*\end{pmatrix}\overset{d}{\to}\cN\big(\zero_{K_j+1},\bI_{K_j+1}\big)
\end{equation*}
where $K_j=|\cS_j|$ and
\begin{equation}
\begin{aligned}
    \bPhi_{j}^{(2)}
    =&\,\bPhi_j^*
    +\bJ_\gamma\big\{(\bgamma_j^*)^\T\otimes \bI_K\big\}\bPhi_X^{(2)}
        \big\{\bgamma_j^*\otimes \bI_K\big\}\bJ_\gamma^\T\\
    &\,+\Big\{\sum_{k=1}^K\sum_{r=1}^{K-1}1_{\{j=j_{kr}\}}\bPhi_j^*\bPhi_{X,kr}^{(2)}\Big\}\bJ_\gamma^\T+\bJ_\gamma\Big\{\sum_{k=1}^K\sum_{r=1}^{K-1}1_{\{j=j_{kr}\}}\bPhi_j^*\bPhi_{X,kr}^{(2)}\Big\}^\T.
\end{aligned}\label{eq_define_cov_ic2}
\end{equation}
Here $\bPhi_X^{(2)}$ and $\bPhi_{X,kr}^{(2)}$ are defined in \eqref{eq_asymp_cov_X_2} and \eqref{eq_define_phi_X2_kr}, respectively.

\subsubsection{Asymptotic Distributions for Latent Variables}\label{thm_asymp_id2_part_2}
We establish the asymptotic normality of $\hat\bu^{\star(2)}_i = \mathop{\arg\max}_{\bu\in\RR^{K}}\sum_{j=1}^q\ell_{ij}\left(\big(\hat\bgamma_j^{(2)}\big)^\T\bu + \hat\beta_j^{(2)}\right)$. The asymptotic normality of $\hat\bu^{(2)}$ then follows from the argument used in Section~\ref{supp_sec_prove_thm_asymp_bu}, together with $\|\hat\bSigma^{(2)} - \bSigma_u^*\| = \cOp(\bar\delta_{nq})$ and the positive definiteness and bounded eigenvalues of $\bSigma_u^*$.

Similar to the derivation in Section~\ref{thm_asymp_id1_part_2}, one can verify that with \eqref{eq_transform_theta_Star_2}, there is
\begin{align*}
    \hat\bu^{\star(2)}_i - \bu_i^*& = \big(\hat\bX^{(2)} + \bI_K\big)^{-\T}\big\{\bu_i^{\star}(\hat\btheta^*) - \hat\bmu^*\big\} - \bu_i^*  \\& = \bu_i^{\star}(\hat\btheta^*)  - \bu_i^* - \big(\hat\bX^{(2)} \big)^\T\bu_i^* -   \hat\bmu^* + \cOp(\bar\delta_{nq}^2).
\end{align*}
The remainder term follows from \eqref{eq_bound_x_2} and Lemma~\ref{lemma_asymptotic_normality_star}. Together with \eqref{eq_accurate_control_hsigma_star_3}, \eqref{eq_asymp_cov_X_2}, and the scaling condition in part~$(ii)$ of Theorem~\ref{thm_asymp_id2}, this gives $\big(\bPsi_i^{(2)}\big)^{-1/2}\big(\hat\bu^{\star(2)}_i - \bu_i^*\big) \overset{d}{\to} \cN\big(\zero_K,\bI_K\big)$,
with 
\begin{equation}
    \bPsi_i^{(2)} = \bPsi_i^* + \big\{\bI_K\otimes (\bu_i^*)^\T\big\}\bPhi^{(2)}_X\big\{\bI_K\otimes \bu_i^*\big\}.\label{eq_define_cov_psi_i_2}
\end{equation}
Here $(\hat\bX^{(2)})^\T\bu_i^*$ is asymptotically independent of $\bu_i^{\star}(\hat\btheta^*)$, because Lemma~\ref{lemma_asymptotic_normality_star} makes each $\hat\btheta_j^*$ asymptotically independent of $\bu_i^{\star}(\hat\btheta^*)$.

\subsubsection{Posterior Asymptotic Normality}\label{thm_asymp_id2_part_3}
The proof is omitted because it closely mirrors the argument in Section~\ref{thm_asymp_id1_part_3}. The only difference is that the working prior now has the plug-in covariance matrix $\hat\bSigma^{(2)}$. By \eqref{eq_transform_theta_Star_2_sigma}, $\|\hat\bR^{(2)}_1\hat\bR_2^{(2)}(\hat\bSigma^*)^{1/2} - \bI_K\| = \|\hat\bX^{(2)}\| = \cOp(\bar\delta_{nq})$, and \eqref{eq_ic4_sigma_consistency} and the transformation above give $\|\hat\bSigma^{(2)} - \bSigma_u^*\| = \cOp(\bar\delta_{nq})$. Because $\bSigma_u^*$ is positive definite with bounded eigenvalues under the modified Assumption~\ref{assump_standard_id}, the argument in Section~\ref{thm_asymp_id1_part_3} applies with only notational changes.



\section{Proof of Technical Lemmas}
\subsection{Proof of Lemmas in Section~\ref{supp_sec_prelim}}

\subsubsection{Proof of Lemma~\ref{lemma_bu_star_theta_true}}\label{supp_sec_prove_lemma_bu_star_theta_true}
The first part of Lemma~\ref{lemma_bu_star_theta_true} is implied by the proof of Lemma~\ref{lemma_neighbour_control} in Section~\ref{supp_sec_prove_lemma_neighbour_control} by taking $\epsilon = 0$. Now we show the asymptotic expansion.
Apply the integral mean value theorem, and we have the expansion
\begin{align}
    \sum_{j=1}^q\ell_{ij}^{\prime}\big({\bgamma_j^*}^\T\bu_i^{*} + \beta_j^*\big)\bgamma_j^*=&\,\sum_{j=1}^q\ell_{ij}^{\prime}\big({\bgamma_j^*}^\T\bu_i^{*} + \beta_j^*\big)\bgamma_j^* - \sum_{j=1}^q\ell_{ij}^{\prime}({\bgamma_j^*}^\T\bu_i^{\star}+\beta_j^*)\bgamma_j^*\nonumber \\=&\, \sum_{j=1}^q\int_{0}^1-\ell_{ij}^{\prime\prime}\big({\bgamma_j^*}^\T\big\{\bu_i^{*} + s(\bu_i^{\star} - \bu_i^*)\big\} + \beta_j^*\big)\bgamma_j^*{\bgamma_j^*}^\T(\bu_i^{\star} - \bu_i^*)ds.\label{eq_bound_neighb_theta_2der_diff_true}
\end{align}
Here the first equality follows from $\partial_{\bu}\ell_i(\bu_i^\star;\btheta^*) =\sum_{j=1}^q\ell_{ij}^{\prime}({\bgamma_j^*}^\T\bu_i^{\star}+\beta_j^*)\bgamma_j^* = \zero_K$ as $\bu^\star_i$ is the maximizer of $\ell_i(\bu;\btheta^*)$. For the integral, we decompose it as
\begin{align}
    &\,-\sum_{j=1}^q\ell_{ij}^{\prime\prime}\big(\eta_{ij}^*)\bgamma_j^*{\bgamma_j^*}^\T(\bu_i^{\star} - \bu_i^*) \nonumber\\&\,+\sum_{j=1}^q\int_{0}^1\Big(\ell_{ij}^{\prime\prime}\big(\eta_{ij}^*) -\ell_{ij}^{\prime\prime}\big[{\bgamma_j^*}^\T\big\{\bu_i^{*} + s(\bu_i^{\star} - \bu_i^*)\big\} + \beta_j^*\big]\Big)\bgamma_j^*{\bgamma_j^*}^\T(\bu_i^{\star} - \bu_i^*)ds.\label{eq_integraL_res_u_star}
\end{align}
Since $|\ell_{ij}^{(3)}(\eta)|\le B_n$, $\max_{i\in[n]}\|\bu_i^{\star} - \bu_i^*\| = \cOp(\sqrt{\log n/q}\kappa_n)$ and $\|\btheta_j^{*}\|\le C$, we have
\begin{align*}
    &\,\max_{i\in[n],j\in[q]}\max_{s\in[0,1]}\left|\ell_{ij}^{\prime\prime}\big({\bgamma_j^*}^\T\{\bu_i^{*} + s(\bu_i^{\star} - \bu_i^*)\} + \beta_j^*\big) - \ell_{ij}^{\prime\prime}\big({\bgamma_j^*}^\T\bu_i^{*}  + \beta_j^*\big)\right| \\\le &\,B_n\max_{i\in[n],j\in[q]}\left|{\bgamma_j^*}^\T(\bu_i^{\star} - \bu_i^{*})\right| = \cOp\big(B_n\kappa_n\sqrt{\log n/q}\big).
\end{align*}
Therefore, we know 
\begin{align*}
    &\,\max_{i\in[n]}\Big\|\sum_{j=1}^q\int_{0}^1\Big(\ell_{ij}^{\prime\prime}\big(\eta_{ij}^*) -\ell_{ij}^{\prime\prime}\big[{\bgamma_j^*}^\T\big\{\bu_i^{*} + s(\bu_i^{\star} - \bu_i^*)\big\} + \beta_j^*\big]\Big)\bgamma_j^*{\bgamma_j^*}^\T(\bu_i^{\star} - \bu_i^*)ds\Big\| \\&\,\le \sum_{j=1}^q\cOp\Big(B_n\kappa_n\sqrt{\log n/q}\max_{j\in[q]}\|\bgamma_j^*\|^2\|\bu_i^{\star}-\bu_i^*\|\Big) = \cOp\big(\kappa_n^2B_n\log n\big).
\end{align*}
One then has $\max_{i\in[n]}\|\eqref{eq_integraL_res_u_star}\| = \cOp(\kappa_n^2B_n\log n)$. Similarly, $\|\eqref{eq_integraL_res_u_star}\| = \cOp(\kappa_n^2B_n)$ for each $i\in[n]$ and $\sum_{i\in[n]}\|\eqref{eq_integraL_res_u_star}\|^2 = \cOp(n\kappa_n^4B_n^2)$.

Plugging into \eqref{eq_bound_neighb_theta_2der_diff_true} and inverting $\sum_{j=1}^q-\ell_{ij}^{\prime\prime}(\eta_{ij}^*)\bgamma_j^*{\bgamma_j^*}^\T$, one has
\begin{equation*}
        \bu_i^{\star} - \bu_i^* = \big\{\sum_{j=1}^q-\ell_{ij}^{\prime\prime}(\eta_{ij}^*)\bgamma_j^*(\bgamma_j^*)^\T\big\}^{-1}\Big\{\sum_{j=1}^q\ell_{ij}^{\prime}(\eta_{ij}^*)\bgamma_j^* - \eqref{eq_integraL_res_u_star}\Big\}.
    \end{equation*}
With the bounds for \eqref{eq_integraL_res_u_star} above and $\lambda_{\min}\{\sum_{j=1}^q-\ell_{ij}^{\prime\prime}(\eta_{ij}^*)\bgamma_j^*(\bgamma_j^*)^\T\big\}\ge q b_n\lambda_{\min}(\bSigma_{\gamma}^*)/2$ when $q$ is large enough, by Assumptions~\ref{assumption_psd_covariance} and~\ref{assumption_smoothness}, we prove the expansion with the residual bounded as required in Lemma~\ref{lemma_bu_star_theta_true}.


To sum the expansion \eqref{eq_expansion_u_star}, let
$\bA_i=-\sum_j\ell_{ij}''(\eta_{ij}^*)\bgamma_j^*(\bgamma_j^*)^\T$,
$\bb_i=\sum_j\ell_{ij}'(\eta_{ij}^*)\bgamma_j^*$, and
$\bar\bA_i=\EE(\bA_i\mid\mathscr U_n)$. Conditional on $\mathscr U_n$, the vectors $\bar\bA_i^{-1}\bb_i$ are independent and mean zero. Assumptions~\ref{assumption_psd_covariance} and~\ref{assumption_smoothness} give $\|\bA_i^{-1}\|,\|\bar\bA_i^{-1}\|\le C/(qb_n)$,
$\EE(\|\bA_i-\bar\bA_i\|^2\mid\mathscr U_n)\le CqB_n^2$, and
$\EE(\|\bb_i\|^2\mid\mathscr U_n)\le CqB_n$. Hence we know\begin{align*}
\Big\|\sum_i\bar\bA_i^{-1}\bb_i\Big\|&=\cOp\big(\kappa_n\sqrt{n/q}\big),\\
\Big\|\sum_i(\bA_i^{-1}-\bar\bA_i^{-1})\bb_i\Big\| &\le \frac{C}{q^2b_n^2}\sum_i\|\bA_i-\bar\bA_i\|\,\|\bb_i\| =\cOp\left(\frac{nB_n^{3/2}}{qb_n^2}\right),\end{align*} which yields\begin{equation*}
\Big\|\sum_i(\bu_i^\star-\bu_i^*)\Big\| =\cOp\left(\kappa_n\sqrt{n/q}+\frac{n\kappa_n^3\sqrt{B_n}}{q}\right).
\end{equation*}
Here we use $\|\sum_i\cR_i^\star\|\le\sqrt n(\sum_i\|\cR_i^\star\|^2)^{1/2}$ and $b_n\le1\le B_n$.

Following a similar procedure, one can prove
\begin{equation}
    \left\|\sum_{i=1}^n\bu_i^*\left(\bu_i^{\star} - \bu_i^*\right)^\T\right\| = \cOp\big(\kappa_n\sqrt{n/q} + \kappa_n^3\sqrt{B_n}n/q\big).\label{eq_bound_sum_asymp_uu}
\end{equation}
Therefore, for $\|\sum_{i=1}^n\big(\bu_i^{\star} {\bu_i^{\star}}^\T - \bu_i^{*} {\bu_i^{*}}^\T\big)\|$, we conclude
\begin{align*}
    \left\|\sum_{i=1}^n\left(\bu_i^{\star} {\bu_i^{\star}}^\T - \bu_i^{*} {\bu_i^{*}}^\T\right)^\T\right\| \le &\, \left\|\sum_{i=1}^n\bu_i^*\left(\bu_i^{\star} - \bu_i^*\right)^\T\right\| + \left\|\sum_{i=1}^n\left(\bu_i^{\star} - \bu_i^*\right)(\bu_i^*)^\T\right\| \\&\, + \left\|\sum_{i=1}^n\left(\bu_i^{\star} - \bu_i^{*} \right)\left(\bu_i^{\star} - \bu_i^{*} \right)^\T\right\|\\\le &\,\cOp\big(\kappa_n\sqrt{n/q} + \kappa_n^3\sqrt{B_n}n/q\big),
\end{align*}where the last inequality follows from \eqref{eq_bound_sum_asymp_uu}, $\big(n^{-1}\sum_{i=1}^n\|\bu_i^{\star} - \bu_i^*\|^2\big)^{1/2} = \cOp(\kappa_nq^{-1/2})$ and Cauchy--Schwarz inequality.


\subsubsection{Proof of Lemma~\ref{lemma_first_order_base}}\label{supp_sec_prove_lemma_first_order_base}

We calculate a Laplace expansion for $\cS_{\theta_j}(\btheta,\bmu,\bSigma)$ as posterior expectations following~\citep{miyata2004fully}, and verify that its remainder is uniform. All bounds below are simultaneous over $i\in[n]$, $j\in[q]$, $\btheta\in\Xi$ and $(\bmu,\bSigma)\in\bar\cK_\rho$. A similar argument applies with $\bar\cK_\rho$ replaced by any fixed compact subset of $\RR^K\times\{\bSigma:\bSigma\succ\zero\}$, with differences only in the constants.

For each subject $i\in[n]$, define
\begin{equation*}
Z_i(\btheta,\bmu,\bSigma)
:=\int_{\RR^K}\exp\{\ell_i(\bu;\btheta)\}\pwk(\bu\mid\bmu,\bSigma)\,d\bu,
\end{equation*}
so that $\cL(\btheta,\bmu,\bSigma)=-\sum_{i=1}^n \log Z_i(\btheta,\bmu,\bSigma)$. 
For each $j\in[q]$, let
\[
N_{ij}(\btheta,\bmu,\bSigma)
:=\int_{\RR^K} \ell^{\prime}_{ij}(\bgamma_j^\T\bu+\beta_j) \begingroup \renewcommand{\arraystretch}{0.7}\begin{pmatrix}
        1\\\bu
    \end{pmatrix}\endgroup\exp\{\ell_i(\bu;\btheta)\}\pwk(\bu\mid\bmu,\bSigma)\,d\bu.
\]
Then we know $\cS_{\theta_j}(\btheta,\bmu,\bSigma)
= \partial_{\btheta_j}\cL(\btheta,\bmu,\bSigma)
= -\sum_{i=1}^n \frac{N_{ij}(\btheta,\bmu,\bSigma)}{Z_i(\btheta,\bmu,\bSigma)}$.
Write, for short, $g_{ij}(\bu,\btheta) = \ell_{ij}^{\prime}(\bgamma_j^\T\bu + \beta_j)(1,\bu^\T)^\T$ for each $i\in[n]$ and $j\in[q]$, and denote its $k$th-order derivative by $\partial_{\bu}^kg_{ij}$. After enlarging the fixed constant $C_\eta$ in the definitions of $B_n$ and $b_n$ if necessary, the constraints defining $\Xi$ imply $\max_{\btheta\in\Xi}\max_{i\in[n],j\in[q]}
|\bgamma_j^\T\bu_i^\star(\btheta)+\beta_j|
\le C_\eta\sqrt{\log n}$. Let $M_{ij,n}:=\sup_{|\eta|\le C_\eta\sqrt{\log n}}|\ell_{ij}'(\eta)|$. Then the mean-value theorem and Assumption~\ref{assumption_smoothness} give $M_{ij,n}\le |\ell_{ij}'(\eta_{ij}^*)|+CB_n\sqrt{\log n}$.
Set $\widetilde\varepsilon_{nq}=\exp[\{\log(n\vee q)\}^{1/2+\tilde c_{\varepsilon}/2}]$ for some constant $\tilde c_{\varepsilon}>0$, such that every fixed product of powers of $\widetilde\varepsilon_{nq}$, $B_n$, $b_n^{-1}$, and $\log(n\vee q)$ is $o(\varepsilon_{nq})$. This is feasible under Assumption~\ref{assump_scaling}.
The conditional sub-exponential bound, followed by Bernstein's inequality and a union bound over $j$, leads to
\begin{equation}\label{eq_uniform_Mijn}
    \max_{j\in[q]}\frac{1}{n}\sum_{i=1}^nM_{ij,n} =\cOp(B_n\sqrt{\log n}),\qquad
    \max_{i\in[n],j\in[q]}M_{ij,n} =\cOp\big\{\sqrt{B_n}\log(nq)+B_n\sqrt{\log n}\big\}=\cOp(\tilde\varepsilon_{nq}).
\end{equation}
Write $\bH_{i,\btheta}:=-q^{-1}\partial_{\bu}^2\ell_i\{\bu_i^\star(\btheta);\btheta\}$.
By Assumption~\ref{assumption_smoothness}, $\max_{\btheta\in\Xi}\max_{i\in[n],j\in[q]}
|\bgamma_j^\T\bu_i^\star(\btheta)+\beta_j|
\le C_\eta\sqrt{\log n}$, the definition of $\Xi$, and singular-value interlacing,
\begin{equation*}
c_\theta^2b_n \le\lambda_{\min}(\bH_{i,\btheta}) \le\lambda_{\max}(\bH_{i,\btheta}) \le CB_n
\end{equation*}
uniformly over $i\in[n]$ and $\btheta\in\Xi$. Assumption~\ref{assumption_smoothness} implies that third and fourth derivatives of $q^{-1}\ell_i(\cdot;\btheta)$ are bounded by $CB_n$, and subsequently, for $1\le r\le3$,
\begin{equation}\label{eq_uniform_gij_derivatives}
\|\partial_{\bu}^rg_{ij}(\bu,\btheta)\|
\le C\{1+M_{ij,n}+B_n\sqrt{\log n}\}.
\end{equation}
The derivatives of the Gaussian working density divided by the density itself grow only polynomially in $\|\bu\|$, uniformly over $(\bmu,\bSigma)\in\bar\cK_\rho$. 

Fix $i$ and write for short $\bu_0=\bu_i^\star(\btheta)$, $h=-q^{-1}\ell_i(\bu;\btheta)$, $\bH=\bH_{i,\btheta}$, $\bA=\bH^{-1/2}$ and $\delta=q^{-1/2}$, so that $\partial_{\bu}h(\bu_0)=\zero$ and $\bH=q^{-1}\bH_{Lu_iu_i}\{\btheta,\bu^{\star}(\btheta)\}$. Since $\|\partial_{\bu}^3h\|\le CB_n$, there is a fixed $c_r>0$ such that, letting $r_n=c_rb_n/B_n$ be the localization radius, it holds that
\begin{equation*}
\tfrac12\bH\preceq\partial_{\bu}^2h(\bu_0+\bv)\preceq\tfrac32\bH\quad(\|\bv\|\le r_n),
\qquad q\{h(\bu_0+\bv)-h(\bu_0)\}\ge cqb_n\|\bv\|^2.
\end{equation*}
By the convexity of $h$, we know that for every unit vector $\bs$ and every $t\ge r_n$,
\begin{equation*}
q\{h(\bu_0+t\bs)-h(\bu_0)\}\ge cqb_n\{r_n^2+r_n(t-r_n)\}\ge cqb_n^3/B_n^2.
\end{equation*}
Integrating over $\|\bu-\bu_0\|\le (qB_n)^{-1/2}$ gives $Z_i\ge c\exp\{\ell_i(\bu_0;\btheta)\}\pwk(\bu_0\mid\bmu,\bSigma)(qB_n)^{-K/2}$ with $\pwk(\bu_0\mid\bmu,\bSigma)\ge n^{-C}$ on $\bar\cK_\rho$. Assumption~\ref{assumption_smoothness} gives the global envelope $\|g_{ij}(\bu,\btheta)\|\le C\{1+M_{ij,n}+B_n\sqrt{\log n}\}(1+\|\bu\|)e^{C\|\bu\|}$, and the integral is uniformly bounded on $\bar\cK_\rho$. In $\{\|\bu-\bu_0\|>r_n\}$, with the lower bound $Z_i\ge c\exp\{\ell_i(\bu_0;\btheta)\}\pwk(\bu_0\mid\bmu,\bSigma)(qB_n)^{-K/2}$, we know that the scaling factor is at most $(n\vee q)^Ce^{-cqb_n^3/B_n^2}\ll q^{-C}$ for any fixed constant $C$. Here we use $qb_n^3/\{B_n^2\log(n\vee q)\}\to\infty$ under Assumption~\ref{assump_scaling}.

Next, with change of variables $\bu=\bu_0+\delta\bA\bz$, we set
\begin{equation*}
W_\delta(\bz)=\frac{\pwk(\bu_0+\delta\bA\bz\mid\bmu,\bSigma)}{\pwk(\bu_0\mid\bmu,\bSigma)}
\exp\big[-q\{h(\bu_0+\delta\bA\bz)-h(\bu_0)\}+\tfrac12\|\bz\|^2\big].
\end{equation*}
To simplify the integral expansion, we let $\EE_0$ denote expectation under $\bZ\sim\cN(\zero_K,\bI_K)$, and $\Pi_i(\cdot\mid\btheta,\bmu,\bSigma)$ denote the posterior measure on $\bu_i$ induced by the integrand in the marginal likelihood. We have, for any function $F$,
\begin{equation}\label{eq_exact_ratio_representation}
\EE_{\Pi_i}F(\bU)=\frac{\EE_0\{F(\bu_0+\delta\bA\bZ)W_\delta(\bZ)\}}{\EE_0 W_\delta(\bZ)}.
\end{equation}
 Split $\RR^K$ into $\{\|\bz\|\le C_T\sqrt{\log(n\vee q)}\}$, $\{\|\bz\|>C_T\sqrt{\log(n\vee q)},\,\|\delta\bA\bz\|\le r_n\}$ and $\{\|\delta\bA\bz\|> r_n\}$ for some constant $C_T$. The first region lies in the local ball of radius $r_n$ because $qb_n^3/B_n^2\gg\log(n\vee q)$. On the second region, one can verify that $e^{-\|\bz\|^2/2}W_\delta(\bz)\le Ce^{-\|\bz\|^2/8}$, so enlarging $C_T$ makes its contribution smaller than any fixed power of $q^{-1}$. The third region is controlled above.
 
 On the first region, by Assumption~\ref{assumption_smoothness}, Taylor expansion of $h$ yields
\begin{equation*}
W_\delta(\bz)=1+\delta p_1(\bz)+\delta^2R_W(\bz),\;\text{ for }\;
p_1(\bz)=\{\partial_{\bu}\log\pwk(\bu_0\mid\bmu,\bSigma)\}^\T\bA\bz
-\tfrac16\partial_{\bu}^3h(\bu_0)[\bA\bz,\bA\bz,\bA\bz],
\end{equation*}
and $|R_W(\bz)|\le C\tilde\varepsilon_{nq}(1+\|\bz\|^{8})$. Since $p_1$ is odd, we obtain $\EE_0W_{\delta}(\bZ)=1+\cO(\delta^2\tilde\varepsilon_{nq})\ge 1/2$. Centering the integrand at $\bu_0$, \eqref{eq_uniform_gij_derivatives} gives
\begin{equation*}
\bG_{ij}(\bz):=g_{ij}(\bu_0+\delta\bA\bz,\btheta)-g_{ij}(\bu_0,\btheta)
=\delta\bL_{ij}(\bz)+\delta^2\bQ_{ij}(\bz)+\delta^3\bR_{ij}(\bz),
\end{equation*}
where $\bL_{ij}(\bz)=\partial_{\bu}g_{ij}(\bu_0,\btheta)\bA\bz$ is odd, $\bQ_{ij}(\bz)=\frac12\partial_{\bu}^2g_{ij}(\bu_0,\btheta)[\bA\bz,\bA\bz]$, and $\|\bR_{ij}(\bz)\|\le C\{1+M_{ij,n}+B_n\sqrt{\log n}\}b_n^{-3/2}\|\bz\|^3$. Integrating each term over standard Gaussian and dividing by $\EE_0W_{\delta}(\bZ)$, the order-$\delta$ term cancels and one obtains \begin{align}
\sup_{\substack{\btheta\in\Xi\\(\bmu,\bSigma)\in\bar\cK_\rho}} \Big\| \frac{N_{ij}(\btheta,\bmu,\bSigma)}{Z_i(\btheta,\bmu,\bSigma)} -g_{ij}\{\bu_i^\star(\btheta),\btheta\} -q^{-1}\cA_{ij}(\btheta,\bmu,\bSigma) \Big\| =\cOp\left[\frac{\tilde\varepsilon_{nq}}{q^{3/2}} \{1+M_{ij,n}+B_n\sqrt{\log n}\}\right], \label{eq_uniform_first_moment_expansion}\\
\sup_{\substack{\btheta\in\Xi\\(\bmu,\bSigma)\in\bar\cK_\rho}} \|\cA_{ij}(\btheta,\bmu,\bSigma)\| =\cOp\left[\tilde\varepsilon_{nq} \{1+M_{ij,n}+B_n\sqrt{\log n}\}\right],
\label{eq_uniform_first_moment_coefficient}
\end{align}
uniformly over $i\in[n]$ and $j\in[q]$.

Note here $\varepsilon_{nq}$ dominates fixed product of $\tilde\varepsilon_{nq}$, $B_n$, $b_n^{-1}$ and $\log(n\vee q)$. Here $\cA_{ij}(\btheta,\bmu,\bSigma)$ collects the $\delta^2$ coefficient. Thus, combining \eqref{eq_uniform_first_moment_expansion} and \eqref{eq_uniform_first_moment_coefficient} gives the tailored first-order bound
\begin{equation}\label{eq_uniform_first_moment_laplace}
\sup_{\substack{\btheta\in\Xi\\(\bmu,\bSigma)\in\bar\cK_\rho}} \left\| \frac{N_{ij}(\btheta,\bmu,\bSigma)}{Z_i(\btheta,\bmu,\bSigma)} -g_{ij}\{\bu_i^\star(\btheta),\btheta\} \right\|
=\cOp\Big(\frac{\varepsilon_{nq}}{q}\Big).
\end{equation}
Summing \eqref{eq_uniform_first_moment_laplace} over $i$ yields
\begin{equation*}
\sup_{\substack{\btheta\in\Xi\\(\bmu,\bSigma)\in\bar\cK_\rho}}
\max_{j\in[q]} \left\| \{\cS_\theta(\btheta,\bmu,\bSigma) -\bS_{L\theta}(\btheta,\bu^\star(\btheta))\}_{\cK_j^+} \right\|
=\cOp\left(\frac{nB_n\sqrt{\log n}}{q}\varepsilon_{nq}\right).
\end{equation*}
Since all $q$ blocks have fixed dimension, the $\ell_2$ bound follows immediately.

For $F(\bu)=\bu$ and $F(\bu)=(\bu-\bmu)(\bu-\bmu)^\T$, because their coefficients grow at most polynomially in $\|\bu_0\|+\|\bmu\|$, the term $\{1+M_{ij,n}+B_n\sqrt{\log n}\}$ no longer occurs and thus we get
\begin{align*}
\sup_{\substack{\btheta\in\Xi\\(\bmu,\bSigma)\in\bar\cK_\rho}} \max_{i\in[n]} \left\|\EE_{\Pi_i}(\bU)-\bu_i^\star(\btheta)\right\|
&=\cOp(q^{-1}\varepsilon_{nq}),\\
\sup_{\substack{\btheta\in\Xi\\(\bmu,\bSigma)\in\bar\cK_\rho}} \max_{i\in[n]} \Big\|\EE_{\Pi_i} \{(\bU-\bmu)(\bU-\bmu)^\T\} -\{\bu_i^\star(\btheta)-\bmu\} \{\bu_i^\star(\btheta)-\bmu\}^\T\Big\| &=\cOp(q^{-1}\varepsilon_{nq}).
\end{align*}
Because $\cS_\mu(\btheta,\bmu,\bSigma) = -\sum_{i=1}^n\bSigma^{-1}
\{\EE_{\Pi_i}(\bU)-\bmu\}$, $\cS_\Sigma(\btheta,\bmu,\bSigma) = -\frac12\sum_{i=1}^n[
\bSigma^{-1}\EE_{\Pi_i}
\{(\bU-\bmu)(\bU-\bmu)^\T\}\bSigma^{-1}-\bSigma^{-1}
]$, and the eigenvalues of $\bSigma$ are uniformly bounded above and away from zero on $\bar\cK_\rho$, the last two assertions follow by summing over $i$.



\subsubsection{Proof of Lemma~\ref{coro_first_order_concern}}\label{sub_prove_coro_first_order_concern}

Note that $\|\bS_{L\theta}(\btheta^*,\bu^*)\| = \|\Lb_{\eta}^\T\bU^{c*}\|_{F}$ and $\|\bS_{Lu}(\btheta^*,\bu^*)\| = \|\Lb_{\eta}\bGamma^*\|_F$ where $\Lb_{\eta} = \{\ell_{ij}^{\prime}(\eta_{ij}^*)\}_{n\times q}$ is given in Lemma~\ref{lemma_concentration_l_mat}, $\bGamma^* = (\bgamma_1^*,\dots,\bgamma_q^*)^\T$ and 
\begin{equation*}
    \bU^{c*} = \begin{pmatrix}
        1 & \cdots & 1\\
        \bu_1^* & \cdots & \bu_n^*
    \end{pmatrix}^\T.
\end{equation*}
By Assumption~\ref{assumption_psd_covariance} and because $K$ is fixed, $\|\bGamma^*\|_{\Fn}\le C\sqrt{q}$ for some constant $C$. By Assumption~\ref{assump_standard_id},
\[
    \|\bU^{c*}\|_{\Fn}
    =\left(n+\sum_{i=1}^n\|\bu_i^*\|^2\right)^{1/2}
    =\cOp(\sqrt n).
\]
Given $\{\bu_i^*\}_{i=1}^n$, the entries of $\Lb_{\eta}$ are independent and mean zero. Their conditional second moments are bounded by $CB_n$ because their conditional sub-exponential norms are bounded by $C\sqrt{B_n}$. Therefore,
\begin{equation*}
    \EE\left[\|\Lb_{\eta}^\T\bU^{c*}\|_{F}^2\Big|\{\bU_i = \bu_i^*\}_{i=1}^n\right]
    \le CB_n\sum_{j=1}^q\sum_{i=1}^n\big(1+\|\bu_i^*\|^2\big)
    =\cOp(nqB_n).
\end{equation*}
This proves $\big\|\bS_{L\theta}(\btheta^*, \bu^*)\big\| = \cOp\big(\sqrt{qnB_n}\big)$. The same calculation with the roles of $i$ and $j$ interchanged proves the stated bound for $\bS_{Lu}(\btheta^*, \bu^*)$.
The first part of Lemma~\ref{coro_first_order_concern} then follows from $\|\bA\bB\|_{\Fn}\le \|\bA\|\|\bB\|_{\Fn}$ and $(iii)$ of Lemma~\ref{lemma_concentration_l_mat}.

Note that each entry of $\bS_{L\theta}(\btheta^*,\bu^*)$ and $\bS_{Lu}(\btheta^*,\bu^*)$ is a summation of $\ell_{ij}^{\prime}(\eta_{ij}^*)$, whose sub-exponential norms are bounded by $C\sqrt{B_n}$ uniformly $i\in[n]$ and $j\in[q]$. To prove the second part of Lemma~\ref{coro_first_order_concern}, we apply the Bernstein inequality for sums of independent sub-exponential random variables:
\begin{lemma}\em Suppose Assumption~\ref{assumption_smoothness} holds and $\log(n\vee q)/\sqrt{n\wedge q}\to 0$ as $n,q\to\infty$.\begin{enumerate}[label=$(\roman*)$]

    \item For any deterministic sequence $\{a_i\}_{i=1}^n$ with $\max_{1\le i\le n}|a_i|\le A_n$ for some $A_n$, it holds that
    \begin{align*}
        \forall j\in[q],\;\Big| \sum_{i=1}^n a_i \ell_{ij}^{\prime}(\eta_{ij}^*) \Big| =\cOp\big(A_n\sqrt{nB_n}\big), \;\max_{1\le j\le q}\Big| \sum_{i=1}^n a_i \ell_{ij}^{\prime}(\eta_{ij}^*) \Big| =\cOp\big(A_n\sqrt{nB_n\log q}\big).
    \end{align*}
    \item For any deterministic sequence $\{a_j\}_{j=1}^q$ with $\max_{1\le j\le q}|a_j|\le A_q$ for some $A_q$, it holds that 
    \begin{align*}
        \forall i\in[n],\;\Big| \sum_{j=1}^q a_j \ell_{ij}^{\prime}(\eta_{ij}^*) \Big| =\cOp\big(A_q\sqrt{qB_n}\big), \; \max_{1\le i\le n}\Big| \sum_{j=1}^q a_j \ell_{ij}^{\prime}(\eta_{ij}^*) \Big| = \cOp\big(A_q\sqrt{qB_n\log n}\big).
    \end{align*}
    \item 
    For any deterministic sequence $\{a_{ij}\}_{i=1,j=1}^{n,q}$ with $\max_{1\le i\le n,1\le j\le q}|a_{ij}|\le A_{nq}$, it holds that
    \begin{align*}
        \left| \sum_{i=1}^n\sum_{j=1}^q a_{ij} \ell_{ij}^{\prime}(\eta_{ij}^*) \right| =\cOp\big( A_{nq}\sqrt{nqB_n}\big),\quad \max_{1\le i\le n}\max_{1\le j\le q}\left|\ell_{ij}^{\prime}(\eta_{ij}^*)\right|=\cOp\big(\sqrt{B_n}\log (nq)\big).
    \end{align*}
    \end{enumerate}\label{lemma_concentration}
\end{lemma}
\begin{proof}
    See Section~\ref{supp_sec_prove_lemma_concentration}.
\end{proof}
Conditioning on $\{\bu_i^*\}_{i=1}^n$ and $\cE_\eta$, the intercept coordinates of $\bS_{L\theta}$ are unweighted sums over $i$, whereas each loading coordinate is a sum with weights $u_{ir}^*$. By Assumption~\ref{assump_standard_id}, we know $\max_{i\in[n],r\in[K]}|u_{ir}^*|=\cOp(\sqrt{\log n})$ and $\max_{r\in[K]}\sum_{i=1}^n(u_{ir}^*)^2=\cOp(n)$.
Conditional Bernstein inequalities and a union bound over the $q(K+1)$ coordinates therefore yield
\begin{equation*}
\|\bS_{L\theta}(\btheta^*,\bu^*)\|_\infty
=\cOp\left\{\sqrt{nB_n\log q}+\sqrt{B_n\log n}\log q\right\}.
\end{equation*} Under Assumption~\ref{assump_scaling}, the second term is absorbed by the first term, and we obtain the stated bound for $n^{-1/2}\big\|\bS_{L\theta}(\btheta^*, \bu^*)\big\|_{\infty}$. The bound for $q^{-1/2}\big\|\bS_{Lu}(\btheta^*, \bu^*)\big\|_{\infty}$ can be obtained similarly.

\subsubsection{Proof of Lemma~\ref{lemma_first_order_concern}}\label{supp_sec_prove_lemma_first_order_concern}

We first prove the case when $\bmu = \zero_K$ and $\bSigma = \bI_K$. 
    Following the notation in Lemma~\ref{lemma_bu_star_theta_true}, let $\bu_i^\star$ be defined as
\begin{equation}
    \bu_i^{\star} =\mathop{\arg\max}_{\bu\in\RR^K} \ell_i(\bu ;\btheta^*).\label{eq_opti_intergral_mvt_at_true}
\end{equation}
By Lemma~\ref{lemma_bu_star_theta_true}, we know that $\|\bu_i^{\star} - \bu_i^*\| = \cOp(\kappa_n/\sqrt{q})$ and $\max_{i\in[n]}\|\bu_i^{\star} - \bu_i^*\|\le \cOp(\kappa_n\sqrt{\log n/q})$. For $\bu^{\star} = \big((\bu_1^{\star})^\T,\cdots, (\bu_n^{\star})^\T\big)^\T$, invoke Lemma~\ref{lemma_first_order_base} to get
\begin{equation*}
   \big\|\cS_{\theta}(\btheta^*) - \bS_{L\theta}\big(\btheta^*,\bu^\star\big)\big\| = \cOp(nB_n\sqrt{\log n}\varepsilon_{nq}/\sqrt{q}),\end{equation*} and \begin{equation*}\big\|\cS_{\theta}(\btheta^*) - \bS_{L\theta}\big(\btheta^*,\bu^\star\big)\big\|_{\infty} = \cOp(nB_n\sqrt{\log n}\varepsilon_{nq}/q).
\end{equation*}
Next, for the $j$-th block of $\bS_{L\theta}\big(\btheta^*,\bu^\star\big)$, it holds that
\begin{align*}
    &\,\big\|\bS_{L\theta_j}\big(\btheta^*,\bu^\star\big) - \bS_{L\theta_j}\big(\btheta^*,\bu^*\big)\big\|\nonumber \\=&\, \left\|\sum_{i=1}^n\ell_{ij}^{\prime}\big({\bgamma_j^*}^\T\bu_i^\star + \beta_j^*\big)\begin{pmatrix}
            1\\\bu_i^\star
        \end{pmatrix} - \ell_{ij}^{\prime}\big({\bgamma_j^*}^\T\bu_i^* + \beta_j^*\big)\begin{pmatrix}
            1\\\bu_i^*\end{pmatrix}\right\|\nonumber\\\le &\, \underbrace{\Big\|\sum_{i=1}^n\ell_{ij}^{\prime}\big({\bgamma_j^*}^\T\bu_i^* + \beta_j^*\big)\big\{\bu_i^\star-\bu_i^*\big\}\Big\|}_{\alpha_1} +\underbrace{\Big|\sum_{i=1}^n\ell_{ij}^{\prime}\big({\bgamma_j^*}^\T\bu_i^\star + \beta_j^*\big)-\ell_{ij}^{\prime}\big({\bgamma_j^*}^\T\bu_i^* + \beta_j^*\big)\Big|}_{\alpha_2}\\&\,+ \underbrace{\Big\|\sum_{i=1}^n\big\{\ell_{ij}^{\prime}\big({\bgamma_j^*}^\T\bu_i^\star + \beta_j^*\big)-\ell_{ij}^{\prime}\big({\bgamma_j^*}^\T\bu_i^* + \beta_j^*\big)\big\}\bu_i^\star\Big\|}_{\alpha_3}.
\end{align*}
To bound $\alpha_1$--$\alpha_3$, we introduce the following lemma.
\begin{lemma}\em\label{lemma_aux_score_bound_pre}
    Under Assumptions~\ref{assump_standard_id}--\ref{assump_scaling}, letting $\bH_{Lu_iu_i}^* = -\sum_{t=1}^q \ell_{it}^{\prime\prime}(\eta_{it}^*) \bgamma_t^* (\bgamma_t^*)^\T $, we have the following:
    \begin{enumerate}[label = $(\roman*)$, leftmargin=0.6cm]
       
\item For any $j\in[q]$, for any $r\in[K]$, 
\begin{align*}
    \left|\sum_{i=1}^n \ell_{ij}^{\prime}(\eta_{ij}^*) \be_r^\T \big(\bH_{Lu_iu_i}^*\big)^{-1} \Big\{\sum_{t=1}^q \bgamma_t^* \ell_{it}^{\prime}(\eta_{it}^*) \Big\}\right|=\cOp \left( \frac{B_n}{b_n}\big(\sqrt{n /q} + n/q\big)  + \varepsilon_{nq}n/q\right);
\end{align*}
\begin{align*}
    \max_{j\in[q]}\left|\sum_{i=1}^n \ell_{ij}^{\prime}(\eta_{ij}^*) \be_r^\T \big(\bH_{Lu_iu_i}^*\big)^{-1} \Big\{\sum_{t=1}^q \bgamma_t^* \ell_{it}^{\prime}(\eta_{it}^*) \Big\}\right|=\cOp \left(\frac{B_n}{b_n}\big\{\sqrt{\log q}( \sqrt{n /q} + n/q) \big\}+ \varepsilon_{nq}n/q\right);
\end{align*}
\begin{align*}
    \sum_{j=1}^q\left|\sum_{i=1}^n \ell_{ij}^{\prime}(\eta_{ij}^*) \be_r^\T \big(\bH_{Lu_iu_i}^*\big)^{-1} \Big\{\sum_{t=1}^q \bgamma_t^* \ell_{it}^{\prime}(\eta_{it}^*) \Big\}\right|^2=\cOp \left(\frac{B_n^2}{b_n^2}\big( n + n^2/q \big)+ \varepsilon_{nq}^2n^2/q\right);
\end{align*}


\item For any $j\in[q]$, for any $r\in[K]$, 
\begin{equation*}
    \left|\sum_{i=1}^n \ell_{ij}^{\prime\prime}(\eta_{ij}^*)  (\bu_i^{\star})^\T \big(\bH_{Lu_iu_i}^*\big)^{-1} \Big\{\sum_{t=1}^q \bgamma_t^* \ell_{it}^{\prime}(\eta_{it}^*) \Big\}\right| = \cOp\left( B_n\kappa_n\sqrt{n/q}+ \varepsilon_{nq}n/q\right);
\end{equation*}
\begin{align*}
    \max_{j\in[q]}\left|\sum_{i=1}^n \ell_{ij}^{\prime\prime}(\eta_{ij}^*)  (\bu_i^{\star})^\T \big(\bH_{Lu_iu_i}^*\big)^{-1} \Big\{\sum_{t=1}^q \bgamma_t^* \ell_{it}^{\prime}(\eta_{it}^*) \Big\}\right| = \cOp\left( B_n\kappa_n\sqrt{n\log n/q}+ \varepsilon_{nq}n/q \right);
\end{align*}
\begin{equation*}
    \sum_{j=1}^q\left|\sum_{i=1}^n \ell_{ij}^{\prime\prime}(\eta_{ij}^*)  (\bu_i^{\star})^\T \big(\bH_{Lu_iu_i}^*\big)^{-1}\Big\{\sum_{t=1}^q \bgamma_t^* \ell_{it}^{\prime}(\eta_{it}^*) \Big\}\right|^2 = \cOp\big(B_n^2\kappa_n^2n+ \varepsilon_{nq}^2n^2/q\big).
\end{equation*}

\end{enumerate}
\end{lemma}
\begin{proof}
See Section~\ref{supp_sec_prove_lemma_aux_score_bound_pre}.
\end{proof}

For $\alpha_1$, with $\eta_{ij}^* = {\bgamma_j^*}^\T\bu_i^* + \beta_j^*$ and the expansion in Lemma~\ref{lemma_bu_star_theta_true}, one has
\begin{align*}
    \alpha_1 = \left\|\sum_{i=1}^n\ell_{ij}^{\prime}(\eta_{ij}^*)\big\{\sum_{t=1}^q\ell_{it}^{\prime\prime}(\eta_{it}^*)\bgamma_t^*{\bgamma_t^*}^\T\big\}^{-1}\sum_{t=1}^q\ell_{it}^{\prime}(\eta_{it}^*){\bgamma_t^*}\right\| + \cOp(\kappa_n^3B_nn/q)
    \end{align*}
Here, $\cOp(\kappa_n^3B_nn/q)$ follows from $\|\sum_{i=1}^n\ell_{ij}'(\eta_{ij}^*)\cR_i^{\star}\|\le \big(\sum_{i=1}^n\ell_{ij}'(\eta_{ij}^*)^2\big)^{1/2}\big(\sum_{i=1}^n\|\cR_i^{\star}\|^2\big)^{1/2}= \cOp(\sqrt{nB_n}\times \kappa_n^2B_n\sqrt{n}/(qb_n))$ by the Cauchy--Schwarz inequality, part~$(ii)$ of Lemma~\ref{lemma_concentration_l_mat}, and Lemma~\ref{lemma_bu_star_theta_true}.
By $(i)$ of Lemma~\ref{lemma_aux_score_bound_pre}, 
\begin{equation*}\alpha_1 = \cOp\Big(\frac{B_n}{b_n}\sqrt{n/q} + \kappa_n^3B_nn/q + \varepsilon_{nq}n/q\Big).\end{equation*}
Similarly, with $(i)$ of Lemma~\ref{lemma_aux_score_bound_pre} and the expansion in Lemma~\ref{lemma_bu_star_theta_true}, one can obtain $\max_{j\in[q]}\alpha_1 = \cOp\big(\sqrt{\log q}(B_n/b_n\sqrt{n/q} + \kappa_n^3B_nn/q)+ \varepsilon_{nq}n/q\big)$ and $\sum_{j=1}^q\alpha_1^2 = \cOp(B_n^2n/b_n^2+\kappa_n^6B_n^2n^2/q + \varepsilon_{nq}^2n^2/q)$.

For $\alpha_2$, apply Taylor expansion on $\sum_{i=1}^n\ell_{ij}^{\prime}\big({\bgamma_j^*}^\T\bu_i^{\star} + \beta_j^*)$ and we have
\begin{align*}
    \sum_{i=1}^n\ell_{ij}^{\prime}\big({\bgamma_j^*}^\T\bu_i^{\star} + \beta_j^*) = &\,\sum_{i=1}^n\ell_{ij}^{\prime}\big({\bgamma_j^*}^\T\bu_i^{*} + \beta_j^*) + \sum_{i=1}^n\ell_{ij}^{\prime\prime}\big({\bgamma_j^*}^\T\bu_i^{*} + \beta_j^*){\bgamma_j^*}^\T(\bu_i^{\star} - \bu_i^*)\\&\,+\frac{1}{2}\sum_{i=1}^n\ell_{ij}^{(3)}\big({\bgamma_j^*}^\T\{\bu_i^{*} + s_i(\bu_i^{\star} - \bu_i^*)\} + \beta_j^*\big)\left\{{\bgamma_j^*}^\T(\bu_i^{\star} - \bu_i^*)\right\}^2,
\end{align*}
for some $s_i\in[0,1]$. By $|\ell_{ij}^{(3)}(x)| \le B_n$ when $|x|\lesssim\sqrt{\log n}$, $\|\bgamma_j^*\| = \cO(1)$, and $\sum_{i=1}^n\|\bu_i^{\star} - \bu_i^*\|^2 = \cOp(n\kappa_n^2/q)$ from Lemma~\ref{lemma_bu_star_theta_true}, one has 
\begin{equation*}
    \alpha_2 \le \left\|\sum_{i=1}^n\ell_{ij}^{\prime\prime}\big(\eta_{ij}^*){\bgamma_j^*}^\T(\bu_i^{\star} - \bu_i^*)\right\| + \cOp(\kappa_n^2B_nnq^{-1}).
\end{equation*}
Apply the expansion in Lemma~\ref{lemma_bu_star_theta_true} again to obtain 
\begin{align*}
\alpha_2 \le &\,\left\|\sum_{i=1}^n\ell_{ij}^{\prime\prime}\big(\eta_{ij}^*){\bgamma_j^*}^\T\big\{\sum_{t=1}^q\ell_{it}^{\prime\prime}(\eta_{it}^*)\bgamma_t^*{\bgamma_t^*}^\T\big\}^{-1}\sum_{t=1}^q\ell_{it}^{\prime}(\eta_{it}^*){\bgamma_t^*}\right\| + \cOp\big(\kappa_n^3B_n^{3/2}nq^{-1}\big) \\
=&\, \cOp\big(B_n\kappa_n\sqrt{n/q} + \kappa_n^3B_n^{3/2}nq^{-1} + \varepsilon_{nq}nq^{-1}\big).
\end{align*}
The last inequality follows from part~$(ii)$ of Lemma~\ref{lemma_aux_score_bound_pre}. Similarly,
$\max_{j\in[q]}\alpha_2 = \cOp\big(\sqrt{\log q}\{B_n\kappa_n\sqrt{n/q} + \kappa_n^3B_n^{3/2}nq^{-1}\}+ \varepsilon_{nq}n/q\big)$ and
$\sum_{j=1}^q\alpha_2^2 = \cOp\big(B_n^2\kappa_n^2n + \kappa_n^6B_n^3n^2q^{-1}+ \varepsilon_{nq}^2n^2/q\big)$.

For $\alpha_3$, note that 
\begin{equation}\label{eq_bound_lalpha_3}\begin{aligned}
    &\,\Big\|\sum_{i=1}^n\big\{\ell_{ij}^{\prime}\big({\bgamma_j^*}^\T\bu_i^\star + \beta_j^*\big)-\ell_{ij}^{\prime}\big({\bgamma_j^*}^\T\bu_i^* + \beta_j^*\big)\big\}(\bu_i^\star - \bu_i^*)\Big\|  \\\le &\,\left\|B_n\sum_{i=1}^n(\bu_i^\star - \bu_i^*)(\bu_i^\star - \bu_i^*)^\T\bgamma_j^*\right\|\le \cOp(\kappa_n^2B_nnq^{-1}).
\end{aligned}\end{equation}
The last inequality follows from $|\ell_{ij}^{\prime\prime}(x)|\le B_n$ when $x=(\bgamma_j^*)^\T\bu_{i}^{*} + s(\bgamma_j^*)^\T(\bu_i^{\star}- \bu_i^*) + \beta_j^*$ for $s\in[0,1]$ by Assumption~\ref{assumption_smoothness} and  $(\sum_{i=1}^n\|\bu_i^{\star} - \bu_i^*\|^2)^{1/2} = \cOp(\kappa_n\sqrt{n/q})$ by Lemma~\ref{lemma_bu_star_theta_true}. Subsequently,
\begin{align*}
    \alpha_3 \le &\,\Big\|\sum_{i=1}^n\big\{\ell_{ij}^{\prime}\big({\bgamma_j^*}^\T\bu_i^\star + \beta_j^*\big)-\ell_{ij}^{\prime}\big({\bgamma_j^*}^\T\bu_i^* + \beta_j^*\big)\big\}\bu_i^*\Big\| + \cOp(\kappa_n^2B_nnq^{-1})\\\le &\,\cOp\big(B_n\kappa_n\sqrt{n/q} + \kappa_n^3B_n^{3/2}nq^{-1} + \varepsilon_{nq}nq^{-1}\big).
\end{align*}
The last inequality follows by the same argument used for $\alpha_2$ and from $\big(\sum_{i=1}^n\|\bu_i^*\|^2\big)^{1/2} = \cOp(\sqrt{n})$ under Assumption~\ref{assump_standard_id}. Because \eqref{eq_bound_lalpha_3} holds uniformly over $j\in[q]$, the same uniform and squared-sum bounds obtained for $\alpha_2$ also hold for $\alpha_3$.

To conclude, we have
\begin{equation*}
    \big\|\bS_{L\theta_j}\big(\btheta^*,\bu^\star\big) - \bS_{L\theta_j}\big(\btheta^*,\bu^*\big)\big\|= \cOp\big(B_n\kappa_n\sqrt{n/q} + \kappa_n^3B_n^{3/2}nq^{-1} + \varepsilon_{nq}nq^{-1}\big),
\end{equation*}
\begin{equation*}
    \max_{j\in[q]}\big\|\bS_{L\theta_j}\big(\btheta^*,\bu^\star\big) - \bS_{L\theta_j}\big(\btheta^*,\bu^*\big)\big\|= \cOp\Big(\sqrt{\log q}\big\{B_n\kappa_n\sqrt{n/q} + \kappa_n^3B_n^{3/2}nq^{-1}\big\}+ \varepsilon_{nq}nq^{-1}\Big),
\end{equation*}
and 
\begin{equation*}
    \big\|\bS_{L\theta}\big(\btheta^*,\bu^\star\big) - \bS_{L\theta}\big(\btheta^*,\bu^*\big)\big\|= \cOp\big(B_n\kappa_n\sqrt{n} + \kappa_n^3B_n^{3/2}n/\sqrt{q} + \varepsilon_{nq}n/\sqrt{q}\big).
\end{equation*}
Finally, we decompose $\cS_{\theta}(\btheta^*)$ into
\begin{equation*}
    \cS_{\theta}(\btheta^*) = \big\{\cS_{\theta}(\btheta^*) - \bS_{L\theta}(\btheta^*,\bu^{\star})\big\} + \big\{\bS_{L\theta}(\btheta^*,\bu^{\star}) - \bS_{L\theta}(\btheta^*,\bu^*)\big\} + \big\{\bS_{L\theta}(\btheta^*,\bu^*)\big\}
\end{equation*}
Lemma~\ref{lemma_first_order_base} bounds the first term, and Lemma~\ref{coro_first_order_concern} bounds the third, proving the result under $(\bmu,\bSigma) = (\zero_K,\bI_K)$. For the uniform bound over $(\bmu,\bSigma)$, note that the above decomposition involves $(\bmu,\bSigma)$ only in $\cS_{\theta}(\btheta^*) - \bS_{L\theta}(\btheta^*,\bu^{\star})$. Lemma~\ref{lemma_first_order_base} bounds this term uniformly over $(\bmu,\bSigma) \in\bar\cK_\rho$ for any fixed $\rho<1$, so the uniform result follows.


\subsubsection{Proof of Lemma~\ref{lemma_laplace_approx_second_order}}\label{supp_sec_prove_lemma_laplace_approx_second_order}

\begin{proof}
Recall that $\pwk(\cdot\mid{\bmu,\bSigma})$ is the density function for multivariate Gaussian distribution with mean $\bmu$ and covariance $\bSigma$. For notational simplicity, we suppress the arguments $(\bmu,\bSigma)$ in $\cL(\btheta,\bmu,\bSigma)$ and write $\pwk(\bu_i)=\pwk(\bu_i|\bmu,\bSigma)$ for this proof.
Recall that in Section~\ref{supp_sec_prove_lemma_first_order_base}, we defined $\Pi_i$ as the posterior and showed that $\EE_0W_\delta(\bZ)=1+\cO(\delta^2\tilde\varepsilon_{nq})$. In what follows, we use an argument  similar to that in Section~\ref{supp_sec_prove_lemma_first_order_base}, partitioning $\RR^K$ into three regions. 
Write $\bz_i(\bu) = (1,\bu^\T)^\T$ and $\bg_{ij}(\bu;\btheta)=\ell_{ij}'(\bgamma_j^\T\bu+\beta_j)\bz_i(\bu)$.
Differentiating $\cS_{\theta_j}=-\sum_{i=1}^n\EE_{\Pi_i}\bg_{ij}(\bU;\btheta)$ under the integral sign gives
\begin{equation}
\begin{aligned}
    \partial_{\btheta_{j_1}\btheta_{j_2}}^2\cL(\btheta,\bmu,\bSigma)
    =&\,\partial_{\btheta){j_2}}\Big\{-\sum_{i=1}^n\frac{\int_{\bu_i\in\RR^K}\bu_i\exp\{\ell_{i}(\bu_i|\btheta)\} \ell_{ij}^{\prime}(\eta_{ij}) \pwk(\bu_i)d\bu_i}{\int_{\bu_i\in\RR^K}\exp\{\ell_{i}(\bu_i)\}\pwk(\bu_i|\btheta)d\bu_i}\Big\}^\T\\=&\,\sum_{i=1}^n\Big[
        1_{\{j_1=j_2\}}\,
        \EE_{\Pi_i}\!\left\{-\ell_{ij_1}''(\eta_{ij_1})\bz_i(\bu_i)\bz_i(\bu_i)^\T\right\}\\
        &\qquad-\operatorname{Cov}_{\Pi_i}\!\left\{\bg_{ij_1}(\bu_i;\btheta),\bg_{ij_2}(\bu_i;\btheta)\right\}
    \Big].
\end{aligned}
\label{eq_marginal_hessian_covariance}
\end{equation}Let $\bz_i^\star=\bz_i\{\bu_i^\star(\btheta)\}$ and $\eta_{ij}^\star(\btheta)=\bgamma_j^\T\bu_i^\star(\btheta)+\beta_j$,
and define the $(K+1)\times K$ Jacobian
\begin{equation*}
    \bD_{ij}^\star
    :=\left.\partial_{\bu}\bg_{ij}(\bu;\btheta)\right|_{\bu=\bu_i^\star(\btheta)}
    =\ell_{ij}''\{\eta_{ij}^\star(\btheta)\}\bz_i^\star\bgamma_j^\T
      +\begin{pmatrix}\zero_K^\T\\ \ell_{ij}'\{\eta_{ij}^\star(\btheta)\}\bI_K
      \end{pmatrix}.
\end{equation*}
Fix $i,j_1,j_2$ and write $\bu_0=\bu_i^\star(\btheta)$, $h(\bu)=-q^{-1}\ell_i(\bu;\btheta)$, $ \bH=\partial_{\bu}^2h(\bu_0) =   q^{-1}\bH_{Lu_iu_i}\{\btheta,\bu^\star(\btheta)\}$, $\bA=\bH^{-1/2},$ and $ \delta=q^{-1/2}$. With the change of variables $\bu=\bu_0+\delta\bA\bz$, and letting $\EE_0$ denote expectation under $\bZ\sim\cN(\zero,\bI_K)$, we have in the first region $\{\|\bz\|\le C_T\sqrt{\log(n\vee q)}\}$ the expansion $\pwk(\bz) = 1 + \delta p_1(\bz) + \delta^2 R_W(\bz)$ where $p_1(\bz)=\{\partial_{\bu}\log\pwk(\bu_0)\}^{\T}\bA\bz +\frac16\,\partial_{\bu}^3\{q^{-1}\ell_i\}(\bu_0) [\bA\bz,\bA\bz,\bA\bz]$ and $|R_W(\bz)|\le \tilde\varepsilon_{nq}^C(1 + \|\bz\|^C)$. For $a=1,2$, let $\bG_a(\bz)=\bg_{ij_a}(\bu_0+\delta\bA\bz;\btheta) -\bg_{ij_a}(\bu_0;\btheta)$, and we have
\begin{equation*}
    \bG_a(\bz) = \delta\big(\bD_{ij_a}^\star\bA\bz\big) + \delta^2\big\{\frac12\,\partial_{\bu}^2\bg_{ij_a}(\bu_0;\btheta)
       [\bA\bz,\bA\bz]\big\} + \delta^3\bR_a(\bz).
\end{equation*}
Write $\bL_a(\bz)=\bD_{ij_a}^\star\bA\bz$ and $ \bQ_a(\bz)=\frac12\,\partial_{\bu}^2\bg_{ij_a}(\bu_0;\btheta)
       [\bA\bz,\bA\bz]$ for short. Because $\|\bR_a(\bz)\|+|R_W(\bz)|$ is bounded by a fixed product of $\tilde\varepsilon_{nq}$, $B_n$, $b_n^{-1}$, and $\|\bz\|^c$ for some constant $c$, we have, for $W_{\delta}(\bZ)$ defined in Section~\ref{supp_sec_prove_lemma_first_order_base},
\begin{align*}
    \frac{\EE_0\{\bG_1(\bZ)\bG_2(\bZ)^\T W_{\delta}(\bZ)\}}{\EE_o[W_{\delta}(\bZ)]} = &\, \delta^2\EE_0\{\bL_1(\bZ)\bL_2(\bZ)^\T\} \\&\,+ \delta^3\EE_0\!\left[
 \bL_1(\bZ)\bQ_2(\bZ)^\T+\bQ_1(\bZ)\bL_2(\bZ)^\T
 +\bL_1(\bZ)\bL_2(\bZ)^\T p_1(\bZ)
\right] + \cOp(\delta^4\varepsilon_{nq}).
\end{align*}
The second term on the right side vanishes because all three terms inside the expectation are odd, and by the definition of $\bL_a$, we know 
\begin{equation*}
    \EE_{\Pi_i}(\bG_1\bG_2^\T) = \frac{\EE_0\{\bG_1(\bZ)\bG_2(\bZ)^\T W_{\delta}(\bZ)\}}{\EE_0W_{\delta}(\bZ)}  = \delta^2\bD_{ij_1}^\star\bH^{-1}(\bD_{ij_2}^\star)^\T + \cOp(\delta^4\varepsilon_{nq}).
\end{equation*}
With $\operatorname{Cov}_{\Pi_i}\{\bg_{ij_1}(\bU;\btheta),\bg_{ij_2}(\bU;\btheta)\} =\EE_{\Pi_i}(\bG_1\bG_2^\T)  -\EE_{\Pi_i}(\bG_1)\EE_{\Pi_i}(\bG_2)^\T$ and $\EE_{\Pi_i}(\bG_a) = \cOp(\tilde\varepsilon_{nq}^Cq^{-1})$, we arrive at
\begin{equation}
\begin{aligned}
    \operatorname{Cov}_{\Pi_i}\!\left\{\bg_{ij_1}(\bu_i;\btheta),\bg_{ij_2}(\bu_i;\btheta)\right\}=\bD_{ij_1}^\star
      \left\{\bH_{Lu_iu_i}\big(\btheta,\bu^\star(\btheta)\big)\right\}^{-1}
      (\bD_{ij_2}^\star)^\T+\cOp(q^{-2}\varepsilon_{nq}),
\end{aligned} \label{eq_laplace_hessian_covariance}
\end{equation}
uniformly in $i,j_1,j_2,\btheta,\bmu$, and $\bSigma$. Finally, the first-order Laplace approximation applied to the expectation in
\eqref{eq_marginal_hessian_covariance} yields
\begin{equation}
    \EE_{\Pi_i}\!\left\{-\ell_{ij}''(\eta_{ij})\bz_i(\bu_i)\bz_i(\bu_i)^\T\right\}
    =-\ell_{ij}''\{\eta_{ij}^\star(\btheta)\}\bz_i^\star(\bz_i^\star)^\T
      +\cOp(\varepsilon_{nq}q^{-1}),
    \label{eq_laplace_hessian_diagonal}
\end{equation}

Therefore, $\bH_{L\theta_ju_i}\{\btheta,\bu^\star(\btheta)\}=-\bD_{ij}^\star$.
Substituting \eqref{eq_laplace_hessian_diagonal} and \eqref{eq_laplace_hessian_covariance} into \eqref{eq_marginal_hessian_covariance} and canceling the identical terms therefore gives
\begin{equation}
\begin{aligned}
    \partial_{\btheta_{j_1}\btheta_{j_2}}^2\cL(\btheta,\bmu,\bSigma)
    =&\,
    \left\{\bH_{L\theta\theta}\big(\btheta,\bu^\star(\btheta)\big)\right\}_{\cK_{j_1}^+,\cK_{j_2}^+}\\
    &-\left\{\bH_{L\theta u}\big(\btheta,\bu^\star(\btheta)\big)\right\}_{\cK_{j_1}^+,\,}
      \left\{\bH_{Luu}\big(\btheta,\bu^\star(\btheta)\big)\right\}^{-1}
      \left\{\bH_{Lu\theta}\big(\btheta,\bu^\star(\btheta)\big)\right\}_{,\,\cK_{j_2}^+}\\
    &+1_{\{j_1=j_2\}}\cOp(\varepsilon_{nq}nq^{-1})+\cOp(\varepsilon_{nq}nq^{-2}).
\end{aligned}
\label{eq_approx_aa2}
\end{equation}
Since the leading terms on the right side form the $(j_1,j_2)$th $(K + 1)\times (K+1)$ block of $\bH_{-\theta}^{\star}(\btheta)$, the remaining error is the sum of a block-diagonal term of order $\cOp(\varepsilon_{nq}nq^{-1})$ per block and a dense term of order $\cOp(\varepsilon_{nq}nq^{-2})$ per block. Because the remainders are uniform over $\btheta\in\Xi$ and
$(\bmu,\bSigma)\in\bar\cK_\rho$, we conclude that
\begin{equation*}
    \big\|\cH_{\theta\theta}(\btheta) - \bH_{-\theta}^{\star}(\btheta)\big\|\le \|\bR_1\| + \|\bR_2\|_{\Fn} = \cOp(\varepsilon_{nq}nq^{-1}).
\end{equation*}
Indeed, $\bR_1$ has nonzero entries only in the block-diagonal positions, with each diagonal block bounded by $\cOp(\varepsilon_{nq}nq^{-1})$ uniformly, so $\|\bR_1\|=\cOp(\varepsilon_{nq}nq^{-1})$. The matrix $\bR_2$ has all blocks bounded by $\cOp(\varepsilon_{nq}nq^{-2})$ uniformly; hence $\|\bR_2\|\le \|\bR_2\|_{\Fn}=\cOp(\varepsilon_{nq}nq^{-1})$. The maximum-row-sum bound is also $\cOp(\varepsilon_{nq}nq^{-1})$, because each block row contains one $\cOp(\varepsilon_{nq}nq^{-1})$ diagonal block and $q-1$ blocks of order $\cOp(\varepsilon_{nq}nq^{-2})$.
    \end{proof}


\subsection{Proof of Lemmas in Section~\ref{supp_sec_main_lemma}}

\subsubsection{Proof of Lemma~\ref{lemma_concentration_l_mat}}\label{supp_sec_prove_lemma_concentration_l_mat}

By Assumption~\ref{assumption_smoothness}, conditional on $\mathscr{U}_n = \sigma(\bU_1,\ldots,\bU_n)$ and on the event $\mathcal E_\eta$, the variables $\ell'_{ij}(\eta_{ij}^*)$ are independent, have conditional mean zero, and satisfy $\max_{i\in[n],j\in[q]}\|\ell'_{ij}(\eta_{ij}^*)\|_{\psi_1\mid \mathscr U_n} \le C\sqrt{B_n}$. Thus $\EE(\{\ell'_{ij}(\eta_{ij}^*)\}^2\mid \mathscr U_n)\le C B_n$ and $\big\|\{\ell'_{ij}(\eta_{ij}^*)\}^2-\EE(\{\ell'_{ij}(\eta_{ij}^*)\}^2\mid \mathscr U_n)\big\|_{\psi_{1/2}\mid \mathscr U_n} \le C B_n$, uniformly over $(i,j)$. Since $\PP_U(\mathcal E_{\eta})\to1$, it suffices to establish the bounds under the conditional response law on $\mathcal E_{\eta}$. We suppress this conditioning below.

We first prove the row and column bounds. For each fixed $i\in[n]$, let $S_i:=\sum_{j=1}^q \{\ell_{ij}'(\eta_{ij}^*)\}^2=\|(\Lb_\eta)_{i,}\|^2$.
By the Bernstein inequality for sums of independent centered sub-Weibull\((1/2)\) variables, for all \(t>0\),
\begin{equation*}
    \PP\{S_i-\EE S_i\ge t\}\le 2\exp\Big[
        -c\min\Big\{\frac{t^2}{qB_n^2},\,\Big(\frac{t}{B_n}\Big)^{1/2}\Big\}
    \Big].
\end{equation*}
Since $\EE S_i\le CqB_n$, taking $t=C_\delta qB_n$ gives $\PP\{S_i\ge C_\delta qB_n\}\le q^{-\delta}$ after enlarging $C_\delta$. Hence, for each fixed $i\in[n]$, we have
\begin{equation*}
    \|(\Lb_\eta)_{i,}\|=\cOp(\sqrt{qB_n}).
\end{equation*}
For the uniform row bound, take $t=C_\delta B_n\{\sqrt{q\log n}+(\log n)^2\}$. For sufficiently large \(C_\delta\), $\PP\{S_i-\EE S_i\ge t\}\le 2n^{-(\delta+1)}$. By the union bound,
\begin{equation*}
\begin{aligned}
    \PP\left[
        \max_{1\le i\le n}S_i\ge C_\delta B_n\{q+(\log n)^2\}
    \right]&\le
    \sum_{i=1}^n
    \PP\left[    S_i\ge C_\delta B_n\{q+(\log n)^2\}
    \right]  \\ &\le 2n^{-\delta}.
\end{aligned}
\end{equation*}
Therefore, $\max_{1\le i\le n}\|(\Lb_\eta)_{i,}\| = \cOp\big(\sqrt{qB_n}+\sqrt{B_n}\log n\big)$. This proves part \((i)\). Part \((ii)\) follows by symmetry.

We next prove the spectral norm bound in part \((iii)\). Let $\bar Z_{ij}:=B_n^{-1/2}\ell'_{ij}(\eta_{ij}^*)$ and $\bar\Lb_\eta:=B_n^{-1/2}\Lb_\eta$.
Then $\bar Z_{ij}$ are independent, mean zero, and uniformly sub-exponential with $\max_{i,j}\|\bar Z_{ij}\|_{\psi_1}\le C$ and $\max_{i,j}\EE \bar Z_{ij}^2\le C$.  For a truncation level \(\nu>0\), define
\begin{equation*}
    \bar Z_{ij}^{(\nu)} := \operatorname{sgn}(\bar Z_{ij})\{|\bar Z_{ij}|\wedge \nu\},\qquad W_{ij} := \bar Z_{ij}^{(\nu)}-\EE\bar Z_{ij}^{(\nu)}, \qquad \Wb:=(W_{ij})_{n\times q}.
\end{equation*}
The variables $W_{ij}$ are independent, mean zero, bounded by $2\nu$, and satisfy $\max_{i,j}\EE W_{ij}^2\le C$. Consider the self-adjoint dilation
\begin{equation*}
    \bX := \begin{pmatrix}
        \zero_{q\times q} & \Wb^\T\\ \Wb & \zero_{n\times n}
    \end{pmatrix}.
\end{equation*}
Then $\|\bX\|=\|\Wb\|$. Moreover, we have
\begin{equation*}
     \tilde\sigma := \max_a \left(\sum_b\EE X_{ab}^2\right)^{1/2} \le C\sqrt{n\vee q},
\end{equation*}
and its entries are bounded by \(2\nu\). By Corollary 3.12 of \cite{bandeira2016sharp}, for all $t>0$, it holds that
\begin{equation*}
    \PP\left\{\|\Wb\| \ge C\sqrt{n\vee q}+t \right\}
    \le (n+q)\exp\left(-\frac{ct^2}{\nu^2}\right).
\end{equation*}
It remains to pass from $\Wb$ to $\bar\Lb_\eta$. { Since $\bar Z_{ij}$ are uniformly sub-exponential,} the truncation is inactive with high probability,
\begin{equation*}
    \PP\big\{\bar\Lb_\eta\neq (\bar Z_{ij}^{(\nu)})_{n\times q}\big\} \le \sum_{i=1}^n\sum_{j=1}^q \PP\{|\bar Z_{ij}|>\nu\} \le
    Cnq\exp(-c\nu).
\end{equation*}
In addition, since $\EE\bar Z_{ij}=0$, we have $\big|\EE\bar Z_{ij}^{(\nu)}\big| = \big|\EE\big[\bar Z_{ij}^{(\nu)}-\bar Z_{ij}\big]\big|\le \EE\big[|\bar Z_{ij}|1_{\{|\bar Z_{ij}|>\nu\}}\big]\le Ce^{-c\nu}$, which yields
\begin{equation*}
    \big\|\big(\EE\bar Z_{ij}^{(\nu)}\big)_{n\times q}\big\|\le C\sqrt{nq}\,e^{-c\nu}.
\end{equation*}
Combining the last three bounds gives
\begin{equation*}
    \PP\left\{ \|\bar\Lb_\eta\|\ge C\sqrt{n\vee q}+C\sqrt{nq}e^{-c\nu}+t
    \right\} \le (n+q)\exp\left(-\frac{ct^2}{\nu^2}\right) +Cnq\exp(-c\nu).
\end{equation*}
If we take $\nu=C_\delta\log(nq)$ with $C_\delta$ large enough that $\sqrt{nq}e^{-c\nu}=o(1)$ and $Cnqe^{-c\nu}=o(1)$, and $t=C_\delta\nu\sqrt{\log(n+q)}$, we know
\begin{equation*}
    \|\bar\Lb_\eta\| = \cOp\left(
        \sqrt{n\vee q} +\log(nq)\sqrt{\log(n+q)} \right).
\end{equation*}
Since $\log(nq)\sqrt{\log(n+q)}=o(\sqrt{n\vee q})$, this leads to $\|\bar\Lb_\eta\|=\cOp(\sqrt{n\vee q})$ and therefore $\|\Lb_\eta\| = \cOp\big(\sqrt{(n\vee q)B_n}\big)$. This proves part \((iii)\).
It remains to prove the uniform bound in part \((iv)\). Let $\bGamma=(\bgamma_1,\ldots,\bgamma_q)^\T$, $\bGamma^*=(\bgamma_1^*,\ldots,\bgamma_q^*)^\T$, $\bbeta=(\beta_1,\ldots,\beta_q)^\T$, $\bbeta^*=(\beta_1^*,\ldots,\beta_q^*)^\T$, \(\bU=(\bu_1,\ldots,\bu_n)^\T\), and \(\bU^*=(\bu_1^*,\ldots,\bu_n^*)^\T\). Define $\Delta_\eta(\btheta,\bu) := \{\eta_{ij}-\eta_{ij}^*\}_{n\times q}$.
Then
\begin{equation*}
    \Delta_\eta(\btheta,\bu) = \bU(\bGamma-\bGamma^*)^\T +(\bU-\bU^*)(\bGamma^*)^\T +\one_n(\bbeta-\bbeta^*)^\T.
\end{equation*}
For $(\btheta,\bu)$ satisfying \eqref{eq_feasible_theta}, we have $\|\bGamma-\bGamma^*\|_{\Fn}\le \|\btheta-\btheta^*\|\le \sqrt q\,\epsilon$ and $\|\bbeta-\bbeta^*\|\le \|\btheta-\btheta^*\|\le \sqrt q\,\epsilon$, and $\|\bU-\bU^*\|_{\Fn}\le \|\bu-\bu^*\|\le \sqrt n\,\epsilon$.
Also, note that $\|\bU\|_{\Fn}\le C\sqrt{n\log n}$ by $\|\bu\|_\infty\le C\sqrt{\log n}$, and $\|\bGamma^*\|_{\Fn}\le C\sqrt q$ by Assumption~\ref{assumption_psd_covariance}. Therefore, 
\begin{align*}
    \|\Delta_\eta(\btheta,\bu)\|_{\Fn} &\le \|\bU\|_{\Fn}\|\bGamma-\bGamma^*\|_{\Fn} +\|\bU-\bU^*\|_{\Fn}\|\bGamma^*\|_{\Fn} +\|\one_n\|\|\bbeta-\bbeta^*\|  \\ &\le
    C\sqrt{nq\log n}\,\epsilon.
\end{align*}
This bound is deterministic and uniform over all $(\btheta,\bu)$ satisfying \eqref{eq_feasible_theta}.

By the mean value theorem, for each $(i,j)$, there is some $\tilde\eta_{ij}$ between $\eta_{ij}$ and $\eta_{ij}^*$ such that $\ell_{ij}^{\prime}(\eta_{ij}) - \ell_{ij}^{\prime}(\eta_{ij}^*) = \ell_{ij}^{\prime\prime}(\tilde\eta_{ij}) (\eta_{ij}-\eta_{ij}^*)$. Since $\|\btheta\|_\infty\le C$ and $\|\bu\|_\infty\le C\sqrt{\log n}$, the whole segment between \(\eta_{ij}\) and \(\eta_{ij}^*\) lies in the interval where Assumption~\ref{assumption_smoothness} applies. Thus $|\ell_{ij}^{\prime\prime}(\tilde\eta_{ij})| \le B_n$ uniformly over $i,j$ and over all feasible $(\btheta,\bu)$. Consequently,
\[
\begin{aligned}
    \sup_{(\btheta,\bu)\text{ in }\eqref{eq_feasible_theta}}
    \|\Lb_\eta(\btheta,\bu)-\Lb_\eta\|
    &\le
    \sup_{(\btheta,\bu)\text{ in }\eqref{eq_feasible_theta}}
    \|\Lb_\eta(\btheta,\bu)-\Lb_\eta\|_{\Fn}  \\
    &\le
    B_n \sup_{(\btheta,\bu)\text{ in }\eqref{eq_feasible_theta}} \|\Delta_\eta(\btheta,\bu)\|_{\Fn}  \\ &\le C B_n\sqrt{nq\log n}\,\epsilon.
\end{aligned}
\]
Combining this with part \((iii)\) completes the proof.

\subsubsection{Proof of Lemma~\ref{lemma_convexity2}}\label{supp_sec_prove_lemma_convexity2}

As in Lemma~\ref{lemma_laplace_approx_second_order}, let $\bu_i^{\star}(\btheta) = \arg\max_{\bu\in\RR^{K}}\ell_i(\bu;\btheta)$ and $\bu^{\star}(\btheta) = \big(\bu_1^{\star}(\btheta)^\T, \dots, \bu_n^{\star}(\btheta)^\T\big)^\T$. The proof of the first lower bound consists of two steps. \begin{itemize}
\item \textbf{Step 1.} We establish the following relationship:
\begin{equation}
\begin{aligned}
   &\,\bD_m\big(\btheta,\bu^{\star}(\btheta)\big)^\T\bS_m\left\{\begin{pmatrix}
       \bH_{L\theta\theta}&\bH_{L\theta u}\\\bH_{Lu\theta}&\bH_{Luu}
   \end{pmatrix}\circ\big(\btheta,\bu^{\star}(\btheta)\big)\right\}\bS_m \bD_m\big(\btheta,\bu^{\star}(\btheta)\big)\\ = &\,\begin{pmatrix}
       n^{-1}{\bH}_{-\theta}^{\star}(\btheta)&\zero\\\zero&q^{-1}\bH_{Luu}\big(\btheta,\bu^{\star}(\btheta)\big)
   \end{pmatrix},
   \end{aligned}\label{eq_obtain_schur}
\end{equation}
for any $\btheta$, and 
\begin{equation}
    \big\{\bD_m(\btheta,\bu)\big\}^\T \times \bar\bv_{rl}(\btheta) = \begin{pmatrix}
            q^{-1/2}\bv_{rl}(\btheta)\\\zero_{nK}
        \end{pmatrix}\text{ for any }\btheta,\bu,\label{eq_relate_eigen_vectors}
\end{equation}
where $\bar\bv_{rl}(\btheta)$, $\bS_m$, and $\bD_m(\btheta,\bu)$ are defined below in \eqref{eq_define_v_bar_pq}, \eqref{eq_define_sm}, and \eqref{eq_define_bd_trans}, respectively. Here $\circ \big(\btheta,\bu^{\star}(\btheta)\big)$ indicates that the preceding matrix is evaluated at $\btheta$ and $\bu^{\star}(\btheta)$. This establishes $n^{-1}\bH_{-\theta}^{\star}(\btheta)$ as the Schur complement of an augmented matrix, which we study in the next step.
\item \textbf{Step 2.} We derive that for $\lambda^{\prime} \asymp b_n$ and $c^{\prime}\asymp b_n$, the lower  bound 
\begin{equation}
    \lambda_{\min}\left\{\bS_m\begin{pmatrix}
       \bH_{L\theta\theta}&\bH_{L\theta u}\\\bH_{Lu\theta}&\bH_{Luu}
   \end{pmatrix}\circ\big(\btheta,\bu^{\star}(\btheta)\big)\bS_m + c^{\prime}\sum_{r=0}^K\sum_{l=1}^K \bar\bv_{rl}(\btheta)\big\{\bar\bv_{rl}(\btheta)\big\}^\T\right\}\ge \lambda^{\prime}\label{eq_standard_convex_aug}\end{equation}
holds with probability approaching $1$ when $n,q\to\infty$. 
\end{itemize}
\bigskip

\noindent\underline{\large \it Executing Step 1.}    Define 
    \begin{equation}
        \bar\bv_{rl}(\btheta) = \left(q^{-1/2}\{\bv_{rl}(\btheta)\}^\T, \zero_{nK}\right)^\T\in\RR^{qK+q+nK}.\label{eq_define_v_bar_pq}
    \end{equation}
    for $r\in\{0\}\cup [K]$ and $l\in[K]$, \begin{equation}
        \bS_m = \begin{pmatrix}
            n^{-1/2}\bI_{qK+q}&\zero\\\zero&q^{-1/2}\bI_{nK}
        \end{pmatrix}\label{eq_define_sm}
    \end{equation}
    and
    \begin{equation}
        \bD_m(\btheta,\bu) = \begin{pmatrix}
            \bI_{qK + q} &\zero \\  -\sqrt{q/n}\{\bH_{Luu}(\btheta,\bu)\}^{-1}\bH_{Lu\theta}(\btheta,\bu) & \bI_{nK}
        \end{pmatrix}.\label{eq_define_bd_trans}
    \end{equation}
    Simple algebra yields \eqref{eq_obtain_schur} and \eqref{eq_relate_eigen_vectors}.
    

Before we start Step 2, we show that $\bD_m(\btheta,\bu)$ is invertible.
When $\btheta \in\cB_{\epsilon}(\btheta^*)$, invoke Lemma~\ref{lemma_neighbour_control} to obtain
\begin{equation}
    \|\bu^{\star}\|_{\infty} \le C\sqrt{\log n}\quad\text{ and }\quad n^{-1/2}\big\|\bu^{\star} - \bu^*\big\| = \cOp(\epsilon\kappa_n\sqrt{B_n}\vee \kappa_n/\sqrt{q}).\label{eq_control_btheta_in_prove_con}
\end{equation}
For $\bH_{Luu}(\btheta,\bu^{\star}(\btheta))$, its $i$-th diagonal block can be bounded as
\begin{equation}
    \begin{aligned}
        &\,\lambda_{\min}\Big\{\bH_{Lu_iu_i}(\btheta,\bu^{\star})\Big\}= \lambda_{\min}\Big\{\sum_{j=1}^q-\ell_{ij}^{\prime\prime}\big(\eta_{ij}^{\star}(\btheta)\big)\bgamma_j\bgamma_j^\T\Big\}\\\overset{(i)}{\ge} &\,\lambda_{\min}\Big\{\sum_{j=1}^q-\ell_{ij}^{\prime\prime}\big(\eta_{ij}^{\star}(\btheta)\big)\bgamma_j^*{\bgamma_j^*}^\T\Big\} - \Big\|\sum_{j=1}^q\ell_{ij}^{\prime\prime}\big(\eta_{ij}^{\star}(\btheta)\big)\bgamma_j^*{\bgamma_j^*}^\T - \sum_{j=1}^q\ell_{ij}^{\prime\prime}\big(\eta_{ij}^{\star}(\btheta)\big)\bgamma_j\bgamma_j^\T\Big\|\\\overset{(ii)}{\ge}&\,b_n\lambda_{\min}\Big\{\sum_{j=1}^q\bgamma_j^*{\bgamma_j^*}^\T\Big\}- 2\Big\|\sum_{j=1}^q\ell_{ij}^{\prime\prime}\big(\eta_{ij}^{\star}(\btheta)\big)\bgamma_j^*(\bgamma_j^* - \bgamma_j)^\T \Big\|  \\&\,- \Big\|\sum_{j=1}^q\ell_{ij}^{\prime\prime}\big(\eta_{ij}^{\star}(\btheta)\big)(\bgamma_j^* - \bgamma_j)(\bgamma_j^* - \bgamma_j)^\T \Big\|\\\overset{(iii)}{\ge}&\,\Theta\{q(b_n-B_n\epsilon)\}=\Theta(qb_n).
    \end{aligned}\label{eq_reference_eigen_lowerbound}
\end{equation}
Here $\eta_{ij}^{\star}(\btheta)=\bgamma_j^\T\bu_i^{\star}(\btheta)+\beta_j$ for each $i\in[n]$ and $j\in[q]$. Step~$(i)$ follows from Weyl's inequality. Step~$(ii)$ uses $\|\btheta\|_{\infty}\le C$, $\|\bu^{\star}(\btheta)\|_{\infty}\le C\sqrt{\log n}$ from \eqref{eq_control_btheta_in_prove_con}, the bound $-\ell_{ij}^{\prime\prime}(\eta_{ij}^{\star}(\btheta))\ge b_n$, and a telescoping expansion. Step~$(iii)$ follows because $\lim_{q\to\infty}q^{-1}\sum_{j=1}^q\bgamma_j^*{\bgamma_j^*}^\T=\bSigma_\gamma^*$ is positive definite, $|\ell_{ij}^{\prime\prime}(\eta_{ij}^{\star}(\btheta))|\le B_n$, $\|\btheta-\btheta^*\|\le\sqrt q\,\epsilon$, and $\epsilon\ll b_n^2/B_n^2$ by \eqref{eq_key_scaling_equation_epsilon}, together with the Cauchy--Schwarz inequality.

For $\bH_{L\theta u}(\btheta,\bu^{\star}(\btheta)) = \bar\bH_{L\theta u}(\btheta,\bu^{\star}(\btheta)) + \bJ_{L\theta u}(\btheta,\bu^{\star}(\btheta))$, for each $i\in[n]$ and $j\in[q]$,
\begin{equation*}
    \left\{\bar\bH_{L\theta u}(\btheta,\bu^{\star}(\btheta))\right\}_{\cK_j^+, \cK_i} = -\ell_{ij}^{\prime\prime}(\eta_{ij}^{\star}(\btheta))\begingroup \renewcommand{\arraystretch}{0.7}\begin{pmatrix}
        1\\\bu_i^{\star}(\btheta)
    \end{pmatrix}\endgroup\bgamma_j^\T.
\end{equation*}
Since $|\ell_{ij}^{\prime\prime}(\eta_{ij}^{\star}(\btheta))| \le B_n$, $\|\bu_i^{\star}(\btheta)\|=\cOp(1)$ and $\|\bgamma_j\|\le C$, we have $\big\|\{\bar\bH_{L\theta u}(\btheta,\bu^{\star}(\btheta))\}_{\cK_j^+, \cK_i}\big\| = \cOp(B_n)$.
Since $\{\bU_i\}_{i=1}^n$ are independent, one can then obtain
\begin{equation*}
    \left\|\bar\bH_{L\theta u}(\btheta,\bu^{\star}(\btheta))\right\| = \cOp(\sqrt{nqB_n}).
\end{equation*}
For $\bJ_{L\theta u}(\btheta,\bu^{\star}(\btheta))$, by (iv) of Lemma~\ref{lemma_concentration_l_mat}, one has
\begin{equation*}
    \left\|\bJ_{L\theta u}(\btheta,\bu^{\star})\right\| = \cOp\left(\sqrt{(n\vee q)B_n} + B_n\sqrt{nq}\epsilon\right).
\end{equation*}
We then conclude that
\begin{align*}
    \left\|\bH_{L\theta u}(\btheta,\bu^{\star})\right\| = \cOp\left(\sqrt{nq}B_n \right) .
\end{align*}
  Subsequently, it is not hard to check that there exists some absolute constants $c,C>0$ such that
    \begin{equation}
       c\frac{b_n}{\sqrt{B_n} + B_n\epsilon}\le \sigma_{\min}\big\{\bD_m(\btheta,\bu^{\star})\big\}\le \sigma_{\max}\big\{\bD_m(\btheta,\bu^{\star})\big\}\le C\frac{\sqrt{B_n} + B_n\epsilon}{b_n} \label{eq_benign_bD_scale}
    \end{equation}holds with probability approaching $1$. 

\medskip
\noindent\underline{\large \it Executing Step 2.} For $r\in\{0\}\cup [K]$ and $l\in[K]$, let
    \begin{equation}
        \tilde\bv_{rl}(\btheta,\bu) = \left(q^{-1/2}\{\bv_{rl}(\btheta)\}^\T, -n^{-1/2}u_{1r}\be_l^\T,\dots,-n^{-1/2}u_{nr}\be_l^\T\right)^\T\in\RR^{qK+q+nK}.\label{eq_define_u_pq}
    \end{equation}
For brevity, we omit explicit dependence on $\btheta$ or $\big(\btheta,\bu^{\star}(\btheta)\big)$ in the following when it is clear from the context.
Assemble vectors $\tilde \bv_{rl}$ and $\bar \bv_{rl}$ into matrices as follows:
\begin{align*}
    &\tilde \bV = \left(\tilde\bv_{01},\dots,\tilde\bv_{0K},\tilde\bv_{11},\dots,\tilde\bv_{KK}\right)\in\RR^{(qK+q+nK)\times (K^2+K)};\\
    &\bar \bV = \left(\bar\bv_{01},\dots,\bar\bv_{0K},\bar\bv_{11},\dots,\bar\bv_{KK}\right)\in\RR^{(qK+q+nK)\times (K^2+K)}.
\end{align*}
Define $\tilde\bV_\theta = \tilde\bV_{1:qK+q,}$ and $\tilde\bV_u = \tilde\bV_{(qK+q+1):(qK+q+nK),}$.
By simple algebra, 
\begin{align*}
    &\,\bS_m\begin{pmatrix}
       \bH_{L\theta\theta}&\bar\bH_{L\theta u}\\\bar\bH_{Lu\theta}&\bH_{Luu}
   \end{pmatrix}\bS_m + \tilde c^{\prime}\sum_{r=0}^K\sum_{l=1}^K \tilde\bv_{rl}\tilde\bv_{rl}^\T \\=&\, \sum_{i=1}^n\sum_{j=1}^q(-\ell_{ij}^{\prime\prime}(\eta_{ij}^{\star}) - \tilde c^{\prime})\bS_m\begin{pmatrix}
       \be_j^{(q)}\otimes \begingroup \renewcommand{\arraystretch}{0.7}\begin{pmatrix}
            1\\\bu_i^{\star}
        \end{pmatrix}\endgroup\\\be_i^{(n)}\otimes \bgamma_j
   \end{pmatrix}\begin{pmatrix}
       \be_j^{(q)}\otimes \begingroup \renewcommand{\arraystretch}{0.7}\begin{pmatrix}
            1\\\bu_i^{\star}
        \end{pmatrix}\endgroup\\\be_i^{(n)}\otimes \bgamma_j
	   \end{pmatrix}^\T\bS_m\quad(:=\alpha_1)\\
    &\,+\tilde c^{\prime}\begin{pmatrix}
	    \bI_q\otimes n^{-1}\sum_{i=1}^n\bz_i^\star(\bz_i^\star)^\T&\zero\\
        \zero&\bI_n\otimes q^{-1}\sum_{j=1}^q\bgamma_j\bgamma_j^\T
	    \end{pmatrix}\quad(:=\alpha_2)\\
    &\,+\tilde c^{\prime}\begin{pmatrix}
	        \tilde\bV_{\theta}(\tilde\bV_{\theta})^\T&\zero\\
            \zero&\tilde\bV_{u}(\tilde\bV_{u})^\T
	    \end{pmatrix}\quad(:=\alpha_3),
\end{align*}
where $\bz_i^\star=(1,(\bu_i^\star)^\T)^\T$. Recall that $\bar\bH_{L\theta u} = \bH_{L\theta u}(\btheta,\bu) - \bJ_{L\theta u}(\btheta,\bu)$ is the curvature-only part of $\bH_{L\theta u}$.
By construction, $\alpha_1$ and $\alpha_3$ are positive semidefinite when $0<\tilde c^{\prime}<b_n$, because $-\ell_{ij}^{\prime\prime}\{\eta_{ij}^{\star}(\btheta)\}\ge b_n$ for $\btheta\in\cB_{\epsilon}(\btheta^*)$ under Assumption~\ref{assumption_smoothness}. For $\alpha_2$, the law of large numbers, \eqref{eq_control_btheta_in_prove_con}, and a telescoping expansion give
\begin{equation}
    \lambda_{\min}\left\{n^{-1}\sum_{i=1}^n\bz_i^\star(\bz_i^\star)^\T\right\}
    \ge \frac{1}{2}\lambda_{\min}\left\{\begin{pmatrix}1&\zero_K^\T\\\zero_K&\bSigma_u^*\end{pmatrix}\right\}
    =:c_u>0
    \label{eq_define_bar_theta_mtheta_star2}
\end{equation}
with probability tending to one. The same loading-matrix perturbation used in \eqref{eq_reference_eigen_lowerbound} gives
\[
    \lambda_{\min}\left\{q^{-1}\sum_{j=1}^q\bgamma_j\bgamma_j^\T\right\}
    \ge \frac{1}{2}\lambda_{\min}(\bSigma_\gamma^*)=:c_\gamma>0.
\]
Consequently, taking $\tilde c^{\prime}=2b_n/3$,
\begin{align*}
    \lambda_{\min}\left\{\bS_m\begin{pmatrix}
       \bH_{L\theta\theta}&\bar\bH_{L\theta u}\\\bar\bH_{Lu\theta}&\bH_{Luu}
   \end{pmatrix}\bS_m + \tilde c^{\prime}\sum_{r=0}^K\sum_{l=1}^K \tilde\bv_{rl}\tilde\bv_{rl}^\T\right\}&\ge \lambda_{\min}(\alpha_2) \\&\ge b_n\min\left\{\lambda_{\min}(\bSigma_u^*), \lambda_{\min}(\bSigma_\gamma^*)\right\}/3\\&=\tilde\lambda^{\prime}.
\end{align*}
with probability tending to one. Next, we show that a similar lower bound exists when replacing $\tilde \bv_{rl}$ with $\bar \bv_{rl}$ and $\bar\bH_{L\theta u}$ with $\bH_{L\theta u}$. First, the preceding result implies
\begin{equation*}
    \bx ^\T\bS_m\begin{pmatrix}
       \bH_{L\theta\theta}&\bar\bH_{L\theta u}\\\bar\bH_{Lu\theta}&\bH_{Luu}
   \end{pmatrix}\bS_m \bx\ge \tilde\lambda^{\prime}\|\bx\|^2,
\end{equation*}
for any $\bx \perp \mathrm{col}(\tilde\bV)$. For general $\bx\in\RR^{qK+q+nK}$, decompose it as $\bx = \by + \bz$, where $\by\in\mathrm{col}(\tilde\bV)$ and $\bz \perp \mathrm{col}(\tilde\bV)$. This decomposition is unique, with $\by = \tilde\bV(\tilde\bV^\T\tilde\bV)^{-1}\tilde\bV^\T \bx$ and $\bz = \bx - \by$. By definition,
\begin{equation*}
    \bS_m\begin{pmatrix}
       \bH_{L\theta\theta}&\bar\bH_{L\theta u}\\\bar\bH_{Lu\theta}&\bH_{Luu}
   \end{pmatrix}\bS_m\by = 0,
\end{equation*}
because $\by\in\mathrm{col}(\tilde\bV)$.
Then one has
\begin{equation}\begin{aligned}
     \bx ^\T\left\{\bS_m\begin{pmatrix}
       \bH_{L\theta\theta}&\bH_{L\theta u}\\\bH_{Lu\theta}&\bH_{Luu}
   \end{pmatrix}\bS_m+c^{\prime}\bar\bV\bar\bV^\T \right\}\bx\ge  &\, \tilde\lambda^{\prime}\|\bz\|^2 + c^{\prime}\bx^\T \bar\bV\bar\bV^\T\bx\\&\, - 2(nq)^{-1/2}\|\bH_{L\theta u} - \bar\bH_{L\theta u}\|\|\bx\|^2.\end{aligned}\label{eq_final_push_to_convexity}
\end{equation}
Choose $c^{\prime}\le\tilde c^{\prime}$ below. Write $\bar\bV = \bar\bV_y + \bar\bV_z$, where
\[
    \bar\bV_y = \tilde\bV(\tilde\bV^\T\tilde\bV)^{-1}\tilde\bV^\T\bar\bV,\qquad
    \bar\bV_z = \left\{\bI_{qK+q+nK}-\tilde\bV(\tilde\bV^\T\tilde\bV)^{-1}\tilde\bV^\T\right\}\bar\bV.
\]
Then
\begin{equation}
    \bx^\T \bar\bV\bar\bV^\T\bx \ge \underbrace{\tfrac{1}{2}\big\|\bar\bV_y^\T\bx\big\|^2}_{\beta_1} - \underbrace{\big\|\bar\bV_z^\T\bx\big\|^2}_{\beta_2}.\label{eq_final_push_to_convexity2}
\end{equation}
To control $\beta_1$ and $\beta_2$, first define
\begin{equation*}
    \zeta_1 = \sigma_{K^2 + K}\big(\tilde\bV\big)\quad\text{ and }\quad\zeta_2 := \sigma_{\min}\left\{(\tilde\bV^\T\tilde\bV)^{-1}\tilde\bV^\T\bar\bV\right\}.
\end{equation*}
Define $\bar\gamma_{lh} = q^{-1}\sum_{j=1}^q\gamma_{jl}\gamma_{jh}$ and $\bar u_{lh}^{\star} = n^{-1}\sum_{i=1}^nu_{il}^{\star}u_{ih}^{\star}$ for $l,h\in[K]$, $\bar\bGamma = q^{-1}\sum_{j=1}^q\bgamma_{j}\bgamma_{j}^\T$, and $\bar \bU^{\star} = n^{-1}\sum_{i=1}^n\bu_i^{\star}{\bu_i^{\star}}^\T$. We also set $\bar u_{l0}^{\star} = \bar u_{0l}^{\star} = n^{-1}\sum_{i=1}^nu_{il}^{\star}$ and $\bar u_{00}^{\star} = 1$. Then
\begin{equation*}
    \tilde\bV^\T\tilde\bV = \bI_{K+1}\otimes \bar\bGamma + \begin{pmatrix}
        \bar u_{00}^{\star}\bI_K & \bar u_{01}^{\star}\bI_K&\cdots &\bar u_{0K}^{\star}\bI_K\\\bar u_{10}^{\star}\bI_K & \bar u_{11}^{\star}\bI_K&\cdots &\bar u_{1K}^{\star}\bI_K\\\ddots & \ddots &\ddots & \ddots\\\bar u_{K0}^{\star}\bI_K & \bar u_{K1}^{\star}\bI_K&\cdots &\bar u_{KK}^{\star}\bI_K
    \end{pmatrix}
\end{equation*}
Under Assumptions~\ref{assump_standard_id} and~\ref{assumption_psd_covariance}, the arguments used for \eqref{eq_reference_eigen_lowerbound} and \eqref{eq_define_bar_theta_mtheta_star2} show that, for some absolute constant $c_1>0$, $\lambda_{\min}(\tilde\bV^\T\tilde\bV)\ge c_1^2$ with probability tending to one. Thus $\PP(\zeta_1\ge c_1)\to1$. For $\zeta_2$, note that
\begin{equation*}
    \tilde\bV^\T\bar\bV = \bI_{K+1}\otimes \bar\bGamma.
\end{equation*}
Similarly, there exists a constant $c_2>0$ such that $\sigma_{\min}(\tilde\bV^\T\bar\bV)\ge c_2$ with probability tending to one. The preceding display and \eqref{eq_control_btheta_in_prove_con} also give $\lambda_{\max}(\tilde\bV^\T\tilde\bV)=\cOp(1)$. After decreasing $c_2$ if necessary, we therefore have $\zeta_2\ge c_2$ with probability tending to one.

    For $\beta_1$, since $\bz\perp\mathrm{col}(\tilde\bV)$ and $\bar\bV_y$ is the projection of $\bar\bV$ onto $\mathrm{col}(\tilde\bV)$,
    \begin{equation}
        \beta_1 = \frac{1}{2}\big\{\by^\T\tilde\bV\big\}\big\{(\tilde\bV^\T\tilde\bV)^{-1}\tilde\bV^\T\bar \bV\big\}\big\{\bar\bV^\T\tilde\bV(\tilde\bV^\T\tilde\bV)^{-1}\big\}\big\{\tilde\bV^\T\by\big\}\ge \frac{c_1^2c_2^2}{2}\|\by\|^2.\label{eq_final_push_to_convexity3}
    \end{equation}
    For $\beta_2$, since $\by\in\mathrm{col}(\tilde\bV)$ and $\bar\bV_z$ is the projection of $\bar\bV$ onto $\mathrm{col}(\tilde\bV)^{\perp}$, we have
    \begin{equation}
        \beta_2 = \|\bar\bV_z^\T\bz\|^2 \le \|\bar\bV_z\|^2 \|\bz\|^2\le \|\bar\bV\|^2\|\bz\|^2\le C_3\|\bz\|^2,\label{eq_final_push_to_convexity4}
    \end{equation}
    for some absolute constant $C_3$ due to the boundedness of $\bgamma$.

    Furthermore, for $\|\bH_{L\theta u} - \bar\bH_{L\theta u}\| = \|\bJ_{L\theta u}\|$, $(iv)$ of Lemma~\ref{lemma_concentration_l_mat} gives $(nq)^{-1/2}\|\bJ_{L\theta u}\| = \cOp\big(\sqrt{B_n/(n\wedge q)} + B_n^2\epsilon\sqrt{\log n}/b_n + B_n\kappa_n\sqrt{\log n/q}\big)$. By Assumption~\ref{assump_scaling} and scaling condition \eqref{eq_key_scaling_equation_epsilon}, we get $(nq)^{-1/2}\|\bJ_{L\theta u}\| = o_p(b_n)$, and therefore,
    \begin{equation}
        (nq)^{-1/2}\|\bH_{L\theta u} - \bar\bH_{L\theta u}\|\|\bx\|^2= o_p\big(b_n\|\bx\|^2\big).\label{eq_final_push_to_convexity5}
    \end{equation}
    Combining \eqref{eq_final_push_to_convexity}--\eqref{eq_final_push_to_convexity5}, we arrive at
    \begin{equation*}
         \bx ^\T\left\{\bS_m\begin{pmatrix}
       \bH_{L\theta\theta}&\bH_{L\theta u}\\\bH_{Lu\theta}&\bH_{Luu}
   \end{pmatrix}\bS_m+c^{\prime}\bar\bV\bar\bV^\T \right\}\bx\ge \left(\tilde\lambda^{\prime} - c^{\prime}C_3\right)\|\bz\|^2 + \frac{c^{\prime}(c_1c_2)^2}{2}\|\by\|^2 - o_p\big(b_n\big)\|\bx\|^2.
    \end{equation*}
    Since $c_1,c_2,C_3$ are constants independent of $n,q$ and the choice of $\btheta$, take $c^{\prime}= C_3^{-1}\tilde\lambda^{\prime}/2$ and we obtain \eqref{eq_standard_convex_aug} with
    \begin{equation*}
        \lambda^{\prime} = \min\left\{\tilde\lambda^{\prime}/2 ,\; C_3^{-1}\tilde\lambda^{\prime}(c_1c_2)^2/4\right\}\asymp b_n.
    \end{equation*}

\medskip
\noindent\underline{\large \it Combining Steps 1 and 2.}
By Step~2, with probability approaching one,
\begin{equation*}
\lambda_{\min}\left[\bS_m \begin{pmatrix}
       \bH_{L\theta\theta}&\bH_{L\theta u}\\
       \bH_{Lu\theta}&\bH_{Luu}
\end{pmatrix} \circ\big(\btheta,\bu^{\star}(\btheta)\big) \bS_m + cb_n\sum_{r=0}^{K}\sum_{l=1}^{K} \bar\bv_{rl}(\btheta)\big\{\bar\bv_{rl}(\btheta)\big\}^{\T}\right]\ge \lambda'
\end{equation*} 
uniformly over $\btheta\in\cB_{\epsilon}(\btheta^*)$, for some constant $c>0$ and some $\lambda'\asymp b_n$. The scaled lower-right block is positive definite by \eqref{eq_reference_eigen_lowerbound}, and taking its Schur complement gives
\begin{equation*}
\lambda_{\min}\Big[n^{-1}\bH_{-\theta}^{\star}(\btheta) + cb_nq^{-1}\sum_{r=0}^{K}\sum_{l=1}^{K} \bv_{rl}(\btheta)\big\{\bv_{rl}(\btheta)\big\}^{\T}\Big] \ge \lambda'.
\end{equation*}
To replace $\bv_{rl}(\btheta)\big\{\bv_{rl}(\btheta)\big\}^\T$ by $\bv_{rl}(\btheta^*)\big\{\bv_{rl}(\btheta^*)\big\}^\T$, set $\Delta\bv_{rl}=\bv_{rl}(\btheta)-\bv_{rl}(\btheta^*)$ and note that
\begin{align*}
    &\,\Big\|q^{-1}\bv_{rl}(\btheta)\big\{\bv_{rl}(\btheta)\big\}^\T-q^{-1}\bv_{rl}(\btheta^*)\big\{\bv_{rl}(\btheta^*)\big\}^\T\Big\|\\
    &\qquad\le 2q^{-1}\|\bv_{rl}(\btheta^*)\|\,\|\Delta\bv_{rl}\|
       +q^{-1}\|\Delta\bv_{rl}\|^2
       =\cO(\epsilon)=o(b_n).
\end{align*}
Here $q^{-1/2}\|\bv_{rl}(\btheta^*)\|=\cO(1)$ and
$q^{-1/2}\|\Delta\bv_{rl}\|\le\epsilon$, while
$\epsilon=o( b_n^2/B_n^2)=o(b_n)$ by \eqref{eq_key_scaling_equation_epsilon} and $b_n\le1\le B_n$. Weyl's inequality completes the replacement.

\medskip 

For the lower bound on the minimum eigenvalue of $\bH_{-\theta}^{\star}(\btheta^*)$ alone, we let $\bar\bH_L^*$ be obtained from $\bH_L(\btheta^*,\bu^*)$ by replacing $\bH_{L\theta u}(\btheta^*,\bu^*),\bH_{Lu\theta}(\btheta^*,\bu^*)$ with  $\bar\bH_{L\theta u}(\btheta^*,\bu^*)$ and $\bar\bH_{Lu\theta}(\btheta^*,\bu^*)$.
By construction, one can verify $\bS_m\bar\bH_L^*\bS_m$ is positive semidefinite. Next, part (iii) of Lemma~\ref{lemma_concentration_l_mat} gives $\|\bS_m\{\bH_L(\btheta^*,\bu^*)-\bar\bH_L^*\}\bS_m\| = \cOp(\sqrt{B_n/(n\wedge q)})$. By Weyl's inequality, we know $\lambda_{\min}(\bS_m\bH_L^*\bS_m)\ge -\cOp(\sqrt{B_n/(n\wedge q)})$. Then by \eqref{eq_obtain_schur},
\[
\bD_m^\T\bS_m\bH_L^*\bS_m\bD_m = \begin{pmatrix}
n^{-1}\bH_{-\theta}^\star(\btheta^*)&\zero\\ \zero&q^{-1}\bH_{Luu}^* \end{pmatrix}.
\]
Together with \eqref{eq_benign_bD_scale}, this yields $\lambda_{\min}\{n^{-1}\bH_{-\theta}^\star(\btheta^*)\} \ge -\cOp(B_n^{5/2}/(b_n^2\sqrt{n\wedge q}))$.
Because $\varepsilon_{nq}=(n\vee q)^{o(1)}$, Assumption~\ref{assump_scaling} gives $n\wedge q\ge(n\vee q)^a$ for some fixed $a>0$, which yields $\varepsilon_{nq}/q\ll(n\wedge q)^{-1/2}$. Lemma~\ref{lemma_laplace_approx_second_order} therefore finishes the proof.


\subsubsection{Proof of Lemma~\ref{lemma_average_converge_eta}}\label{supp_sec_prove_lemma_average_converge_eta}
Fix $(\bmu,\bSigma)\in\bar\cK_\rho$ with $\rho<1$, and suppress these arguments in $\cL$. 
Denote the $i$th marginal integral by $Z_i(\btheta) = \int_{\RR^K}\exp\{\ell_i(\bu;\btheta)\}\pwk(\bu\mid\bmu,\bSigma)\,d\bu $. We first compare $\cL(\btheta)$ and $-\sum_{i=1}^n\ell_i\{\bu_i^\star(\btheta);\btheta\}$ over $\btheta\in\Xi$. Since $\bu_i^\star(\btheta)$ maximizes $\ell_i(\cdot;\btheta)$, $Z_i(\btheta)$ is at most $\exp[\ell_i\{\bu_i^\star(\btheta);\btheta\}]$.
For a lower bound, restrict the integral to the ball of radius $r_{nq}=(nq)^{-1/2}$ centered at $\bu_i^\star(\btheta)$. The constraints defining $\Xi$ and Assumption~\ref{assumption_smoothness} imply that $\|\partial_{\bu}^2\ell_i(\bu;\btheta)\|$ on this ball is bounded by $Cq\exp(C\sqrt{\log n})$. Since $\partial_{\bu}\ell_i\{\bu_i^\star(\btheta);\btheta\}=\zero$, Taylor's theorem bounds the decrease in $\ell_i(\bu;\btheta)$ from its maximum by $Cqr_{nq}^2\exp(C\sqrt{\log n})\le C$ on this ball for sufficiently large $n$. Moreover, $\|\bu\|\le C_u\sqrt{\log n}+r_{nq}$, so the Gaussian density $\pwk$ is bounded below by $cn^{-C}$, uniformly over $(\bmu,\bSigma)\in\bar\cK_\rho$. Therefore, we have
\begin{equation*}
    c n^{-C}(nq)^{-K/2}
    \exp[\ell_i\{\bu_i^\star(\btheta);\btheta\}]
    \le Z_i(\btheta)
    \le \exp[\ell_i\{\bu_i^\star(\btheta);\btheta\}] .
\end{equation*}
 Taking logarithms and summing over $i$ yields
\begin{equation*}
    0\le \cL(\btheta) + \sum_{i=1}^n\ell_i\{\bu_i^\star(\btheta);\btheta\} \le Cn\log(nq).
\end{equation*}
Let $\eta_{ij}^\star = (\bgamma_j^*)^\T\bu_i^\star(\btheta^*)+\beta_j^*$ and $\hat\eta_{ij}^\star =\hat\bgamma_j^\T\bu_i^\star(\hat\btheta)+\hat\beta_j$ for each $i\in[n]$ and $j\in[q]$. Since $\btheta^*\in\Xi$ with probability tending to one, $\hat\btheta\in\Xi$, and $\cL(\hat\btheta)\le\cL(\btheta^*)$, we therefore obtain
\begin{align*}
    -\sum_{i=1}^n\sum_{j=1}^q \ell_{ij}(\hat\eta_{ij}^\star) +\sum_{i=1}^n\sum_{j=1}^q \ell_{ij}(\eta_{ij}^\star) 
    &= \cL(\hat\btheta)-\cL(\btheta^*) + \cOp(n\log(nq))\\
    &= \cOp\big(n\log(n\vee q)\big).
\end{align*}Here, the last inequality follows from $\log(nq)\asymp \log (n\vee q)$. A second-order Taylor expansion about $\eta_{ij}^{\star}$ gives
\begin{equation}
     -\underbrace{\sum_{i=1}^n\sum_{j=1}^q\ell_{ij}^{\prime}(\eta_{ij}^{\star}) (\hat\eta_{ij}^{\star} - \eta_{ij}^{\star})}_{\alpha_1} + \frac{1}{2}\underbrace{\sum_{i=1}^n\sum_{j=1}^q\big\{-\ell_{ij}^{\prime\prime}(\tilde\eta_{ij}^{\star})\big\} (\hat\eta_{ij}^{\star} - \eta_{ij}^{\star})^2}_{\alpha_2}\le \cOp(n\log(n\vee q)).\label{eq_taylor_expansio_Average}
\end{equation}
Here $\tilde\eta_{ij}^{\star} = \eta_{ij}^{\star} + s_{ij}(\hat\eta_{ij}^{\star} - \eta_{ij}^{\star})$ for some $s_{ij}\in[0,1]$. For $\ell_{ij}^{\prime}(\eta_{ij}^{\star})$, where $\eta_{ij}^{\star}= {\bgamma_j^*}^\T\bu_{i}^{\star}(\btheta^*) + \beta_j^*$, invoke Lemma~\ref{lemma_bu_star_theta_true} and part~(iv) of Lemma~\ref{lemma_concentration_l_mat} to obtain
\begin{equation*}
    \left\|\big\{\ell_{ij}^{\prime}(\eta_{ij}^{\star})\big\}_{n\times q}\right\|\le \cOp(\sqrt{(n\vee q)B_n} + B_n\sqrt{nq\log n}\times \kappa_n/\sqrt{q}) = \cOp(B_n\kappa_n\sqrt{(n\vee q)\log n}).
\end{equation*}
This bound is deliberately coarse because it is used only to localize the analysis.
The matrix $\{\hat\eta_{ij}^{\star}-\eta_{ij}^{\star}\}_{n\times q}$ has rank at most $2K+1$. Hence the trace inequality
$|\tr(\bA^\T\bB)|\le\|\bA\|\|\bB\|_*\le\sqrt{\operatorname{rank}(\bB)}\|\bA\|\|\bB\|_{\Fn}$ gives
\begin{align*}
    |\alpha_1|
    \le &\, \sqrt{2K+1}\left\|\big\{\ell_{ij}^{\prime}(\eta_{ij}^{\star})\big\}_{n\times q}\right\|\left\|\left\{\hat\eta_{ij}^{\star} - \eta_{ij}^{\star}\right\}_{n\times q}\right\|_{\Fn}\\
    = &\, \cOp\left(B_n\kappa_n\sqrt{(n\vee q)\log n}\Big\{\sum_{i=1}^n\sum_{j=1}^q\left(\hat\eta_{ij}^{\star} - \eta_{ij}^{\star}\right)^2\Big\}^{1/2}\right).
\end{align*}
For $\alpha_2$, $\hat\eta_{ij}^{\star}=\hat\bgamma_j^\T\bu_{i}^{\star}(\hat\btheta)+\hat\beta_j$ is bounded by $C\sqrt{\log n}$ uniformly because $\hat\btheta\in\Xi$. Because $\btheta^*\in\Xi$, Assumption~\ref{assumption_smoothness} gives
$\min_{i\in[n],j\in[q]}\{-\ell_{ij}^{\prime\prime}(\tilde\eta_{ij}^{\star})\}\ge b_n$. Therefore
\begin{equation*}
    \alpha_2 \ge b_n\sum_{i=1}^n\sum_{j=1}^q\left(\hat\eta_{ij}^{\star} - \eta_{ij}^{\star}\right)^2.
\end{equation*}
Plug the estimates for $\alpha_1$ and $\alpha_2$ into \eqref{eq_taylor_expansio_Average} and with the scaling condition in Assumption~\ref{assump_scaling}, we arrive at \begin{equation}
\begin{aligned}
    \sum_{i=1}^n\sum_{j=1}^q\left(\hat\eta_{ij}^{\star}  - \eta_{ij}^{\star}\right)^2 &= \cOp\left\{\frac{\kappa_n^2B_n^2(n\vee q)\log n}{b_n^2} +\frac{n\log(n\vee q)}{b_n}\right\}\\
    &= \cOp\left\{\frac{\kappa_n^2B_n^2(n\vee q)}{b_n^2}\log (n\vee q)\right\}. \end{aligned} \label{eq_average_eta_consis}
\end{equation}


\subsubsection{Proof of Lemma~\ref{lemma_aux_estimator_property}}\label{supp_sec_prove_lemma_aux_estimator_property}
When $\btheta\in\cB_{\epsilon}(\btheta^*)$, we know $\|\btheta - \btheta^*\|\le \sqrt{q}\epsilon$. By Lemma~\ref{lemma_neighbour_control}, we have $\|\bu_i^{\star}(\btheta) - \bu_i^*\| = \cOp(\epsilon\kappa_n\sqrt{B_n}\vee \kappa_n/\sqrt{q})$. 
Assumption~\ref{assump_standard_id} gives $\bSigma_u^* = \bI_K$. Hence, $\|n^{-1}\sum_{i=1}^n\bu_i^*\| = \cOp(n^{-1/2})$ and $\|n^{-1}\sum_{i=1}^n\bu_i^*(\bu_i^*)^\T - \bI_K\| = \cOp(n^{-1/2})$. 
Therefore, by the Cauchy--Schwarz inequality,
\begin{equation}\label{eq_bound_u_sum_ms}\begin{aligned}
    \left\|\frac{1}{n}\sum_{i=1}^n\bu_i^{\star}(\btheta)\right\|
    &=\cOp\left(\epsilon\kappa_n\sqrt{B_n}+\frac{\kappa_n}{\sqrt q}+\frac{1}{\sqrt n}\right),\\
    \left\|\frac{1}{n}\sum_{i=1}^n\bu_i^{\star}(\btheta)\{\bu_i^{\star}(\btheta)\}^\T-\bI_K\right\|
    &=\cOp\left(\epsilon\kappa_n\sqrt{B_n}+\frac{\kappa_n}{\sqrt q}+\frac{1}{\sqrt n}\right).
\end{aligned}\end{equation} 
Indeed, $(\hat\bmu^*,\hat\bSigma^*)\in\cK_C$ minimizes a differentiable function over the convex compact set $\cK_C$, so constrained optimality gives the variational inequality
\begin{equation*}
    \big\langle \cS_{\mu}(\hat\btheta^*,\hat\bmu^*,\hat\bSigma^*),\bmu - \hat\bmu^*\big\rangle + \big\langle \cS_{\Sigma}(\hat\btheta^*,\hat\bmu^*,\hat\bSigma^*),\bSigma - \hat\bSigma^*\big\rangle \; \ge\; 0\qquad\text{for all }(\bmu,\bSigma)\in\cK_C.
\end{equation*}
Let $\bd(\btheta) = n^{-1}\sum_{i=1}^n\bu_i^\star(\btheta)$ and $\bS(\btheta)=n^{-1}\sum_{i=1}^n\{\bu_i^\star(\btheta)-\bd(\btheta)\}\{\bu_i^\star(\btheta)-\bd(\btheta)\}^\T$. Then for $G(\bmu,\bSigma)=\tfrac12\log\det\bSigma+
\tfrac12\operatorname{tr}\{\bSigma^{-1}[\bS(\btheta)+(\bmu-\bd(\btheta))(\bmu-\bd(\btheta))^\T]\}$, we know by Lemma~\ref{lemma_first_order_base} that its derivatives with respect to $\bmu$ and $\bSigma$ can be uniformly approximated by $\cS_{\mu}$ and $\cS_{\Sigma}$ up to an error of order $\cOp(\varepsilon_{nq}/q)$. Taking $C>1$ so that $(\zero_K,\bI_K)\in\cK_C$. Equations \eqref{eq_bound_u_sum_ms} ensure that $(\bd(\btheta),\bS(\btheta))$ also lies in the interior of $\cK_C$.
Then we have 
\begin{equation*}
\begin{aligned}
\langle\partial_\mu G,\bmu-\bd(\btheta)\rangle +\langle\partial_\Sigma G,\bSigma-\bS(\btheta)\rangle &= \frac12\big\{\bmu-\bd(\btheta)\big\}^\T
\left\{ \bSigma^{-1} +\bSigma^{-1}\bS(\btheta)\bSigma^{-1} \right\}\big\{\bmu-\bd(\btheta)\big\}\\
&\quad+ \frac12 \left\| \bSigma^{-1/2}\big\{\bSigma-\bS(\btheta)\big\}\bSigma^{-1/2} \right\|_{\Fn}^{2}\\
&\ge c\left( \big\|\bmu-\bd(\btheta)\big\|^2+\|\bD\|_{\Fn}^2 \right),
\end{aligned}
\end{equation*}
 Thus we know 
\begin{equation*}
\big\|\hat\bmu^* - \bd(\btheta)\big\| + \big\|\hat\bSigma^* - \bS(\btheta)\big\| = \cOp(\varepsilon_{nq}/q),
\end{equation*}
which, together with \eqref{eq_bound_u_sum_ms}, yields $\|\hat\bmu^*\| = \cOp(\epsilon\kappa_n\sqrt{B_n} + \kappa_n/\sqrt{q} + n^{-1/2} + \varepsilon_{nq}/q)= o_p(1)$ and $\|\hat\bSigma^* - \bI_K\| = \cOp(\epsilon\kappa_n\sqrt{B_n} + \kappa_n/\sqrt{q} + n^{-1/2} + \varepsilon_{nq}/q) = o_p(1)$. Since $(\zero_K,\bI_K)$ is an interior point of $\cK_C$, the constraints on $(\bmu,\bSigma)$ are inactive with probability tending to one, so first-order conditions $\cS_{\mu} = \zero$ and $\cS_{\Sigma}=\zero$ hold.

The first-order condition for $\hat\btheta^*$ in \eqref{eq_aux_estimator}, however, is substantially more challenging to establish.
To address this, we construct a second auxiliary estimator for the estimator defined in \eqref{eq_aux_estimator}:
\begin{equation}
    \tilde\btheta^* = \mathop{\arg\min}_{\btheta\in \cB_{\epsilon}(\btheta^*)}\left\|\hat\cS_{Q}^*(\btheta)\right\|_{\infty}^{2},\label{eq_dbl_aux_estimator}
\end{equation}
where $\hat\cS_{Q}^*(\btheta)$ is defined as
\begin{equation*}
    \hat\cS_{Q}^*(\btheta) := n^{-1}\partial_{\btheta}\cQ^*(\btheta,\hat\bmu^*,\hat\bSigma^*).
\end{equation*}
 The primary reason we consider this estimator is that, by definition, it holds that
\begin{equation}
    \left\|\hat\cS_{Q}^*(\tilde\btheta^*) - \hat\cS_{Q}^*(\btheta^*)\right\|_{\infty}\le 2\left\|\hat\cS_{Q}^*(\btheta^*)\right\|_{\infty}=2\left\|n^{-1}\partial_{\btheta}\cQ^*(\btheta^*,\hat\bmu^*,\hat\bSigma^*)\right\|_{\infty},\label{eq_zeta_diff_definition}
\end{equation}
where the right-hand side is controlled by the following lemma.
\begin{lemma}\em\label{lemma_zeta_bound_S_Q}
    Under the setting of Lemma~\ref{lemma_aux_estimator_property}, for
    {$r_{\theta,nq}$ defined in Equation~\eqref{eq_define_r_theta_nq}, one has}
    \begin{equation*}
        \left\|n^{-1}\partial_{\btheta}\cQ^*(\btheta^*,\hat\bmu^*,\hat\bSigma^*)\right\|_{\infty} = \cOp(r_{\theta,nq}\sqrt{\log q/n}).
    \end{equation*}
\end{lemma}
    \begin{proof}
        Invoke the expression of $\partial_{\btheta}P^*(\btheta)$ in \eqref{eq_first_order_p_star} and one immediately has 
        \begin{equation*}
            n^{-1}\partial_{\btheta}\cQ^*(\btheta^*,\hat\bmu^*,\hat\bSigma^*) = n^{-1}\cS_{\theta}(\btheta^*,\hat\bmu^*,\hat\bSigma^*) + n^{-1}\partial_{\btheta}P^*(\btheta^*) = n^{-1}\cS_{\theta}(\btheta^*,\hat\bmu^*,\hat\bSigma^*).
        \end{equation*}
        The result now follows from Lemma~\ref{lemma_first_order_concern}.
    \end{proof}
Next, we apply the integral mean value theorem for $\hat\cS_{Q}^*$:
\begin{equation*}
    \hat\cS_{Q}^*(\tilde\btheta) - \hat\cS_{Q}^*(\btheta^*) = \int_{0}^1n^{-1}\partial_{\btheta\btheta}^2\cQ^*\big(\btheta^* + s(\tilde\btheta - \btheta^*),\hat\bmu^*,\hat\bSigma^*\big)ds\times (\tilde\btheta - \btheta^*).
\end{equation*}
For the remainder of this proof, write $\tilde\btheta=\tilde\btheta^*$.
Lemma~\ref{lemma_laplace_approx_second_order} provides an approximation for $\partial_{\btheta}^2\cQ^*(\btheta,\hat\bmu^*,\hat\bSigma^*)$ through $\bH_{-\theta}^{\star}(\btheta)$. We let 
\begin{equation*}
    \bH_{Q2}^{\star}(\btheta) := \left\{\bH_{L\theta u}\big(\btheta,\bu^{\star}(\btheta)\big)\right\}\,\left\{\bH_{Lu u}\big(\btheta,\bu^{\star}(\btheta)\big)\right\}^{-1}\,\left\{\bH_{Lu\theta}\big(\btheta,\bu^{\star}(\btheta)\big)\right\}.
\end{equation*}
Define $\tilde\cH_{Q}^*(\btheta,\hat\bmu^*,\hat\bSigma^*) = \int_{0}^1n^{-1}\partial_{\btheta\btheta}^2\cQ^*\big(\btheta^* + s(\btheta - \btheta^*),\hat\bmu^*,\hat\bSigma^*\big)ds$, $\tilde\cH_{Q2}^*(\btheta) = \int_{0}^1n^{-1}\bH_{Q2}^{\star}\big(\btheta^* + s(\btheta - \btheta^*)\big)ds$, and $\tilde\cH_{Q1}^*(\btheta,\hat\bmu^*,\hat\bSigma^*) = \tilde\cH_{Q}^*(\btheta,\hat\bmu^*,\hat\bSigma^*) + \tilde\cH_{Q2}^*(\btheta)$. With these definitions, the expansion can be equivalently expressed as 
\begin{subequations}\label{eq_expand_tild_theta}
    \begin{align}
    \hat\cS_{Q}^*(\tilde\btheta) - \hat\cS_{Q}^*(\btheta^*) &= \tilde\cH_{Q}^*(\tilde\btheta,\hat\bmu^*,\hat\bSigma^*)\times (\tilde\btheta - \btheta^*)\label{eq_expand_tild_theta_S1} \\&=\tilde\cH_{Q1}^*(\tilde\btheta,\hat\bmu^*,\hat\bSigma^*)\times (\tilde\btheta - \btheta^*) - \tilde\cH_{Q2}^*(\tilde\btheta)\times (\tilde\btheta - \btheta^*) .\label{eq_expand_tild_theta_S2}
\end{align}
\end{subequations}
We present the following technical lemmas to characterize the terms in this expansion. 
\begin{lemma}\label{lemma_aux_of_aux_estimator_property}\em Under the setting of Lemma~\ref{lemma_aux_estimator_property}, the following statements hold for every $\btheta\in\cB_{\epsilon}(\btheta^*)$:
\begin{enumerate}[label=$\textnormal{(\roman*)}$]
\item $\PP\big[\lambda_{\min}\{\tilde\cH_{Q}^*(\btheta,\hat\bmu^*,\hat\bSigma^*)\}\ge\lambda\big]\to1$ as $n,q\to\infty$, for some $\lambda\asymp b_n$;
\item $\big\|\{\tilde\cH_{Q1}^*(\btheta,\hat\bmu^*,\hat\bSigma^*)\}^{-1}\big\|_{\infty}= \cOp(b_n^{-1})$ and
\begin{equation*}
    \big\|\tilde\cH_{Q2}^*(\btheta)(\btheta - \btheta^*) \big\|_{\infty}
    = \cOp\Big(\frac{B_n^2}{b_n\sqrt q}\|\btheta-\btheta^*\|\Big).
\end{equation*}
\end{enumerate}
\end{lemma}
\begin{proof}
    See Section~\ref{supp_sec_prove_lemma_aux_of_aux_estimator_property}.
\end{proof}
Subsequently, invoke expansion \eqref{eq_expand_tild_theta} and one has
\begin{align*}
    \frac{\|\tilde\btheta - \btheta^*\|}{\sqrt{q}}\le&\,q^{-1/2}\lambda_{\min}\left\{\tilde\cH_{Q}^*(\tilde\btheta,\hat\bmu^*,\hat\bSigma^*)\right\}^{-1}\left\|\hat\cS_{Q}^*(\tilde\btheta) - \hat\cS_{Q}^*(\btheta^*)\right\|\\\overset{(i)}{\le}&\,q^{-1/2}\lambda_{\min}\left\{\tilde\cH_{Q}^*(\tilde\btheta,\hat\bmu^*,\hat\bSigma^*)\right\}^{-1}\left\|\hat\cS_{Q}^*(\tilde\btheta) - \hat\cS_{Q}^*(\btheta^*)\right\|_{\infty}(qK+q)^{1/2}\\
    \overset{(ii)}{\le}&\,4K^{1/2}\lambda_{\min}\left\{\tilde\cH_{Q}^*(\tilde\btheta,\hat\bmu^*,\hat\bSigma^*)\right\}^{-1}\left\|n^{-1}\partial_{\btheta}\cQ^*(\btheta^*,\hat\bmu^*,\hat\bSigma^*)\right\|_{\infty}\\
    \overset{(iii)}{=}&\,\cOp\big\{r_{\theta,nq}\sqrt{\log q}/(b_n\sqrt{n})\big\}.
\end{align*}
Here $(i)$ follows from H\"older's inequality, $(ii)$ from \eqref{eq_zeta_diff_definition}, and $(iii)$ from Lemma~\ref{lemma_zeta_bound_S_Q} and part~$(i)$ of Lemma~\ref{lemma_aux_of_aux_estimator_property}. Similarly, expansion \eqref{eq_expand_tild_theta_S2} gives
\begin{align*}
    \big\|\tilde\btheta - \btheta^*\big\|_{\infty}\le&\,\left\|\big\{\tilde\cH_{Q1}^*(\tilde\btheta,\hat\bmu^*,\hat\bSigma^*)\big\}^{-1}\right\|_{\infty}\Big\|\hat\cS_{Q}^*(\tilde\btheta) - \hat\cS_{Q}^*(\btheta^*)\Big\|_{\infty} \\&\,+ \left\|\big\{\tilde\cH_{Q1}^*(\tilde\btheta,\hat\bmu^*,\hat\bSigma^*)\big\}^{-1}\right\|_{\infty}\Big\|\tilde\cH_{Q2}^*(\tilde\btheta)\times (\tilde\btheta - \btheta^*) \Big\|_{\infty}\\
    =&\,\cOp\Big\{\frac{r_{\theta,nq}}{b_n}\sqrt{\frac{\log q}{n}} + \frac{B_n^2}{b_n^2}\epsilon\Big\}.
\end{align*}
The last equality follows from Lemma~\ref{lemma_zeta_bound_S_Q} and part~$(ii)$ of Lemma~\ref{lemma_aux_of_aux_estimator_property}. 

Assumption~\ref{assump_scaling} and \eqref{eq_key_scaling_equation_epsilon} imply that $\epsilon$ satisfies
\begin{equation*}
    \frac{r_{\theta,nq}\sqrt{\log q}}{b_n\sqrt n}\ll \epsilon\ll (\log n)^{-1}\wedge\frac{b_n^2}{B_n^2},
\end{equation*}
the above estimates imply 
\begin{equation*}
    q^{-1/2}\big\|\tilde\btheta - \btheta^*\big\|\ll \epsilon\quad\text{ and }\quad\big\|\tilde\btheta - \btheta^*\big\|_{\infty}\ll 1.
\end{equation*}
Therefore $\tilde\btheta$ is an interior point of problem \eqref{eq_dbl_aux_estimator}. We next show that its score is zero without differentiating the nonsmooth maximum norm. Suppose, to the contrary, that $\hat\cS_Q^*(\tilde\btheta)\ne\zero$. By \eqref{eq_convex_aux_obj}, the Jacobian
$\partial_{\btheta}\hat\cS_Q^*(\tilde\btheta)$ is positive definite. Set
\[
    \bd=-\{\partial_{\btheta}\hat\cS_Q^*(\tilde\btheta)\}^{-1}\hat\cS_Q^*(\tilde\btheta).
\]
Because $\tilde\btheta$ is interior, $\tilde\btheta+t\bd$ is feasible for all sufficiently small $t>0$, and Taylor's theorem gives
\[
    \hat\cS_Q^*(\tilde\btheta+t\bd)
    =(1-t)\hat\cS_Q^*(\tilde\btheta)+o(t).
\]
Its maximum norm is therefore strictly smaller for sufficiently small $t>0$, contradicting the definition of $\tilde\btheta$. Hence
\begin{equation*}
    \hat\cS^*_{Q}(\tilde\btheta) = \zero.
\end{equation*}
Strong convexity then shows that $\tilde\btheta$ is the unique minimizer of $\cQ^*$ inside $\cB_{\epsilon}(\btheta^*)$, so $\hat\btheta^* = \tilde\btheta$ with probability tending to one.
Next, as in the proof of Lemma~\ref{lemma_zeta_bound_S_Q}, $\|\hat\cS_{Q}^*(\btheta^*)\| = \cOp(r_{\theta,nq}\sqrt{q/n})$. Together with $\hat\cS^*_{Q}(\hat\btheta^*) = \zero$, the same argument used to bound $\|\tilde\btheta - \btheta^*\|$ gives
\begin{equation*}
    \big\|\hat\btheta^* - \btheta^*\big\| = \cOp\big(r_{\theta,nq}\sqrt{q/(nb_n^2)}\big),
\end{equation*}
which proves $(i)$ of Lemma~\ref{lemma_aux_estimator_property}. 

For part $(ii)$, substituting the established rate $q^{-1/2}\|\hat\btheta^* - \btheta^*\| = \cOp(r_{\theta,nq}b_n^{-1}n^{-1/2})$ into the preceding bounds $\|\hat\bmu^*\| = \cOp(\epsilon\kappa_n\sqrt{B_n}\vee \kappa_n/\sqrt{q}\vee n^{-1/2})$ and $\|\hat\bSigma^* - \bI_K\| = \cOp(\epsilon\kappa_n\sqrt{B_n}\vee \kappa_n/\sqrt{q}\vee n^{-1/2}) $ gives the result.

For part~$(iii)$, use $\hat\cS_Q^*(\hat\btheta^*)=\zero$ in \eqref{eq_expand_tild_theta_S2}, together with Lemmas~\ref{lemma_zeta_bound_S_Q} and~\ref{lemma_aux_of_aux_estimator_property}. The sharper bound in part~$(ii)$ of Lemma~\ref{lemma_aux_of_aux_estimator_property} and the average rate above yield
\begin{align*}
    \|\hat\btheta^*-\btheta^*\|_\infty
   =\cOp\Big(\frac{r_{\theta,nq}}{\sqrt n}
    \left\{\frac{\sqrt{\log q}}{b_n}+\frac{B_n^2}{b_n^3}\right\}\Big).
\end{align*}
This proves part~$(iii)$ without changing the localization radius $\epsilon$. Part~$(iv)$ follows from $\tilde\btheta = \hat\btheta^*$ and $\hat\cS_{Q}^*(\tilde\btheta)=\zero$ with probability tending to one, together with the first-order conditions for $\hat\bmu^*$ and $\hat\bSigma^*$.

\subsubsection{Proof of Lemma~\ref{lemma_average_eta}}\label{supp_sec_prove_lemma_average_eta}

Write $\bu^{\star}_i = \bu_i^{\star}(\btheta^*)$ for short. Note that $P^*(\btheta) = 0$ implies
    \begin{align}
        \sum_{j=1}^q\bgamma_j^*\big(\beta_j  -\beta_j^*\big) = \zero\text{ and }\sum_{j=1}^q\bgamma_j^*\big(\bgamma_j  -\bgamma_j^*\big)^\T = \zero.\label{eq_identify_mean}
    \end{align}
    By \eqref{eq_identify_mean}, one has 
    \begin{align*}
        \sum_{i=1}^n\Big\|\sum_{j=1}^q\bgamma_j^*(\bgamma_j^*)^\T(\bu_i - \bu_i^{\star})\Big\|^2
        = &\,\sum_{i=1}^n\Big\|\sum_{j=1}^q\bgamma_j^*\big\{\bgamma_j^\T\bu_i + \beta_j - (\bgamma_j^*)^\T\bu_i^{\star} - \beta_j^*\big\}\Big\|^2\\
        \le &\, \sum_{i=1}^n\left(\sum_{j=1}^q\big\|\bgamma_j^*\big\|^2\right)
        \left(\sum_{j=1}^q\big\{\bgamma_j^\T\bu_i + \beta_j - (\bgamma_j^*)^\T\bu_i^{\star} - \beta_j^*\big\}^{2}\right)\\
        = &\,\cO\big(q\iota_{nq}^2\big),
    \end{align*}
    where
    \begin{equation*}
        \iota_{nq} = \Bigg\{\sum_{i=1}^n\sum_{j=1}^q\big(\bgamma_j^\T\bu_i + \beta_j - (\bgamma_j^*)^\T\bu_i^{\star}(\btheta^*)  -\beta_j^*\big)^2\Bigg\}^{1/2}
        =o_p(\sqrt{nq}).
    \end{equation*}
    By Assumption~\ref{assumption_psd_covariance}, $\lambda_{\min}\big\{\sum_{j=1}^q\bgamma_j^*(\bgamma_j^*)^\T\big\} = \Theta(q)$, and consequently, 
    \begin{equation*}
        \Big\|\sum_{j=1}^q\bgamma_j^*(\bgamma_j^*)^\T(\bu_i - \bu_i^{\star})\Big\|^2\ge \lambda_{\min}\big\{\sum_{j=1}^q\bgamma_j^*(\bgamma_j^*)^\T\big\}^2 \big\|\bu_i - \bu_i^{\star}\big\|^2,
    \end{equation*}
which implies
    \begin{equation}\label{eq_bound_u_star_from_Avera}
        n^{-1}\sum_{i=1}^n\big\|\bu_i-\bu_i^{\star}\big\|^2 = \cOp\big(\iota_{nq}^2/(nq)\big).
    \end{equation}
Next, let $\bP_{\Gamma^*} = \bGamma^*\{(\bGamma^*)^\T\bGamma^*\}^{-1}(\bGamma^*)^\T$ and $\bP_{\Gamma^*}^{\perp} = \bI_q-\bP_{\Gamma^*}$. Let $\bU = (\bu_1,\dots,\bu_n)^\T$, $\bU^* = (\bu_1^*,\dots,\bu_n^*)^\T$, and $\bU^{\star} = (\bu_1^{\star},\dots,\bu_n^{\star})^\T$. Then
\begin{align*}
    \bE := &\,\bbeta\one_n^\T + \bGamma\bU^\T - \bbeta^*\one_n^\T - \bGamma^*\bU^{\star\T} \\
    = &\, (\bbeta - \bbeta^*)\one_n^\T + \bGamma^*\big(\bU-\bU^{\star}\big)^\T + (\bGamma - \bGamma^*)\bU^\T.
\end{align*}
Then $\bP_{\Gamma^*}^{\perp}\bE =  (\bGamma - \bGamma^*)\bU^\T + (\bbeta - \bbeta^*)\one_n^\T$. Here, we use $\sum_{j=1}^q\bgamma_j^*(\beta_j-\beta_j^*)=0$ and $\sum_{j=1}^q\bgamma_j^*(\bgamma_j-\bgamma_j^*)^\T=0$, which follow from $P^*(\btheta) = 0$. Therefore,
\begin{equation*}
    \big\|(\bGamma - \bGamma^*)\bU^\T + (\bbeta - \bbeta^*)\one_n^\T\big\|_{\Fn}^2\le \big\|\bE\big\|_{\Fn}^2= \iota_{nq}^2.
\end{equation*}
For $(\one_n,\bU)$, note that \eqref{eq_bound_u_star_from_Avera} and Lemma~\ref{lemma_bu_star_theta_true} give
\begin{equation*}
    n^{-1/2}\big\|\bU - \bU^{*}\big\|_{\Fn}
    \le n^{-1/2}\big\|\bU - \bU^{\star}\big\|_{\Fn} + n^{-1/2}\big\|\bU^* - \bU^{\star}\big\|_{\Fn}
    = \cOp\big\{\iota_{nq}/\sqrt{nq} + \kappa_n/\sqrt{q}\big\}=o_p(1).
\end{equation*}
Moreover, the law of large numbers and Assumption~\ref{assump_standard_id} imply
\begin{equation*}
    n^{-1}(\one_n,\bU^*)^\T(\one_n,\bU^*)\overset{p}{\longrightarrow}\begin{pmatrix}
        1 & \zero_K^\T\\ \zero_K & \bI_K
    \end{pmatrix}.
\end{equation*}
Weyl's inequality therefore yields, for some constant $c>0$,
\begin{equation}
    \label{eq_lower_bound_U_c}\lambda_{\min}\big\{n^{-1}(\one_n,\bU)^\T(\one_n,\bU)\big\}\ge c
\end{equation}
with probability approaching $1$ as $n,q\to\infty$. Consequently,
\begin{equation*}
    \big\|(\bGamma - \bGamma^*)\bU^\T + (\bbeta - \bbeta^*)\one_n^\T\big\|_{\Fn}^2
    \ge cn\big(\|\bGamma- \bGamma^*\|_{\Fn}^2 + \|\bbeta^* - \bbeta\|^2\big),
\end{equation*}
which gives 
\begin{equation*}
    q^{-1/2}\big\|\btheta - \btheta^*\big\|=\cOp\left(\frac{\iota_{nq}}{\sqrt{nq}}\right).
\end{equation*}

\subsection{Proofs of Lemmas in Sections~\ref{supp_sec_main_proof} and~\ref{sec_res_pract_id}}

\subsubsection{Proof of Lemma~\ref{lemma_residual_in_asym_expand}}\label{sec_prove_lemma_residual_in_asym_expand}

As in the statement of the lemma, write
\[
    \Delta_{\theta j} := \hat\btheta_j-\btheta_j^*,
    \qquad
    \Delta_{u i} := \hat\bu_i^\star-\bu_i^*,
\]
and let $\Delta_{\beta j}$ and $\Delta_{\gamma j}$ denote the intercept and loading components of $\Delta_{\theta j}$, so that $\|\Delta_{\theta j}\|^2 = \Delta_{\beta j}^2 + \|\Delta_{\gamma j}\|^2$. All stochastic orders below are taken conditionally on $\{\bu_i^*\}_{i=1}^n$. Recall from \eqref{eq_consis_ave_all} that
\begin{equation}
    q^{-1}\sum_{j=1}^q\big\|\Delta_{\theta j}\big\|^2=\cOp(\delta_{nq}^2),
    \qquad
    n^{-1}\sum_{i=1}^n\big\|\Delta_{u i}\big\|^2=\cOp(\delta_{nq}^2),
    \label{eq_second_moment_block_errors}
\end{equation}
and that Assumption~\ref{assump_standard_id} gives $n^{-1}\sum_{i=1}^n(1+\|\bu_i^*\|)^m=\cOp(1)$ for every fixed positive integer $m$, together with $\max_{i\in[n]}\|\bu_i^*\|=\cOp(\sqrt{\log n})$.

At fixed $\hat\bu^\star$, the integral form of Taylor expansion gives
\begin{equation}
    \bS_{L\theta_j}(\hat\btheta,\hat\bu^\star) -\bS_{L\theta_j}(\btheta^*,\hat\bu^\star) =
    \left\{\int_0^1\bH_{L\theta_j\theta_j}\big(\btheta^*+s(\hat\btheta-\btheta^*),\hat\bu^\star\big)\,ds\right\} \Delta_{\theta j}.
    \label{eq_taylor_expand_theta_j}
\end{equation}
The argument of $\ell_{ij}''$ in the integrand is $\{\bgamma_j^*+s(\hat\bgamma_j-\bgamma_j^*)\}^{\T}\hat\bu_i^\star
    +\beta_j^*+s(\hat\beta_j-\beta_j^*)$. By Assumption~\ref{assumption_smoothness}, the norm of the right side is at least
\begin{equation*}
        b_n\lambda_{\min}\left\{\sum_{i=1}^n\begingroup\renewcommand{\arraystretch}{0.8}\begin{pmatrix}1&(\hat\bu_i^{\star})^\T\\\hat\bu_i^{\star}&\hat\bu_i^{\star}(\hat\bu_i^{\star})^\T\end{pmatrix}\endgroup\right\} \big\|\Delta_{\theta j}\big\|\ge cnb_n\big\|\Delta_{\theta j}\big\|,
\end{equation*}
where the inequality follows from \eqref{eq_second_moment_block_errors} and the same argument used to derive \eqref{eq_lower_bound_U_c} in Section~\ref{supp_sec_prove_lemma_average_eta}. Thus we have
\begin{equation*}
    \left\|\bS_{L\theta_j}(\hat\btheta,\hat\bu^\star)
    -\bS_{L\theta_j}(\btheta^*,\hat\bu^\star)\right\|
    \ge cnb_n\|\Delta_{\theta j}\|.
\end{equation*}
Because $\cS_{\theta_j}(\hat\btheta)=\zero$ with probability approaching one, the left side of \eqref{eq_taylor_expand_theta_j} can be decomposed as
\begin{align}
    &\bS_{L\theta_j}(\hat\btheta,\hat\bu^\star)-\bS_{L\theta_j}(\btheta^*,\hat\bu^\star)\nonumber\\ =\,&
    \big\{\bS_{L\theta_j}(\hat\btheta,\hat\bu^\star)-\cS_{\theta_j}(\hat\btheta)\big\} + \big\{\bS_{L\theta_j}(\btheta^*,\bu^*)-\bS_{L\theta_j}(\btheta^*,\hat\bu^\star)\big\} -\bS_{L\theta_j}(\btheta^*,\bu^*). \label{eq_pointwise_score_decomposition}
\end{align}
Lemma~\ref{lemma_first_order_base} bounds the first term by $\cOp(nB_n\varepsilon_{nq}\sqrt{\log n}/q)$ {uniformly over} $j\in[q]$. For the second term, the mean-value theorem and the Cauchy--Schwarz inequality give
\begin{align*}
    &\left\|\bS_{L\theta_j}(\btheta^*,\hat\bu^\star)
    -\bS_{L\theta_j}(\btheta^*,\bu^*)\right\|\\
    &\quad\le C B_n\sum_{i=1}^n
    \big(1+\|\bu_i^*\|+\|\Delta_{u i}\|\big)\|\Delta_{u i}\|
    +\big\|(\Lb_\eta)_{,j}\big\|
    \left(\sum_{i=1}^n\|\Delta_{u i}\|^2\right)^{1/2},
\end{align*}
where $\Lb_\eta$ is the score matrix of Lemma~\ref{lemma_concentration_l_mat}. By \eqref{eq_second_moment_block_errors} and $n^{-1}\sum_{i=1}^n(1+\|\bu_i^*\|)^2=\cOp(1)$, the first sum is $\cOp(nB_n\delta_{nq})$ and the second is $\cOp(\sqrt n\,\delta_{nq})\|(\Lb_\eta)_{,j}\|$. 
Combining these bounds, we know that for all $j\in[q]$,
\begin{equation}
    \big\|\Delta_{\theta j}\big\|\le C\left\{ \frac{B_n}{b_n}\delta_{nq} +\frac{\big\|\bS_{L\theta_j}(\btheta^*,\bu^*)\big\|}{nb_n} +\frac{\delta_{nq}\big\|(\Lb_\eta)_{,j}\big\|}{\sqrt n\,b_n} +\frac{B_n\varepsilon_{nq}\sqrt{\log n}}{qb_n}
    \right\}.
    \label{eq_theta_block_error_inequality}
\end{equation}

For  $\bS_{L\theta_j}(\btheta^*,\bu^*)$, because, conditionally on $\{\bu_i^*\}_{i=1}^n$, the variables $\ell_{ij}'(\eta_{ij}^*)$ are independent and mean zero, and Assumption~\ref{assumption_smoothness} bounds their conditional sub-exponential norms by $C\sqrt{B_n}$, we therefore have its conditional second moment bounded by $CB_n\sum_{i=1}^n(1+\|\bu_i^*\|^2)$. Thus $\|\bS_{L\theta_j}(\btheta^*,\bu^*)\|=\cOp(\sqrt{nB_n})$. Part $(i)$ of Lemma~\ref{lemma_concentration_l_mat} give $\max_{j\in[q]}\|(\Lb_\eta)_{,j}\|=\cOp(\sqrt{nB_n}+\sqrt{B_n}\log q)$. Therefore, by the definition of $\bar\delta_{nq}$, combining these bounds gives $\|\bDelta_{\theta j}\|=\cOp(\bar\delta_{nq})$.

In addition, applying Rosenthal-type inequality, we know $\EE[\|\bS_{L\theta_j}(\btheta^*,\bu^*)\|^4\mid \{\bu_i^*\}_{i=1}^n]\le CB_n^2\{(\sum_{i=1}^n(1+\|\bu_i^*\|^2))^2+\sum_{i=1}^n(1+\|\bu_i^*\|^2)^2\}=\cOp(n^2B_n^2)$ and $\EE[\|(\Lb_\eta)_{,j}\|^4\mid \{\bu_i^*\}_{i=1}^n]\le C\{(\sum_{i=1}^n\EE[\{\ell_{ij}'(\eta_{ij}^*)\}^2])^2+\sum_{i=1}^n\EE[\{\ell_{ij}'(\eta_{ij}^*)\}^4]\}=\cO(n^2B_n^2)$. Averaging over $j$ leads to $q^{-1}\sum_{j=1}^q\big\{\|\bS_{L\theta_j}(\btheta^*,\bu^*)\|^4+\|(\Lb_\eta)_{,j}\|^4\big\}=\cOp(n^2B_n^2)$.
Subsequently, we have \begin{equation}
    q^{-1}\sum_{j=1}^q\|\Delta_{\theta j}\|^4=\cOp(\bar\delta_{nq}^4).\label{eq_fourth_moment_block_errors}
\end{equation}

A similar upper bound can be derived for $\|\Delta_{u i}\|$ as
\begin{equation}
    \big\|\Delta_{u i}\big\|\le C\left\{ \frac{B_n}{b_n}\big(1+\|\bu_i^*\|\big)\delta_{nq} + \frac{\big\|\bS_{Lu_i}(\btheta^*,\bu^*)\big\|}{qb_n} + \frac{\delta_{nq}\big\|(\Lb_\eta)_{i,}\big\|}{\sqrt q\,b_n}
    \right\}\quad\text{ for all }\quad i\in[n].
    \label{eq_u_block_error_inequality}
\end{equation}
And we can obtain analogously that, for each $i\in[n]$,
\begin{equation}
    \|\Delta_{u i}\|=\cOp(\delta_{nq});\quad \|\Delta_{u i}\|^4 = \cOp(\delta_{nq}^4).\label{eq_fourth_moment_latent_errors}
\end{equation}
 Applying the conditional Bernstein bounds to \eqref{eq_theta_block_error_inequality} and \eqref{eq_u_block_error_inequality} also yields
\begin{equation}\label{eq_max_block_errors_asym_expand}
\max_{j\in[q]}\|\Delta_{\theta j}\|=\cOp(\bar\delta_{nq}\sqrt{\log q}),\qquad
\max_{i\in[n]}\|\Delta_{u i}\|=\cOp(\bar\delta_{nq}\sqrt{\log n}).
\end{equation}

Next, we compute the Taylor remainders in \eqref{eq_asymp_expand}.
For $t\in[0,1]$, set
\begin{equation*}
    \bgamma_j(t)=\bgamma_j^*+t\Delta_{\gamma j}, \qquad \bu_i(t)=\bu_i^*+t\Delta_{u i}, \qquad  \bu_i^c(t)=\big(1,\bu_i(t)^\T\big)^\T,
\end{equation*}
and let $\eta_{ij}(t)=\bgamma_j(t)^\T\bu_i(t)+\beta_j^*+t\Delta_{\beta j}$. Subsequently, the derivatives with respect to $t$ are 
\begin{equation*}
    \dot\eta_{ij}(t)=\Delta_{\beta j}+\Delta_{\gamma j}^\T\bu_i(t)+\bgamma_j(t)^\T\Delta_{u i},
    \quad\text{ and }\quad
    \ddot\eta_{ij}(t)=2\Delta_{\gamma j}^\T\Delta_{u i}.
\end{equation*}
The second derivatives of $-\ell_{ij}'(\eta_{ij}(t))\bu_i^c(t)$ and $-\ell_{ij}'(\eta_{ij}(t))\bgamma_j(t)$ along this path are
\begin{align*}
    \frac{d^2}{dt^2}\left\{-\ell_{ij}'\big(\eta_{ij}(t)\big)\bu_i^c(t)\right\}
    ={}&-\ell_{ij}^{(3)}\big(\eta_{ij}(t)\big)\dot\eta_{ij}(t)^2\bu_i^c(t)
    -\ell_{ij}''\big(\eta_{ij}(t)\big)\ddot\eta_{ij}(t)\bu_i^c(t)\\
    &-2\ell_{ij}''\big(\eta_{ij}(t)\big)\dot\eta_{ij}(t)
    \begin{pmatrix}0\\\Delta_{u i}\end{pmatrix},\\
    \frac{d^2}{dt^2}\left\{-\ell_{ij}'\big(\eta_{ij}(t)\big)\bgamma_j(t)\right\}
    ={}&-\ell_{ij}^{(3)}\big(\eta_{ij}(t)\big)\dot\eta_{ij}(t)^2\bgamma_j(t)
    -\ell_{ij}''\big(\eta_{ij}(t)\big)\ddot\eta_{ij}(t)\bgamma_j(t)\\
    &-2\ell_{ij}''\big(\eta_{ij}(t)\big)\dot\eta_{ij}(t)\Delta_{\gamma j}.
\end{align*}
By \eqref{eq_max_block_errors_asym_expand} and Assumption~\ref{assump_scaling}, $\max_{j\in[q]}\|\Delta_{\theta j}\|+\max_{i\in[n]}\|\Delta_{u i}\|=o_p(1)$. Uniformly in $t\in[0,1]$, we have 
    $\|\bu_i^c(t)\|\le C(1+\|\bu_i^*\|)$, $\|\bgamma_j(t)\|\le C$, $|\dot\eta_{ij}(t)|\le C\{(1+\|\bu_i^*\|)\|\Delta_{\theta j}\|+\|\Delta_{u i}\|\}$, and $|\ddot\eta_{ij}(t)|\le2\|\Delta_{\theta j}\|\|\Delta_{u i}\|$. Under Assumption~\ref{assumption_smoothness}, the derivatives can be bounded by
\begin{equation*}
    \Big\|\frac{d^2}{dt^2}\left\{-\ell_{ij}'\big(\eta_{ij}(t)\big)\bu_i^c(t)\right\}\Big\|\le CB_n\Big\{(1+\|\bu_i^*\|)^3\|\Delta_{\theta j}\|^2+(1+\|\bu_i^*\|)\|\Delta_{u i}\|^2+(1+\|\bu_i^*\|)^2\|\Delta_{\theta j}\|\|\Delta_{u i}\|\Big\}
\end{equation*}
and
\begin{equation*}
    \Big\|\frac{d^2}{dt^2}\left\{-\ell_{ij}'\big(\eta_{ij}(t)\big)\bgamma_j(t)\right\}\Big\|\le CB_n\Big\{(1+\|\bu_i^*\|)^2\|\Delta_{\theta j}\|^2+\|\Delta_{u i}\|^2+(1+\|\bu_i^*\|)\|\Delta_{\theta j}\|\|\Delta_{u i}\|\Big\},
\end{equation*}
respectively. Hence the integral form of the second-order Taylor formula gives
\begin{align}
    \big\|(\bR_{L\theta})_{\cK_j^+}\big\|
    &\le CB_n\sum_{i=1}^n\Big\{(1+\|\bu_i^*\|)^3\|\Delta_{\theta j}\|^2+(1+\|\bu_i^*\|)\|\Delta_{u i}\|^2+(1+\|\bu_i^*\|)^2\|\Delta_{\theta j}\|\|\Delta_{u i}\|\Big\},
    \label{eq_residual_theta_direct}\\
    \big\|(\bR_{Lu})_{\cK_i}\big\|
    &\le CB_n\sum_{j=1}^q\Big\{(1+\|\bu_i^*\|)^2\|\Delta_{\theta j}\|^2+\|\Delta_{u i}\|^2+(1+\|\bu_i^*\|)\|\Delta_{\theta j}\|\|\Delta_{u i}\|\Big\}.
    \label{eq_residual_u_direct}
\end{align}

Fix a $j\in[q]$, by $n^{-1}\sum_{i=1}^n(1+\|\bu_i^*\|)^m=\cOp(1)$, the first term in \eqref{eq_residual_theta_direct} can be bounded by $\cOp(n)\|\Delta_{\theta j}\|^2=\cOp(n\bar\delta_{nq}^2)$. For the second term, the Cauchy--Schwarz inequality and \eqref{eq_fourth_moment_latent_errors} give $\sum_{i=1}^n(1+\|\bu_i^*\|)\|\Delta_{u i}\|^2 = \cOp(n\bar\delta_{nq}^2)$.
For the third term, $\|\Delta_{\theta j}\|\{\sum_i(1+\|\bu_i^*\|)^4\}^{1/2}\{\sum_i\|\Delta_{u i}\|^2\}^{1/2}=\cOp(\bar\delta_{nq}\times n\delta_{nq})=\cOp(n\bar\delta_{nq}^2)$. This proves $\|(\bR_{L\theta})_{\cK_j^+}\|=\cOp(nB_n\bar\delta_{nq}^2)$. 

For $\|\bR_{L\theta}\|$, note that the three summations in \eqref{eq_residual_theta_direct} do not depend on $j$. With $\sum_{i=1}^n(1+\|\bu_i^*\|)^3=\cOp(n)$, $\sum_{i=1}^n(1+\|\bu_i^*\|)\|\Delta_{u i}\|^2=\cOp(n\bar\delta_{nq}^2)$, $\sum_{i=1}^n(1+\|\bu_i^*\|)^2\|\Delta_{u i}\|=\cOp(n\delta_{nq})$, \eqref{eq_second_moment_block_errors} and \eqref{eq_fourth_moment_block_errors}, squaring \eqref{eq_residual_theta_direct} and summing over $j$ to get
\begin{align*}
    \big\|\bR_{L\theta}\big\|^2\le{}& CB_n^2\Bigg[\Big\{\sum_{i=1}^n(1+\|\bu_i^*\|)^3\Big\}^2\sum_{j=1}^q\|\Delta_{\theta j}\|^4
    +q\Big\{\sum_{i=1}^n(1+\|\bu_i^*\|)\|\Delta_{u i}\|^2\Big\}^2\\
    &\qquad\quad+\Big\{\sum_{i=1}^n(1+\|\bu_i^*\|)^2\|\Delta_{u i}\|\Big\}^2\sum_{j=1}^q\|\Delta_{\theta j}\|^2\Bigg]
    =\cOp\big(B_n^2n^2q\bar\delta_{nq}^4\big).
\end{align*}
With $\max_{j\in[q]}\|\hat\btheta_j - \btheta_j^*\| = \cOp(\bar\delta_{nq}\sqrt{\log q})$ and \eqref{eq_residual_theta_direct}, one can check $\max_{j\in[q]}\|(\bR_{L\theta})_{\cK_j^+}\|=\cOp(nB_n\log q\bar\delta_{nq}^2)$. The bounds for $\bR_{Lu}$ can be obtained analogously from \eqref{eq_residual_u_direct}.

\subsubsection{Proof of Lemma~\ref{lemma_approx_hessian_inverse}}\label{sec_prove_lemma_approx_hessian_inverse}
Lemma~\ref{lemma_approx_hessian_inverse} provides appropriately scaled approximations to the inverse of $\bH_L^* + \bH_P^*$. The proof consists of three steps.
\begin{itemize}
    \item \textbf{Step 1.} We establish that, for some constant $\gamma_L>0$,
    \begin{equation}
        \PP\left\{\lambda_{\min}\big[\bS_m\big(\bH_L^* + \bH_P^*\big)\bS_m\big]\ge \gamma_Lb_n\right\}\to 1 \text{ as }n,q\to\infty.\label{eq_standard_convex_aug2}
    \end{equation}
    \item \textbf{Step 2.} We study the following block form of $\bH_L^*$:
    \begin{equation}
        \bH_L^* + \bH_P^* = \begin{pmatrix}
            \bH_{\theta\theta}^* & \bH_{\theta u}^*\\
            \bH_{u\theta}^* & \bH_{u u}^*
        \end{pmatrix},\label{eq_block_HL_plus_HP}
    \end{equation}
with detailed expressions for the blocks provided in \eqref{eq_block_HL_plus_HP_decomp} below. Using this block structure, we derive an explicit expression for the inverse $\big(\bH_L^* + \bH_P^*\big)^{-1}$ and establish sharp bounds for the matrices appearing in the inversion.
    \item \textbf{Step 3.} Based on the formula for $\big(\bH_L^* + \bH_P^*\big)^{-1}$, we derive $(a)$--$(d)$.
\end{itemize}

\medskip
\noindent\underline{\large \it Executing Step 1.}
This step follows a similar argument in Step 3 in the proof of Lemma~\ref{lemma_convexity2}. For completeness, we provide the details. Start with the matrix $\bar\bH_L^* + \bphi_L^*(\bphi_L^*)^\T$, where 
\begin{equation*}
    \bar\bH_{L}^* = \begin{pmatrix}
        \bH_{L\theta\theta}(\btheta^*,\bu^*) & \bar\bH_{L\theta u}(\btheta^*,\bu^*)\\
        \bar\bH_{Lu\theta}(\btheta^*,\bu^*) & \bH_{Lu u}(\btheta^*,\bu^*)
    \end{pmatrix},
\end{equation*}
and
\begin{equation*}
    \bphi_L^* = \big(\bphi_{01}(\btheta^*),\dots,\bphi_{0K}(\btheta^*),\bphi_{11}(\btheta^*),\dots,\bphi_{1K}(\btheta^*),\dots,\bphi_{K1}(\btheta^*),\dots,\bphi_{KK}(\btheta^*)\big),
\end{equation*}
for some constant $c'\in(0,1)$ used below.
Let $\bphi_{L\theta}^* = (\bphi_L^*)_{1:qK+q,}$ and $\bphi_{Lu}^* = (\bphi_L^*)_{(qK+q+1):(qK+q + nK),}$.
With $\eta_{ij}^* = (\bgamma_j^*)^\T\bu_i^* + \beta_j^*$, the matrix $\bar\bH_L^* + c'b_n\bphi_L^*(\bphi_L^*)^\T$ can be expanded as follows:
\begin{subequations}
    \begin{align}
    &\,\bS_m\bar\bH_L^*\bS_m + c'b_n\bphi_L^*(\bphi_L^*)^\T \nonumber\\=&\, \sum_{i=1}^n\sum_{j=1}^q\big(-\ell_{ij}^{\prime\prime}(\eta_{ij}^*) - c^{\prime} b_n\big)\bS_m\begin{pmatrix}
        \be_j^{(q)}\otimes\begingroup\renewcommand{\arraystretch}{0.7}\begin{pmatrix}
            1\\\bu_i^*
        \end{pmatrix}\endgroup\\\be_i^{(n)}\otimes\bgamma_j^*
    \end{pmatrix}\begin{pmatrix}
        \be_j^{(q)}\otimes\begingroup\renewcommand{\arraystretch}{0.7}\begin{pmatrix}
            1\\\bu_i^*
        \end{pmatrix}\endgroup\\\be_i^{(n)}\otimes\bgamma_j^*
    \end{pmatrix}^\T\bS_m\label{eq_expand_Hl_phi_1}\\&\,+ c^{\prime}b_n\begin{pmatrix}
        \frac{1}{n}\bI_q\otimes \sum_{i=1}^n\begingroup\renewcommand{\arraystretch}{0.7}\begin{pmatrix}
            1&(\bu_i^*)^\T\\\bu_i^*  &\bu_i^*(\bu_i^*)^\T
        \end{pmatrix}\endgroup&\zero\\\zero
        &\frac{1}{q}\bI_n\otimes \sum_{j=1}^q\bgamma_j^*(\bgamma_j^*)^\T
    \end{pmatrix}\label{eq_expand_Hl_phi_2}\\&\,+c'b_n\begin{pmatrix}
        \bphi_{L\theta}^*\big(\bphi_{L\theta}^*\big)^\T&\zero\\\zero & \bphi_{Lu}^*\big(\bphi_{Lu}^*\big)^\T
        \end{pmatrix}\label{eq_expand_Hl_phi_3}.
\end{align}\end{subequations}
The matrix in \eqref{eq_expand_Hl_phi_1} is positive semidefinite because $b_n\le -\ell_{ij}^{\prime\prime}(\eta_{ij}^*)$ uniformly by Assumption~\ref{assumption_smoothness}. Next, the smallest eigenvalue of \eqref{eq_expand_Hl_phi_2} satisfies
\begin{align*}
    \lambda_{\min}\eqref{eq_expand_Hl_phi_2}\ge&\, c^{\prime}b_n\min\left[\lambda_{\min}\left\{\frac{1}{n} \sum_{i=1}^n\begingroup\renewcommand{\arraystretch}{0.7}\begin{pmatrix}
            1&(\bu_i^*)^\T\\\bu_i^*  &\bu_i^*(\bu_i^*)^\T
        \end{pmatrix}\endgroup\right\},\lambda_{\min}\left\{\frac{1}{q}\sum_{j=1}^q\bgamma_j^*(\bgamma_j^*)^\T\right\}\right]\\\ge &\,c^{\prime}b_n\min\left\{\lambda_{\min}\big(\bSigma_u^*\big),\lambda_{\min}\big(\bSigma_{\gamma}^*\big)\right\}/2,
\end{align*}
which is of order $b_n$ by Assumptions~\ref{assump_standard_id} and~\ref{assumption_psd_covariance}. The second inequality holds for all sufficiently large $n$ and $q$. Finally, \eqref{eq_expand_Hl_phi_3} is positive semidefinite. Combining these bounds gives
\begin{equation*}
    \PP\left(\lambda_{\min}\left[\bS_m\left\{\bar\bH_L^* + c'b_n\bphi_L^*(\bphi_L^*)^\T\right\}\bS_m\right]\ge \gamma b_n\right)\to 1
\end{equation*}
as $n,q\to\infty$, for some constant $\gamma>0$. Next, we establish a similar lower bound with $\bar\bH_L^*+c'b_n\bphi^*_L(\bphi_L^*)^\T$ replaced by $\bH_{L}^* + \bH_P^*$. For any $\bx\perp \mathrm{span}\big\{\bphi_{rl}(\btheta^*):r\in\{0\}\cup[K],l\in[K]\big\}$,
\begin{equation*}
    \bx^\T\bS_m\bar\bH_L^*\bS_m\bx = \bx^\T\bS_m\left\{\bar\bH_L^* + c'b_n\bphi_L^*(\bphi_L^*)^\T\right\}\bS_m\bx\ge \gamma b_n\|\bx\|^2.
\end{equation*}
For an arbitrary $\bx\in\RR^{qK+q+nK}$, write $\bx = \by + \bz$, where $\by\in\mathrm{col}(\bphi_L^*)$ and $\bz\perp\mathrm{col}(\bphi_L^*)$. This decomposition is unique, with $\by = \bphi_L^*\big\{(\bphi_L^*)^\T\bphi_L^*\big\}^{-1}(\bphi_L^*)^\T \bx$ and $\bz = \bx - \by$. Since $\bS_m\bar\bH_L^*\bS_m\by = \zero$, we have
\begin{equation}
     \bx ^\T\bS_m\left\{\bH_L^*+\bH_P^*\right\}\bS_m\bx\ge \gamma b_n\|\bz\|^2 + \bx^\T\bS_m\bH_P^*\bS_m\bx - \bx^\T\bS_m(\bar\bH_L^* - \bH_L^*)\bS_m\bx.\label{eq_final_push_to_convexity_true}
\end{equation}
For $\bH_P^* = cb_n\sum_{r=0}^K\sum_{l=1}^K\bpsi_{rl}^*(\bpsi_{rl}^*)^\T$, define the normalized column matrix
\begin{equation*}
\widetilde\bpsi_L^* = \bS_m\big(\bpsi_{01}^*,\dots,\bpsi_{0K}^*,\bpsi_{11}^*,\dots,\bpsi_{1K}^*,\dots,\bpsi_{K1}^*,\dots,\bpsi_{KK}^*\big),
\end{equation*}
so that $\bS_m\bH_P^*\bS_m=cb_n\widetilde\bpsi_L^*(\widetilde\bpsi_L^*)^\T$. Write $\widetilde\bpsi_L^* = \widetilde\bpsi_{Ly}^* + \widetilde\bpsi_{Lz}^*$, where
\begin{align*}
\widetilde\bpsi_{Ly}^*&=\bphi_L^*\{(\bphi_L^*)^\T\bphi_L^*\}^{-1}(\bphi_L^*)^\T\widetilde\bpsi_L^*,\\
\widetilde\bpsi_{Lz}^*&=\left[\bI_{q(K+1)+nK}-\bphi_L^*\{(\bphi_L^*)^\T\bphi_L^*\}^{-1}(\bphi_L^*)^\T\right]\widetilde\bpsi_L^*.
\end{align*}
Then
\begin{equation}
    \bx^\T \widetilde\bpsi_L^*(\widetilde\bpsi_L^*)^\T\bx
    \ge \underbrace{\frac{1}{2}\big\|(\widetilde\bpsi_{Ly}^*)^\T\bx\big\|^2}_{\beta_1} - \underbrace{\big\|(\widetilde\bpsi_{Lz}^*)^\T\bx\big\|^2}_{\beta_2}.\label{eq_final_push_to_convexity2_true}
\end{equation}
To control $\beta_1$ and $\beta_2$, we first bound two quantities:
\begin{equation*}
    \zeta_1 = \sigma_{K^2 + K}\big(\bphi_L^*\big)\quad\text{ and }\quad\zeta_2 := \sigma_{\min}\left[\{(\bphi_L^*)^\T\bphi_L^*\}^{-1}(\bphi_L^*)^\T\widetilde\bpsi_L^*\right].
\end{equation*}
Define $\bar\gamma_{lh}^* = q^{-1}\sum_{j=1}^q\gamma_{jl}^*\gamma_{jh}^*$ and $\bar u_{lh}^{*} = n^{-1}\sum_{i=1}^nu_{il}^{*}u_{ih}^{*}$ for $l,h\in[K]$, $\bar\bGamma^* = q^{-1}\sum_{j=1}^q\bgamma_{j}^*(\bgamma_{j}^*)^\T$, and $\bar \bU^{*} = n^{-1}\sum_{i=1}^n\bu_i^{*}(\bu_i^{*})^\T$. We also set $\bar u_{l0}^* = \bar u_{0l}^* = n^{-1}\sum_{i=1}^nu_{il}^{*}$ and $\bar u_{00}^{*} = 1$. The canonical identifiability condition \eqref{eq_id_st} implies that $\bar u_{lh}^* = \mathds{1}_{\{l=h\}}$. Then 
\begin{equation*}
    (\bphi_L^*)^\T\bphi_L^* = \bI_{K+1}\otimes \bar\bGamma^* + \bI_{K(K+1)}
\end{equation*}
Under Assumption~\ref{assumption_psd_covariance}, we know that for some absolute constant $c_1$, $\lambda_{\min}((\bphi_L^*)^\T\bphi_L^*)\ge c_1^2$ when $n,q$ are large enough. Then $\zeta_1\ge {c_1}$. For $\zeta_2$, note that 
\begin{equation*}
    \zeta_2 \ge \frac{\sigma_{\min}\big\{(\bphi_L^*)^\T\widetilde\bpsi_L^*\big\}}{\sigma_{\max}\big\{(\bphi_L^*)^\T\bphi_L^*\big\}}.
\end{equation*}
Since $\|\btheta^*\|_{\infty}\le C$, we have $\lambda_{\max}((\bphi_L^*)^\T\bphi_L^*) = \cO(1)$. We now study $(\bphi_L^*)^\T\widetilde\bpsi_L^*$. For $r\in\{0\}\cup[K]$ and $l\in[K]$, let $\iota_{rl} = rK + l$ and compute its $\iota_{rl}$th column as follows.
\begin{enumerate}[label=(\arabic*)]
    \item when $r = 0$,
    \begin{equation*}
        (\bphi_L^*)^\T\big(\widetilde\bpsi_L^*\big)_{,\iota_{0l}} = -\be_l^{(K^2+K)};
    \end{equation*}
    \item when $1\le r=l\le K$,
    \begin{equation*}
         (\bphi_L^*)^\T\big(\widetilde\bpsi_L^*\big)_{,\iota_{rr}} = -\be_{\iota_{rr}}^{(K^2+K)};
    \end{equation*}
    \item when $1\le r< l\le K$
    \begin{equation*}
         (\bphi_L^*)^\T\big(\widetilde\bpsi_L^*\big)_{,\iota_{rl}} = -\big\{\be_{\iota_{rl}}^{(K^2+K)} + \be_{\iota_{lr}}^{(K^2+K)}\big\};
    \end{equation*}
    \item when $1\le l<r\le K$, 
    \begin{equation*}
         (\bphi_L^*)^\T\big(\widetilde\bpsi_L^*\big)_{,\iota_{rl}} = \sum_{k=1}^K\bar \gamma_{lk}^*\be_{\iota_{rk}}^{(K^2+K)} - \sum_{k=1}^K\bar \gamma_{rk}^*\be_{\iota_{lk}}^{(K^2+K)};
    \end{equation*}
\end{enumerate}
Let $\bM=(\bphi_L^*)^\T\widetilde\bpsi_L^*$. We show that there exists a constant $c_2>0$ such that $\|\bM\ba\|\ge c_2\|\ba\|$ for every $\ba=(a_{01},\ldots,a_{0K},a_{11},\ldots,a_{1K},\ldots,a_{K1},\ldots,a_{KK})^\T \in\RR^{K^2+K}$. By the column formulas for $\bM$, $\sum_{l=1}^K [\bM\ba]_l^2 = \sum_{l=1}^K a_{0l}^2$.
It suffices to control the coordinates indexed by $\iota_{rl}$ for $r,l\in[K]$. 
Arrange these \(K^2\) coordinates into a \(K\times K\) matrix according to the index \(\iota_{rl}\). Specifically, $\bM\ba$ can be written as
\begin{equation*}
-\Big\{ \sum_{r=1}^K a_{rr}\be_r\be_r^\T + \sum_{1\le r<l\le K} a_{rl}\big(\be_r\be_l^\T+\be_l\be_r^\T\big) \Big\} + \Big\{ \sum_{1\le l<r\le K} a_{rl}\big(\be_r\be_l^\T-\be_l\be_r^\T\big) \Big\}\bar\bGamma^* . \end{equation*}
By $\|\btheta^*\|_\infty\le C$, there exist constants $0<c_\Gamma<C_\Gamma<\infty$ such that, for all sufficiently large \(n,q\), we have $c_\Gamma\le \lambda_{\min}(\bar\bGamma^*)\le \lambda_{\max}(\bar\bGamma^*)\le C_\Gamma$. Subsequently, we have
\[
\begin{aligned}
\sum_{r=1}^K\sum_{l=1}^K [\bM\ba]_{\iota_{rl}}^2  \ge
\lambda_{\min}(\bar\bGamma^*) &\Bigg\| -\Big\{ \sum_{r=1}^K a_{rr}\be_r\be_r^\T + \sum_{1\le r<l\le K} a_{rl}\big(\be_r\be_l^\T+\be_l\be_r^\T\big) \Big\}(\bar\bGamma^*)^{-1/2} \\&\quad
+ \Big\{ \sum_{1\le l<r\le K} a_{rl}\big(\be_r\be_l^\T-\be_l\be_r^\T\big) \Big\}(\bar\bGamma^*)^{1/2} \Bigg\|_{\Fn}^2 .
\end{aligned}
\]
The first matrix inside the norm is symmetric before multiplication by $(\bar\bGamma^*)^{-1/2}$, while the second matrix is skew-symmetric before multiplication by $(\bar\bGamma^*)^{1/2}$. One can then check the Frobenius inner product of these two terms are zero.
Thus,
\begin{align*}
\sum_{r=1}^K\sum_{l=1}^K [\bM\ba]_{\iota_{rl}}^2 &\ge
\frac{\lambda_{\min}(\bar\bGamma^*)}{\lambda_{\max}(\bar\bGamma^*)} \Big\| \sum_{r=1}^K a_{rr}\be_r\be_r^\T + \sum_{1\le r<l\le K} a_{rl}\big(\be_r\be_l^\T+\be_l\be_r^\T\big) \Big\|_{\Fn}^2 \\
&\quad + \lambda_{\min}^2(\bar\bGamma^*) \Big\| \sum_{1\le l<r\le K} a_{rl}\big(\be_r\be_l^\T-\be_l\be_r^\T\big) \Big\|_{\Fn}^2 .
\end{align*}
For the two terms on the right side, $\big\|\sum_{r=1}^K a_{rr}\be_r\be_r^\T + \sum_{1\le r<l\le K} a_{rl}\big(\be_r\be_l^\T+\be_l\be_r^\T\big)\big\|_{\Fn}^2 = \sum_{r=1}^K a_{rr}^2 + 2\sum_{1\le r<l\le K}a_{rl}^2$ and $\big\|\sum_{1\le l<r\le K} a_{rl}\big(\be_r\be_l^\T-\be_l\be_r^\T\big)\big\|_{\Fn}^2= 2\sum_{1\le l<r\le K}a_{rl}^2$, which therefore implies
\begin{equation*}
\sum_{r=1}^K\sum_{l=1}^K [\bM\ba]_{\iota_{rl}}^2 \ge
\frac{c_\Gamma}{C_\Gamma} \Big( \sum_{r=1}^K a_{rr}^2 + \sum_{1\le r<l\le K}a_{rl}^2 \Big) + c_\Gamma^2 \sum_{1\le l<r\le K}a_{rl}^2 .
\end{equation*}
We thus conclude that $\sigma_{\min}\{(\bphi_L^*)^\T\bpsi_L^*\} \ge \left\{\min\left(1,\frac{c_\Gamma}{C_\Gamma},c_\Gamma^2\right)\right\}^{1/2}
=:c_2'>0$. Therefore, $\zeta_2$ can be lower bounded by some positive constant, and for brevity, we simply denote this constant by $c_2$.

    For $\beta_1$, since $\bz\perp\mathrm{col}(\bphi_L^*)$ and $\widetilde\bpsi_{Ly}^*$ is the projection of $\widetilde\bpsi_L^*$ onto $\mathrm{col}(\bphi_L^*)$,
    \begin{equation}
        \beta_1 =\frac{1}{2}\big\|({\widetilde\bpsi_{Ly}^*})^\T\by\big\|^2
        \ge \frac{1}{2}c_1^2c_2^2\|\by\|^2.\label{eq_final_push_to_convexity3_true}
    \end{equation}
    For $\beta_2$, since $\by\in\mathrm{col}(\bphi_L^*)$ and $\widetilde\bpsi_{Lz}^*$ lies in $\mathrm{col}^{\perp}(\bphi_L^*)$, we have
    \begin{equation}
        \beta_2 = \big\|(\widetilde\bpsi_{Lz}^*)^\T\bz\big\|^2 \le \|\widetilde\bpsi_{L}^*\|^2 \|\bz\|^2\le C_3\|\bz\|^2,\label{eq_final_push_to_convexity4_true}
    \end{equation}
    for some absolute constant $C_3$, by the boundedness of the loadings and the canonical empirical moments.
    Combining \eqref{eq_final_push_to_convexity_true}--\eqref{eq_final_push_to_convexity4_true}, we arrive at
    \begin{equation*}
         \bx ^\T\bS_m\left\{\bH_L^*+\bH_P^* \right\}\bS_m\bx\ge \left(\gamma b_n - cb_nC_3\right)\|\bz\|^2 + \frac{cb_n}{2}(c_1c_2)^2\|\by\|^2 - \|\bx\|^2\big\|\bS_m(\bar\bH_{L}^* - \bH_L^*)\bS_m\big\|.
    \end{equation*}
    By the appropriate spectral-norm conclusion of Lemma~\ref{lemma_concentration_l_mat}, the last term is at least
    $-\cOp\{\sqrt{B_n/(n\wedge q)}\}(\|\bz\|^2 + \|\by\|^2)$ and is therefore negligible relative to the preceding terms under Assumption~\ref{assump_scaling}.
    Since $c_1,c_2,C_3$ do not depend on $n$ or $q$, taking $c=\gamma/(2C_3)$ proves \eqref{eq_standard_convex_aug2} for a sufficiently small constant $\gamma_L>0$.

\medskip
\noindent\underline{\large \it Executing Step 2.} Now we study the inversion of the block form given in \eqref{eq_block_HL_plus_HP}. Define
\begin{equation*}
    \bpsi_L^*:=\sqrt{cb_n}\big(\bpsi_{01}^*,\dots,\bpsi_{0K}^*,\bpsi_{11}^*,\dots,\bpsi_{1K}^*,\dots,\bpsi_{K1}^*,\dots,\bpsi_{KK}^*\big),
\end{equation*} and let $\bpsi_{L\theta}^* = \big(\bpsi_{L}^*\big)_{1:qK+q,}$ and $\bpsi_{Lu}^* = \big(\bpsi_{L}^*\big)_{(qK+q+1):(qK+q+nK),}$. The blocks are defined by
\begin{subequations}
    \label{eq_block_HL_plus_HP_decomp}\begin{align}
        &\bH_{\theta\theta}^* = \bH_{L\theta\theta}^* + \bpsi_{L\theta}^*\big(\bpsi_{L\theta}^*\big)^\T;\\
        &\bH_{\theta u}^* = \big(\bH_{u\theta}^*\big)^\T = \bH_{L\theta u}(\btheta^*,\bu^*)+\bpsi_{L\theta}^*\big(\bpsi_{Lu}^*\big)^\T;\\
        &\bH_{uu}^* = \bH_{Luu}^* + \bpsi_{Lu}^*\big(\bpsi_{Lu}^*\big)^\T.
    \end{align}
\end{subequations}
 Here we use
\begin{equation*}
    \bH_P^* = \bpsi_L^*(\bpsi_L^*)^\T = \begin{pmatrix}
        \bpsi_{L\theta}^*(\bpsi_{L\theta}^*)^\T & \bpsi_{L\theta}^*(\bpsi_{Lu}^*)^\T\\
        \bpsi_{Lu}^*(\bpsi_{L\theta}^*)^\T & \bpsi_{Lu}^*(\bpsi_{Lu}^*)^\T 
    \end{pmatrix}
\end{equation*}
Applying the block-matrix inversion formula gives
\begin{align*}
    \big(\bH_L^* + \bH_P^*\big)^{-1} = \begin{pmatrix}
           \bH_{-\theta}^* & -(\bH_{\theta\theta}^*)^{-1} \bH_{\theta u}^*\bH_{-u}^*\\
            -(\bH_{uu}^*)^{-1}\bH_{u\theta}^*\bH_{-\theta}^* & \bH_{-u}^{*}
        \end{pmatrix},
\end{align*}
where $\bH_{-\theta}^*$ and $\bH_{-u}^{*}$ are the inverses of the Schur complement of $\bH_L^* + \bH_P^*$ given as
\begin{align*}
    \bH_{-\theta}^* = \big\{\bH_{\theta\theta}^* - \bH_{\theta u}^*(\bH_{uu}^*)^{-1}\bH_{u\theta}^*\big\}^{-1}= (\bH_{\theta\theta}^*)^{-1} + (\bH_{\theta\theta}^*)^{-1} \bH_{\theta u}^*\bH_{-u}^*\bH_{u \theta}^*(\bH_{\theta\theta}^*)^{-1} ;\\ \bH_{-u}^* =\big\{\bH_{uu}^* - \bH_{u\theta}^*(\bH_{\theta\theta}^*)^{-1}\bH_{\theta u}^*\big\}^{-1} = (\bH_{uu}^*)^{-1} + (\bH_{uu}^*)^{-1}\bH_{u\theta}^*\bH_{-\theta}^*\bH_{\theta u}^*(\bH_{uu}^*)^{-1}.
\end{align*}
For $\bH_{-\theta}^*$ and $\bH_{-u}^*$, by $\bS_m^{-1}(\bH_L^* + \bH_P^*)^{-1}\bS_m^{-1} = \cOp(b_n^{-1})$ with probability approaching $1$ as established in Step 1, we know that the Schur complements satisfy
\begin{equation}
    \|\bH_{-\theta}^*\| = \cOp(n^{-1}b_n^{-1})\text{ and }\|\bH_{-u}^*\| = \cOp(q^{-1}b_n^{-1}).\label{eq_step2_bound_schir}
\end{equation}
We first use the Woodbury identity to decompose $(\bH_{\theta\theta}^*)^{-1} = \big\{\bH_{L\theta\theta}^* + \bpsi_{L\theta}^*\big(\bpsi_{L\theta}^*\big)^\T\big\}^{-1}$ as 
\begin{equation}
    (\bH_{\theta\theta}^*)^{-1} = \big(\bH_{L\theta\theta}^*\big)^{-1} - \big(\bH_{L\theta\theta}^*\big)^{-1}\bpsi_{L\theta}^*\left\{\bI_{K^2+K} + \big(\bpsi_{L\theta}^*\big)^\T\big(\bH_{L\theta\theta}^*\big)^{-1}\bpsi_{L\theta}^*\right\}^{-1}\big(\bpsi_{L\theta}^*\big)^\T\big(\bH_{L\theta\theta}^*\big)^{-1}.\label{eq_wood_expand}
\end{equation}
By the construction of $\bpsi_{L\theta}^*$ and $\|\btheta^*\|_{\infty}\le C$, we have
\begin{equation}
    \big\|\big(\bpsi_{L\theta}^*\big)_{\cK_j^+,}\big\| = \cO(\sqrt{nb_n/q})\text{ and }\big\|\bpsi_{L\theta}^*\big\| = \cO(\sqrt{nb_n}).\label{eq_step2_bound_psi_theta}
\end{equation}
Next, the structure of $\bH_{L\theta\theta}^*$ allows us to focus on the diagonal blocks, where we have the following bounds
\begin{align*}
    \min_{j\in[q]}\lambda_{\min}\big\{(\bH_{L\theta\theta}^*)_{\cK_j^+,\cK_j^+}\big\}=&\,\min_{j\in[q]}\lambda_{\min}\left\{\sum_{i=1}^n-\ell_{ij}^{\prime\prime}\big(\eta_{ij}^*\big)\begingroup\renewcommand{\arraystretch}{0.7}\begin{pmatrix}1\\\bu_i^{*}\end{pmatrix}\begin{pmatrix}1\\\bu_i^{*}\end{pmatrix}^\T\endgroup\right\}\\\overset{(i)}{\ge}&\,b_n\lambda_{\min}\left\{\sum_{i=1}^n\begingroup\renewcommand{\arraystretch}{0.7}\begin{pmatrix}1\\\bu_i^{*}\end{pmatrix}\begin{pmatrix}1\\\bu_i^{*}\end{pmatrix}^\T\endgroup\right\}\\\overset{(ii)}{=}&\,\Theta_p(b_nn).
\end{align*}
and 
\begin{align*}
    \max_{j\in[q]}\lambda_{\max}\big\{(\bH_{L\theta\theta}^*)_{\cK_j^+,\cK_j^+}\big\}=&\,\max_{j\in[q]}\lambda_{\max}\left\{\sum_{i=1}^n-\ell_{ij}^{\prime\prime}\big(\eta_{ij}^*\big)\begingroup\renewcommand{\arraystretch}{0.7}\begin{pmatrix}1\\\bu_i^{*}\end{pmatrix}\begin{pmatrix}1\\\bu_i^{*}\end{pmatrix}^\T\endgroup\right\}\\\overset{(iii)}{\le}&\,B_n\lambda_{\max}\left\{\sum_{i=1}^n\begingroup\renewcommand{\arraystretch}{0.7}\begin{pmatrix}1\\\bu_i^{*}\end{pmatrix}\begin{pmatrix}1\\\bu_i^{*}\end{pmatrix}^\T\endgroup\right\}\\\overset{(iv)}{=}&\,\cOp(B_nn).
\end{align*}
Here $(i)$ and $(iii)$ follow from Assumption~\ref{assumption_smoothness} and $|\eta_{ij}^*|\le C\sqrt{\log n}$ uniformly over $i\in[n],j\in[q]$; $(ii)$ and $(iv)$ hold for sufficiently large $n$ and $q$ by Assumption~\ref{assump_standard_id}. Consequently, by the block-diagonal structure, for $\bH_{L\theta_j\theta_j}^* :=(\bH_{L\theta\theta}^*)_{\cK_j^+,\cK_j^+}$, one has
\begin{equation}
    \Big\|\bH_{L\theta_j\theta_j}^*\Big\| = \cOp(nB_n);\; \Big\|\big(\bH_{L\theta_j\theta_j}^*\big)^{-1}\Big\| = \cOp((nb_n)^{-1}),\label{eq_step2_bound_hl_theta_j}
\end{equation}
and for $\bH_{L\theta\theta}^*$, one has
\begin{equation}
    \Big\|\bH_{L\theta\theta}^*\Big\| = \cOp(nB_n);\; \Big\|\big(\bH_{L\theta\theta}^*\big)^{-1}\Big\| = \cOp((nb_n)^{-1}).\label{eq_step2_bound_hl_theta}
\end{equation}
Now go back to \eqref{eq_wood_expand} and we first note that 
\begin{equation*}
    \left\|\big(\bpsi_{L\theta}^*\big)^\T\big(\bH_{L\theta\theta}^*)^{-1}\bpsi_{L\theta}^*\right\| \le \big\|\bpsi_{L\theta}^*\big\|^2 \big\|\big(\bH_{L\theta\theta}^*\big)^{-1}\big\| = \cOp(1).
\end{equation*}
Therefore, $\Theta_p(1)\le \lambda_{\min}\big(\bI_{K^2+K} + \big(\bpsi_{L\theta}^*\big)^\T\big(\bH_{L\theta\theta}^*\big)^{-1}\bpsi_{L\theta}^*\big)\le \lambda_{\max}\big(\bI_{K^2+K} + \big(\bpsi_{L\theta}^*\big)^\T\big(\bH_{L\theta\theta}^*\big)^{-1}$ $\bpsi_{L\theta}^*\big) \le \cOp(1)$, which gives \begin{equation}
    \left\|\big(\bI_{K^2+K} + \big(\bpsi_{L\theta}^*\big)^\T\big(\bH_{L\theta\theta}^*\big)^{-1}\bpsi_{L\theta}^*\big)^{-1}\right\| = \cOp(1).\label{eq_bound_wood_inv}
\end{equation}
With \eqref{eq_step2_bound_psi_theta}, \eqref{eq_step2_bound_hl_theta_j} and \eqref{eq_step2_bound_hl_theta}, we arrive at
\begin{equation}
    \begin{aligned}\left\|\left\{(\bH_{\theta\theta}^*)^{-1} - \big(\bH_{L\theta\theta}^*\big)^{-1}\right\}_{\cK_j^+,}\right\| \le &\, \big\|\big(\bH_{L\theta_j\theta_j}^*\big)^{-1}\big\|\big\|\{\bpsi_{L\theta}^*\}_{\cK_j^+,}\big\|\big\|\bpsi_{L\theta}^*\big\|\big\|\big(\bH_{L\theta\theta}^*\big)^{-1}\big\|\\&\quad\times \left\|\left\{\bI_{K^2+K} + \big(\bpsi_{L\theta}^*\big)^\T\big(\bH_{L\theta\theta}^*\big)^{-1}\bpsi_{L\theta}^*\right\}^{-1}\right\|\\= &\,\cOp\left(\frac{1}{n\sqrt{q}b_n}\right),\end{aligned}\label{eq_step2_bound_hl_theta_inv_j}
\end{equation}
and 
\begin{equation}
    \begin{aligned}\big\|(\bH_{\theta\theta}^*)^{-1}\big\| \le &\,\big\|\big(\bH_{L\theta\theta}^*\big)^{-1}\big\| +\big\| \big(\bH_{L\theta\theta}^*\big)^{-1}\big\|^2\big\|\bpsi_{L\theta}^*\big\|^2\\&\quad\times\left\|\left\{\bI_{K^2+K} + \big(\bpsi_{L\theta}^*\big)^\T\big(\bH_{L\theta\theta}^*\big)^{-1}\bpsi_{L\theta}^*\right\}^{-1}\right\|= \cOp\left(\frac{1}{nb_n}\right).\end{aligned}\label{eq_step2_bound_hl_theta_inv}
\end{equation}
Next, we note that
\begin{equation}
    \big\|\big(\bH_{\theta u}^*\big)_{\cK_j^+,}\big\| = \cOp(B_n\sqrt{n}),\;\big\|\big(\bH_{\theta u}^*\big)_{,\cK_i}\big\| = \cOp(B_n\sqrt{q}),\text{ and }\big\|\bH_{\theta u}^*\big\| = \cOp(B_n\sqrt{nq}).\label{eq_step2_bound_hl_thetau}
\end{equation}
One can verify the bounds by (i) $\bH_{\theta u}^* = \bar\bH_{L\theta u}(\btheta^*,\bu^*) + \bJ_{L\theta u}(\btheta^*,\bu^*) + \bpsi_{L\theta}^*(\bpsi_{Lu}^*)^\T$; (ii) each entry in $\bar\bH_{L\theta u}(\btheta^*,\bu^*) $ is bounded by $\max_{i\in[n],j\in[q]}|\ell_{ij}^{\prime\prime}(\eta_{ij}^*)|\|\btheta_j^*\|_{\infty}(1+\|\bu_i^*\|_{\infty}) = \cOp(B_n\sqrt{\log n})$ by Assumptions~\ref{assump_standard_id}--\ref{assumption_smoothness}; (iii) $\bJ_{L\theta u}(\btheta^*,\bu^*)$ can be bounded according to Lemma~\ref{lemma_concentration_l_mat}; and (iv) \eqref{eq_step2_bound_psi_theta}.

Invoke the expansion for $\bH_{-\theta}^*$ and with \eqref{eq_step2_bound_schir}, \eqref{eq_step2_bound_hl_theta_inv_j}, \eqref{eq_step2_bound_hl_theta_inv} and \eqref{eq_step2_bound_hl_thetau}, we arrive at
\begin{align*}
&\,\left\|\left\{\bH_{-\theta}^* - \big(\bH_{L\theta\theta}^*\big)^{-1}\right\}_{\cK_j^+,}\right\| = \left\|\left\{(\bH_{\theta\theta}^*)^{-1} - \big(\bH_{L\theta\theta}^*\big)^{-1}\right\}_{\cK_j^+,}\right\|\\&\,+\left\|\big\{(\bH_{\theta\theta}^*)^{-1}\big\}_{\cK_j^+,} \bH_{\theta u}^*\bH_{-u}^*\bH_{u \theta}^*(\bH_{\theta\theta}^*)^{-1}\right\|\\=&\,\cOp\big(b_n^{-1}n^{-1}q^{-1/2}\big) + \left\|\left\{(\bH_{\theta\theta}^*)^{-1} - \big(\bH_{L\theta\theta}^*\big)^{-1}\right\}_{\cK_j^+,}\bH_{\theta u}^*\bH_{-u}^*\bH_{u \theta}^*(\bH_{\theta\theta}^*)^{-1}\right\|\\&\,+ \left\|\left\{ \big(\bH_{L\theta\theta}^*\big)^{-1}\right\}_{\cK_j^+,}\bH_{\theta u}^*\bH_{-u}^*\bH_{u \theta}^*(\bH_{\theta\theta}^*)^{-1}\right\|\\=&\, \cOp\Big(n^{-1}q^{-1/2}\frac{B_n^2}{b_n^3}\Big) + \left\| \big(\bH_{L\theta_j\theta_j}^*\big)^{-1}\big(\bH_{\theta u}^*\big)_{\cK_j^+,}\bH_{-u}^*\bH_{u \theta}^*(\bH_{\theta\theta}^*)^{-1}\right\|\\=&\,\cOp\left(\frac{\kappa_n^3\sqrt{B_n}}{n\sqrt{q}}\right).\end{align*}
Similarly, we can obtain a similar result for $\bH_{uu}^*$ and $\bH_{-u}^*$. To summarize, we have for each $i\in[n]$ and $j\in[q]$ that
    \begin{subequations}\label{subeq_step2_bound}
    \begin{align}
       & \Big\|\big\{(\bH_{\theta\theta}^*)^{-1} - \big(\bH_{L\theta\theta}^*\big)^{-1}\big\}_{\cK_j^+,}\Big\| = \cOp\big(b_n^{-1}n^{-1}q^{-1/2}\big);\label{subeq_step2_bound_hl_theta_inv_j}\\
       &\Big\|\big\{\bH_{-\theta}^* - \big(\bH_{L\theta\theta}^*\big)^{-1}\big\}_{\cK_j^+,}\Big\| = \cOp\big(\kappa_n^3\sqrt{B_n}n^{-1}q^{-1/2}\big);\label{subeq_step2_bound_hl_m_theta_inv_j}\\
       &\big\|(\bH_{L\theta_j\theta_j}^*)^{-1}\big\|  = \cOp\big((b_nn)^{-1}\big);\;\big\|(\bH_{\theta\theta}^*)^{-1}\big\|  = \cOp\big((b_nn)^{-1}\big);\label{subeq_step2_bound_hl_m_theta_inv}\\&\;\big\|\bH_{-\theta}^*\big\| = \cOp\big((b_nn)^{-1}\big);\nonumber\\
       & \Big\|\big\{(\bH_{uu}^*)^{-1} - \big(\bH_{Luu}^*\big)^{-1}\big\}_{\cK_i,}\Big\| = \cOp\big(b_n^{-1}n^{-1/2}q^{-1}\big);\label{subeq_step2_bound_hl_u_inv_j}\\
       &\Big\|\big\{\bH_{-u}^* - \big(\bH_{Luu}^*\big)^{-1}\big\}_{\cK_i,}\Big\| = \cOp\big(\kappa_n^3\sqrt{B_n}n^{-1/2}q^{-1}\big);\label{subeq_step2_bound_hl_m_u_inv_j}\\
       &\big\|(\bH_{Lu_iu_i}^*)^{-1}\big\|  = \cOp\big((qb_n)^{-1}\big);\label{subeq_step2_bound_hl_m_u_inv}\\&\;\big\|(\bH_{uu}^*)^{-1}\big\|  = \cOp\big((qb_n)^{-1}\big);\;\big\|\bH_{-u}^*\big\| = \cOp\big((qb_n)^{-1}\big);\nonumber\\
       &\big\|\big(\bH_{\theta u}^*\big)_{\cK_j^+,}\big\| = \cOp(B_n\sqrt{n}),\;\big\|\big(\bH_{\theta u}^*\big)_{,\cK_i}\big\| = \cOp(B_n\sqrt{q});\label{subeq_step2_bound_hl_thetau}\\&\;\big\|\bH_{\theta u}^*\big\| = \cOp(B_n\sqrt{nq}).\nonumber
    \end{align}
\end{subequations}

\medskip
\noindent\underline{\large \it Executing Step 3. } Now we use \eqref{subeq_step2_bound} to derive the results. Recall that
\begin{align*}
    \big(\bH_L^* + \bH_P^*\big)^{-1} = \begin{pmatrix}
           \bH_{-\theta}^* & -(\bH_{\theta\theta}^*)^{-1} \bH_{\theta u}^*\bH_{-u}^*\\
            -(\bH_{uu}^*)^{-1}\bH_{u\theta}^*\bH_{-\theta}^* & \bH_{-u}^{*}
        \end{pmatrix},
\end{align*}
For $(a)$, we note that 
\begin{align*}
    &\,\Big\|\bS_m^{-1}\big\{(\bH_{L}^* + \bH_P^*)^{-1}\bS_m^{-1} - n(\bH_{L\theta\theta}^*)^{-1}\big\}_{\cK_j^+,}\Big\| \\\le&\, \Big\|n\big\{\bH_{-\theta}^* - (\bH_{L\theta\theta}^*)^{-1}\big\}_{\cK_j^+,}\Big\| + \Big\|(nq)^{1/2}\big\{(\bH_{\theta\theta}^*)^{-1} \bH_{\theta u}^*\bH_{-u}^*\big\}_{\cK_j^+,}\Big\| \\\overset{(i)}{\le }&\,\cOp\big(\kappa_n^3\sqrt{B_n}q^{-1/2}\big) + (nq)^{1/2}\Big\|\big\{(\bH_{\theta\theta}^*)^{-1} - (\bH_{L\theta\theta}^*)^{-1} \big\}_{\cK_j^+,}\bH_{\theta u}^*\bH_{-u}^*\Big\|\\&\,+(nq)^{1/2} \Big\|\big\{(\bH_{L\theta\theta}^*)^{-1}\big\}_{\cK_j^+,} \bH_{\theta u}^*\bH_{-u}^*\Big\|\\\overset{(ii)}{\le }&\,\cOp\big(\kappa_n^3\sqrt{B_n}q^{-1/2}\big) + (nq)^{1/2}\Big\|(\bH_{L\theta_j\theta_j}^*)^{-1}(\bH_{\theta u}^*)_{\cK_j^+,}\bH_{-u}^*\Big\|\\ \overset{(iii)}{=} &\,\cOp\big(\kappa_n^3\sqrt{B_n}q^{-1/2}\big).
\end{align*}
Here $(i)$ follows from \eqref{subeq_step2_bound_hl_m_theta_inv_j}; $(ii)$ follows from \eqref{subeq_step2_bound_hl_theta_inv_j}, \eqref{subeq_step2_bound_hl_m_u_inv}, and \eqref{subeq_step2_bound_hl_thetau}; and $(iii)$ follows from \eqref{subeq_step2_bound_hl_m_theta_inv}, \eqref{subeq_step2_bound_hl_m_u_inv}, and \eqref{subeq_step2_bound_hl_thetau}. Part~$(c)$ of Lemma~\ref{lemma_approx_hessian_inverse} follows similarly.

Next, for part $(b)$, we introduce the following lemma.
\begin{lemma}\em\label{lemma_aux_score_bound}
    Suppose Assumptions~\ref{assump_standard_id}--\ref{assumption_smoothness} hold.
    \begin{enumerate}[label = $(\roman*)$]
       \item For any $r,l\in[K]$ \begin{equation*}
    \Big|\sum_{i=1}^n u_{il}^* \be_r^\T\Big\{\sum_{t=1}^q \ell_{it}^{\prime\prime}\big(\eta_{it}^*\big) \bgamma_t^* (\bgamma_t^*)^\T \Big\}^{-1} \Big\{\sum_{t=1}^q \bgamma_t^* \ell_{it}^{\prime}(\eta_{it}^*) \Big\}\Big| =\cOp\big( \kappa_n\sqrt{n/q}\big);
\end{equation*}

\item For any $r\in[K]$, 
\begin{align*}
    \sum_{j=1}^q\Big|\sum_{i=1}^n \ell_{ij}^{\prime}(\eta_{ij}^*) \be_r^\T \Big\{\sum_{t=1}^q \ell_{it}^{\prime\prime}(\eta_{it}^*) \bgamma_t^* (\bgamma_t^*)^\T \Big\}^{-1} \Big\{\sum_{t=1}^q \bgamma_t^* \ell_{it}^{\prime}(\eta_{it}^*) \Big\}\Big|^2\\=\cOp \Big( \frac{B_n^2q}{b_n^2}\big(\sqrt{n /q} + n/q\big)^2\Big);
\end{align*}

\item For any $r\in[K]$, 
\begin{align*}
    \sum_{j=1}^q\Big|\sum_{i=1}^n \ell_{ij}^{\prime\prime}(\eta_{ij}^*) u_{ir}^* (\bgamma_j^*)^\T \Big\{\sum_{t=1}^q \ell_{it}^{\prime\prime}(\eta_{it}^*) \bgamma_t^* (\bgamma_t^*)^\T \Big\}^{-1} \Big\{\sum_{t=1}^q \bgamma_t^* \ell_{it}^{\prime}(\eta_{it}^*) \Big\}\Big|^2 \\= \cOp\big( B_n^2\kappa_n^2{n^2/q}\big);
\end{align*}

\item For any $r,l\in[K]$, 
\begin{align*}
    \Big|\sum_{j=1}^q \gamma_{jl}^* \be_r^\T \Big\{\sum_{t=1}^n \ell_{tj}^{\prime\prime}(\eta_{tj}^*) \bu_t^* (\bu_t^*)^\T \Big\}^{-1} \Big\{\sum_{t=1}^n \bu_t^* \ell_{tj}^{\prime}(\eta_{tj}^*) \Big\}\Big| = \cOp\big(\kappa_n\sqrt{q/n}\big);
\end{align*}

\item For any $r\in[K]$,
\begin{align*}
 \sum_{i=1}^n\Big|\sum_{j=1}^q \ell_{ij}^{\prime}(\eta_{ij}^*) \be_r^\T \Big\{\sum_{t=1}^n \ell_{tj}^{\prime\prime}(\eta_{tj}^*) \bu_t^* (\bu_t^*)^\T \Big\}^{-1} \Big\{\sum_{t=1}^n \bu_t^* \ell_{tj}^{\prime}(\eta_{tj}^*) \Big\}\Big|^2 \\ =\cOp \Big(\frac{B_n^2n}{b_n^2}\big(\sqrt{q /n} + q/n \big)^2\Big);
\end{align*}
\item For any $r\in[K]$,
\begin{align*}
    \sum_{i=1}^n\Big|\sum_{j=1}^q \ell_{ij}^{\prime\prime}(\eta_{ij}^*) \gamma_{jr}^* (\bu_i^*)^\T \Big\{\sum_{t=1}^n \ell_{tj}^{\prime\prime}(\eta_{tj}^*) \bu_t^* (\bu_t^*)^\T \Big\}^{-1} \Big\{\sum_{t=1}^n \bu_t^* \ell_{tj}^{\prime}(\eta_{tj}^*) \Big\}\Big|^2 \\= \cOp\big(\kappa_n^2B_n^2{q^2/n}\big);
\end{align*}
\end{enumerate}
\end{lemma}
\begin{proof}
    These results can be verified similar to those in the proof of Lemma~\ref{lemma_aux_score_bound_pre}.
\end{proof}

We now prove part $(b)$ of Lemma~\ref{lemma_approx_hessian_inverse}. Let $\bS_{L\theta}^* = \{\bS_L(\btheta^*,\bu^*)\}_{1:qK+q}$, $\bS_{L\theta_j}^* = (\bS_{L\theta}^*)_{\cK_j^+,}$, $\bS_{Lu}^* = \{\bS_L(\btheta^*,\bu^*)\}_{(qK+q+1):(qK+q+nK)}$, and $\bS_{Lu_i}^* = (\bS_{Lu}^*)_{\cK_i,}$. For each $i\in[n]$ and $j\in[q]$,
\begin{equation*}
\big\|\bS_{L\theta_j}^*\big\| = \cOp\big(\sqrt{nB_n}\big);\;    \big\|\bS_{L\theta}^*\big\| = \cOp\big(\sqrt{nqB_n}\big);\;\big\|\bS_{Lu_i}^*\big\| = \cOp\big(\sqrt{qB_n}\big);\;    \big\|\bS_{Lu}^*\big\| = \cOp\big(\sqrt{nqB_n}\big),
\end{equation*}
by Lemma~\ref{coro_first_order_concern}.
Consequently, for part $(b)$, we calculate
\begin{align*}
    &\,\Big\|\big\{(\bH_{L}^* + \bH_P^*)^{-1} - (\bH_{L\theta\theta}^*)^{-1}\big\}_{\cK_j^+,}\bS_L(\btheta^*,\bu^*)\Big\| \\\le&\,\underbrace{\Big\|\big\{\bH_{-\theta}^* - (\bH_{L\theta\theta}^*)^{-1}\big\}_{\cK_j^+,}\bS_{L\theta}^*\Big\|}_{\zeta_1} + \underbrace{\Big\|\big\{(\bH_{\theta\theta}^*)^{-1} \bH_{\theta u}^*\bH_{-u}^*\big\}_{\cK_j^+,}\bS_{Lu}^*\Big\|}_{\zeta_2} .
\end{align*}
For $\zeta_1$, invoke the expansion of $\bH_{-\theta}^*$ and we know
\begin{align*}
\zeta_1\le \underbrace{\Big\|\big\{(\bH_{\theta\theta}^*)^{-1} - (\bH_{L\theta\theta}^*)^{-1}\big\}_{\cK_j^+,}\bS_{L\theta}^*\Big\|}_{\eta_1} + \underbrace{\Big\|\big\{(\bH_{\theta\theta}^*)^{-1} \bH_{\theta u}^*\bH_{-u}^*\bH_{u \theta}^*(\bH_{\theta\theta}^*)^{-1}\big\}_{\cK_j^+,}\bS_{L\theta}^*\Big\|}_{\eta_2}.
\end{align*}
For $\eta_1$, invoke the expansion of $\bH_{\theta\theta}^*$ and we know
\begin{align*}
    \eta_1 \le&\, \Big\|\big\{\big(\bH_{L\theta\theta}^*\big)^{-1} - (\bH_{L\theta\theta}^*)^{-1}\big\}_{\cK_j^+,}\bS_{L\theta}^*\Big\| \\&\, +\Big\|\big(\bH_{L\theta_j\theta_j}^*\big)^{-1}\big(\bpsi_{L\theta}^*\big)_{\cK_j^+,}\left\{\bI_{K^2+K} + \big(\bpsi_{L\theta}^*\big)^\T\big(\bH_{L\theta\theta}^*\big)^{-1}\bpsi_{L\theta}^*\right\}^{-1}\big(\bpsi_{L\theta}^*\big)^\T\big(\bH_{L\theta\theta}^*\big)^{-1}\bS_{L\theta}^*\Big\| \\\overset{(i)}{\le} &\, \cOp\big((nqb_n)^{-1/2}\big) \Big\|\big(\bpsi_{L\theta}^*\big)^\T\big(\bH_{L\theta\theta}^*\big)^{-1}\bS_{L\theta}^*\Big\|\\\overset{(ii)}{\le }&\,\cOp\big(\kappa_n(nq)^{-1/2}\big).
\end{align*}
Here $(i)$ is due to \eqref{eq_step2_bound_psi_theta}, \eqref{eq_step2_bound_hl_theta_j}, and \eqref{eq_bound_wood_inv}; $(ii)$ is due to $(iv)$ of Lemma~\ref{lemma_aux_score_bound}. For $\eta_2$, we compute
\begin{align*}
    \eta_2\le&\, \Big\|\big\{(\bH_{\theta\theta}^*)^{-1} - (\bH_{L\theta\theta}^*)^{-1}\big\}_{\cK_j^+,} \bH_{\theta u}^*\bH_{-u}^*\bH_{u \theta}^*(\bH_{\theta\theta}^*)^{-1}\bS_{L\theta}^*\Big\|\\&\,+\Big\|\big\{(\bH_{L\theta\theta}^*)^{-1}\big\}_{\cK_j^+,} \bH_{\theta u}^*\bH_{-u}^*\bH_{u \theta}^*(\bH_{\theta\theta}^*)^{-1}\bS_{L\theta}^*\Big\|\\\le&\, \Big\|\big\{(\bH_{\theta\theta}^*)^{-1} - (\bH_{L\theta\theta}^*)^{-1}\big\}_{\cK_j^+,} \bH_{\theta u}^*\bH_{-u}^*\bH_{u \theta}^*(\bH_{L\theta\theta}^*)^{-1}\bS_{L\theta}^*\Big\|\\&\,+ \Big\|\big\{(\bH_{\theta\theta}^*)^{-1} - (\bH_{L\theta\theta}^*)^{-1}\big\}_{\cK_j^+,} \bH_{\theta u}^*\bH_{-u}^*\bH_{u \theta}^*\big\{(\bH_{\theta\theta}^*)^{-1} - (\bH_{L\theta\theta}^*)^{-1}\big\}\bS_{L\theta}^*\Big\|\\&\,+\Big\|(\bH_{L\theta_j\theta_j}^*)^{-1} (\bH_{\theta u}^*)_{\cK_j^+,}\bH_{-u}^*\bH_{u \theta}^*(\bH_{L\theta\theta}^*)^{-1}\bS_{L\theta}^*\Big\|\\&\,+ \Big\|(\bH_{L\theta_j\theta_j}^*)^{-1} (\bH_{\theta u}^*)_{\cK_j^+,}\bH_{-u}^*\bH_{u \theta}^*\big\{(\bH_{\theta\theta}^*)^{-1} - (\bH_{L\theta\theta}^*)^{-1}\big\}\bS_{L\theta}^*\Big\|.
\end{align*}
Similar to deriving $\eta_1 = \cOp\big(\kappa_n(nq)^{-1/2}\big)$, we can also show that $\big\|\big\{(\bH_{\theta\theta}^*)^{-1} - (\bH_{L\theta\theta}^*)^{-1}\big\}\bS_{L\theta}^*\big\| = \cOp(\kappa_nn^{-1/2})$, as the bound for $\big\|\big\{(\bH_{\theta\theta}^*)^{-1} - (\bH_{L\theta\theta}^*)^{-1}\big\}_{\cK_j^+,}\big\|$ is uniform over $j\in[q]$. By $(v)$ and $(vi)$ of Lemma~\ref{lemma_aux_score_bound}, we know that
\begin{equation}
    \label{eq_combined_est}\big\|\bH_{u \theta}^*(\bH_{L\theta\theta}^*)^{-1}\bS_{L\theta}^*\big\| = \cOp(\kappa_nB_nq^{1/2} + B_nq/(\sqrt{n}b_n)).
\end{equation}
Combined with \eqref{subeq_step2_bound_hl_theta_inv_j}, \eqref{subeq_step2_bound_hl_m_theta_inv}, \eqref{subeq_step2_bound_hl_m_u_inv}, \eqref{subeq_step2_bound_hl_thetau}, and Lemma~\ref{coro_first_order_concern}, we arrive at
\begin{equation*}
    \zeta_1 = \cOp\Big(\frac{\kappa_n^2B_n}{\sqrt{nq}} +  \frac{\kappa_n^3\sqrt{B_n}}{n}\Big).
\end{equation*}
Similarly, one can derive that 
\begin{equation}
    \Big\|\big\{\bH_{-u}^* - (\bH_{Luu}^*)^{-1}\big\}\bS_{Lu}^*\Big\| = \cOp\big(\kappa_nq^{-1/2}\big),\label{eq_step3_bound_eta_3}
\end{equation}
For $\zeta_2$, we note that
\begin{align*}
    \zeta_2 \le &\, \underbrace{\Big\|\big\{(\bH_{\theta\theta}^*)^{-1}\bH_{\theta u}^*\big\}_{\cK_j^+,}\big\{\bH_{-u}^* - (\bH_{Luu}^*)^{-1}\big\}\bS_{Lu}^*\Big\|}_{\eta_3} \\&\,+ \underbrace{\Big\|\big\{(\bH_{\theta\theta}^*)^{-1} - (\bH_{L\theta\theta}^*)^{-1}\big\}_{\cK_j^+,}\bH_{\theta u}^* (\bH_{Luu}^*)^{-1}\bS_{Lu}^*\Big\|}_{\eta_4}\\&\,+ \underbrace{\Big\|(\bH_{L\theta_j\theta_j}^*)^{-1}\big(\bH_{\theta u}^*\big)_{\cK_j^+,} (\bH_{Luu}^*)^{-1}\bS_{Lu}^*\Big\|}_{\eta_5}.
\end{align*}
By \eqref{subeq_step2_bound_hl_theta_inv_j} and \eqref{subeq_step2_bound_hl_thetau}, one has
\begin{align*}
    \Big\|\big\{(\bH_{\theta\theta}^*)^{-1}\bH_{\theta u}^*\big\}_{\cK_j^+,}\Big\| \le&\, \Big\|\big\{(\bH_{L\theta\theta}^*)^{-1}-(\bH_{\theta\theta}^*)^{-1}\big\}_{\cK_j^+,}\bH_{\theta u}^*\Big\| + \Big\|\big\{(\bH_{L\theta\theta}^*)^{-1}\bH_{\theta u}^*\big\}_{\cK_j^+,}\Big\|\\=&\,\cOp\big(B_n/(b_n\sqrt{n})\big) + \Big\|\big\{(\bH_{L\theta_j\theta_j}^*)^{-1}\big(\bH_{\theta u}^*\big)_{\cK_j^+,}\Big\|\\=&\,\cOp\big(B_n/(b_n\sqrt{n})\big).
\end{align*}
Combining this bound with \eqref{eq_step3_bound_eta_3} gives $\eta_3 = \cOp\big(\kappa_n^2\sqrt{B_n}(nq)^{-1/2}\big)$. For $\eta_4$, parts~$(ii)$ and~$(iii)$ of Lemma~\ref{lemma_aux_score_bound} give, similarly to \eqref{eq_combined_est},
\begin{equation*}
    \big\|\bH_{ \theta u}^*(\bH_{Luu}^*)^{-1}\bS_{Lu}^*\big\| = \cOp(\kappa_nB_nn^{1/2} + B_nn/(\sqrt{q}b_n)).
\end{equation*}
Then with \eqref{subeq_step2_bound_hl_theta_inv_j}, we conclude that $\eta_4 = \cOp\big((nq)^{-1/2}\kappa_n^2(\sqrt{B_n} + \sqrt{n/q})\big)$. Finally, for $\eta_5$, with \eqref{subeq_step2_bound_hl_m_theta_inv} and $\big\|\big(\bH_{\theta u}^*\big)_{\cK_j^+,} (\bH_{Luu}^*)^{-1}\bS_{Lu}^*\big\| = \cOp(\kappa_nB_n\sqrt{n/q} + B_nn/(qb_n))$, we know $\eta_5 = \cOp\big((nq)^{-1/2}\kappa_n^2(\sqrt{B_n} + \sqrt{n/q})\big)$. Then we arrive at 
\begin{equation}
    \zeta_2 = \cOp\Big(\frac{\kappa_n^2B_n}{\sqrt{nq}} +  \frac{\kappa_n^2}{q}\Big).
\end{equation}
Combining the estimates for $\zeta_1$ and $\zeta_2$ proves part~(b) of Lemma~\ref{lemma_approx_hessian_inverse}. Part~(d) follows by symmetry.

\subsubsection{Proof of Lemma~\ref{lemma_bound_aug_via_id}}\label{sec_prove_lemma_bound_aug_via_id}
For each $j\in[q]$, one can compute
\begin{align*}
    \big\{\bH_P^*\big\}_{\cK_j^+,}\begingroup\renewcommand{\arraystretch}{0.6}\begin{pmatrix}
            \hat\btheta - \btheta^*\\\hat\bu^{\star} - \bu^{*}
        \end{pmatrix}\endgroup =&\, \sum_{r=0}^K\sum_{l=1}^Kcb_n\big\{\bpsi_{rl}^*\big\}_{\cK_j^+}\big(\bpsi_{rl}^*\big)^\T\begingroup\renewcommand{\arraystretch}{0.6}\begin{pmatrix}
            \hat\btheta - \btheta^*\\\hat\bu^{\star} - \bu^{*}
        \end{pmatrix}\endgroup\\
        \overset{(i)}{=}&\, \sum_{r=1}^K\sum_{1\le l<r}cb_n\big\{\bpsi_{rl}^*\big\}_{\cK_j^+}\big(\bpsi_{rl}^*\big)^\T\begingroup\renewcommand{\arraystretch}{0.6}\begin{pmatrix}
            \hat\btheta - \btheta^*\\\hat\bu^{\star} - \bu^{*}
        \end{pmatrix}\endgroup\\
         = &\, \sum_{r=1}^K\sum_{1\le l<r}cb_n\big\{\bpsi_{rl}^*\big\}_{\cK_j^+}\sqrt{n/q}\sum_{s=1}^q\left\{\gamma_{sl}^*(\hat\gamma_{sr} - \gamma_{sr}^*) - \gamma_{sr}^*(\hat\gamma_{sl} - \gamma_{sl}^*)\right\}\\\overset{(ii)}{=}&\,0.
\end{align*}
Here $(i)$ follows because the item-parameter block of $\bpsi_{rl}^*$ is zero unless $1\le l<r\le K$; $(ii)$ follows from \eqref{eq_id_gamma_rotate}.

For each $i\in[n]$, similarly one can compute
\begin{align*}
    &\,\big\{\bH_P^*\big\}_{qK+q+\cK_i,}\begingroup\renewcommand{\arraystretch}{0.6}\begin{pmatrix}
            \hat\btheta - \btheta^*\\\hat\bu^{\star} - \bu^{*}
        \end{pmatrix}\endgroup \\=&\, \sum_{r=0}^K\sum_{l=1}^Kcb_n\big\{\bpsi_{rl}^*\big\}_{qK+q+\cK_i}\big(\bpsi_{rl}^*\big)^\T\begingroup\renewcommand{\arraystretch}{0.6}\begin{pmatrix}
            \hat\btheta - \btheta^*\\\hat\bu^{\star} - \bu^{*}
        \end{pmatrix}\endgroup\\
        {=}&\, \sum_{r=0}^K\sum_{ ( r\vee 1)\le l\le K}cb_n\big\{\bpsi_{rl}^*\big\}_{qK+q+\cK_i}\big(\bpsi_{rl}^*\big)^\T\begingroup\renewcommand{\arraystretch}{0.6}\begin{pmatrix}
            \hat\btheta - \btheta^*\\\hat\bu^{\star} - \bu^{*}
        \end{pmatrix}\endgroup\\
         = &\, cb_n\sqrt{\frac{q}{n}}\sum_{l=1}^K\big\{\bpsi_{0l}^*\big\}_{qK+q+\cK_i}\underbrace{\sum_{s=1}^n\left(\hat u_{sl}^{\star} - u_{sl}^*\right)}_{\alpha_{1l}} \\&\,+ cb_n\sqrt{\frac{q}{n}}\sum_{r=1}^K\big\{\bpsi_{rr}^*\big\}_{qK+q+\cK_i}\underbrace{\sum_{s=1}^n\left(u_{sr}^*\hat u_{sr}^{\star} - (u_{sr}^*)^2\right)}_{\alpha_{2r}}\\&\,+cb_n\sqrt{\frac{q}{n}}\sum_{r=1}^K\sum_{r<l\le K}\big\{\bpsi_{rl}^*\big\}_{qK+q+\cK_i}\underbrace{\sum_{s=1}^n\left\{u_{sr}^*\big(\hat u_{sl}^{\star} - u_{sl}^*\big) + u_{sl}^*\big(\hat u_{sr}^{\star} - u_{sr}^*\big)\right\}}_{\alpha_{3,rl}}.
\end{align*}
Next, we show that each $\alpha$ term above is $\cOp(n\varepsilon_{nq}/q + n\delta_{nq}^2)$, uniformly over its fixed-dimensional index set.
Invoke Lemma~\ref{lemma_first_order_base}, and we know with $\bS_{\mu}(\hat\btheta) = 0$ and $\bS_{\Sigma}(\hat\btheta) = 0$ that $\big\|\sum_{i=1}^n\hat\bu_i^{\star}\big\| = \cOp(n\varepsilon_{nq}/q)$ and $\big\|\sum_{i=1}^n\hat\bu_i^{\star}(\hat\bu_i^{\star})^\T -n\bI_K\big\| = \cOp(n\varepsilon_{nq}/q)$.
With $\sum_{i=1}^n\bu_i^* = \zero_K$, $\sum_{i=1}^n\bu_i^*(\bu_i^*)^\T = n\bI_K$ and $\sum_{s=1}^n\|\hat\bu_s^\star-\bu_s^*\|^2=\cOp(n\delta_{nq}^2)$ by  \eqref{eq_consis_u_star_hat}, we have $|\alpha_{1l}| = \cOp(n\varepsilon_{nq}/q)$, 
\begin{equation*}
    |\alpha_{2r}| =\left|\frac{1}{2}\sum_{s=1}^n\big(\hat u_{sr}^{\star}\big)^2 -\frac{1}{2}\sum_{s=1}^n(\hat u_{sr}^{\star} - u_{sr}^*)^2 - \frac{1}{2}\sum_{s=1}^n\big(u_{sr}^{*}\big)^2 \right| = \cOp(n\varepsilon_{nq}/q + n\delta_{nq}^2),
\end{equation*}
and 
\begin{align*}
    |\alpha_{3,rl}| =&\,\left|\sum_{s=1}^n\hat u_{sr}^{\star}\hat u_{sl}^{\star}-\sum_{s=1}^n(\hat u_{sr}^{\star}-u_{sr}^*) (\hat u_{sl}^{\star} - u_{sl}^*) - \sum_{s=1}^nu_{sr}^*u_{sl}^*\right| =  \cOp(n\varepsilon_{nq}/q + n\delta_{nq}^2).
\end{align*}
Since $\max_{i\in[n]}\|\bu_i^*\| = \cOp(\sqrt{\log n})$ by Assumption~\ref{assump_standard_id}, it holds that $\|\bpsi^*_{rl}\|_{\infty} = O(\sqrt{q\log n/n})$ when $r=0$ or $r\le l$. Therefore, we conclude that 
\begin{align*}
    \left\|\bS_m^{2}\bH_P^*\begingroup\renewcommand{\arraystretch}{0.6}\begin{pmatrix}
            \hat\btheta - \btheta^*\\\hat\bu^{\star} - \bu^{*}
        \end{pmatrix}\endgroup\right\|_{\infty} &= q^{-1}\max_{i\in[n]}\left\|\big\{\bH_P^*\big\}_{qK+q+\cK_i,}\begingroup\renewcommand{\arraystretch}{0.6}\begin{pmatrix}
            \hat\btheta - \btheta^*\\\hat\bu^{\star} - \bu^{*}
        \end{pmatrix}\endgroup\right\|_{\infty}\\&= \cOp\Big(b_n\sqrt{\log n}\big\{\delta_{nq}^2 + \varepsilon_{nq}/q\big\}\Big)
\end{align*}
which completes the proof.
\subsubsection{Proof of Lemma~\ref{lemma_asymptotic_normality_star}}\label{supp_sec_prove_lemma_asymptotic_normality_star}

The proof of Lemma~\ref{lemma_aux_estimator_property} does not rely on a particular identifiability condition, so its conclusions continue to hold under Condition~\ref{condition1} or~\ref{condition2}. Using the argument for \eqref{eq_asymp_expand} and Lemma~\ref{lemma_residual_in_asym_expand}, we obtain the following result under Assumptions~\ref{assump_standard_id}--\ref{assump_scaling}:
\begin{equation*}
    \big\|\hat\btheta_j^* - \btheta^*_j\big\| = \cOp(\bar\delta_{nq})\quad\text{ and }\quad\big\|\hat\bu_i^{\star}(\hat\btheta^*) - \bu_i^*\big\| = \cOp(\bar\delta_{nq})\text{ for }i\in[n],j\in[q].
\end{equation*}
This pointwise consistency does not rely on the identifiability condition imposed. Moreover, the same Taylor argument gives
\begin{equation*}
        \bS_L\big(\hat\btheta^*,\bu^{\star}(\hat\btheta^*)\big) - \bS_{L}(\btheta^*,\bu^*) = \bH_L(\btheta^*,\bu^*) \begin{pmatrix}
            \hat\btheta^* - \btheta^*\\ \bu^{\star}(\hat\btheta^*) - \bu^*
        \end{pmatrix} + \frac{1}{2}\begin{pmatrix}\tilde\bR_{L\theta}\\\tilde\bR_{Lu}\end{pmatrix},
    \end{equation*}
    for residuals $\tilde\bR_{L\theta}$ and $\tilde\bR_{Lu}$ satisfying the same bound as in Lemma~\ref{lemma_residual_in_asym_expand}. Next, to derive the asymptotic expansion, we use an argument slightly different from that used in the proof of Theorem~\ref{thm_asymp} in Section~\ref{supp_sec_prove_asym_0}. Specifically, define
\begin{equation*}
    \tilde\bH_P^* = c b_n\sum_{r=0}^K\sum_{l=1}^K \bar\bv_{rl}(\btheta^*)\{\bar\bv_{rl}(\btheta^*)\}^{\T},
\end{equation*}
where $\bar\bv_{rl} = \bS_m^{-1}\bar\bv_{rl}(\btheta^*) = \sqrt{n}\,\bar\bv_{rl}(\btheta^*)$ is defined in Lemma~\ref{lemma_convexity2}. Then $\bS_m\tilde\bH_P^*\bS_m = cb_n\sum_{r=0}^K\sum_{l=1}^K\bar\bv_{rl}(\btheta^*)\{\bar\bv_{rl}(\btheta^*)\}^\T$ has the same scale as $\bS_m\bH_P^*\bS_m = cb_n\widetilde\bpsi_L^*(\widetilde\bpsi_L^*)^\T$.
We first note that Lemma~\ref{lemma_approx_hessian_inverse} continues to hold with $\tilde\bH_P^*$ in place of $\bH_P^*$. To avoid unnecessary repetition, we only sketch the required modifications.
The proof of Lemma~\ref{lemma_approx_hessian_inverse} consists of three steps. In Step~1, the key task is to establish a lower bound on the minimum eigenvalue of $\bar\bH_L^*+\tilde\bH_P^*$. This is proved in Step~2 of Lemma~\ref{lemma_convexity2}. Thus Step~1 remains valid after replacing $\bH_P^*$ by $\tilde\bH_P^*$. Steps~2 and~3 rely mainly on inverting $\bH_L^*+\bH_P^*$, where the penalty Hessian enters the Woodbury identity only through its low-rank structure. Since $\tilde\bH_P^*$ has the same low-rank form and the same scale as $\bS_m\bH_P^*\bS_m$, the same Woodbury-based argument applies verbatim. Therefore, Lemma~\ref{lemma_approx_hessian_inverse} holds with $\tilde\bH_P^*$ replacing $\bH_P^*$. For the analogue of Lemma~\ref{lemma_bound_aug_via_id}, note that for any $r\in\{0\}\cup[K]$ and $l\in[K]$,
\begin{align*}
    \big\{\bar\bv_{rl}(\btheta^*)\big\}^\T\begin{pmatrix}
            \hat\btheta^* - \btheta^*\\ \bu^{\star}(\hat\btheta^*) - \bu^*
        \end{pmatrix}  = &\, q^{-1/2}\big\{\bv_{rl}(\btheta^*)\big\}^\T\big( \hat\btheta^* - \btheta^*\big)\\
        = &\, q^{-1/2}\sum_{j=1}^q\gamma_{jl}^*\left\{\big(\hat\beta_j^* - \beta_j^*\big)1_{\{r=0\}} + \big(\hat\gamma_{jr}^* - \gamma_{jr}^*\big)1_{\{r\neq0\}}\right\}\\
        = &\, 0,
\end{align*}
because $P^*(\hat\btheta^*)=0$, as established in Step~2 of Section~\ref{supp_sec_prove_lemma_key_first_order}.
Consequently, one can repeat the derivations of \eqref{eq_asymp_expand_2_bgamma} and \eqref{eq_asymp_expand_2_bu} to derive the asymptotic distributions in Lemma~\ref{lemma_asymptotic_normality_star}. Finally, the rate for $\hat\bmu^*$ and $\hat\bSigma^*$ required in the lemma follows directly from Lemma~\ref{lemma_aux_estimator_property}.

To derive the rates for $\hat\bmu^*$ and $\hat\bSigma^*$ in \eqref{eq_accurate_control_hsigma_star_2} and \eqref{eq_accurate_control_hsigma_star_3}, invoke the first-order conditions in Lemma~\ref{lemma_aux_estimator_property}: $\cS_{\mu}(\hat\btheta^*,\hat\bmu^*,\hat\bSigma^*) = \zero$ and $\cS_{\Sigma}(\hat\btheta^*,\hat\bmu^*,\hat\bSigma^*) = \zero$. Because $\|\hat\bmu^*\|$ and $\|\hat\bSigma^* - \bSigma_u^*\|$ are bounded by part~$(ii)$ of Lemma~\ref{lemma_aux_estimator_property}, Lemma~\ref{lemma_first_order_base} gives
    \begin{equation}
        \big\|\hat\bmu^*-n^{-1}\sum_{i=1}^n\bu_i^{\star}(\hat\btheta^*)\big\| = \cOp(\varepsilon_{nq}q^{-1}),\label{eq_first_order_app_to_hat_mu}
    \end{equation}
    and 
    \begin{equation}
        \big\|\hat\bSigma^* - n^{-1}\sum_{i=1}^n\big\{\bu_i^{\star}(\hat\btheta^*) - \hat\bmu^*\big\}\big\{\bu_i^{\star}(\hat\btheta^*)-\hat\bmu^*\big\}^\T\big\| = \cOp(\varepsilon_{nq}q^{-1}).\label{eq_first_order_app_to_hat_sigma}
    \end{equation}
The preceding argument gives the expansion $\bu_i^{\star}(\hat\btheta^*) - \bu_i^* = -\big(\bH_{Lu_iu_i}^*\big)^{-1}\bS_{Lu_i}^* +\cOp\big(\tau_{nq}\sqrt{\log n/B_n}\big)$.
In addition, one can similarly show that the remainders are also controlled in squared sum over $i\in[n]$.
Subsequently, one can calculate
\begin{align*}
    \Big\|\sum_{i=1}^n\big\{\bu_i^{\star}(\hat\btheta^*) -\bu_i^*\big\}\Big\| \le&\, \Big\|\sum_{i=1}^n\sum_{j=1}^q\Big(-\sum_{s=1}^q\ell_{is}^{\prime\prime}(\eta_{is}^*)\bgamma_s^*{\bgamma_s^*}^\T\Big)^{-1}\ell_{ij}^{\prime}(\eta_{ij}^*)\bgamma_j^*\Big\|\\ &\,+\cOp\big(n\tau_{nq}\sqrt{\log n/B_n}\big) .
\end{align*}
To bound the right side, let $\bb_i = \sum_{j=1}^q\ell_{ij}^{\prime}(\eta_{ij}^*)\bgamma_j^*$ , $\bA_i := -\sum_{j=1}^q\ell_{ij}^{\prime\prime}(\eta_{ij}^*)\bgamma_j^*(\bgamma_j^*)^\T$, and $\bar\bA_i = \EE[\bA_i\mid\bu_i^*]$, which is deterministic conditionally on $\{\bu_i^*\}_{i=1}^n$ and satisfies $\lambda_{\min}(\bar\bA_i)\ge qb_n\lambda_{\min}(\bSigma_{\gamma}^*)/2$. Conditionally on $\{\bu_i^*\}_{i=1}^n$, the summands of $\sum_{i=1}^n\bar\bA_i^{-1}\bb_i$ are independent and mean zero with $\EE\|\bb_i\|^2 = \cO(qB_n)$. Thus, one can obtain $\|\sum_{i=1}^n\bar\bA_i^{-1}\bb_i\| = \cOp\big(\kappa_n\sqrt{n/q}\big)$. Then it suffices to bound $\|\sum_{i=1}^n(\bA_i^{-1} - \bar\bA_i^{-1})\bb_i\|$. With Bernstein inequality and a union bound over $i\in[n]$, we know $\max_{i\in[n]}\|\bA_i-\bar\bA_i\| = \cOp(B_n\sqrt{q\log n})$. Then one can check with $\tau_{nq}\ge\varepsilon_{nq}\kappa_n^3B_n^2/q$ and $b_n\le 1\le B_n\le B_n^{3/2}\varepsilon_{nq}$ that $\|\sum_{i=1}^n(\bA_i^{-1} - \bar\bA_i^{-1})\bb_i\| = \cOp\big(n\tau_{nq}\sqrt{\log n/B_n}\big)$.
Therefore, we conclude that 
\begin{equation}\label{eq_first_order_app_to_hat_mu_1}
    \Big\|\sum_{i=1}^n\big\{\bu_i^{\star}(\hat\btheta^*) -\bu_i^*\big\}\Big\| = \cOp\big(n\tau_{nq}\sqrt{\log n/B_n}\big)
\end{equation}
Similarly, one has
    \begin{equation}\label{eq_first_order_app_to_hat_sigma_2}
        n^{-1}\big\|\sum_{i=1}^n\bu_i^{\star}(\hat\btheta^*)\big\{\bu_i^{\star}(\hat\btheta^*)\big\}^\T - \sum_{i=1}^n\bu_i^*\big\{\bu_i^*\big\}^\T\big\| = \cOp\big(\tau_{nq}\sqrt{\log n/B_n}\big).
    \end{equation}
    Under the constraints $n^{-1}\sum_{i=1}^n\bu_i^* = \zero$ and $n^{-1}\sum_{i=1}^n\bu_i^*(\bu_i^*)^\T = \bI_K$, \eqref{eq_first_order_app_to_hat_mu}--\eqref{eq_first_order_app_to_hat_sigma_2} imply \eqref{eq_accurate_control_hsigma_star_2}. Under the constraints $n^{-1}\sum_{i=1}^n\bu_i^* = \zero$ and ${\rm diag}\big\{n^{-1}\sum_{i=1}^n\bu_i^*(\bu_i^*)^\T \big\}=\bI_K$, the same equations imply \eqref{eq_accurate_control_hsigma_star_3}.

\subsubsection{Proof of Lemma~\ref{lemma_linear_equation_perturb}}\label{prove_lemma_linear_equation_perturb}

Let \(\bP=\bH^\T(\bH\bH^\T)^{-1}\bH\) and
\(\widetilde\bP=\widetilde\bH^\T(\widetilde\bH\widetilde\bH^\T)^{-1}\widetilde\bH\)
be the orthogonal projectors onto the row spaces. Weyl's inequality implies
\(\operatorname{rank}(\widetilde\bH)=K-1\) for sufficiently small \(\epsilon\).
The resolvent identity and
\[
    \|\widetilde\bH\widetilde\bH^\T-\bH\bH^\T\|
    \le (2\|\bH\|+\epsilon)\epsilon
\]
give \(\|(\widetilde\bH\widetilde\bH^\T)^{-1}\|=O(1)\) and
\[
    \|(\widetilde\bH\widetilde\bH^\T)^{-1}-(\bH\bH^\T)^{-1}\|
    =O(\epsilon).
\]
Expanding \(\widetilde\bP-\bP\) and applying submultiplicativity therefore gives
\[
    \|\widetilde\bP-\bP\|=O(\epsilon).
\]
Because \(\bI_K-\bP\) and \(\bI_K-\widetilde\bP\) are the rank-one
projectors onto \(\ker(\bH)\) and \(\ker(\widetilde\bH)\), respectively,
\[
    \sin\angle(\bx_h,\widetilde\bx_h)
    =\|\widetilde\bP-\bP\|
    =O(\epsilon),
\]
which proves the result.

    \subsection{Proof of Additional Technical Lemmas}

    \subsubsection{Proof of Lemma~\ref{lemma_concentration}}\label{supp_sec_prove_lemma_concentration}
Conditional on $\{\bu_i^*\}_{i=1}^n$, apply the Bernstein inequality for weighted sums of independent, mean-zero sub-exponential variables:
\[
\mathbb{P} \left\{ \left| \sum_{i=1}^n a_i \ell_{ij}^{\prime}(\eta_{ij}^*) \right| \geq t \,\middle|\, \bu_1^*,\ldots,\bu_n^* \right\}
\leq 2 \exp \left[ -c \min \left\{ \frac{t^2}{C^2B_n\sum_{i=1}^na_i^2}, \frac{t}{C\sqrt{B_n}\max_i|a_i|} \right\} \right].
\]
Here $c>0$ is an absolute constant and $C\sqrt{B_n}$ is the uniform upper bound for the sub-exponential norm in Assumption~\ref{assumption_smoothness}. Taking $t$ to be a sufficiently large multiple of $A_n\sqrt{nB_n}$ gives the pointwise bound in part~(i). A union bound gives
\[
\mathbb{P} \left\{\max_{1\le j\le q} \left| \sum_{i=1}^n a_i\ell_{ij}^{\prime}(\eta_{ij}^*) \right| \geq t \,\middle|\, \bu_1^*,\ldots,\bu_n^* \right\}
\leq 2q\exp \left[ -c \min \left\{ \frac{t^2}{C^2B_nA_n^2n}, \frac{t}{C\sqrt{B_n}A_n} \right\} \right].
\]
Under $\log q/\sqrt n\to0$, taking $t$ to be a sufficiently large multiple of $A_n\sqrt{nB_n\log q}$ proves the uniform bound in part~(i). Part~(ii) follows by interchanging $i$ and $j$, and the weighted double-sum assertion in part~(iii) follows by applying the same inequality to the $nq$ entries. Finally, a union bound for individual sub-exponential tails gives
$\max_{i,j}|\ell_{ij}'(\eta_{ij}^*)|=\cOp\{\sqrt{B_n}\log(nq)\}$.

\subsubsection{Proof of Lemma~\ref{lemma_aux_score_bound_pre}}\label{supp_sec_prove_lemma_aux_score_bound_pre}
First, note that by Assumptions~\ref{assumption_psd_covariance} and~\ref{assumption_smoothness}, we have
\begin{equation}\label{eq_sub_exp_1}\min_{i\in[n]}\lambda_{\min}\{\sum_{j=1}^q-\ell_{ij}^{\prime\prime}(\eta_{ij}^*)\bgamma_j^*(\bgamma_j^*)^\T\big\}\ge q b_n\lambda_{\min}(\bSigma_{\gamma}^*) := Cqb_n,\end{equation} for some absolute constant $C$ as $-\ell_{ij}^{\prime\prime}(\eta_{ij}^*)\ge b_n$ uniformly over $i\in[n]$ and $j\in[q]$. Conditional on $\{\bu_i^*\}_{i=1}^n$, the variables $\{\ell_{ij}^{\prime}(\eta_{ij}^*)\}_{i\in[n],j\in[q]}$ are independent and mean-zero with
\begin{equation}
    \EE\left[\{\ell_{ij}^{\prime}(\eta_{ij}^*)\}^2\mid \bu_1^*,\dots,\bu_n^*\right]\le C B_n, \qquad  \EE\left[\{\ell_{ij}^{\prime}(\eta_{ij}^*)\}^4\mid \bu_1^*,\dots,\bu_n^*\right]\le C B_n^2,\label{eq_sub_exp_sum}
\end{equation}

Similar to the proof in Section~\ref{supp_sec_prove_lemma_asymptotic_normality_star}, write $\bA_i := -\sum_{t=1}^q \ell_{it}^{\prime\prime}(\eta_{it}^*)\bgamma_t^*(\bgamma_t^*)^\T$, $\bb_i := \sum_{t=1}^q\bgamma_t^*\ell_{it}^{\prime}(\eta_{it}^*)$, and $\bb_i^{(-j)} := \sum_{t\neq j}\bgamma_t^*\ell_{it}^{\prime}(\eta_{it}^*)$, so that $\bb_i = \bb_i^{(-j)} + \bgamma_j^*\ell_{ij}^{\prime}(\eta_{ij}^*)$. We further introduce
\begin{equation*}
    \bar\bA_i \; := \; \EE\big[\bA_i\mid \bu_1^*,\dots,\bu_n^*\big],\quad\text{with}\quad \lambda_{\min}(\bar\bA_i)\ge qb_n\lambda_{\min}(\bSigma_{\gamma}^*)/2 =: Cqb_n,
\end{equation*}
where the lower bound follows from $-\ell_{it}^{\prime\prime}(\eta_{it}^*)\ge b_n$. Applying the matrix Bernstein inequality conditionally on $\{\bu_i^*\}_{i=1}^n$ and taking a union bound over $i\in[n]$, gives
\begin{equation}\label{eq_curvature_fluct}
    \max_{i\in[n]}\big\|\bA_i - \bar\bA_i\big\| = \cOp\big(B_n\sqrt{q\log n}\big),\qquad
    \max_{i\in[n]}\big\|\bA_i^{-1} - \bar\bA_i^{-1}\big\| = \cOp\Big(\frac{B_n\sqrt{q\log n}}{(qb_n)^2}\Big).
\end{equation}

For part $(i)$, decompose
\begin{equation*}
    \sum_{i=1}^n \ell_{ij}^{\prime}(\eta_{ij}^*) \be_r^\T\bA_i^{-1}\bb_i
    = \underbrace{\sum_{i=1}^n \ell_{ij}^{\prime}(\eta_{ij}^*)\be_r^\T\bar\bA_i^{-1}\bb_i^{(-j)}}_{T_1}
    +\underbrace{\sum_{i=1}^n \ell_{ij}^{\prime}(\eta_{ij}^*)\be_r^\T\big(\bA_i^{-1}-\bar\bA_i^{-1}\big)\bb_i^{(-j)}}_{T_2}
    +\underbrace{\sum_{i=1}^n \big\{\ell_{ij}^{\prime}(\eta_{ij}^*)\big\}^2\be_r^\T\bA_i^{-1}\bgamma_j^*}_{T_3}.
\end{equation*}
Conditionally on $\{\bu_i^*\}_{i=1}^n$, the summands in $T_1$ are independent and mean zero, as $\bar\bA_i^{-1}$ is deterministic and $\ell_{ij}^{\prime}(\eta_{ij}^*)$ is independent of $\bb_i^{(-j)}$ with mean zero. With $\EE\|\bb_i^{(-j)}\|^2 = \cO(qB_n)$ and $\|\bar\bA_i^{-1}\|\le (Cqb_n)^{-1}$ from \eqref{eq_sub_exp_sum}, this yields
\begin{equation*}
    |T_1| = \cOp\Big(\Big\{\sum_{i=1}^n\EE\big[\{\ell_{ij}^{\prime}(\eta_{ij}^*)\}^2\big]\big\|\bar\bA_i^{-1}\big\|^2\EE\big\|\bb_i^{(-j)}\big\|^2\Big\}^{1/2}\Big) = \cOp\Big(\frac{B_n}{b_n}\sqrt{\frac nq}\Big).
\end{equation*}For $T_2$, by \eqref{eq_curvature_fluct}, $\sum_i|\ell_{ij}^{\prime}(\eta_{ij}^*)| = \cOp(n\sqrt{B_n})$ and $\max_i\|\bb_i^{(-j)}\| = \cOp(\sqrt{qB_n\log n})$, we know
\begin{equation*}
    |T_2| = \cOp\left(\frac{nB_n^2\log n}{qb_n^2}\right) = \cOp\left(\varepsilon_{nq}\frac nq\right),
\end{equation*}
where the last follows the scaling condition in Assumption~\ref{assump_scaling}.
For $T_3$, with $\max_i\|\bA_i^{-1}\|=\cOp\{(qb_n)^{-1}\}$ and $\sum_i\{\ell_{ij}^{\prime}(\eta_{ij}^*)\}^2 = \cOp(nB_n)$, we arrive at $|T_3| = \cOp\{(B_n/b_n)(n/q)\}$. 
Combining the three terms proves the first assertion of part $(i)$.

For the second assertion, part $(ii)$ of Lemma~\ref{lemma_concentration_l_mat} gives that $\max_{j\in[q]}\big\{\sum_{i=1}^n\ell_{ij}^{\prime}(\eta_{ij}^*)^2\big\}^{1/2} = \cOp(\sqrt{nB_n} + \sqrt{B_n}\log q)$. The preceding derivation then verifies the second assertion in part $(i)$. The third follows similarly.

For part $(ii)$, with the convention $u_{i0}^*=1$, decompose analogously, writing $\bar\ell_{ij}^{\prime\prime} = \EE[\ell_{ij}^{\prime\prime}(\eta_{ij}^*)\mid\bu_i^*]$,
\begin{align*}
    \sum_{i=1}^n \ell_{ij}^{\prime\prime}(\eta_{ij}^*)u_{ir}^*(\bgamma_j^*)^\T\bA_i^{-1}\bb_i
    = &\, \sum_{i=1}^n \ell_{ij}^{\prime\prime}(\eta_{ij}^*)u_{ir}^*(\bgamma_j^*)^\T\bar\bA_i^{-1}\bb_i^{(-j)}\\
    &\,+\sum_{i=1}^n \ell_{ij}^{\prime\prime}(\eta_{ij}^*)u_{ir}^*(\bgamma_j^*)^\T\big(\bA_i^{-1}-\bar\bA_i^{-1}\big)\bb_i^{(-j)}\\
    &\,+\sum_{i=1}^n \ell_{ij}^{\prime\prime}(\eta_{ij}^*)\ell_{ij}^{\prime}(\eta_{ij}^*)u_{ir}^*(\bgamma_j^*)^\T\bA_i^{-1}\bgamma_j^*.
\end{align*}
Following a similar procedure above, one can bound the three terms by $\cOp(B_n\kappa_n\sqrt{n/q})$, $\cOp(\varepsilon_{nq}n/q)$, $\cOp(nB_n\kappa_n/q)$, respectively.
The second and third assertions follow as in part $(i)$, using $\max_{i\in[n]}\|\bu_i^*\| =\cOp(\sqrt{\log n})$.

\subsubsection{Proof of Lemma~\ref{lemma_aux_of_aux_estimator_property}}\label{supp_sec_prove_lemma_aux_of_aux_estimator_property}

     The first part is a result of \eqref{eq_convex_aux_obj} and the continuity of eigenvalues. Now we focus on the second part of the lemma. Write $\btheta^s = \btheta^* + s(\btheta - \btheta^*)$, and, for symbolic subscripts $x,y\in\{\theta,u\}$, define $\bH_{Lxy}^s = \bH_{Lxy}(\btheta^s,\bu^{\star}(\btheta^s))$. We begin by providing a list of estimates for these quantities. Specifically, uniformly over $s\in[0,1]$, we have
     \begin{enumerate}[label=(\alph*)]
    \item for each $j\in[q]$, $\Theta_p(nb_n)\le\lambda_{\min}\big(\{\bH_{L\theta\theta}^s\}_{\cK_j^+,\cK_j^+}\big)\le \lambda_{\max}\big(\{\bH_{L\theta\theta}^s\}_{\cK_j^+,\cK_j^+}\big)\le \cOp(B_nn)$;
    \item $\big\|\bH_{L\theta\theta}^s\big\|=\cOp(B_nn)$, $\big\|(\bH_{L\theta\theta}^s)^{-1}\big\| = \cOp((b_nn)^{-1})$, $\big\|(\bH_{L\theta\theta}^s)^{-1}\big\|_{\infty} = \cOp((b_nn)^{-1})$;
    \item for each $i\in[n]$, $\Theta_p(qb_n)\le\lambda_{\min}\big(\{\bH_{Luu}^s\}_{\cK_i,\cK_i}\big)\le \lambda_{\max}\big(\{\bH_{Luu}^s\}_{\cK_i,\cK_i}\big)\le \cOp(qB_n)$;
    \item $\big\|\bH_{Luu}^s\big\|=\cOp(qB_n)$, $\big\|(\bH_{Luu}^s)^{-1}\big\| = \cOp((b_nq)^{-1})$, $\big\|(\bH_{Luu}^s)^{-1}\big\|_{\infty} = \cOp((qb_n)^{-1})$;
    \item $\big\|(\bar\bH_{L\theta u}^s)_{\cK_j^+,}\big\| = \cOp(B_n\sqrt{n})$, $\big\|(\bar\bH_{L\theta u}^s)_{,\cK_i}\big\| = \cOp(B_n\sqrt{q})$, $\big\|\bar\bH_{L\theta u}^s\big\| = \cOp(B_n\sqrt{nq})$;
    \item $\max_{j\in[q]}\big\|(\bar\bH_{L\theta u}^s)_{\cK_j^+,}\big\| = \cOp(B_n\sqrt{n})$, $\max_{i\in[n]}\big\|(\bar\bH_{L\theta u}^s)_{,\cK_i}\big\| = \cOp(B_n\sqrt{q\log n})$;
    \item $\|(\bV^*)_{\cK_j^+}\| = \cO(\sqrt{b_n/q})$, $\|\bV^*\| = \cO(\sqrt{b_n})$, $\|\bV^*\|_{\infty} = \cO(\sqrt{b_n/q})$, $\|(\bV^*)^\T\|_{\infty} = \cO(\sqrt{qb_n})$.
    \end{enumerate}
    The matrix $\bV^*$ in (g) is defined as in the proof of Lemma~\ref{lemma_convexity2}, but with the scaling
    \begin{equation*}
        \bV^* = \sqrt{2c_P/q}\left(\bv_{01}^*,\dots,\bv_{0K}^*,\bv_{11}^*,\dots,\bv_{1K}^*,\dots,\bv_{KK}^*\right),
    \end{equation*}
    where, for $r\in\{0\}\cup[K]$ and $l\in[K]$,
    \begin{equation}
        \bv_{rl}^*=\bv_{rl}(\btheta^*)
        =\left(\gamma_{1l}^*\big(\be_{r+1}^{(K+1)}\big)^\T, \dots, \gamma_{ql}^*\big(\be_{r+1}^{(K+1)}\big)^\T\right)^\T \in \RR^{q(K+1)}.
    \end{equation}
To verify (a)--(g), first note that, under the setup of Lemma~\ref{lemma_aux_estimator_property}, Lemma~\ref{lemma_neighbour_control} gives
\begin{equation}
    \big\|\btheta^s\big\|_{\infty}\le C;\; \big\|\bu^{\star}(\btheta^s)\big\|_{\infty}\le C\sqrt{\log n};\; \big\|\btheta^s - \btheta^*\|\le \sqrt{q}\epsilon;\; \big\|\bu^{\star}(\btheta^s) - \bu^*\big\| = \cOp(\sqrt{nB_n}\kappa_n\epsilon).\label{eq_bounded_execute_step2}
\end{equation}
For any $s\in[0,1]$, these bounds hold uniformly. For (a), one has
\begin{align*}
    \min_{s\in[0,1]}\lambda_{\min}\big(\bH_{L\theta_j\theta_j}\big)=&\,\min_{s\in[0,1]}\lambda_{\min}\left\{\sum_{i=1}^n-\ell_{ij}^{\prime\prime}\big(\eta_{ij}^{\star}(\btheta^s)\big)\begingroup\renewcommand{\arraystretch}{0.7}\begin{pmatrix}1\\\bu_i^{\star}(\btheta^s)\end{pmatrix}\begin{pmatrix}1\\\bu_i^{\star}(\btheta^s)\end{pmatrix}^\T\endgroup\right\}\\\overset{(i)}{\ge} &\,\min_{s\in[0,1]}\lambda_{\min}\Big\{\sum_{i=1}^n-\ell_{ij}^{\prime\prime}\big(\eta_{ij}^{\star}(\btheta^s)\big)\begingroup\renewcommand{\arraystretch}{0.7}\begin{pmatrix}1\\\bu_i^*\end{pmatrix}\begin{pmatrix}1\\\bu_i^*\end{pmatrix}^\T\endgroup\Big\} \\&\,- \min_{s\in[0,1]}\Big\|\sum_{i=1}^n\ell_{ij}^{\prime\prime}\big(\eta_{ij}^{\star}(\btheta^s)\big)\begingroup\renewcommand{\arraystretch}{0.7}\begin{pmatrix}1\\\bu_i^*\end{pmatrix}\begin{pmatrix}1\\\bu_i^*\end{pmatrix}^\T\endgroup \\
    &\, \qquad\qquad-  \sum_{i=1}^n\ell_{ij}^{\prime\prime}\big(\eta_{ij}^{\star}(\btheta^s)\big)\begingroup\renewcommand{\arraystretch}{0.7}\begin{pmatrix}1\\\bu_i^{\star}(\btheta^s)\end{pmatrix}\begin{pmatrix}1\\\bu_i^{\star}(\btheta^s)\end{pmatrix}^\T\endgroup\Big\|\\\overset{(ii)}{\ge}&\,b_n\lambda_{\min}\Big\{\begingroup\renewcommand{\arraystretch}{0.7}\begin{pmatrix}n&\sum_{i=1}^n(\bu_i^*)^\T\\\sum_{i=1}^n\bu_i^*& \sum_{i=1}^n\bu_i^*{\bu_i^*}^\T\end{pmatrix}\endgroup\Big\} \\&\, - 2B_n\max_{s\in[0,1]}\Big\|\sum_{i=1}^n\bu_i^*(\bu_i^* - \bu_i^{\star}(\btheta^s))^\T \Big\|  \\&\,- B_n\max_{s\in[0,1]}\Big\|\sum_{i=1}^n(\bu_i^* - \bu_i^{\star}(\btheta^s))(\bu_i^* - \bu_i^{\star}(\btheta^s))^\T \Big\|\\\overset{(iii)}{\ge}&\,\Theta_p(nb_n)-\cOp\big(nB_n\kappa_n\sqrt{B_n}\epsilon+nB_n^2\kappa_n^2\epsilon^2\big)=\Theta_p(nb_n).
\end{align*}
Here $\eta_{ij}^{\star}(\btheta) =\bgamma_j^\T\bu_i^{\star}(\btheta) + \beta_j$. Step $(i)$ uses Weyl's theorem; step $(ii)$ uses $B_n\ge -\ell_{ij}^{\prime\prime}(\eta_{ij}^{\star}(\btheta^s))\ge b_n$ and telescoping; and step $(iii)$ uses Assumption~\ref{assump_standard_id}, \eqref{eq_bounded_execute_step2}, and the Cauchy--Schwarz inequality. Assumption~\ref{assump_scaling} makes both perturbation terms $o_p(nb_n)$, proving the lower bound in (a). The upper bound is obtained similarly:
\begin{align*}
\max_{s\in[0,1]}\lambda_{\max}\big(\bH_{L\theta_j\theta_j}\big)\le &\,\max_{s\in[0,1]}\lambda_{\max}\left\{\sum_{i=1}^n-\ell_{ij}^{\prime\prime}\big(\eta_{ij}^{\star}(\btheta^s)\big)\begingroup\renewcommand{\arraystretch}{0.7}\begin{pmatrix}1\\\bu_i^*\end{pmatrix}\begin{pmatrix}1\\\bu_i^*\end{pmatrix}^\T\endgroup\right\} \\&\,+ \max_{s\in[0,1]}\left\|\sum_{i=1}^n\ell_{ij}^{\prime\prime}\big(\eta_{ij}^{\star}(\btheta^s)\big)\begingroup\renewcommand{\arraystretch}{0.7}\begin{pmatrix}1\\\bu_i^*\end{pmatrix}\begin{pmatrix}1\\\bu_i^*\end{pmatrix}^\T\endgroup- \sum_{i=1}^n\ell_{ij}^{\prime\prime}\big(\eta_{ij}^{\star}(\btheta^s)\big)\begingroup\renewcommand{\arraystretch}{0.7}\begin{pmatrix}1\\\bu_i^{\star}(\btheta^s)\end{pmatrix}\begin{pmatrix}1\\\bu_i^{\star}(\btheta^s)\end{pmatrix}^\T\endgroup\right\|\\
\le&\,\cOp\big(nB_n+nB_n\kappa_n\sqrt{B_n}\epsilon\big)=\cOp(nB_n).
\end{align*}
Then (b) follows from the block-diagonal structure of $\bH_{L\theta\theta}^s$; (c) and (d) can be checked by symmetry.

For (e), note that
\begin{equation*}
    \max_{s\in[0,1]}\left\|\ell_{ij}^{\prime\prime}(\eta_{ij}^{\star}(\btheta^s))\begingroup\renewcommand{\arraystretch}{0.7}\begin{pmatrix}1\\\bu_i^{\star}(\btheta^s)\end{pmatrix}\endgroup(\bgamma_j^s)^\T\right\| = \cOp(B_n),
\end{equation*}
by \eqref{eq_bounded_execute_step2} and Assumption~\ref{assumption_smoothness}. Then $\max_{s\in[0,1]}\big\|\big(\bar\bH_{L\theta u}\big)_{\cK_j^+,}\big\| = \cOp(B_n\sqrt{n})$. 
By symmetry, $\max_{s\in[0,1]}\big\|(\bar\bH_{L\theta u})_{,\cK_i}\big\| = \cOp(B_n\sqrt{q})$. A similar argument gives $\big\|\bar\bH_{L\theta u}\big\| = \cOp(B_n\sqrt{nq})$. The bounds in (f) are uniform versions of the same estimates.
Finally, part (g) can be checked by the definition of $\bv_{rl}^*$ and Assumption~\ref{assumption_psd_covariance}.

     Now we compute $\big\|\{\tilde\cH_{Q1}^*(\btheta,\hat\bmu^*,\hat\bSigma^*)\}^{-1}\big\|_{\infty}$. Let $\bH=\bH_{L\theta u}$, $\bJ=\bJ_{L\theta u}$, and $\bA=\bH_{Luu}^{-1}$. Then $\bar\bH=\bH-\bJ$ and $-\bH\bA\bH^\T+\bar\bH\bA\bar\bH^\T=-\bJ\bA\bH^\T-\bH\bA\bJ^\T+\bJ\bA\bJ^\T$. Since $\bH_{Q2}^{\star}$ contains the full off-diagonal Hessian blocks, adding $n^{-1}\bH_{Q2}^{\star}$ cancels the entire Schur term in the penalized Hessian approximation. Define
     \begin{equation}\begin{aligned}
    \bH_{Q1}^{\star}(\btheta)
    &= n^{-1}\bH_{L\theta\theta}\big(\btheta,\bu^{\star}(\btheta)\big)
       +2c_Pq^{-1}\sum_{r=0}^K\sum_{l=1}^K\bv_{rl}^*(\bv_{rl}^*)^\T\\
    &= n^{-1}\bH_{L\theta\theta}\big(\btheta,\bu^{\star}(\btheta)\big)+\bV^*(\bV^*)^\T
       =:\bH_{Q11}^{\star}(\btheta).
\end{aligned}\label{eq_define_H_Q1_star}\end{equation} Set $\bH_Q^\star(\btheta)=\bH_{Q1}^\star(\btheta)-n^{-1}\bH_{Q2}^\star(\btheta)$.
     Since $\partial_{\btheta\btheta}^2\cQ^* = \partial_{\btheta\btheta}^2\cL + \partial_{\btheta\btheta}^2P^*$, Lemma~\ref{lemma_laplace_approx_second_order} implies, uniformly over $s\in[0,1]$, that
     \begin{subequations}\begin{align}
    &\Big\|\bH_{Q}^{\star}(\btheta) - n^{-1}\partial_{\btheta\btheta}^2\cQ^*(\btheta,\hat\bmu^*,\hat\bSigma^*)\Big\| = \cOp(\varepsilon_{nq}q^{-1});\label{eq_approx_h_Q_integral_2}\\&\Big\|\bH_{Q}^{\star}(\btheta) - n^{-1}\partial_{\btheta\btheta}^2\cQ^*(\btheta,\hat\bmu^*,\hat\bSigma^*)\Big\|_{\infty} = \cOp(\varepsilon_{nq}q^{-1}).\label{eq_approx_h_Q_integral_inf}\end{align}
\end{subequations}
Subsequently, using $\tilde\cH_{Q}^*(\btheta,\hat\bmu^*,\hat\bSigma^*) = \int_{0}^1n^{-1}\partial_{\btheta\btheta}^2\cQ^*\big(\btheta^s,\hat\bmu^*,\hat\bSigma^*\big)ds$ and $\tilde\cH_{Q2}^*(\btheta) = \int_{0}^1n^{-1}\bH_{Q2}^{\star}\big(\btheta^s\big)ds$, define $\tilde\cH_{Q1}^*(\btheta,\hat\bmu^*,\hat\bSigma^*) = \tilde\cH_{Q}^*(\btheta,\hat\bmu^*,\hat\bSigma^*) + \tilde\cH_{Q2}^*(\btheta)$. Then
\begin{subequations}\begin{align}
    &\Big\|\tilde\cH_{Q1}^*(\btheta,\hat\bmu^*,\hat\bSigma^*) - \int_0^1\bH_{Q1}^{\star}(\btheta^s)ds\Big\| = \cOp(\varepsilon_{nq}q^{-1});\label{eq_approx_h_Q1_integral_2}\\&\Big\|\tilde\cH_{Q1}^*(\btheta,\hat\bmu^*,\hat\bSigma^*) - \int_0^1\bH_{Q1}^{\star}(\btheta^s)ds\Big\|_{\infty} = \cOp(\varepsilon_{nq}q^{-1}).\label{eq_approx_h_Q1_integral_inf}\end{align}
\end{subequations}
This implies that $\tilde\bH_{Q1}^{\star}(\btheta) := \int_0^1\bH_{Q1}^{\star}(\btheta^s)ds$ is the leading term in $\tilde\cH_{Q1}^*(\btheta,\hat\bmu^*,\hat\bSigma^*)$. We therefore calculate the inverse of $\tilde\bH_{Q1}^{\star}(\btheta)$.

By \eqref{eq_define_H_Q1_star}, we let $\tilde\bH_{Q11}^{\star}(\btheta) = n^{-1}\tilde\bH_{L\theta\theta}(\btheta) + \bV^*(\bV^*)^\T$ where $\tilde\bH_{L\theta\theta}(\btheta)$ is defined as $\int_0^1\bH_{L\theta\theta}(\btheta^s)ds$. By (a) and (b), one can check
\begin{enumerate}[label=(\alph*$^{\prime}$)]
    \item for each $j\in[q]$, $\Theta_p(b_nn)\le\lambda_{\min}\big(\{\tilde\bH_{L\theta\theta}\}_{\cK_j^+,\cK_j^+}\big)\le \lambda_{\max}\big(\{\tilde\bH_{L\theta\theta}\}_{\cK_j^+,\cK_j^+}\big)\le \cOp(B_nn)$;
    \item $\big\|\tilde\bH_{L\theta\theta}\big\|=\cOp(B_nn)$, $\big\|\tilde\bH_{L\theta\theta}^{-1}\big\| = \cOp((b_nn)^{-1})$, $\big\|\tilde\bH_{L\theta\theta}^{-1}\big\|_{\infty} = \cOp((b_nn)^{-1})$;
    \end{enumerate}
Next, apply Woodbury identity to write the inverse of $\tilde\bH_{Q11}^{\star}(\btheta)$ as
\begin{align*}
    n^{-1}\left\{\tilde\bH_{Q11}^{\star}(\btheta)\right\}^{-1} =  \tilde\bH_{L\theta\theta}^{-1} - \tilde\bH_{L\theta\theta}^{-1}\bV^*\left\{n^{-1}\bI_{K^2+K} + (\bV^*)^\T\tilde\bH_{L\theta\theta}^{-1}\bV^*\right\}^{-1}(\bV^*)^\T\tilde\bH_{L\theta\theta}^{-1}.
\end{align*}
By (g) and (b$^{\prime}$), we know $\|(\bV^*)^\T\tilde\bH_{L\theta\theta}^{-1}\bV^*\|\le \|\bV^*\|^2\|\tilde\bH_{L\theta\theta}^{-1}\|\le \cOp(n^{-1})$. Since $(\bV^*)^\T\tilde\bH_{L\theta\theta}^{-1}\bV^*$ is positive semidefinite, we know
\begin{equation*}
    \left\|\left(n^{-1}\bI_{K^2+K} + (\bV^*)^\T\tilde\bH_{L\theta\theta}^{-1}\bV^*\right)^{-1}\right\| = \lambda_{\min}\left\{n^{-1}\bI_{K^2+K} + (\bV^*)^\T\tilde\bH_{L\theta\theta}^{-1}\bV^*\right\}^{-1} = \cOp(n).
\end{equation*}
Then with (g), (a$^{\prime}$) and (b$^{\prime}$), we conclude that 
\begin{equation}\begin{aligned}
    &\,\left\|\left\{\tilde\bH_{Q11}^{\star}(\btheta)\right\}^{-1}\right\|\\\le &\,n\big\|\tilde\bH_{L\theta\theta}^{-1}\big\| + n\left\|\tilde\bH_{L\theta\theta}^{-1}\bV^*\left\{n^{-1}\bI_{K^2+K} + (\bV^*)^\T\tilde\bH_{L\theta\theta}^{-1}\bV^*\right\}^{-1}(\bV^*)^\T\tilde\bH_{L\theta\theta}^{-1}\right\|\\\le &\,n\big\|\tilde\bH_{L\theta\theta}^{-1}\big\|^2 \big\|\bV^*\big\|\big\|(\bV^*)^\T\big\|\left\|\left\{n^{-1}\bI_{K^2+K} + (\bV^*)^\T\tilde\bH_{L\theta\theta}^{-1}\bV^*\right\}^{-1}\right\| + \cOp(b_n^{-1})\\= &\,\cOp(b_n^{-1}).
\end{aligned}\label{eq_bound_tilde_HQ1_star_inv}\end{equation}
and
\begin{equation}\begin{aligned}
    &\,\left\|\left\{\tilde\bH_{Q11}^{\star}(\btheta)\right\}^{-1}\right\|_{\infty}\\\le &\,n\big\|\tilde\bH_{L\theta\theta}^{-1}\big\|_{\infty} + n\left\|\tilde\bH_{L\theta\theta}^{-1}\bV^*\left\{n^{-1}\bI_{K^2+K} + (\bV^*)^\T\tilde\bH_{L\theta\theta}^{-1}\bV^*\right\}^{-1}(\bV^*)^\T\tilde\bH_{L\theta\theta}^{-1}\right\|_{\infty}\\\le &\,n\big\|\tilde\bH_{L\theta\theta}^{-1}\big\|_{\infty}^2 \big\|\bV^*\big\|_{\infty}\big\|(\bV^*)^\T\big\|_{\infty}(K^2+K)^{1/2}\left\|\left\{n^{-1}\bI_{K^2+K} + (\bV^*)^\T\tilde\bH_{L\theta\theta}^{-1}\bV^*\right\}^{-1}\right\| \\&\,+ \cOp(b_n^{-1})\\= &\,\cOp(b_n^{-1}).
\end{aligned}\label{eq_bound_tilde_HQ1_star_inv_inf}\end{equation}
Since $\bH_{Q1}^{\star}=\bH_{Q11}^{\star}$, the perturbation is only the Laplace-approximation error. Define $\tilde\bH_{Q1}^{\delta}=\tilde\bH_{Q11}^{\star}-\tilde\cH_{Q1}^*$, suppressing the arguments. By \eqref{eq_approx_h_Q1_integral_inf}, we know 
\begin{equation}
    \|\tilde\bH_{Q1}^{\delta}\|_{\infty} =\cOp(\varepsilon_{nq}/q)=o_p(b_n),\;\text{ and }\; \rho_{nq}:=\| (\tilde\bH_{Q11}^{\star})^{-1}\tilde\bH_{Q1}^{\delta}\|_{\infty} =\cOp\big(\varepsilon_{nq}/(qb_n)\big)=o_p(1).
    \label{eq_bound_tilde_HQ1_star_inv_2}
\end{equation}
Here the bound for $\rho_{nq}$ uses \eqref{eq_bound_tilde_HQ1_star_inv_inf} and Assumption~\ref{assump_scaling}. The Neumann series therefore converges in $\|\cdot\|_{\infty}$ with probability tending to one, which yields
\begin{equation*}
    \| (\tilde\cH_{Q1}^*)^{-1}\|_{\infty}
    \le \frac{\| (\tilde\bH_{Q11}^{\star})^{-1}\|_{\infty}}{1-\rho_{nq}}
    =\cOp(b_n^{-1}).
\end{equation*}
This proves the first part.

     Now we prove the bound for $\tilde\cH_{Q2}^*(\btheta)(\btheta - \btheta^*)$. Let $\btheta^s = \btheta^* + s(\btheta - \btheta^*)$. Since $\tilde\cH_{Q2}^*(\btheta) = n^{-1}\int_0^1\bH_{Q2}^{\star}(\btheta^s)ds$, we have
     \begin{align*}
         \big\|\tilde\cH_{Q2}^*(\btheta) (\btheta - \btheta^*)\big\|_{\infty}
         \le&\, \max_{s\in[0,1]}n^{-1}\big\|\bH^{\star}_{Q2}(\btheta^s)(\btheta - \btheta^*)\big\|_{\infty} \\
         \le&\,\max_{s\in[0,1]}\max_{j\in[q]}n^{-1}\left\|\left\{\bH_{L\theta u}^s\left(\bH_{Lu u}^s\right)^{-1}\bH_{Lu\theta}^s (\btheta - \btheta^*)\right\}_{\cK_j^+,}\right\|_{\infty} \\
         \le&\,\max_{s\in[0,1]}\max_{j\in[q]}\frac{C}{n}\big\|\{\bH_{L\theta u}^s\}_{\cK_j^+,}\big\|\big\|(\bH_{Lu u}^s)^{-1}\big\|\big\|\bH_{Lu\theta}^s (\btheta - \btheta^*)\big\| \\
         \le&\,\max_{s\in[0,1]}\max_{j\in[q]}\frac{C}{n}\big\|\{\bH_{L\theta u}^s\}_{\cK_j^+,}\big\|\big\|(\bH_{Lu u}^s)^{-1}\big\|\big\|\bH_{Lu\theta}^s\big\|\big\|\btheta - \btheta^*\big\| \\
         \le&\,\cOp\left(\frac{B_n^2}{b_n\sqrt q}\|\btheta-\btheta^*\|\right).
     \end{align*}
     For the last inequality, (d) gives the inverse bound. Bounded loadings and intercepts and $\max_s n^{-1}\sum_i\|\bu_i^{\star}(\btheta^s)\|^2=\cOp(1)$ imply $\max_{s,j}\|(\eta_{ij}^{\star}(\btheta^s)-\eta_{ij}^*)_{i\in[n]}\|=\cOp(\sqrt n)$. Taylor expansion of the scores and the uniform column bound in Lemma~\ref{lemma_concentration_l_mat} therefore give $\max_{s,j}\|\{\bJ_{L\theta u}^s\}_{\cK_j^+,}\|=\cOp(B_n\sqrt n)$. Combining this with (e)--(f) yields $\max_{s,j}\|\{\bH_{L\theta u}^s\}_{\cK_j^+,}\|=\cOp(B_n\sqrt n)$; bounding the spectral norm by the Frobenius norm then gives $\max_s\|\bH_{Lu\theta}^s\|=\cOp(B_n\sqrt{nq})$. Thus the stated bound for $\tilde\cH_{Q2}^*(\btheta)(\btheta-\btheta^*)$ is unchanged.

\section{Additional Simulation Results}\label{supp_Sec_simu}

\subsection{Additional Results from Section~\ref{sec_simu}}

\begin{table}[h]
\centering
\caption{Empirical coverage probabilities for model parameters under the linear model. Parentheses report Monte Carlo standard errors multiplied by \(1000\). The nominal coverage level is \(0.95\).}
\label{tab:coverage-gamma-beta-linear}
\small
\setlength{\tabcolsep}{3.5pt} 
\begin{tabular}{@{}l*{9}{c}@{}}
\toprule
 & \multicolumn{3}{c}{Gaussian} & \multicolumn{3}{c}{Uniform} & \multicolumn{3}{c}{Gaussian Mixture} \\
\cmidrule(lr){2-4}\cmidrule(lr){5-7}\cmidrule(lr){8-10}
$n$ & $q=150$ & $q=300$ & $q=600$ & $q=150$ & $q=300$ & $q=600$ & $q=150$ & $q=300$ & $q=600$ \\
\midrule
300 & 0.941 & 0.944 & 0.945 & 0.943 & 0.945 & 0.946 & 0.943 & 0.945 & 0.944 \\
& (1.45) & (0.67) & (0.46) & (1.32) & (0.65) & (0.47) & (1.36) & (0.66) & (0.47) \\
600 & 0.944 & 0.946 & 0.947 & 0.945 & 0.947 & 0.947 & 0.943 & 0.947 & 0.947 \\
& (1.41) & (0.63) & (0.46) & (1.31) & (0.65) & (0.44) & (1.30) & (0.64) & (0.46) \\
1000 & 0.943 & 0.947 & 0.948 & 0.944 & 0.948 & 0.949 & 0.943 & 0.946 & 0.948 \\
& (1.35) & (0.67) & (0.46) & (1.32) & (0.65) & (0.44) & (1.34) & (0.68) & (0.45) \\
\bottomrule
\end{tabular}
\end{table}

\begin{table}[h]
\centering
\caption{Empirical coverage probabilities for latent variables $\bU_i$ under the linear model. Parentheses report Monte Carlo standard errors multiplied by \(1000\). The nominal coverage level is \(0.95\).}
\label{tab:coverage-u-linear}
\small
\setlength{\tabcolsep}{3.5pt} 
\begin{tabular}{@{}l*{9}{c}@{}}
\toprule
 & \multicolumn{3}{c}{Gaussian} & \multicolumn{3}{c}{Uniform} & \multicolumn{3}{c}{Gaussian Mixture} \\
\cmidrule(lr){2-4}\cmidrule(lr){5-7}\cmidrule(lr){8-10}
$n$ & $q=150$ & $q=300$ & $q=600$ & $q=150$ & $q=300$ & $q=600$ & $q=150$ & $q=300$ & $q=600$ \\
\midrule
300 & 0.940 & 0.945 & 0.947 & 0.941 & 0.943 & 0.946 & 0.940 & 0.946 & 0.946 \\
& (1.19) & (0.74) & (0.80) & (1.09) & (0.88) & (0.84) & (1.01) & (0.76) & (0.82) \\
600 & 0.940 & 0.945 & 0.947 & 0.940 & 0.945 & 0.947 & 0.941 & 0.945 & 0.947 \\
& (0.89) & (0.56) & (0.51) & (0.79) & (0.56) & (0.52) & (0.86) & (0.48) & (0.51) \\
1000 & 0.940 & 0.945 & 0.947 & 0.941 & 0.946 & 0.948 & 0.940 & 0.945 & 0.947 \\
& (0.61) & (0.45) & (0.45) & (0.70) & (0.45) & (0.42) & (0.64) & (0.43) & (0.38) \\
\bottomrule
\end{tabular}
\end{table}

We report finite-sample coverage under the linear link (Tables~\ref{tab:coverage-gamma-beta-linear} and~\ref{tab:coverage-u-linear}) and the Poisson link (Tables~\ref{tab:coverage-gamma-beta-poisson} and~\ref{tab:coverage-u-poisson}) across different problem sizes and latent distributions. The results follow the same pattern as those for the logistic link in Section~\ref{sec_simu} where empirical coverage is close to the nominal 95\% level. Coverage is most accurate under the linear link, which represents an easier setting because $\kappa_n\asymp 1$, whereas $\kappa_n\asymp n^{o(1)}$ under the Bernoulli and Poisson models. Overall, the coverage results under both nonlinear links remain satisfactory across the settings considered.

\begin{table}[h]
\centering
\caption{Empirical coverage probabilities for model parameters under the Poisson model. Parentheses report Monte Carlo standard errors multiplied by \(1000\). The nominal coverage level is \(0.95\).}
\label{tab:coverage-gamma-beta-poisson}
\small
\setlength{\tabcolsep}{3.5pt} 
\begin{tabular}{@{}l*{9}{c}@{}}
\toprule
 & \multicolumn{3}{c}{Gaussian} & \multicolumn{3}{c}{Uniform} & \multicolumn{3}{c}{Gaussian Mixture} \\
\cmidrule(lr){2-4}\cmidrule(lr){5-7}\cmidrule(lr){8-10}
$n$ & $q=150$ & $q=300$ & $q=600$ & $q=150$ & $q=300$ & $q=600$ & $q=150$ & $q=300$ & $q=600$ \\
\midrule
300 & 0.918 & 0.926 & 0.930 & 0.929 & 0.938 & 0.941 & 0.922 & 0.929 & 0.933 \\
& (1.87) & (0.82) & (0.56) & (1.55) & (0.76) & (0.49) & (1.69) & (0.81) & (0.54) \\
600 & 0.922 & 0.931 & 0.935 & 0.936 & 0.941 & 0.944 & 0.927 & 0.934 & 0.937 \\
& (1.83) & (0.83) & (0.56) & (1.52) & (0.70) & (0.48) & (1.78) & (0.82) & (0.53) \\
1000 & 0.923 & 0.936 & 0.937 & 0.932 & 0.941 & 0.944 & 0.929 & 0.936 & 0.939 \\
& (1.70) & (0.86) & (0.57) & (1.72) & (0.80) & (0.49) & (1.76) & (0.87) & (0.53) \\
\bottomrule
\end{tabular}
\end{table}

\begin{table}[h]
\centering
\caption{Empirical coverage probabilities for latent variables $\bU_i$ under the Poisson model. Parentheses report Monte Carlo standard errors multiplied by \(1000\). The nominal coverage level is \(0.95\).}
\label{tab:coverage-u-poisson}
\small
\setlength{\tabcolsep}{3.5pt} 
\begin{tabular}{@{}l*{9}{c}@{}}
\toprule
 & \multicolumn{3}{c}{Gaussian} & \multicolumn{3}{c}{Uniform} & \multicolumn{3}{c}{Gaussian Mixture} \\
\cmidrule(lr){2-4}\cmidrule(lr){5-7}\cmidrule(lr){8-10}
$n$ & $q=150$ & $q=300$ & $q=600$ & $q=150$ & $q=300$ & $q=600$ & $q=150$ & $q=300$ & $q=600$ \\
\midrule
300 & 0.924 & 0.933 & 0.935 & 0.929 & 0.936 & 0.940 & 0.929 & 0.934 & 0.935 \\
& (1.33) & (0.87) & (0.96) & (1.43) & (0.94) & (0.73) & (1.32) & (0.94) & (0.89) \\
600 & 0.931 & 0.935 & 0.940 & 0.933 & 0.941 & 0.943 & 0.930 & 0.938 & 0.940 \\
& (0.88) & (0.63) & (0.56) & (0.91) & (0.59) & (0.50) & (1.01) & (0.59) & (0.62) \\
1000 & 0.935 & 0.938 & 0.941 & 0.933 & 0.940 & 0.945 & 0.936 & 0.940 & 0.944 \\
& (0.70) & (0.48) & (0.48) & (0.62) & (0.51) & (0.43) & (0.58) & (0.54) & (0.52) \\
\bottomrule
\end{tabular}
\end{table}

\newpage
\subsection{Simulation Results under Practical Identifiability Condition}\label{supp_simu_2}
We assess the finite-sample accuracy of the frequentist and Bayesian asymptotic approximations for the latent variables $\bU_i$. Under the framework of Section~\ref{subsec_latent_recover}, the frequentist and Bayesian asymptotic covariance matrices are asymptotically equivalent, both concurring with the inverse conditional Fisher information given $\{\bU_i\}_{i\in[n]}$. Under the practical identifiability conditions studied in Section~\ref{sec_res_pract_id}, however, the two approaches generally yield different covariance expressions. To examine this distinction empirically, we conduct additional simulations under the identifiability constraints introduced in that section. Specifically, we consider the following four patterns:
\begin{equation*}
    \bGamma^{(a)} = \begingroup
    \renewcommand{\arraystretch}{0.45}\begin{pmatrix}
        * & 0 & 0\\ * & * & 0\\ * & * & *\\\vdots & \vdots & \vdots
    \end{pmatrix},\endgroup \quad \bGamma^{(b)}=\begingroup
    \renewcommand{\arraystretch}{0.45}\begin{pmatrix}
        * & 0 & *\\ * & * & 0 \\ * & * & 0 \\\vdots & \vdots & \vdots
    \end{pmatrix}\endgroup\qquad
    \bGamma^{(c)} = \begingroup
    \renewcommand{\arraystretch}{0.45}\begin{pmatrix}
        * & * & 0\\ * & 0 & 0\\ 0 & 0 & *\\ 0 & * & *\\\vdots & \vdots & \vdots
    \end{pmatrix},\endgroup \quad \bGamma^{(d)}=\begingroup
    \renewcommand{\arraystretch}{0.45}\begin{pmatrix}
        * & * & 0\\ * & * & 0\\ * & 0 & *\\ * & 0 & *\\ 0 & * & *\\ 0 & * & * \\\vdots & \vdots & \vdots
    \end{pmatrix}\endgroup
\end{equation*}


\noindent Here `$*$' denotes a free entry, `$0$' denotes an entry fixed at zero, and the vertical ellipses indicate additional unrestricted rows. The patterns $\bGamma^{(a)}$ and $\bGamma^{(b)}$ satisfy the zero-restriction structure in Condition~\ref{condition1}$(iii)$, whereas $\bGamma^{(c)}$ and $\bGamma^{(d)}$ satisfy that in Condition~\ref{condition2}$(iii)$, provided the corresponding rank requirements hold. We generate the model parameters as in Section~\ref{sec_simu}, setting the designated entries of $\bGamma$ to zero according to each pattern. We examine the performance under Bernoulli case and Gaussian distribution as example. In settings~$(a)$ and~$(b)$, the latent variables are simulated as in Section~\ref{sec_simu}. In settings~$(c)$ and~$(d)$, we simulate latent variables from $\cN(\zero_3,\bSigma_u)$, where $\bSigma_u\in\RR^{K\times K}$ has its $(l,h)$th entry being $0.3^{|l-h|}$, and then normalize them to satisfy Condition~\ref{condition2}.  Other aspects of the simulation design remain the same as in Section~\ref{sec_simu}. For frequentist inference, we compute the empirical coverage of nominal 95\% confidence intervals for $\bU_i$, using either the sandwich covariance $\bPsi_i^*$ or the corresponding covariance $\bPsi_i^{(1)}$ or $\bPsi_i^{(2)}$ derived under the relevant identifiability condition. To assess posterior approximation, we report the KL divergence between the plug-in posterior, evaluated by adaptive Gauss--Hermite quadrature, and Gaussian approximations constructed using each of these covariance matrices. The results are reported in Tables~\ref{tab:varcomp-ic1}--\ref{tab:varcomp-ic4}.

The results clearly distinguish the two notions of uncertainty. Across all four settings, confidence intervals for point estimate based on $\bPsi_i^{(1)}$ or $\bPsi_i^{(2)}$ attain coverage generally close to the nominal 95\% level, whereas those based on $\bPsi_i^*$ exhibit substantial undercoverage. Conversely, $\bPsi_i^*$ provides an accurate Gaussian approximation to the plug-in posterior, with the KL divergence decreasing as $q$ increases, while using the frequentist covariance produces considerably larger discrepancies. This agrees with our theory that the uncertainty induced by practical identification affects the frequentist distribution of the point estimator but not the first-order posterior approximation.

\begin{table}[htbp]
\centering
\caption{Comparison between two choices of covariance matrices under setting $(a)$. Boldface marks the theoretically appropriate covariance for each inferential target. }\label{tab:varcomp-ic1}
\small
\begin{tabular}{llccccccc}
\toprule
 & & \multicolumn{3}{c}{Frequentist 95\% CI coverage} && \multicolumn{3}{c}{Posterior approximation: KL divergence} \\
\cmidrule(lr){3-5}\cmidrule(lr){7-9}
$n$ & & $q=150$ & $q=300$ & $q=600$ && $q=150$ & $q=300$ & $q=600$ \\
\midrule
300 & $\bPsi_i^*$ & 0.855 & 0.818 & 0.713 && \textbf{0.051} & \textbf{0.024} & \textbf{0.012} \\
 & $\bPsi_i^{(1)}$ & \textbf{0.974} & \textbf{0.969} & \textbf{0.943} && 1.231 & 0.854 & 1.089 \\
\midrule
600 & $\bPsi_i^*$ & 0.901 & 0.890 & 0.761 && \textbf{0.045} & \textbf{0.024} & \textbf{0.011} \\
 & $\bPsi_i^{(1)}$ & \textbf{0.947} & \textbf{0.954} & \textbf{0.928} && 0.251 & 0.279 & 0.694 \\
\midrule
1000 & $\bPsi_i^*$ & 0.922 & 0.887 & 0.889 && \textbf{0.041} & \textbf{0.022} & \textbf{0.011} \\
 & $\bPsi_i^{(1)}$ & \textbf{0.945} & \textbf{0.985} & \textbf{0.970} && 0.080 & 1.770 & 0.469 \\
\bottomrule
\end{tabular}
\end{table}

\begin{table}[htbp]
\centering
\caption{Comparison between two choices of covariance matrices under setting $(b)$. Boldface marks the theoretically appropriate covariance for each inferential target. }\label{tab:varcomp-ic2}
\small
\begin{tabular}{llccccccc}
\toprule
 & & \multicolumn{3}{c}{Frequentist 95\% CI coverage} & &\multicolumn{3}{c}{Posterior approximation: KL divergence} \\
\cmidrule(lr){3-5}\cmidrule(lr){7-9}
$n$ & & $q=150$ & $q=300$ & $q=600$ && $q=150$ & $q=300$ & $q=600$ \\
\midrule
300 & $\bPsi_i^*$ & 0.715 & 0.752 & 0.605 && \textbf{0.045} & \textbf{0.024} & \textbf{0.012} \\
 & $\bPsi_i^{(1)}$ & \textbf{0.939} & \textbf{0.937} & \textbf{0.955} && 1.062 & 0.774 & 1.604 \\
\addlinespace
600 & $\bPsi_i^*$ & 0.880 & 0.890 & 0.879 && \textbf{0.049} & \textbf{0.025} & \textbf{0.011} \\
 & $\bPsi_i^{(1)}$ & \textbf{0.933} & \textbf{0.946} & \textbf{0.954} && 0.160 & 0.153 & 0.212 \\
\addlinespace
1000 & $\bPsi_i^*$ & 0.903 & 0.891 & 0.642 && \textbf{0.043} & \textbf{0.023} & \textbf{0.012} \\
 & $\bPsi_i^{(1)}$ & \textbf{0.937} & \textbf{0.943} & \textbf{0.910} && 0.089 & 0.155 & 1.220 \\
\bottomrule
\end{tabular}
\end{table}

\begin{table}[htbp]
\centering
\caption{Comparison between two choices of covariance matrices under setting $(c)$. Boldface marks the theoretically appropriate covariance for each inferential target.}\label{tab:varcomp-ic3}
\small
\begin{tabular}{llccccccc}
\toprule
 & & \multicolumn{3}{c}{Frequentist 95\% CI coverage} && \multicolumn{3}{c}{Posterior approximation: KL divergence} \\
\cmidrule(lr){3-5}\cmidrule(lr){7-9}
$n$ & & $q=150$ & $q=300$ & $q=600$ && $q=150$ & $q=300$ & $q=600$ \\
\midrule
300 & $\bPsi_i^*$ & 0.856 & 0.860 & 0.757 && \textbf{0.052} & \textbf{0.031} & \textbf{0.016} \\
 & $\bPsi^{(2)}$ & \textbf{0.933} & \textbf{0.940} & \textbf{0.921} && 0.230 & 0.241 & 0.763 \\
\midrule
600 & $\bPsi_i^*$ & 0.882 & 0.874 & 0.835 && \textbf{0.071} & \textbf{0.027} & \textbf{0.014} \\
 & $\bPsi^{(2)}$ & \textbf{0.935} & \textbf{0.937} & \textbf{0.937} && 0.150 & 0.160 & 0.278 \\
\midrule
1000 & $\bPsi_i^*$ & 0.900 & 0.895 & 0.878 && \textbf{0.077} & \textbf{0.026} & \textbf{0.012} \\
 & $\bPsi^{(2)}$ & \textbf{0.932} & \textbf{0.937} & \textbf{0.932} && 0.106 & 0.083 & 0.182 \\
\bottomrule
\end{tabular}
\end{table}

\begin{table}[htbp]
\centering
\caption{Comparison between two choices of covariance matrices under setting $(d)$. Boldface marks the theoretically appropriate covariance for each inferential target.}\label{tab:varcomp-ic4}
\small
\begin{tabular}{llccccccc}
\toprule
 & & \multicolumn{3}{c}{Frequentist 95\% CI coverage} && \multicolumn{3}{c}{Posterior approximation: KL divergence} \\
\cmidrule(lr){3-5}\cmidrule(lr){7-9}
$n$ & & $q=150$ & $q=300$ & $q=600$ && $q=150$ & $q=300$ & $q=600$ \\
\midrule
300 & $\bPsi_i^*$ & 0.850 & 0.800 & 0.771 && \textbf{0.060} & \textbf{0.033} & \textbf{0.015} \\
 & $\bPsi^{(2)}$ & \textbf{0.922} & \textbf{0.927} & \textbf{0.935} && 0.215 & 0.418 & 0.772 \\
\midrule
600 & $\bPsi_i^*$ & 0.909 & 0.879 & 0.827 && \textbf{0.046} & \textbf{0.024} & \textbf{0.013} \\
 & $\bPsi^{(2)}$ & \textbf{0.937} & \textbf{0.930} & \textbf{0.932} && 0.069 & 0.102 & 0.289 \\
\midrule
1000 & $\bPsi_i^*$ & 0.910 & 0.896 & 0.867 && \textbf{0.072} & \textbf{0.024} & \textbf{0.013} \\
 & $\bPsi^{(2)}$ & \textbf{0.937} & \textbf{0.929} & \textbf{0.936} && 0.092 & 0.074 & 0.194 \\
\bottomrule
\end{tabular}
\end{table}

\end{document}